\documentclass[leqno]{amsart}
\usepackage{url}
\usepackage{amsmath}
\usepackage{mathtools}

\DeclarePairedDelimiter\floor{\lfloor}{\rfloor}
\usepackage{amsfonts}
\usepackage{amssymb}
\usepackage{xcolor}
\usepackage{graphicx}
\usepackage{float}
\usepackage{tabularx}
\usepackage{multirow}

\newcommand{\norm}[1]{\left\lVert#1\right\rVert}
\newcommand{\abs}[1]{\left |#1\right |}
\usepackage{mathtools}
\usepackage{algorithmic, algorithm} 
\theoremstyle{definition}
\newtheorem{theorem}{Theorem}
\newtheorem{corollary}{Corollary}

\newtheorem{definition}{Definition}
\newtheorem{lemma}{Lemma}
\newtheorem{proposition}{Proposition}
\newtheorem{remark}{Remark}

\usepackage[export]{adjustbox}
\usepackage{commath}
\usepackage{hyperref}
\hypersetup{
	colorlinks,
	linkcolor=black,
	anchorcolor=black,
	citecolor=black
}
\newcommand*\diff{\mathop{}\!\mathrm{d}}
\newcommand*\conj[1]{\overline{#1}}

\DeclareFontFamily{U}{mathx}{}
\DeclareFontShape{U}{mathx}{m}{n}{ <-> mathx10 }{}
\DeclareSymbolFont{mathx}{U}{mathx}{m}{n}
\DeclareFontSubstitution{U}{mathx}{m}{n}

\DeclareMathAccent{\widecheck}{0}{mathx}{"71}

\begin{document}

\title[Synchrosqueezed chirplet transforms]{Disentangling modes with crossover instantaneous frequencies by synchrosqueezed chirplet transforms, from theory to application}  

\author{Ziyu Chen}
\address{Department of Mathematics, Duke University, Durham, NC, USA}
\email{ziyu@math.duke.edu}

\author{Hau-Tieng Wu}
\address{Department of Mathematics and Department of Statistical Science, Duke University, Durham, NC, USA. Mathematics Division, National Center for Theoretical Sciences, Taipei, Taiwan}
\email{hauwu@math.duke.edu}

\maketitle

\begin{abstract}
Analysis of signals with oscillatory modes with crossover instantaneous frequencies is a challenging problem in time series analysis. One way to handle this problem is lifting the 2-dimensional time-frequency representation to a 3-dimensional representation, called time-frequency-chirp rate (TFC) representation, by adding one extra chirp rate parameter so that crossover frequencies are disentangled in higher dimension. The chirplet transform is an algorithm for this lifting idea, which leads to a TFC representation. However, in practice, we found that it has a strong ``blurring'' effect in the chirp rate axis, which limits its application in real-world data. Moreover, to our knowledge, we have limited mathematical understanding of the chirplet transform in the literature. Motivated by the need for the real-world data analysis, in this paper, we propose the synchrosqueezed chirplet transform (SCT) that enhances the TFC representation given by the chirplet transform. The resulting concentrated TFC representation has high contrast so that one can better distinguish different modes with crossover instantaneous frequencies. 
The basic idea is to use the phase information in the chirplet transform to determine a reassignment rule that sharpens the TFC representation determined by the chirplet transform. 
We also analyze the chirplet transform and provide theoretical guarantees of SCT.
\end{abstract}

\section{Introduction}\label{Sect:Introduction}

Time series (or signals) are ubiquitous in almost all scientific fields. Often, the observed/recorded signal is composed of multiple components. 
Due to the nonstationary dynamics of our natural system, if the component is oscillatory, it usually oscillates with time-varying amplitudes, frequencies and even oscillatory patterns. On the other hand, if the component is stochastic, it could be a nonstationary random process with time-varying statistical properties. These time-varying quantities encode rich information about the dynamics of the underlying system of interest. By properly extracting and parsing these time-varying quantities, we could peek into nature and decode the underlying dynamics. In this paper, we focus on analyzing signals with oscillatory components and extracting the dynamics.

To properly extract and parse the time-varying quantities of each oscillatory component from the input signal is however not easy. For example, due to the nonstationarity, the traditional Fourier transform cannot properly capture such information, since the Fourier transform is a global operator. An intuitive and successful solution to handle the nonstationarity is ``divide-and-conquer''. By dividing the signal into pieces, and assuming that each piece is stationary, we could apply Fourier transform over each piece to capture the nonstationarity. 
This intuitive idea has led to a rich field called time-frequency (TF) analysis, where the ``spectrum'' at each moment is evaluated so that the time-varying frequency and amplitude could be evaluated. It includes various algorithms, ranging from linear-type, polynomial-type to nonlinear-type algorithms. For example, the short-time Fourier transform (STFT), the continuous wavelet transform (CWT) and the S transform are linear-type algorithms. The Wigner-Ville, Cohen's class and affine class are bilinear-type algorithms \cite{Flandrin:1999}. The polynomial Wigner-Ville distributions are polynomial-type algorithms \cite{boashash1994polynomial, boashash1998polynomial, barkat1999instantaneous}. The reassignment method (RM) \cite{auger1995improving}, the synchrosqueezing transform (SST) \cite{daubechies2011synchrosqueezed} and its generalizations \cite{oberlin2017second}, the Blaschke decomposition \cite{Nahon:2000Thesis}, and the empirical mode decomposition (EMD) \cite{huang1998empirical} and its variations are nonlinear-type algorithms. We refer readers with interest to a recent review of the topic \cite{wu2020current}. In short, these algorithms convert an input signal into a TF representation and/or oscillatory components. From the output, we could more efficiently visualize or extract different time-varying properties of the input signal. 
While the linear-type algorithms have been widely applied, one well-known property shared by these algorithms is the uncertainty principle. This uncertainty principle leads to a blurred TF representation so that the information of interest, like the ridges \cite{carmona1997characterization}, is of low contrast. This  may cause troubles in some applications. This problem can be handled by nonlinear-type approaches; for example, RM \cite{auger1995improving} and SST \cite{daubechies2011synchrosqueezed} were invented for the sake of sharpening the TF representation so that its contrast is enhanced. Even more, signals with chirp components can be further enhanced by higher order SST \cite{Oberlin_Meignen_Perrier:2015}.

While we have advanced the analysis tools from various fronts in the past 20 years, there are still several challenges. In this paper, we focus on the specific crossover frequency challenge. Recall that most of the above algorithms are based on the assumption that different oscillatory components have ``well-separated'' time-varying frequencies in the TF plane. Mathematically, this well-separation condition is defined in the following way. Suppose a signal is a superposition of multiple \textit{intrinsic-mode-type} (IMT) functions \cite{daubechies2011synchrosqueezed}:
\[
f(t) = \sum_{k=1}^K f_k(t) = \sum_{k=1}^K A_k(t)e^{2\pi i \phi_k(t)},
\]
where $A_k(t)>0$ and $\phi_k'(t)>0$ model the time-varying amplitude (or called amplitude modulation (AM)) and the time-varying frequency (or called instantaneous frequency (IF)) of the $k$-th IMT component $f_k(t)$, with some regularity assumptions. Suppose for any $k\neq j$, we have $\abs{\phi_k'(t)-\phi_j'(t)}\geq \Delta$ for some constant $\Delta>0$ for all $t$, the signal $f$ satisfies the well-separation condition. This is the minimal assumption we need for most TF analysis tools mentioned above. When this sell-separation condition is broken, for example, when the time-varying frequencies of two oscillatory components are too close, or even crossover, the algorithm fails. 
Mathematically, when $\phi'_k(t)=\phi'_j(t)$ for some $t$, we say that the $k$-th and $j$-th components have {\em crossover IFs}.

See Figure \ref{fig:intro0} for a synthetic signal as an example, where $f(x) = f_1(x) + f_2(x)$, where $f_1(x) = e^{2\pi i\cdot 4x^2}$, and $f_2(x) = e^{2\pi i [-\pi x^2 + (24+6\pi)x]}$, $x\in [1,5]$, with such a crossover frequency phenomenon. 
The signal is sampled at 100Hz. Note that in this simple case, the chirp rates are not time-varying, and their IFs have a crossing point at $t_0 = 3$ and $\xi_0 = 24$. Figure \ref{fig:intro0} gives the time-frequency representations determined by STFT, RM, SST, and  2nd-order SST with the window $g_0(x) = e^{-\pi x^2}$,
Clearly, the spectrogram (the magnitude of the STFT of the signal) and the TF representation determined by SST are blurred as expected, and the TF representation is sharpened after applying RM and the 2nd-order SST. However, during the period that two IFs crossover, all approaches encounter troubles, and we need a new approach to handle it.

Recently, some methods have been proposed to handle signals with crossover frequencies. In \cite{li2020if}, an IF estimator based on the linear least square fitting and the Viterbi algorithm was introduced. A quasi maximum likelihood - random samples consensus algorithm was proposed for the IF estimation of overlapping signals in the TF plane \cite{djurovic2018qml}. The chirplet transform (CT) was introduced in \cite{mann1992time,mann1992chirplets,mann1995chirplet}, which gives rise to a three dimensional parameter space, i.e., time, frequency (scale) and chirp rate that includes the TF plane (of the STFT) or the time-scale plane (of the CWT) as a subspace. Its discretization was later studied in \cite{xia2000discrete} with several generalizations \cite{bossmann2015asymmetric,tu2017instantaneous} and applications \cite{wang2003manoeuvring,cui2005time,dugnol2008chirplet,hartono2019gear,ghosh2020automated}, which is a non-exhaustive list. Besides its common application to estimate chirp rates, with the extra chirp rate dimension, it is possible to apply it to ``see through'' the crossover IFs.  
Such potential was confirmed in two recent papers. In \cite{zhu2020frequency}, CT was employed to estimate IFs and chirp rates of a signal that can separate crossing IFs in the time-frequency-chirp rate (TFC) space. The CT was also employed to retrieve modes with crossover IFs from a signal with mathematical theorems in \cite{chui2021time, li2021chirplet}. 
We shall mention two relevant papers that also depend on chirplet. In \cite{candes2008detecting}, the authors propose a multiscale chirplet method to ``detect waves'', where the chirplet refers to atoms that oscillate with linearly varying frequency (fixed chirp rate) and exist in short (dyadic) time intervals. In \cite{candes2002multiscale}, the authors propose a tight frame of ``chirplet'' as in \cite{candes2008detecting} for the reconstruction of chirps; however, they do not deal with the problem of decomposing a multi-component signal or of estimating the instantaneous frequency and chirp rate.

While CT has been proposed to analyze signals with crossover IFs, there are several interesting open problems. First, the TFC representation determined by CT with a window $g$ is also impacted by the uncertainty principle; that is, the TFC representation is blurred. See the synthetic example in Figure \ref{fig:intro0}. It seems that the ``blurring'' of the TFC representation is not improved compared with the spectrogram, particularly after projecting it onto the TF domain. In particular, while at time $t_0=3$ and frequency $\xi_0=24$, the TFC representation on the chirp rate axis gives a separation of two components. However, the magnitude at the chirp rate $0$ is nontrivial, which blurs the available information.
It is thus beneficial if we could enhance the TFC representation for some specific applications. This blurring effect also raised our curiosity about the theoretical properties of CT. For example, if the input signal is $L^1$, is the TFC representation in the chirp rate axis continuous and decays to $0$, and what is the decay rate? If it decays fast, does CT map a $L^2$ function to a $L^2$ function? However, to our knowledge, except the closely related results of the oscillatory integration \cite{stein1993harmonic}, the theoretical property of CT is less studied in the literature.
In this paper, we seek to answer these questions, at least partially, and provide a practical algorithm to separate components with crossover IFs from a multicomponent signal. We propose to combine the ``squeezing'' idea with CT, and  synchrosqueezed chirplet transform (SCT) that gives a sharper TFC representation of a signal, even if the signal has crossover IFs, from which one can separate different components. In addition to providing several theoretical properties of CT and a theoretical guarantee of SCT, we provide a series of numerical experiments with real data illustration. We shall mention that a seemingly similar squeezing idea has been shown in \cite{zhu2019multiple}. However, despite the nomination ``squeeze'', the squeezing part was only applied to the frequency axis that is based on the chirp rate information and only the TF representation was considered, and only limited theoretical analysis of the squeezing step is provided.
Last but not the least, the Matlab code used in this paper is made public for the reproducibility purpose.

\begin{figure}[!htbp]
	\centering
	\includegraphics[width=0.325\textwidth]{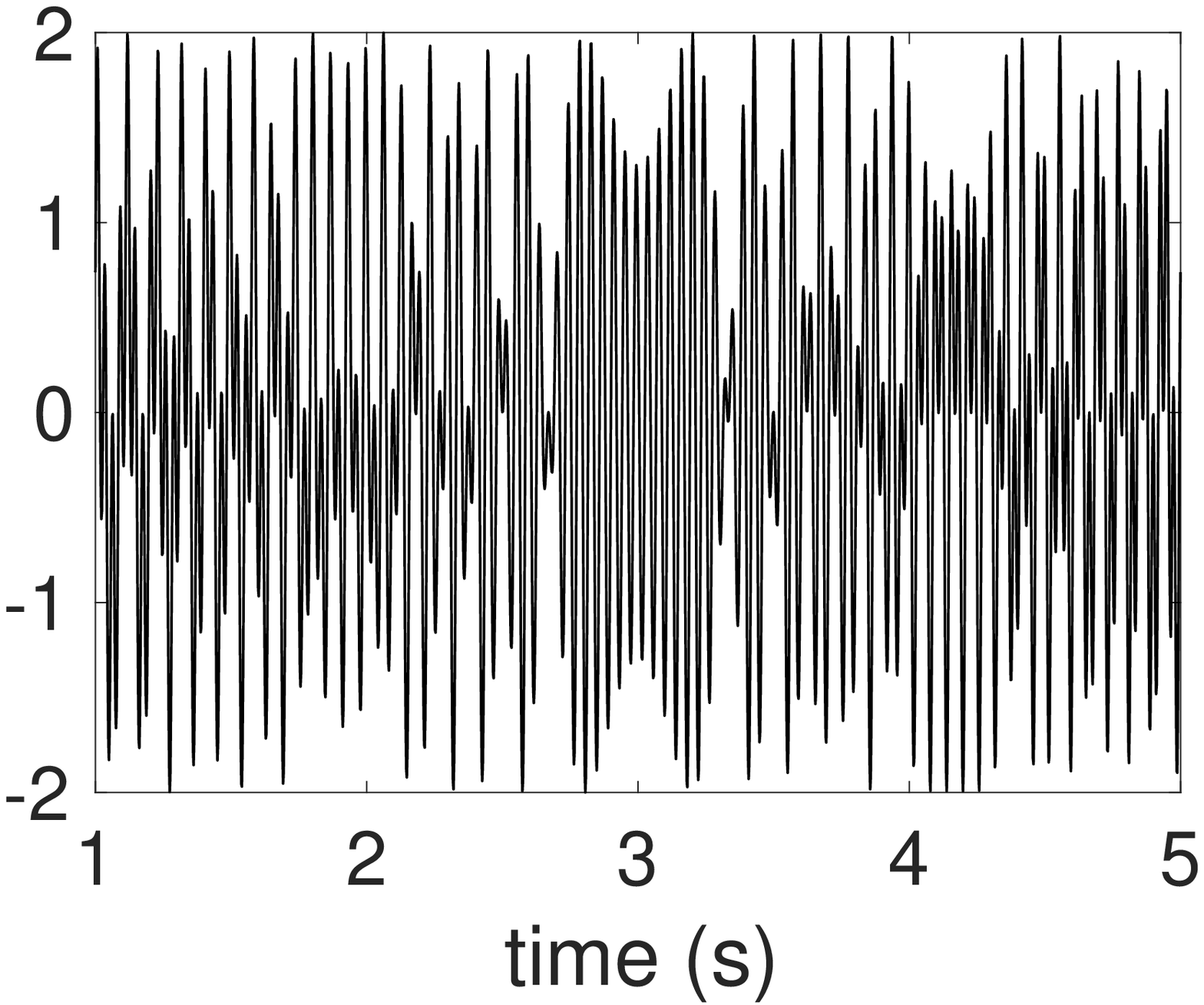}
	\includegraphics[width=0.325\textwidth]{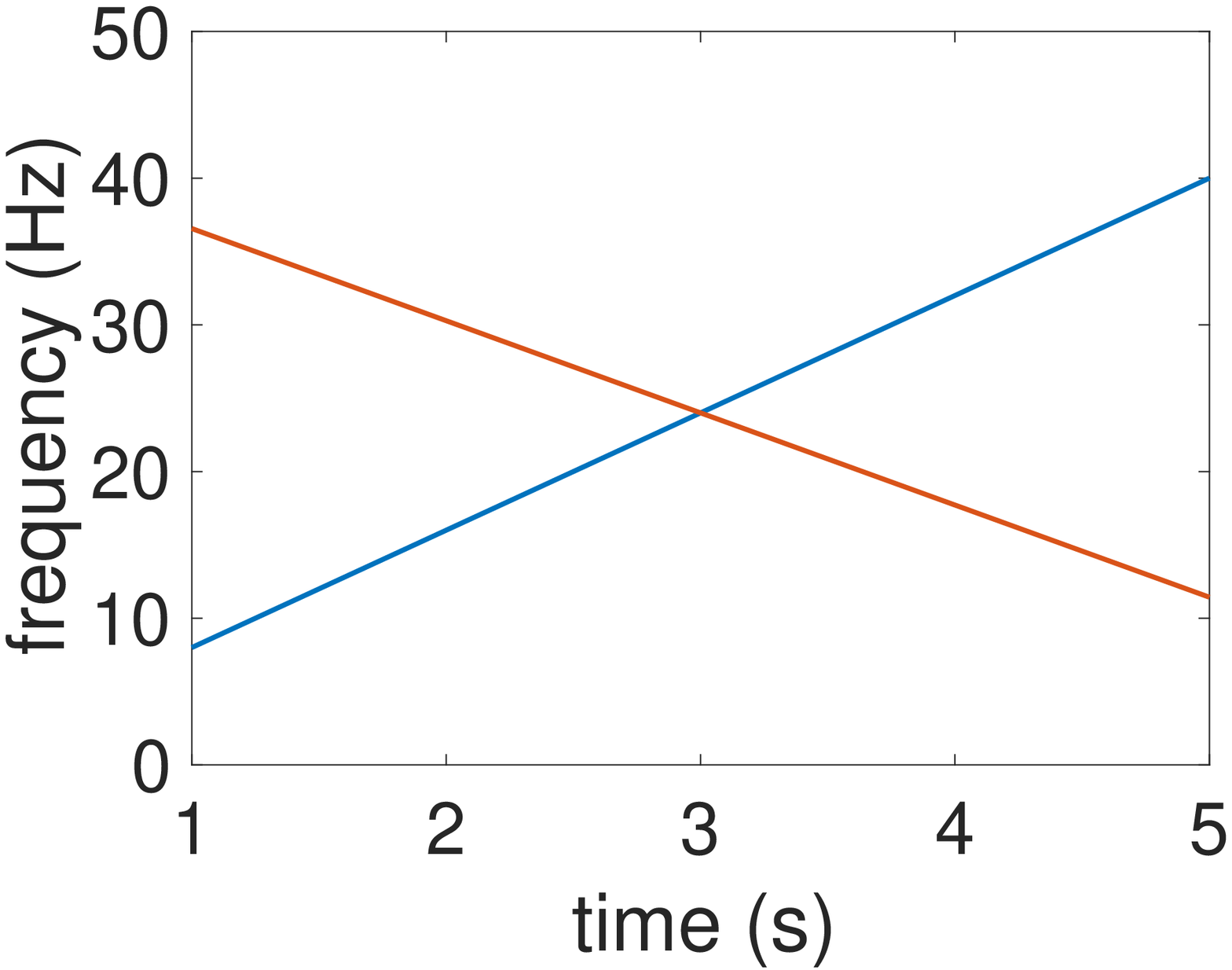}
	\includegraphics[width=0.325\textwidth]{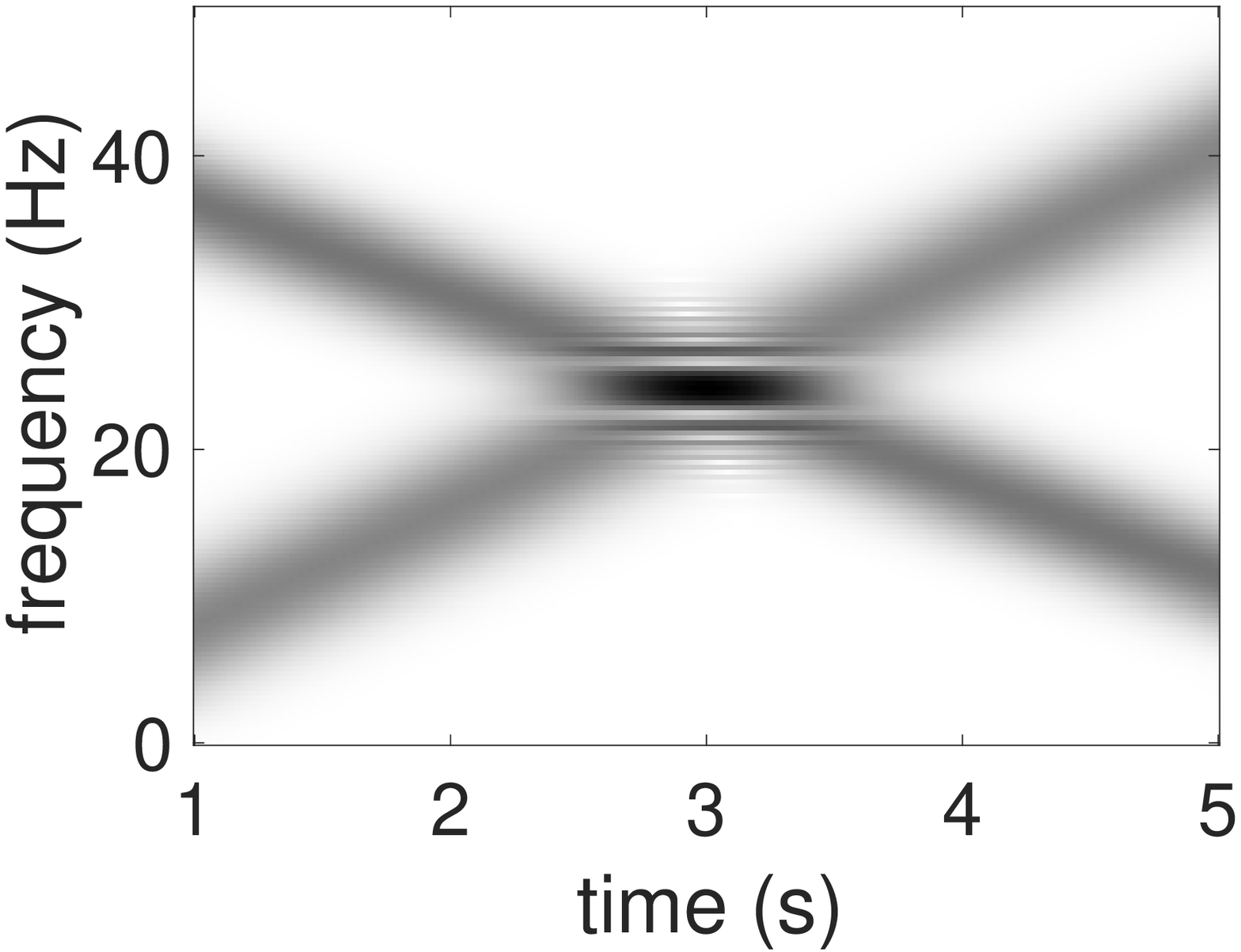}
	\includegraphics[width=0.325\textwidth]{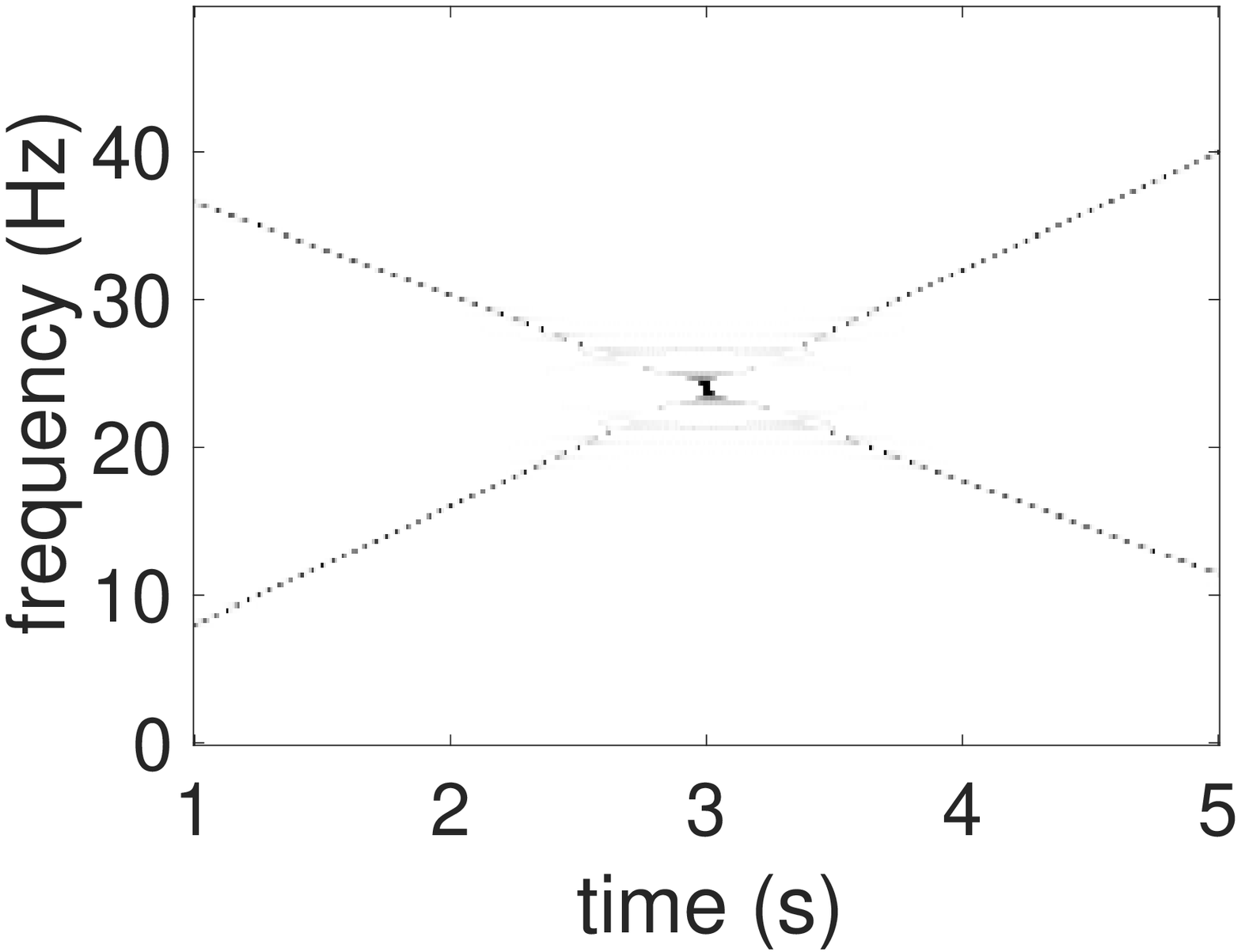}
	\includegraphics[width=0.325\textwidth]{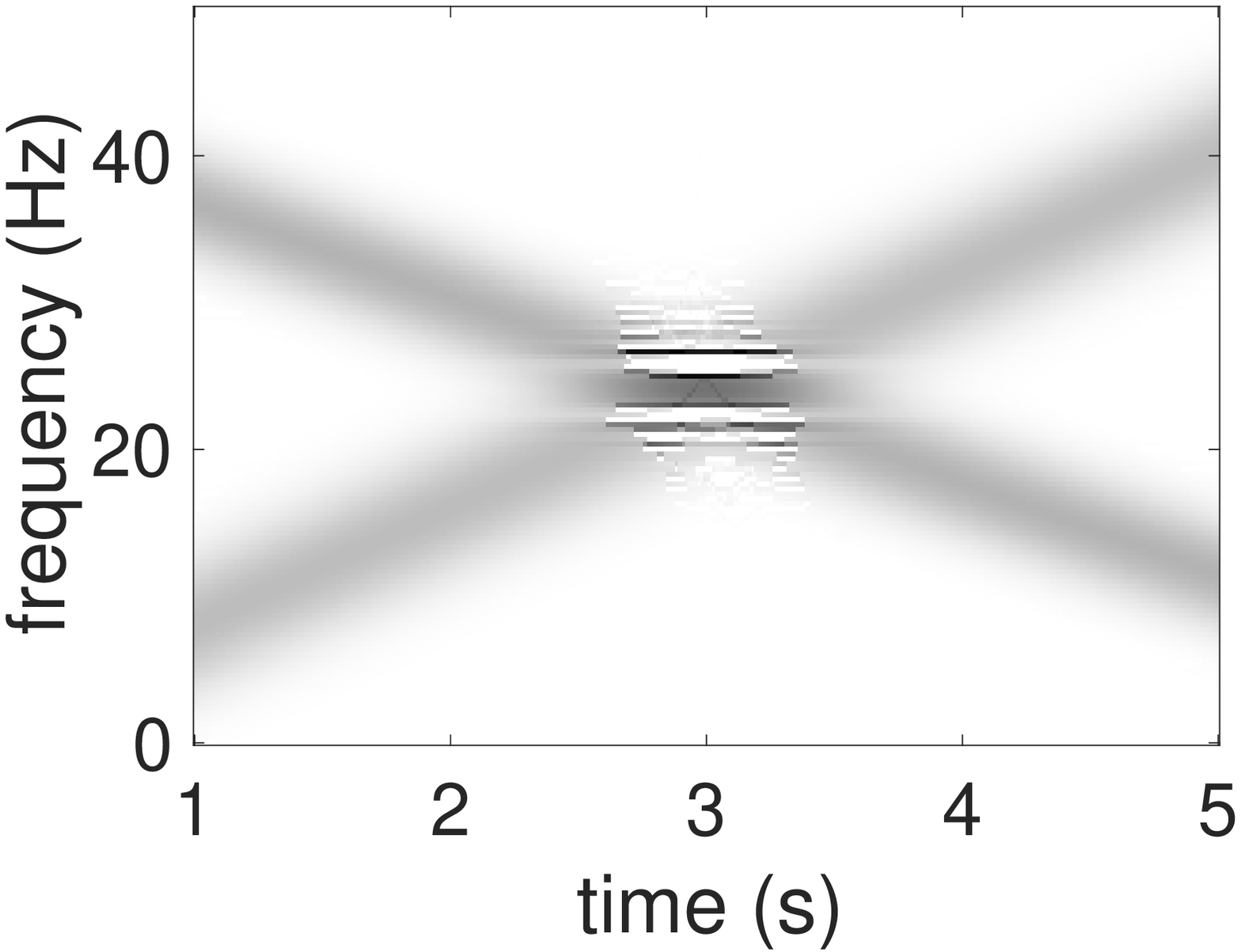}
	\includegraphics[width=0.325\textwidth]{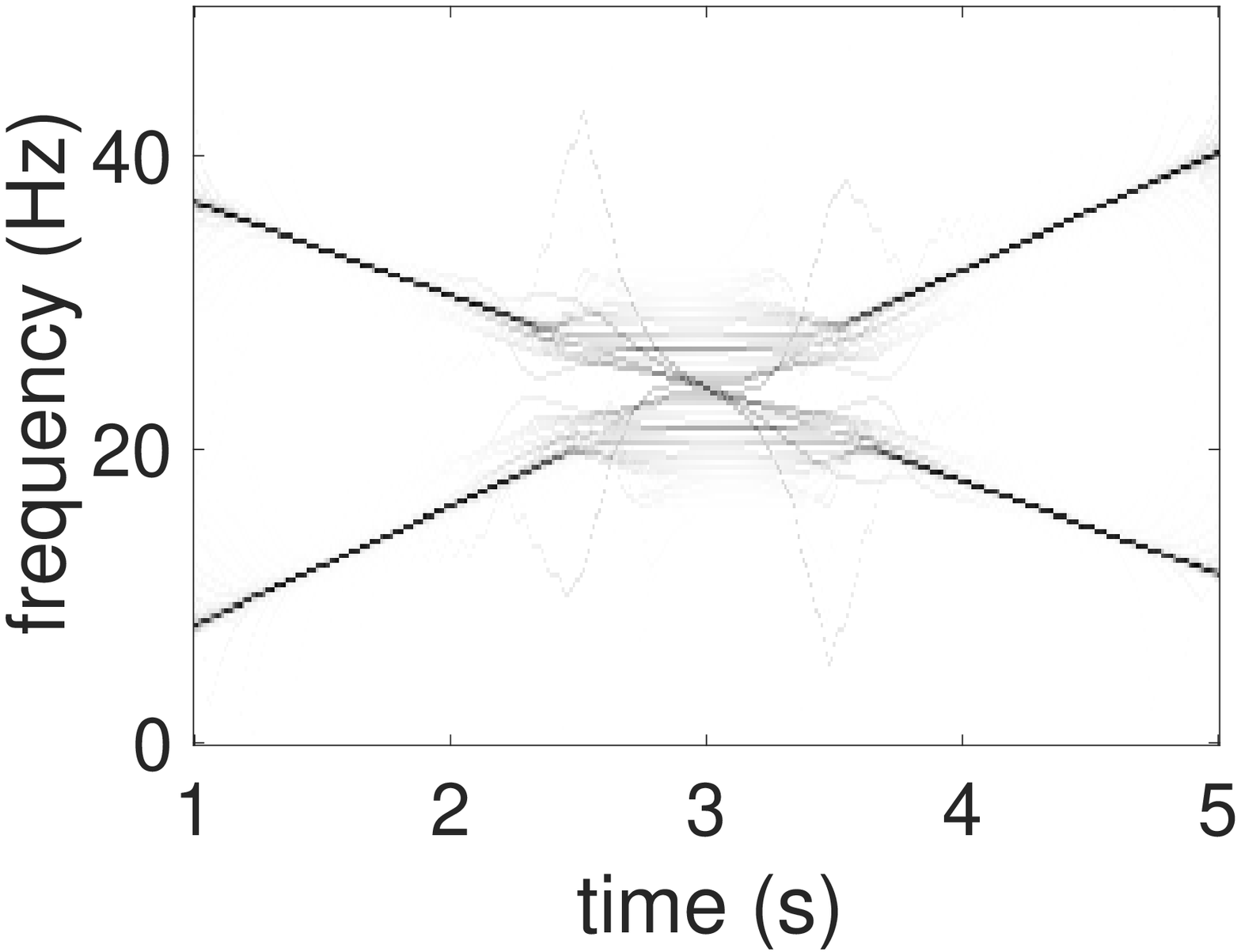}
	\includegraphics[width=0.325\textwidth]{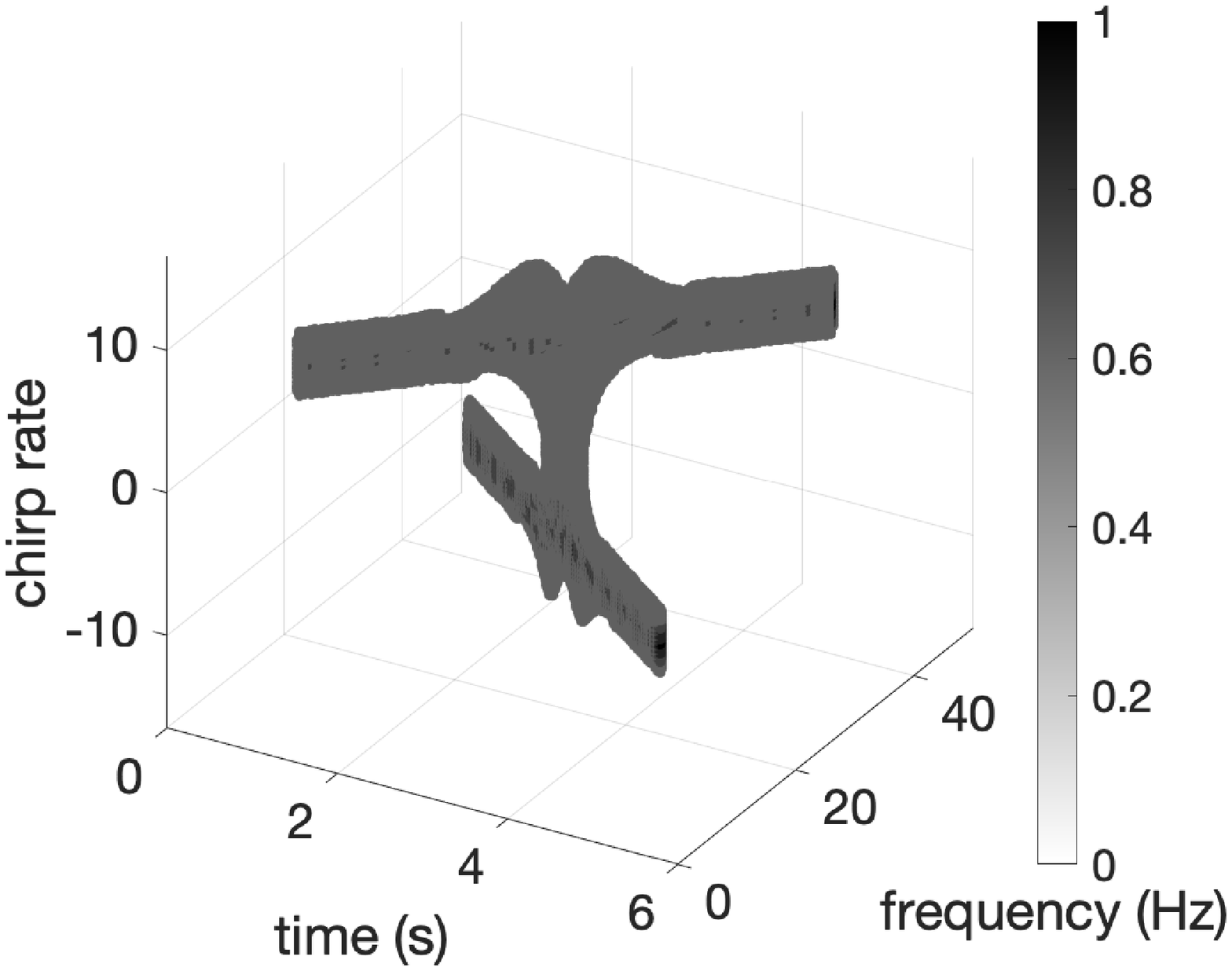}
	\includegraphics[width=0.325\textwidth]{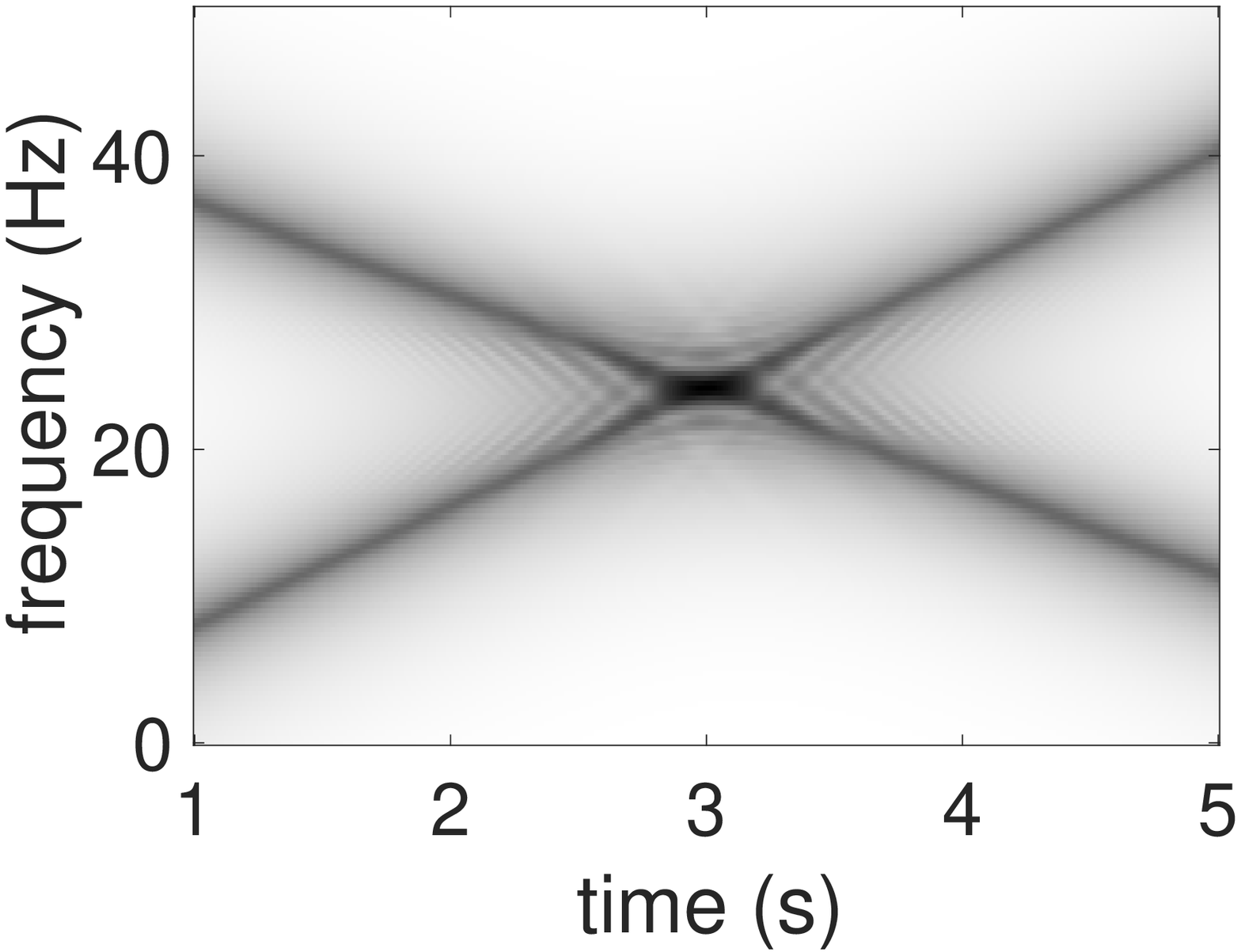}
	\includegraphics[width=.325\textwidth]{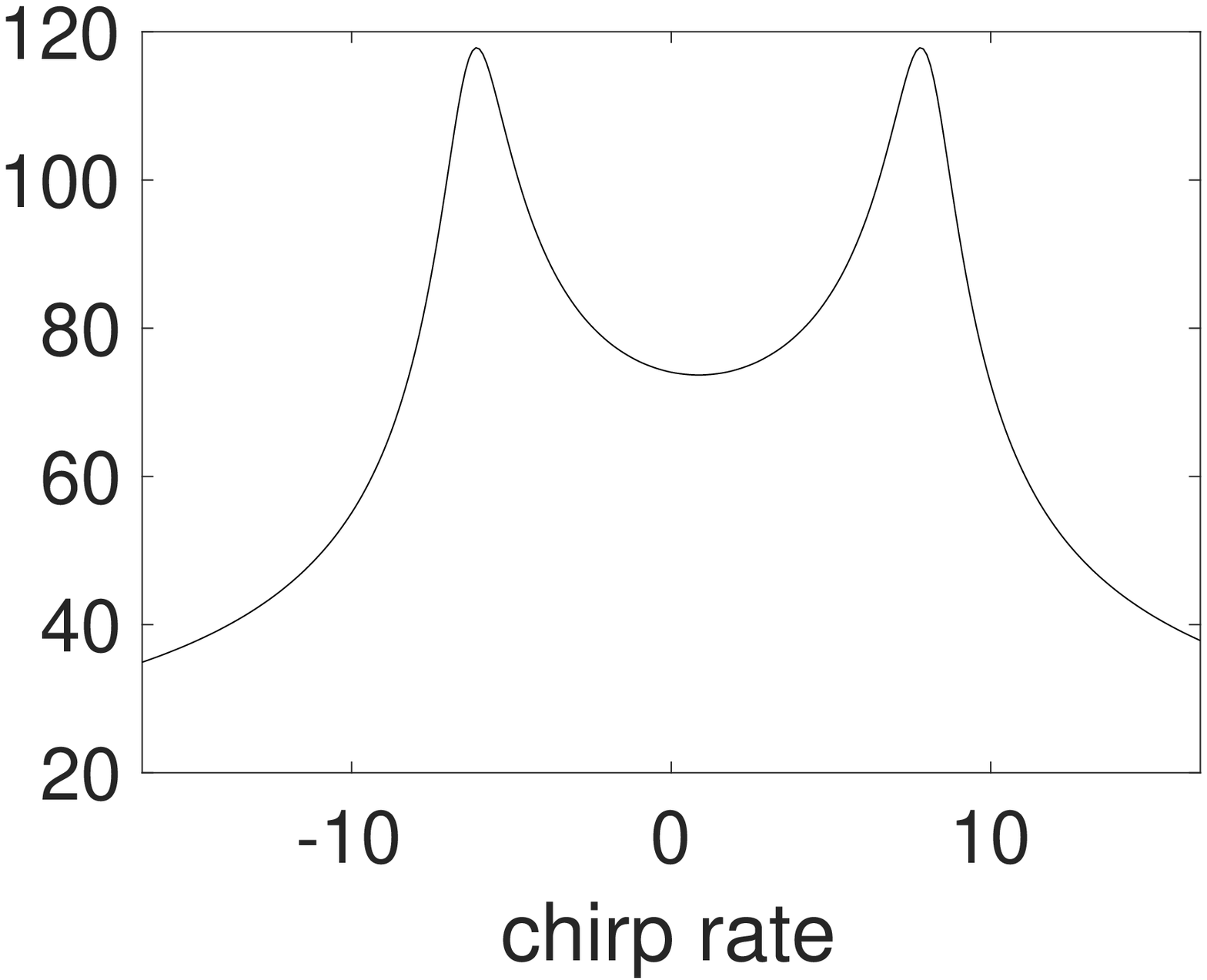}
	\caption{Top row, from left to right: the plot of $\Re(f_1+f_2)$, the plot of the IFs of $f_1$ and $f_2$, and the spectrogram. 
		Note that their IFs have a crossing point at $(t_0,\xi_0) = (3, 24)$. All the plots are generated with the kernel $g_0= e^{-\pi x^2}$.
		Second row, from left to right: the reassignment method, SST, and the 2nd-order SST.
		Bottom row, from left to right:  the 3-dim visualization of the time-frequency-chirp rate representation $\abs{T_f^{(g_0)}(t,\xi,\lambda)}$, the projection of $\abs{T_f^{(g_0)}(t,\xi,\lambda)}$ onto the time-frequency plane, and the plot of $\abs{T_f^{(g_0)}(t_0,\xi_0,\lambda)}$.
		For each time-frequency representation, we set all entries with the magnitude above the 99.99\% percentile of all entries to be the magnitude of the 99.99\% percentile.
	}
	\label{fig:intro0}
\end{figure}

The rest of the paper is organized as follows. In Section \ref{Sect:Chirptransform}, we introduce the chirp transform and study its properties. Based on the theoretical results, SCT and its theoretical support are provided in Section \ref{section:SCT}. We then provide numerical illustration of SCT in Section \ref{Sect:Numerical}. The discussion and future work are given in Section \ref{secttion:discussion future}.

\section{Chirplet transform and its squeezing}\label{Sect:Chirptransform}

\subsection{Chirplet transform}
We start our exploration by introducing the definition of CT.
Denote $\mathcal{S}(\mathbb{R})$ to be the Schwartz space and $\mathcal{S}'(\mathbb{R})$ to be the space of tempered distributions.

\begin{definition}[Chirplet transforms (CT) \cite{mann1995chirplet}]\label{Definition CT}
	The chirplet transform of a tempered distribution $f\in \mathcal{S}'(\mathbb{R})$ associated with a window function $g\in \mathcal{S}(\mathbb{R})$ at $(t,\xi,\lambda)$ is defined as
		\[
		T_{f}^{(g)}(t,\xi,\lambda) :=  \langle f(\cdot),\, g(\cdot-t)e^{2\pi i \xi(\cdot-t)}e^{\pi i \lambda (\cdot-t)^2}\rangle,
		\]
	where $\langle\cdot,\cdot\rangle$ indicates the evaluation of a tempered distribution at a Schwartz function, $t \in \mathbb{R}$ indicates {\em time}, $\xi\in \mathbb{R}$ indicates {\em frequency} and $\lambda \in \mathbb{R}$ indicates the {\em chirp rate}.
\end{definition}

In general, since $g$ is Schwartz, $T_{f}^{(g)}(t,\xi,\lambda)$ is well defined at each $(t,\xi,\lambda)$, and it is a complex function on the TFC domain. Note that when $f$ is a function, 
$T_{f}^{(g)}(t,\xi,\lambda)$ has an integration form $\int_{\mathbb{R}} f(x)\conj{g(x-t)}e^{-2\pi i \xi(x-t)}e^{-\pi i \lambda (x-t)^2}\diff{x}$.

\begin{definition}[Time-frequency-chirp rate representation and time-frequency representation determined by the chirplet transform]
	Following Definition \ref{Definition CT}, we call $|T_{f}^{(g)}(t,\xi,\lambda)|$ the {\em TFC representation} of a signal $f$ with window $g$. 
	We define the {\em TF representation determined by CT} as the projection of the TFC representation onto the TF domain, denoted as $\mathfrak T_{f}^{(g)}(t,\xi)$, which is defined as 
	\[
	\mathfrak T_{f}^{(g)}(t,\xi):=\int_{-\infty}^\infty \abs{T_{f}^{(g)}(t,\xi,\lambda)}\diff{\lambda}\,.
	\] 
\end{definition}
From the dictionary learning perspective, $T_{f}^{(g)}(t,\xi,\lambda)$ comes from fitting an atom with frequency $\xi$ and chirp rate $\lambda$ when the signal is centered at $t$ and truncated by the Gaussian window. 

In the bottom-right subfigure of Figure \ref{fig:intro0}, although we can observe two peaks that correspond to the chirp rates of $f_1$ and $f_2$, those peaks are not sharp in the chirp rate direction. An immediate question we would like to ask is the uncertainty principle associated with CT. We illustrate this point by evaluating CT of a pure linear chirp $f(x) = e^{2\pi i \xi_0 x +\pi i \lambda_0 x^2}$, where $\xi_0,\lambda_0\in \mathbb{R}$, and show that we cannot simultaneously sharply localize a signal in both the time domain and chirp rate domain.
By directly expanding the CT with the window $g(x) = e^{-\pi\alpha x^2}$, where $\alpha>0$, we have
\begin{align}
	T_f^{(g)}(t,\xi,\lambda) &= \int e^{2\pi i \xi_0 x +\pi i \lambda_0 x^2} e^{-\pi\alpha (x-t)^2} e^{-2\pi i \xi (x-t)} e^{-\pi i \lambda(x-t)^2}\diff{x}\nonumber\\
	&=e^{2\pi i \xi_0 t + \pi i \lambda_0 t^2} \int e^{-\pi(\alpha-i\lambda_0 + i\lambda)x^2} e^{-2\pi i (-\xi_0 -\lambda_0 t + \xi)x}\diff{x}\nonumber\\
	&= e^{2\pi i \xi_0 t + \pi i \lambda_0 t^2}
	\frac{1}{\sqrt{\alpha + i(\lambda-\lambda_0)}}e^{\frac{-\pi(
			\xi - \xi_0 -\lambda_0 t)^2}{\alpha+i(\lambda-\lambda_0)}},\label{CT of linear chirp expansion}
\end{align}
whose magnitude is 
\[
\abs{T_f^{(g)}(t,\xi,\lambda)} = 
\frac{1}{\sqrt[4]{\alpha^2 + (\lambda - \lambda_0)^2}} e^{\frac{-\pi \alpha(\xi-\xi_0-\lambda_0 t)^2}{\alpha^2 + (\lambda-\lambda_0)^2}}.
\]
The decay of $\abs{T_f^{(g)}(t,\xi,\lambda)}$ is related to the window $g(x) = e^{-\pi\alpha x^2}$ in the following Lemma.
	\begin{lemma}\label{Proposition slow decay of CT}
		Take $f(x) = e^{2\pi i \xi_0 x +\pi i \lambda_0 x^2}$. For a window $g(x) = e^{-\pi\alpha x^2}$ with fixed $\alpha>0$ and $t\in \mathbb{R}$, when $\xi$ is sufficiently close to $\xi_0+\lambda_0t$, we have
		$$\frac{\sqrt{\alpha}}{(\abs{\lambda-\lambda_0}^2+\alpha^2)^{\frac{1}{4}}}\max_{\lambda} \abs{T_f^{(g)}(t,\xi,\lambda)}\leq\abs{T_f^{(g)}(t,\xi,\lambda)}\leq \frac{e^{\frac{1}{4}}\sqrt{\alpha}}{\sqrt{\abs{\lambda-\lambda_0}}}\max_{\lambda} \abs{T_f^{(g)}(t,\xi,\lambda)}.$$
	\end{lemma}
	
	The proof of Lemma \ref{Proposition slow decay of CT} can be found in \ref{proof:lemma1}. This lemma indicates that for the linear chirp, the CT in the chirp rate axis decays slowly. To be more specific, the window $g(x)$ becomes more localized in time as $\alpha$ gets larger; larger $\alpha$ indicates the magnitude of the chirplet transform decays slower in the chirp rate direction near $\lambda = \lambda_0$ as can be seen in the lower bound. It also explains the blurring of the TFC representation shown in Figure \ref{fig:intro0}.

\subsection{Some properties of chirp transform}
Before introducing our algorithm we discuss some properties of the general chirp transform.
\begin{definition}[Chirp transforms]
	The chirp transform of a function $f\in L^1(\mathbb{R})$ is defined as
	\begin{equation}\label{def:chirp transform}
		Tf(\lambda) = \int_{\mathbb{R}} f(x)e^{-\pi i \lambda x^2}\diff{x},
	\end{equation}
	where $\lambda\in\mathbb{R}$.
\end{definition}
Note that replacing $f(x)$ by $f(x+t)g(x)e^{-2\pi i \xi x}$ in \eqref{def:chirp transform} yields the chirplet transform of $f$ with window $g$, where $g\in\mathcal{S}$, and the results below can be translated to that case.
We start by asking a basic question -- since the chirp transform is similar to the Fourier transform, may we obtain those basic properties of the Fourier transform for the chirp transform? For example, is the chirp transform of a $L^1$ function continuous and decay like that stated in the Riemann-Lebegue theorem? Does the chirp transform preserve the $L^2$ norm like the Plancherel theorem? Or is it bounded from $L^p$ to $L^p$ for some $p$? Note that this problem is closely related to the polynomial Carleson operator and Elias Stein's conjecture. See, for example \cite{lie2020polynomial}, for more details. 
To study these properties, we introduce some function spaces that we will use in the following propositions. We say a function $f(x) \in W_0^{1,1}(\mathbb{R})$, if $\displaystyle\lim_{\abs{x}\to \infty}f(x) =0$, and both $f(x)$ and $f'(x)$ in the weak sense are in $L^1(\mathbb{R})$; a function $f(x) \in C_{0}^{2n}(\mathbb{R})$, if $f$ is $2n$ times continuously differentiable and  $\lim\limits_{\abs{x}\to\infty}f^{(k)}(x) = 0$ for $k = 1,\dots, 2n$. Also, denote $C^\infty_c(\mathbb{R})$ to be the set of smooth functions with compact support.

We start with investigating if the magnitude of the chirp transform has some decay property in terms of chirp rate. Note that the following analysis of the decay rate is a special case of the oscillatory integral analysis widely known in the harmonic analysis community \cite{stein1993harmonic}, and we provide details for the sake of self-containedness.
First we have
\begin{lemma}\label{lemma decay of chirp transform}
	For $-\infty\leq a < b \leq\infty$, we have
	\[
	\left| \int_a^b e^{-\pi i \lambda x^2} \diff{x}\right| \leq \frac{2\sqrt{6}}{\sqrt{\pi}}\lambda^{-\frac{1}{2}}\,.
	\]
\end{lemma}
\begin{proof}
	We only show the case where $0\in (a,b)$, the argument for $0\in \mathbb{R}\backslash(a,b)$ is similar.
	Choose $\delta>0$ and split the integral into 
	\[
	\int_{a}^{-\delta}e^{-\pi i \lambda x^2} \diff{x} + \int_{-\delta}^{\delta}e^{-\pi i \lambda x^2} \diff{x} + \int_{\delta}^{b}e^{-\pi i \lambda x^2} \diff{x}
	\]
	For the first integral,integration by parts gives
	\begin{align*}
		\abs{\int_a^{-\delta}e^{-\pi i \lambda x^2}\diff{x}} &= \abs{\int_a^{-\delta}\frac{1}{-2\pi i \lambda x}\frac{\diff{e^{-\pi i \lambda x^2}}}{\diff{x}}\diff{x}}\\
		&\leq \abs{\frac{e^{-\pi i \lambda x^2}}{2\pi\lambda x}\bigg|_{x = a}^{-\delta}} + \abs{\int_{a}^{-\delta}e^{-\pi i \lambda x^2}\frac{\diff{}}{\diff{x}}(\frac{1}{2\pi i \lambda x})\diff{x}}\\
		&\leq \frac{1}{\pi\lambda\delta} + \frac{1}{2\pi\lambda} \int_{a}^{-\delta}\abs{\frac{\diff{}}{\diff{x}}(\frac{1}{x})}\diff{x}\\
		&= \frac{1}{\pi\lambda\delta} + \frac{1}{2\pi\lambda}\abs{\int_{a}^{-\delta}\frac{\diff{}}{\diff{x}}(\frac{1}{x})\diff{x}} \quad (\text{by the mononoticity of $\frac{1}{x}$ on $(a,-\delta)$})\\
		&\leq \frac{1}{\pi\lambda\delta} + \frac{1}{2\pi\lambda\delta} = \frac{3}{2\pi\lambda\delta}.
	\end{align*}
	By the same argument, we also have 
	\[
	\abs{\int_{\delta}^{b}e^{-\pi i \lambda x^2}\diff{x}} \leq \frac{3}{2\pi\lambda\delta},
	\]
	combined with $\abs{\int_{-\delta}^{\delta}e^{-\pi i \lambda x^2}\diff{x}} \leq 2\delta$, we have 
	\[
	\abs{\int_{a}^{b}e^{-\pi i \lambda x^2}\diff{x}} \leq \frac{3}{\pi\lambda\delta} + 2\delta.
	\]
	Take $\delta = \lambda^{-\frac{1}{2}}\sqrt{\frac{3}{2\pi}}$ that minimizes the right hand side, and we have $$\abs{\int_{a}^{b}e^{-\pi i \lambda x^2}\diff{x}}\leq \frac{2\sqrt{6}}{\sqrt{\pi}}\lambda^{-\frac{1}{2}}.$$
\end{proof}

\begin{proposition}\label{prop CT decay}
	Suppose $f(x) \in W_0^{1,1}(\mathbb{R})$, then $\abs{Tf(\lambda)} \leq \frac{2\sqrt{6}}{\sqrt{\pi}}\norm{f'}_1\lambda^{-\frac{1}{2}}$.
\end{proposition}

\begin{proof}
	When $f\in C^\infty_c(\mathbb{R})\subset W_0^{1,1}(\mathbb{R})$, by the change of variable, we have
	\[
	Tf(\lambda) = \int f(x)\frac{\mathrm{d}}{\diff{x}} \left(\int_{-\infty}^{x}e^{-\pi i \lambda t^2}\diff{t}\right) \diff{x}
	= -\int f'(x)\left(\int_{-\infty}^{x}e^{-\pi i \lambda t^2}\diff{t}\right) \diff{x}\,,
	\]
	so by Lemma \ref{lemma decay of chirp transform}, we have
	\[
	\left|Tf(\lambda)\right|\leq \int \left|f'(x)\right| \left| \int_{-\infty}^x e^{-\pi i \lambda t^2} \diff{t}\right| \diff{x}\leq  \frac{2\sqrt{6}}{\sqrt{\pi}}\lambda^{-\frac{1}{2}}\int\abs{f'(x)}\diff{x}.
	\]
	Since $C^\infty_c(\mathbb{R})$ is dense in $W_0^{1,1}(\mathbb{R})$, we get the proof by a direct extension argument.
\end{proof}

We have several properties of the chirp transforms that are parallel to some properties of the Fourier transforms.  
\begin{proposition}\label{prop unif cont}
	Let $f\in L^1(\mathbb{R})$, then $Tf(\lambda)$ is uniformly continuous.
\end{proposition}

\begin{proof}
	Take $\epsilon>0$. For any $h,r>0$, we have the following control:
	\begin{align*}
		\abs{Tf(\lambda+h)-Tf(\lambda)} &= \abs{\int f(x)(e^{-\pi i (\lambda+h)x^2}-e^{-\pi i \lambda x^2})\diff{x}}\\
		&\leq \int\abs{f(x)}\abs{e^{-\pi i h x^2}-1}\diff{x}\\
		&= \int_{\abs{x}\leq r}\abs{f(x)}\abs{e^{-\pi i h x^2}-1}\diff{x} + \int_{\abs{x}>r}\abs{f(x)}\abs{e^{-\pi i h x^2}-1}\diff{x}\\
		&\leq \int_{\abs{x}\leq r} \abs{f(x)}\pi\abs{h}x^2\diff{x} + 2\int_{\abs{x}>r}\abs{f(x)}\diff{x}.
	\end{align*}
	We only need to choose $h$ and $r$ so that $\int_{\abs{x}>r}\abs{f(x)}\diff{x}<\frac{\epsilon}{4}$, and $\pi\abs{h}r^2 < \frac{\epsilon}{2\norm{f}_1}$.
\end{proof}

By Propositions \ref{prop CT decay} and \ref{prop unif cont}, we know that the chirp transform, and hence CT, has a similar property as the Riemann-Lebesgue Theorem of the Fourier transform.

\begin{proposition}
	For $f\in L^1(\mathbb{R})$, if $x^2f(x)\in L^1(\mathbb{R})$, then $Tf(\lambda)$ is differentiable and 
	\[
	\frac{\diff{Tf(\lambda)}}{\diff{\lambda}} = -\pi i T(x^2 f)(\lambda).
	\]
\end{proposition}

\begin{proof}
	The proof is straightforward by applying the dominated convergence theorem.
\end{proof}

Unfortunately, we do not have a norm-preserving theorem for the chirp transforms similar to the Plancherel theorem in the Fourier transforms partially due to the slow decay rate in $\lambda$. On the contrary, we show that the chirp transforms do not satisfy ``most'' of the weak type inequalities. Before we state the proposition, we recall the definition of the weak type inequalities. 
	Let $(X, \nu_1)$ and $(Y, \nu_2)$ be measure spaces, and let $T$ be an operator from $L^p(X, \nu_1)$ into the space of measurable functions from $Y$ to $\mathbb{C}$. We say that $T$ is of weak-type $(p,q)$, $p, q<\infty$ if for $\lambda\geq 0$ we have
	\[
	\nu_2(\{y\in Y: |Tf(y)|>\lambda\})\leq\left(\frac{C\norm{f}_p}{\lambda}\right)^q.
	\] 
	This is a condition weaker than that if $T$ is bounded from $L^p(X, \nu_1)$ to $L^q(Y, \nu_2)$. In our case of CT, $X=Y=\mathbb{R}$ and $\nu_1,\nu_2$ are the Lebesgue measure on $\mathbb{R}$.

\begin{proposition}
	The chirp transform operator $T$ is not weak-type $(p,q)$, if $\displaystyle \frac{2q}{p}-2q+4 \neq 0$.
\end{proposition}

\begin{proof}
	Fix $p>1$. We show that for any $C>0$, there exists $f\in L^p(\mathbb{R})$ and $\eta>0$, such that 
	\[
	m(\{ \lambda: \abs{Tf(\lambda)}>\eta\}) > \frac{C\norm{f}_p^q}{\eta^q},
	\]
	where $\frac{2q}{p}-2q+4 \neq 0$ and $m$ is the Lebesgue measure.
	Let $f_{\alpha}(x) = \alpha^{\frac{1}{p}}e^{-\frac{\pi\alpha^2 x^2}{p}}$, $\alpha>0$. One can verify that $\norm{f_{\alpha}}_p = 1$ for any $\alpha>0$. It suffices to show that for any $C>0$, we can always find $\alpha, \eta>0$, such that
	\begin{equation}
		m(\{ \lambda: \abs{Tf_{\alpha}(\lambda)}>\eta\}) > \frac{C\norm{f_{\alpha}}_p^q}{\eta^q} = \frac{C}{\eta^q}\,.\label{proof:prop5 eq1}
	\end{equation}
	It is easy to obtain $Tf_{\alpha}(\lambda) = \frac{\alpha^{\frac{1}{p}}}{\sqrt{\frac{\alpha^2}{p}+i\lambda}}$, so $\abs{Tf_{\alpha}(\lambda)} = \frac{\alpha^{\frac{1}{p}}}{\sqrt[4]{\frac{\alpha^4}{p^2} + \lambda^2}}$. By solving the inequality $\frac{\alpha^{\frac{1}{p}}}{\sqrt[4]{\frac{\alpha^4}{p^2} + \lambda^2}} > \eta$, we have 
	\begin{equation}
		m(\{ \lambda: \abs{Tf_{\alpha}(\lambda)}>\eta\}) = 2\sqrt{\frac{\alpha^{\frac{4}{p}}}{\eta^4}-\frac{\alpha^4}{p^2}}\,, \label{proof:prop5 eq2}
	\end{equation}
	for $\frac{\alpha^{\frac{4}{p}}}{\eta^4}-\frac{\alpha^4}{p^2}>0$; otherwise the measure is zero. Note that the condition $\frac{\alpha^{\frac{4}{p}}}{\eta^4}-\frac{\alpha^4}{p^2}>0$ gives $\eta<\alpha^{\frac{1}{p}-1}p^{\frac{1}{2}}$. Select $\eta = \frac{1}{2}\alpha^{\frac{1}{p}-1}p^{\frac{1}{2}}$. We now determine $\alpha$. With \eqref{proof:prop5 eq2} and the selected $\eta$, \eqref{proof:prop5 eq1} becomes
	\[
	2\sqrt{\frac{\alpha^{\frac{4}{p}}}{(\frac{1}{2}\alpha^{\frac{1}{p}-1}p^{\frac{1}{2}})^4}-\frac{\alpha^4}{p^2}} > \frac{C}{(\frac{1}{2}\alpha^{\frac{1}{p}-1}p^{\frac{1}{2}})^q}.
	\]
	Simplifying this inequality gives us
	\[
	\alpha^{\frac{2q}{p}-2q+4}>\frac{2^{2q-2}}{15}C^2p^{2-q}.
	\]
	Therefore, if $\frac{2q}{p}-2q+4>0$, we can set $\alpha$ to be sufficiently large, so that the left hand side is larger than the right hand side; if $\frac{2q}{p}-2q+4<0$, we can set $\alpha$ to be sufficiently small. We thus finish the claim.
\end{proof}

\begin{corollary}
	The chirp transform operator $T$ is not weak-type $(2,2)$.
\end{corollary}

We next show that the decay rate of $Tf(\lambda)$ is related to the vanishing order of $f$ at the origin. Note that this is intuitively correct. Unlike $e^{i\xi_0x}$ for some $\xi_0>0$,  $e^{i\lambda x^2}$ for some $\lambda\neq 0$ does not oscillate around $0$. This is the main resource of the slow decay rate of $Tf(\lambda)$. Thus, if we want to have a faster decay rate of $Tf(\lambda)$, we shall force the function $f$ to be zero around $0$. This intuition is made clear in the following proposition.
\begin{proposition}\label{prop2}
	Suppose $f(x)\in C_{0}^{2n}(\mathbb{R})$, $n\geq 1$ and $f\in L^1(\mathbb{R})$. If we assume that $f(0) = f''(0) = \cdots = f^{(2n-2)}(0) = 0$ and  $f^{(2n)}(0)\neq 0$, then 
	\[
	\abs{Tf(\lambda)}\leq C\lambda^{-\frac{2n+1}{3}},
	\]
	for some contant $C>0$ and $C$ depends on $n$ and $f$.
\end{proposition}
\begin{proof}
	Fix $\lambda>0$. Let $D$ be a differential operator such that $Df := -\frac{1}{2\pi i\lambda x}\frac{\diff{f}}{\diff{x}}$, and $\tilde{D}$ such that $\tilde{D}f := \frac{1}{2\pi i\lambda}\frac{\diff{}}{\diff{x}}(\frac{f}{x})$. First, if $f$ satisfies the above conditions and is zero in $(-\epsilon,\epsilon)$ for some $\epsilon >0$, since $D^{k}(e^{-\pi i \lambda x^2}) = e^{-\pi i \lambda x^2}$ for any $k$ and $x\neq 0$, we have  
	\begin{align*}
		Tf(\lambda) 
		&= \int_{\mathbb{R}\backslash (-\epsilon,\epsilon)} f(x)e^{-\pi i \lambda x^2}\diff{x}\\
		&= \int_{\mathbb{R}\backslash (-\epsilon,\epsilon)} f(x)D^{2n}(e^{-\pi i \lambda x^2})\diff{x}= \int_{\mathbb{R}\backslash (-\epsilon,\epsilon)}\tilde{D}^{2n}(f)e^{-\pi i \lambda x^2}\diff{x}.
	\end{align*}
	One can easily show by induction on $n$ together with the boundedness of the derivatives of $f$ that $\abs{Tf(\lambda)}\leq C\lambda^{-2n}\epsilon^{-(4n-1)}$ for some $C>0$ that depends on $n$ and $f$. 
	
	Then we prove the statement of the proposition. Suppose for some $\epsilon>0$, we can write $f$ in its Taylor expansion:
	\begin{align*}
		f(x) &= f(0) + f'(0)x +  \cdots + \frac{f^{(2n-1)}(0)}{(2n-1)!}x^{2n-1} + \frac{f^{(2n)}(\xi(x))}{(2n)!}x^{2n}\\
		&= f'(0)x+ \cdots + \frac{f^{(2n-3)}(0)}{(2n-3)!}x^{2n-3}+ \frac{f^{(2n-1)}(0)}{(2n-1)!}x^{2n-1} + \frac{f^{(2n)}(\xi(x))}{(2n)!}x^{2n}
	\end{align*}
	for $\abs{x}<2\epsilon$, and $\xi(x)$ is some number between $0$ and $x$.
	Suppose $\psi(x)\in C_c^{\infty}(\mathbb{R})$ with $\psi(x) = 1$ for $\abs{x}\leq 1$, and $\psi(x) = 0$ for $\abs{x}\geq 2$, and we also assume $\psi$ is even. Then we can decompose $f$ into $f(x) = f(x)\psi(\frac{x}{\epsilon})+ f(x)(1-\psi(\frac{x}{\epsilon}))$. From the first part of the proof, we know that $\abs{T((1-\psi(\frac{x}{\epsilon}))f)(\lambda)}\leq C\lambda^{-2n}\epsilon^{-(4n-1)}$. On the other hand, $T(x^{k}\psi(\frac{x}{\epsilon}))(\lambda)= 0$ for odd $k$ by symmetry. What remains is controlling $\frac{f^{(2n)}(\xi(x))}{(2n)!}x^{2n}\psi(\frac{x}{\epsilon})$. By the boundedness of $f^{(2n)}(x)$, we have
	\begin{align*}
		\abs{T\left(\frac{f^{(2n)}(\xi(x))}{(2n)!}x^{2n}\psi\left(\frac{x}{\epsilon}\right)\right)(\lambda)}
		&\leq C\int_{\abs{x}\leq 2\epsilon} x^{2n}\diff{x}\leq C\epsilon^{2n+1}\,,
	\end{align*}
	where $C>0$ depends on $n$ and $f$.
	Combining these estimations, we have
	\begin{align*}
		\abs{Tf(\lambda)}\leq C\left(\lambda^{-2n}\epsilon^{-(4n-1)}+\epsilon^{2n+1}\right)\,.
	\end{align*}
	If we select $\epsilon = \lambda^{-\frac{1}{3}}$, we have
	\begin{align*}
		\abs{Tf(\lambda)}\leq C \lambda^{-\frac{2n+1}{3}}
	\end{align*}
\end{proof}

\section{Proposed algorithm -- Synchrosqueezed chirplet transform}\label{section:SCT}

The above result leads us to consider if it is possible to sharpen the TFC representation by some idea similar to the reassignment technique, like synchrosqueezing \cite{daubechies2011synchrosqueezed}, used to handle the uncertainty principle of the linear-type TF analysis algorithms. Indeed, we ask if it is possible to manipulate the phase information in CT so that the TFC representation can be sharpened. Below we introduce this idea with a linear chirp $f(x)= e^{2\pi i \xi_0 x +\pi i \lambda_0 x^2}$, where $\lambda_0$ is allowed to be large.
Note that the idea of the original SST might not be suitable for our purpose due to the non-trivial chirp rate. Instead, we need to consider the higher order phase information, like that considered in the second order SST \cite{oberlin2017second, behera2018theoretical}, to handle the possibly non-trivial chirp rate. 
We take any real-valued window function $g(x)$ in the Schwartz space. Since 
\begin{align*}
	T_f^{(g)}(t,\xi,\lambda) &= \int e^{2\pi i \xi_0 x +\pi i \lambda_0 x^2} g(x-t) e^{-2\pi i \xi (x-t)} e^{-\pi i \lambda(x-t)^2}\diff{x}\\
	&=e^{2\pi i \xi_0 t + \pi i \lambda_0 t^2} \int g(x) e^{-\pi i(\lambda - \lambda_0)x^2} e^{-2\pi i (\xi -\xi_0 -\lambda_0 t )x}\diff{x}\,
\end{align*}
by a direct expansion we have 
\begin{equation}
	\partial_t T_f^{(g)}(t,\xi,\lambda) = (2\pi i \xi_0 + 2\pi i \lambda_0 t)T_f^{(g)}(t,\xi,\lambda) + 2\pi i \lambda_0 T_f^{(xg)}(t,\xi,\lambda)
	\label{partial_t}
\end{equation}
and
\begin{equation}\label{partialxi}
	\partial_\xi T_f^{(g)}(t,\xi,\lambda) = -2\pi i T_f^{(xg)}(t,\xi,\lambda).
\end{equation}
Then from (\ref{partial_t}) and (\ref{partialxi}) we have
\begin{align*}
	\partial_t\left(\frac{\partial_t T_f^{(g)}(t,\xi,\lambda)}{T_f^{(g)}(t,\xi,\lambda)}\right) &= 2\pi i \lambda_0 \left(1+ \partial_t\left(\frac{T_f^{(xg)}(t,\xi,\lambda)}{T_f^{(g)}(t,\xi,\lambda)}\right)\right)\\
	&= \lambda_0 \left(2\pi i - \partial_t\left(\frac{\partial_{\xi}T_f^{(g)}(t,\xi,\lambda)}{T_f^{(g)}(t,\xi,\lambda)}\right)\right),
\end{align*}
so
\begin{equation*}
	\lambda_0 = \frac{\partial_t\left(\frac{\partial_t T_f^{(g)}(t,\xi,\lambda)}{T_f^{(g)}(t,\xi,\lambda)}\right)}{2\pi i - \partial_t\left(\frac{\partial_{\xi}T_f^{(g)}(t,\xi,\lambda)}{T_f^{(g)}(t,\xi,\lambda)}\right)}. 
	\label{chirprate}
\end{equation*}
In addition, from (\ref{partial_t}) we can solve for the frequency,
\begin{equation*}
	\xi_0 + \lambda_0 t = \frac{\partial_t T_f^{(g)}(t,\xi,\lambda) - 2\pi i \lambda_0 T_f^{(xg)}(t,\xi,\lambda)}{2\pi i T_f^{(g)}(t,\xi,\lambda)}.
	\label{frequency}
\end{equation*}
The above equalities motivate us to propose the reassignment operators:
	\[
	\mu_{f}^{(g)}(t,\xi,\lambda):=\Re{\left(\frac{\partial_t\left(\frac{\partial_t T_f^{(g)}(t,\xi,\lambda)}{T_f^{(g)}(t,\xi,\lambda)}\right)}{2\pi i - \partial_t\left(\frac{\partial_{\xi}T_f^{(g)}(t,\xi,\lambda)}{T_f^{(g)}(t,\xi,\lambda)}\right)}\right)},
	\]
	\begin{align*}
		\omega_f^{(g)}(t,\xi,\lambda):=&\, \Re{\Bigg(\frac{\partial_t T_f^{(g)}(t,\xi,\lambda) - 2\pi i \mu_{f}^{(g)}(t,\xi,\lambda) T_f^{(xg)}(t,\xi,\lambda)}{2\pi i T_f^{(g)}(t,\xi,\lambda)}\Bigg)}\\
		= &\, \Re{\Bigg(\frac{\partial_t T_f^{(g)}(t,\xi,\lambda) - \mu_{f}^{(g)}(t,\xi,\lambda) \partial_{\xi}T_f^{(g)}(t,\xi,\lambda)}{2\pi i T_f^{(g)}(t,\xi,\lambda)}\Bigg)},
	\end{align*}
	where $\Re{(z)}$ denotes the real part of a complex number $z$. We can calculate the partial derivatives, and the expression for $\mu_{f}^{(g)}$ at $(t,\xi,\lambda)$ becomes
	\[
	\mu_{f}^{(g)} =  \Re{\left( \frac{T_{f}^{(g)}\partial_{tt}^2T_{f}^{(g)}-(\partial_t T_f^{(g)})^2}{2\pi i \big(T_{f}^{(g)}\big)^2 +2\pi i \big(T_{f}^{(g)}\partial_t T_{f}^{(xg)}- T_{f}^{(xg)}\partial_t T_{f}^{(g)}\big)}\right)}.
	\]
The above calculations lead us to the following definition.

\begin{definition}[Synchrosqueezed chirplet transform]
	Take $f\in \mathcal{S}'(\mathbb{R})$ and $g\in \mathcal{S}(\mathbb{R})$ as the window. For a small $\alpha>0$,
	define \textit{the synchrosqueezed chirplet transform} (SCT) with {\em resolution $\alpha$} as
	\begin{equation*}
		S_f^{(g,\alpha)}(t,\xi,
		\lambda) := \iint_{\mathbb{R}^2} T_f^{(g)}(t,\eta, \gamma)h_{\alpha}\left(\xi-\omega_{f}^{(g)}(t,\eta,\gamma)\right)h_{\alpha}\left(\lambda-\mu_f^{(g)}(t,\eta,\gamma)\right)\diff{\eta}\diff{\gamma},
	\end{equation*}
	where $h_{\alpha}$ is an ``approximate $\delta$-function'' (i.e. $h$ is smooth and decays fast with $\int h(x)\diff{x} = 1$, so that $h_\alpha(t):= \frac{1}{\alpha}h(\frac{t}{\alpha})$ tends weakly to the delta measure $\delta$ as $\alpha\to 0$).
	Similar to $\mathfrak T_{f}^{(g)}(t,\xi)$, we define the TF representation associated with SCT via projecting the TFC representation onto the TF domain:
	\begin{equation*}
		\mathfrak S_f^{(g,\alpha)}(t,\xi) =\int_{-\infty}^\infty |S_f^{(g,\alpha)}(t,\xi,
		\lambda)| \diff{\lambda}\,.
	\end{equation*}
\end{definition}

A main practical concern of SCT is its numerical stability, since differentiations in the calculation of $\mu_{f}^{(g)}$ and $\omega_{f}^{(g)}$ are numerically unstable. This concern is handled by noting that $\omega_f^{(g)}$ and $\mu_{f}^{(g)}$ can be calculated directly with different window functions without any numerical differentiation. Indeed, note that by a direct expansion, we have at $(t,\xi,\lambda)$
\begin{equation*}
	\partial_t T_f^{(g)} = -T_f^{(g')} + 2\pi i \xi T_f^{(g)} + 2\pi i \lambda T_f^{(xg)},
\end{equation*}
and
\begin{align*}
	\partial_{tt}^2 T_f^{(g)} = &T_f^{(g'')} -4\pi i \xi T_f^{(g')} -2\pi i \lambda T_f^{(xg')} + (2\pi i \xi)^2 T_f^{(g)} \\
	& + 2(2\pi i \xi)(2\pi i \lambda)T_f^{(xg)}- 2\pi i\lambda T_f^{(g+xg')} + (2\pi i \lambda)^2T_f^{(x^2g)}.
\end{align*}
Therefore, at $(t,\xi,\lambda)$ we have
\begin{align*}
	&\mu_{f}^{(g)} = \Re\left(\frac{T_{f}^{(g)}\partial_{tt}^2T_{f}^{(g)}-(\partial_t T_f^{(g)})^2}{2\pi i (T_{f}^{(g)})^2 +2\pi i (T_{f}^{(g)}\partial_t T_{f}^{(xg)}- T_{f}^{(xg)}\partial_t T_{f}^{(g)})}\right)
	= \Re\Big(\frac{M_1}{M_2}\Big),
\end{align*}
where 
\begin{align}
	M_1:=&T_{f}^{(g)}T_{f}^{(g'')}-4\pi i \lambda T_{f}^{(g)}T_{f}^{(xg')} -2\pi i \lambda (T_{f}^{(g)})^2+(2\pi i \lambda)^2T_{f}^{(g)}T_{f}^{(x^2g)}\nonumber\\
	&-(T_{f}^{(g')})^2 - (2\pi i \lambda)^2(T_{f}^{(xg)})^2+4\pi i \lambda T_{f}^{(g')}T_{f}^{(xg)}\nonumber
\end{align} 
and $$M_2:=-2\pi i T_{f}^{(g)}T_{f}^{(xg')}+(2\pi i)^2\lambda T_{f}^{(g)} T_{f}^{(x^2g)}+2\pi i T_{f}^{(xg)}T_{f}^{(g')}-(2\pi i)^2\lambda (T_{f}^{(xg)})^2$$
and
\begin{flalign*}
	\omega_f^{(g)}&= \Re\left(\frac{\partial_t T_f^{(g)} - 2\pi i \mu_{f}^{(g)} T_f^{(xg)}}{2\pi i T_f^{(g)}}\right)
	= \Re\left(\xi + \frac{-T_f^{(g')}+2\pi i \lambda T_f^{(xg)} - 2\pi i \mu_{f}^{(g)} T_f^{(xg)} }{2\pi i T_f^{(g)}}\right)
\end{flalign*}
at $(t,\xi,\lambda)$.\\

\noindent The numerical implementation of SCT will be detailed in Section \ref{sec:numericalSCT}.

\begin{remark}
	We shall compare SCT with SST and its 2nd-order variation. Recall that the Short-time Fourier transform (STFT) of a tempered distribution $f\in \mathcal{S}'(\mathbb{R})$ associated with a window function $g\in \mathcal{S}(\mathbb{R})$ at $(t,\xi)$ is defined as
	\[
	W_{f}^{(g)}(t,\xi) :=  \langle f(\cdot),\, g(\cdot-t)e^{2\pi i \xi(\cdot-t)}\rangle,
	\]
	where $\langle\cdot,\cdot\rangle$ indicates the evaluation of a tempered distribution at a Schwartz function, $t \in \mathbb{R}$ indicates {\em time}, $\xi\in \mathbb{R}$ indicates {\em frequency}.
	With the same $h_{\alpha}$ as in the definition of SCT, the STFT-based SST of $f\in \mathcal{S}'(\mathbb{R})$ with $g\in \mathcal{S}(\mathbb{R})$
	and {\em resolution $\alpha$} is then defined as \cite{thakur2011synchrosqueezing}
	\begin{equation}
		S_f^{1,(g,\alpha)}(t,\xi) := \int_{\mathbb{R}} W_f^{(g)}(t,\eta)h_{\alpha}(\xi-\omega_{f}^{1,(g)}(t,\eta))\diff{\eta},
	\end{equation}
	where 
	\begin{equation*}
		\omega_{f}^{1,(g)}(t,\xi) = \Re\left(\frac{\partial_t W_f^{(g)}(t,\xi)}{2\pi i W_f^{(g)}(t,\xi)}\right);
	\end{equation*}
	the STFT-based 2nd-order SST of $f\in \mathcal{S}'(\mathbb{R})$ with $g\in \mathcal{S}(\mathbb{R})$
	and {\em resolution $\alpha$} is defined as \cite{behera2018theoretical}
	\begin{equation}
		S_f^{2,(g,\alpha)}(t,\xi) := \int_{\mathbb{R}} W_f^{(g)}(t,\eta)h_{\alpha}(\xi-\omega_{f}^{2,(g)}(t,\eta))\diff{\eta},
	\end{equation}
	where
	\begin{equation*}
		\omega_{f}^{2,(g)}(t,\xi) = \omega_{f}^{1,(g)}(t,\xi) + \tilde{q}_f(t,\xi)(t-\tilde{t}_f(t,\xi)),
	\end{equation*}
	and $\tilde{t}_f(t,\xi)$ and $\tilde{q}_f(t,\xi)$ are defined as 
	\begin{equation*}
		\tilde{t}_f(t,\xi) := \Re\left(t- \frac{\partial_{\xi}W_f^{(g)}(t,\xi)}{2\pi i W_f^{(g)}(t,\xi)}\right)\,,\ \ 
		\tilde{q}_f(t,\xi) :=\Re\left(\frac{\partial_t\left(\frac{\partial_t W_f^{(g)}(t,\xi)}{W_f^{(g)}(t,\xi)}\right)}{2\pi i - \partial_t\left(\frac{\partial_{\xi}W_f^{(g)}(t,\xi)}{W_f^{(g)}(t,\xi)}\right)}\right).
	\end{equation*}
	In short, the SST reassigns the TF content based on the first-order Taylor expansion of the phase function, and the 2nd-order SST approximates the phase function by the second-order Taylor expansion, and we can view the 2nd-order SST as a special case of SCT in which the chirp rate is fixed and set to be zero.
\end{remark}

At the first glance, the behavior of $S_f^{(g,\alpha)}(t,\xi,
\lambda)$ might not be clear even if $f$ is a function. To better understand its properties, and design algorithms to solve the challenge of crossover IFs, we consider the following model.

\begin{definition}[\textit{Intrinsic chirp type function}]
	Fix $\epsilon>0$. A function $f: \mathbb{R} \to \mathbb{C}$ is said to be of the $\epsilon$-intrinsic chirp type ($\epsilon$-ICT) if $f(x) = A(x)e^{2\pi i\phi(x)}$ with $A$ and $\phi$ having the following properties:
	\begin{equation*}
		A \in C^1(\mathbb{R})\cap L^{\infty}(\mathbb{R}),\quad \phi\in C^3(\mathbb{R}),
	\end{equation*}
	\begin{equation*}
		\inf_{x\in\mathbb{R}}\phi'(x)>0, \quad \sup_{x\in\mathbb{R}}\phi'(x)<\infty,
	\end{equation*}
	\begin{equation*}
		A(x)>0,\quad\abs{A'(x)},\abs{A''(x)},\abs{\phi'''(x)}\leq\epsilon\phi'(x),\quad \forall x\in\mathbb{R}.
	\end{equation*}
\end{definition}

By definition, an $\epsilon$-ICT function is an oscillatory function that is locally close to a linear chirp function, where the closeness is quantified by $\epsilon$. In addition to the chirp rate variation, the AM variations are both controlled by the IF.

\begin{definition}[\textit{Superposition of well-separated $\epsilon$-ICT components}]
	A function $f:\mathbb{R}\to \mathbb{C}$ is said to be in the space $\mathcal{A}_{\epsilon,\Delta}$ of superpositions of well-separated $\epsilon$-ICT functions, and with separation $\Delta>0$, if there exists a finite $K$, such that
	\[
	f(x) = \sum_{k=1}^K f_k(x) = \sum_{k=1}^K A_k(x) e^{2\pi i \phi_k(x)},
	\]
	where each $f_k$ is an $\epsilon$-ICT function, and their respective phase functions $\phi_k$ satisfy
	\begin{equation*}
		\abs{\phi_k'(t)-\phi_l'(t)}+\abs{\phi_k''(t)-\phi_k''(t)}\geq 2\Delta.
	\end{equation*}
\end{definition}

To describe the main property of SCT, for a Schwartz function $g$, we define
\[
\widecheck{g}(\xi,\lambda):= \int_{\mathbb{R}} g(x)e^{-2\pi i \xi x}e^{-\pi i \lambda x^2}\diff{x}\,.
\]
Now we state the main theorem of SCT.
\begin{theorem}\label{maintheorem}
	Suppose that $f\in\mathcal{A}_{\epsilon,\Delta}$, and pick a window function $g\in\mathcal{S}(\mathbb{R})$ that satisfies $\abs{\widecheck{(x^ng)}(\xi,\lambda)}\leq \frac{G_n}{\sqrt{\abs{\xi}+\abs{\lambda}}}$ for some $G_n>0$, $n=0,1,2$. Let $\tilde{\epsilon} = \epsilon^{\frac{1}{6}}$. Then, provided $\sqrt{\Delta} = \epsilon^{-1}$ and $\epsilon$ (and thus also $\tilde{\epsilon}$) is sufficiently small, the following hold:
	\begin{itemize}
		\item If $(t,\xi,\lambda)\not\in Z_k$ for any $k\in\{1,\dots,K\}$, where
		\begin{equation}
			Z_k:= \{(t,\xi,\lambda):\abs{\xi-\phi_k'(t)}+\abs{\lambda-\phi_k''(t)}< \Delta \}\,,\label{definition Zk}
		\end{equation}
		then $\abs{T_f^{(g)}(t,\xi,\lambda)}\leq \tilde{\epsilon}$.
		\item For each tuple $(t,\xi,\lambda)\in Z_k$ such that $\abs{T_f^{(g)}(t,\xi,\lambda)}> \tilde{\epsilon}$ and $2\pi\abs{1+\partial_t\left(\frac{T_f^{(xg)}(t,\xi,\lambda)}{T_f^{(g)}(t,\xi,\lambda)}\right) }>\tilde{\epsilon}$, we have 
		\begin{equation}
			\abs{\omega_f^{(g)}(t,\xi,\lambda)-\phi_k'(t)}\leq \tilde{\epsilon}\ \mbox{ and }\ \abs{\mu_f^{(g)}(t,\xi,\lambda)-\phi_k''(t)}\leq \tilde{\epsilon}\,.\label{definition omega mu approximation}
		\end{equation}
	\end{itemize}
\end{theorem}

The proof of Theorem \ref{maintheorem} relies on a number of standard estimates like those in \cite{behera2018theoretical}, which is postponed to \ref{section:supp}.
	\begin{remark}
		The condition for the separation parameter, $\sqrt{\Delta} = \epsilon^{-1}$, indicates that in order to allow a small error bound $\tilde{\epsilon}$, $\Delta$ should be large so that the bound of $\abs{T_{f_{k,2}}^{(x^ng)}(t,\xi,\lambda)}$ in the proof of Lemma \ref{lemma:2} in \ref{section:supp} is small. This is reasonable 
		since if we want $\abs{\widecheck{(x^ng)}(\xi,\lambda)}$ to be small, $(\xi,\lambda)$ should be away from the origin. Note that a similar quantitative requirement for $\Delta$ to be large was mentioned in Remark 9 in \cite{behera2018theoretical} for the second-order SST. We provide a specific quantitative relation here.
		
		An example of a window function $g$ that satisfies the window condition in Theorem \ref{maintheorem} is $g(x) = e^{-\pi x^2}$. To see why, 
		first, note that we have $\widecheck{g}(\xi,\lambda) = \int_{\mathbb{R}} e^{-\pi x^2}e^{-2\pi i \xi x}e^{-\pi i \lambda x^2}\diff{x} = \frac{1}{\sqrt{1+i\lambda}}e^{-\frac{\pi\xi^2}{1+i\lambda}}$, so $\abs{\widecheck{g}(\xi,\lambda)} = (1+\lambda^2)^{-\frac{1}{4}}e^{-\frac{\pi\xi^2}{1+\lambda^2}}$. It suffices to show that $(1+\lambda^2)^{-\frac{1}{4}}e^{-\frac{\pi\xi^2}{1+\lambda^2}}\leq\frac{G_0}{\sqrt{\abs{\xi}+\abs{\lambda}}}$ for some constant $G_0>0$, or equivalently $\frac{\abs{\xi}+\abs{\lambda}}{\sqrt{1+\lambda^2}}e^{-\frac{2\pi\xi^2}{1+\lambda^2}}\leq G_0^2$. We could easily see the boundedness of $\frac{\abs{\xi}}{\sqrt{1+\lambda^2}}e^{-\frac{2\pi\xi^2}{1+\lambda^2}}$ by setting $y = \frac{\abs{\xi}}{\sqrt{1+\lambda^2}}$, so that $\frac{\abs{\xi}}{\sqrt{1+\lambda^2}}e^{-\frac{2\pi\xi^2}{1+\lambda^2}} = ye^{-2\pi y^2}$. On the other hand, $\frac{\abs{\lambda}}{\sqrt{1+\lambda^2}}e^{-\frac{2\pi\xi^2}{1+\lambda^2}}\leq \frac{\abs{\lambda}}{\sqrt{1+\lambda^2}}$ is clearly also bounded. We thus get the claim for $n=0$.
		Secondly, 
		\begin{align*}
			\widecheck{(xg)}(\xi,\lambda) &\,= \int_{\mathbb{R}} xe^{-\pi x^2}e^{-2\pi i \xi x}e^{-\pi i \lambda x^2}\diff{x} \\
			&\,= \frac{1}{-2\pi i}\partial_{\xi}\left(\frac{1}{\sqrt{1+i\lambda}}e^{-\frac{\pi\xi^2}{1+i\lambda}}\right)=  \frac{1}{-2\pi i} \frac{-2\pi\xi}{(1+i\lambda)^{\frac{3}{2}}}e^{-\frac{\pi\xi^2}{1+i\lambda}},
		\end{align*} 
		so $\abs{\widecheck{(xg)}(\xi,\lambda)} = \frac{\abs{\xi}}{(1+\lambda^2)^{\frac{3}{4}}}e^{-\frac{\pi\xi^2}{1+\lambda^2}}$. We need to show that $\frac{\abs{\xi}}{(1+\lambda^2)^{\frac{3}{4}}}e^{-\frac{\pi\xi^2}{1+\lambda^2}}\leq\frac{G_1}{\sqrt{\abs{\xi}+\abs{\lambda}}}$ for some constant $G_1>0$, or $\frac{\abs{\xi}^3+\abs{\lambda}\xi^2}{(1+\lambda^2)^{\frac{3}{2}}}e^{-\frac{2\pi\xi^2}{1+\lambda^2}}\leq G_1^2$. Let $y = \frac{\abs{\xi}}{\sqrt{1+\lambda^2}}$, then $\frac{\abs{\xi}^3}{(1+\lambda^2)^{\frac{3}{2}}}e^{-\frac{2\pi\xi^2}{1+\lambda^2}} = y^3e^{-2\pi y^2}$ is bounded. Also, we have $\frac{\abs{\lambda}\xi^2}{(1+\lambda^2)^{\frac{3}{2}}}e^{-\frac{2\pi\xi^2}{1+\lambda^2}} = \frac{\abs{\lambda}}{\sqrt{1+\lambda^2}}y^2e^{-2\pi y^2}$, in which both $\frac{\abs{\lambda}}{\sqrt{1+\lambda^2}}$ and $y^2e^{-2\pi y^2}$ are bounded. We thus get the case when $n=1$.
		Finally, 
		\begin{align*}
			\widecheck{(x^2g)}(\xi,\lambda) &= \int_{\mathbb{R}} x^2e^{-\pi x^2}e^{-2\pi i \xi x}e^{-\pi i \lambda x^2}\diff{x}\\
			&= \left(\frac{1}{-2\pi i}\right)^2\partial^2_{\xi\xi}\left(\frac{1}{\sqrt{1+i\lambda}}e^{-\frac{\pi\xi^2}{1+i\lambda}}\right)\\
			&=  \frac{-1}{4\pi^2} \frac{1}{\sqrt{1+i\lambda}}\left[\frac{4\pi^2\xi^2}{(1+i\lambda)^2}-\frac{2\pi\xi}{1+i\lambda}\right]e^{-\frac{\pi\xi^2}{1+i\lambda}},\end{align*} 
		so $\abs{\widecheck{(x^2g)}(\xi,\lambda)} \leq \frac{1}{4\pi^2}\frac{1}{(1+\lambda^2)^{\frac{1}{4}}}\left[ \frac{4\pi^2\xi^2}{1+\lambda^2} + \frac{2\pi\abs{\xi}}{\sqrt{1+\lambda^2}}\right]e^{-\frac{\pi\xi^2}{1+\lambda^2}}$. We just showed that $\frac{1}{4\pi^2}\frac{1}{(1+\lambda^2)^{\frac{1}{4}}} \frac{2\pi\abs{\xi}}{\sqrt{1+\lambda^2}}e^{-\frac{\pi\xi^2}{1+\lambda^2}}\leq \frac{\tilde{G}_{2,1}}{\sqrt{\abs{\xi}+\abs{\lambda}}}$ for some $\tilde{G}_{2,1}>0$ and it suffices to show that $\frac{1}{4\pi^2}\frac{1}{(1+\lambda^2)^{\frac{1}{4}}} \frac{4\pi^2\xi^2}{1+\lambda^2}e^{-\frac{\pi\xi^2}{1+\lambda^2}}\leq \frac{\tilde{G}_{2,2}}{\sqrt{\abs{\xi}+\abs{\lambda}}}$ for some $\tilde{G}_{2,2}>0$. This is equivalent to $\frac{\abs{\xi}^5+\abs{\lambda}\xi^4}{(1+\lambda^2)^\frac{5}{2}}e^{-\frac{2\pi\xi^2}{1+\lambda^2}}\leq \tilde{G}_{2,2}^2$. Again, let $y = \frac{\abs{\xi}}{\sqrt{1+\lambda^2}}$, then $\frac{\abs{\xi}^5}{(1+\lambda^2)^\frac{5}{2}}e^{-\frac{2\pi\xi^2}{1+\lambda^2}} = y^5 e^{-2\pi y^2}$ is bounded. On the other hand, $\frac{\abs{\lambda}\xi^4}{(1+\lambda^2)^\frac{5}{2}}e^{-\frac{2\pi\xi^2}{1+\lambda^2}} = \frac{\abs{\lambda}}{\sqrt{1+\lambda^2}}y^4 e^{-2\pi y^2}$ is also bounded by some absolute constant. We thus finish the claim.
	\end{remark}

\subsection{Window effect on SCT}

The slow decay of CT shown in Lemma \ref{Proposition slow decay of CT} motivates us to consider SCT. It is a natural question to ask if it is possible to ``speed up'' the decay rate by considering a different window? As is considered in the widely applied multitaper technique \cite{Percival:1993} and its generalization \cite{DaWaWu2016}, we know that different windows, particularly the Hermite windows that are widely used in practice, may provide information of a signal from different aspects. Inspired by the Hermite windows and Proposition \ref{prop2}, we consider windows that have higher vanishing order at the origin.
The first result is the following corollary that comes from Proposition \ref{prop2}.

\begin{corollary}\label{corollary of prop2}
	Suppose $f(x)\in C^{2n}(\mathbb{R})\cap \mathcal{S}'(\mathbb{R})$, $n\geq 1$ and $f\in L^1(\mathbb{R})$. If $g(x) = x^{2n}e^{-\pi\alpha x^2}$, for any $t$ and $\xi$, we have  
	\[
	\abs{T_{f}^{(g)}(t,\xi,\lambda)}\leq C\lambda^{-\frac{2n+1}{3}},
	\]
	for some constant $C>0$, where $C$ depends on $n$ and $f$.
\end{corollary}

This result shows that the higher order of the kernel vanishing at $0$, the faster decay the CT is in the chirp rate axis. This is intuitively correct since at time $0$, the chirp component $e^{i\pi \lambda t^2}$ does not oscillate at all near $0$. Thus, if the kernel does not vanish near or at $0$, the integration against $e^{i\pi \lambda t^2}$ does not vanish. See Proposition \ref{prop2} for details.
Now, we provide an analysis of how the windows that vanish at the origin affect the concentration of the SCT of a two-component signal  in the chirp rate direction, particularly at the time that the crossover IF happens.

\begin{proposition}\label{vanishorder}
	Suppose $f(x) = f_1(x) + f_2(x)$, where $f_1(x) = e^{2\pi i \xi_1 x + \pi i \lambda_1 x^2}$ and $f_2(x) = e^{2\pi i \xi_2 x + \pi i \lambda_2 x^2}$, and $\xi_1, \xi_2, \lambda_1, \lambda_2 \in \mathbb{R}$, $\lambda_1 \neq \lambda_2$. We also assume that at time $t_0$, the IFs of $f_1$ and $f_2$ intersect. Set $\bar{\xi}:=\xi_1+\lambda_1 t_0 = \xi_2 +\lambda_2 t_0$. The window function is $g(x) = x^{n}e^{-\pi\alpha x^2}$, where $\alpha>0$ and $n$ is a nonnegative integer. Then there exists a constant $c_{\alpha}$ that depends only on $\alpha$, such that if $\lambda$ is closer to $\lambda_1$, i.e.  $\abs{\lambda-\lambda_1}<\abs{\lambda-\lambda_2}$, we have
	\begin{equation*}
		\abs{\mu_f^{(g)}(t_0,\bar{\xi},\lambda)-\lambda_1}\leq
		\begin{cases}
			c_{\alpha}\left(\frac{\alpha^2+(\lambda-\lambda_1)^2}{\alpha^2+(\lambda-\lambda_2)^2}\right)^{\frac{n+1}{4}}, & \text{if}\ n\ \text{is even}; \\
			c_{\alpha}\left(\frac{\alpha^2+(\lambda-\lambda_1)^2}{\alpha^2+(\lambda-\lambda_2)^2}\right)^{\frac{n+2}{4}}, & \text{if}\ n\ \text{is odd}.
		\end{cases}
	\end{equation*}
\end{proposition}

\begin{proof}
	We have already computed that for $j = 1,2$
	\begin{equation*}
		T_{f_j}^{(e^{-\pi\alpha x^2})}(t,\xi,\lambda) = e^{2\pi i \xi_j t + \pi i \lambda_j t^2}\frac{1}{\sqrt{\alpha+i(\lambda-\lambda_j)}}e^{\frac{-\pi(\xi-\xi_j-\lambda_j t)^2}{\alpha+i(\lambda-\lambda_j)}},
	\end{equation*}
	and we know from (\ref{partialxi}) that 
	\begin{equation*}
		\partial_\xi T_{f_j}^{(e^{-\pi\alpha x^2})}(t,\xi,\lambda) = -2\pi i T_{f_j}^{(xe^{-\pi\alpha x^2})}(t,\xi,\lambda).
	\end{equation*}
	Therefore,
	\begin{align*}
		T_{f_j}^{(g)}(t,\xi,\lambda) &= \left(\frac{-1}{2\pi i}\right)^n e^{2\pi i \xi_j t + \pi i \lambda_j t^2}\frac{1}{\sqrt{\alpha+i(\lambda-\lambda_j)}}\partial_\xi^n\left(e^{\frac{-\pi(\xi-\xi_j-\lambda_j t)^2}{\alpha+i(\lambda-\lambda_j)}}\right)\\
		&= \left(\frac{-1}{2\pi i}\right)^n e^{2\pi i \xi_j t + \pi i \lambda_j t^2}\frac{1}{\sqrt{\alpha+i(\lambda-\lambda_j)}}e^{\frac{-\pi(\xi-\xi_j-\lambda_j t)^2}{\alpha+i(\lambda-\lambda_j)}}P_{n,j}(\xi-\xi_j-\lambda_j t),
	\end{align*}
	where $P_{n,i}(x)$ is a Hermite-like polynomial, and contains only even powers of $x$ if $n$ is even; or contains only odd powers of $x$ if $n$ is odd. We proceed with $g(x) = x^{2n}e^{-\pi\alpha x^2}$ for $n\geq 0$, and omit the computation for $g(x) = x^{2n+1}e^{-\pi\alpha x^2}$ which would be similar. Let us compute the reassignment rule $\mu_f^{(g)}$ at $(t_0,\bar{\xi},\lambda)$. First we have for $j = 1, 2$
	\begin{align*}
		\partial_t T_{f_j}^{(g)}(t,\xi,\lambda) &= (2\pi i \xi_j + 2\pi i \lambda_j t + \frac{2\pi\lambda_j(\xi-\xi_j-\lambda_j t)}{\alpha+i(\lambda-\lambda_j)})T_{f_j}^{(g)}(t,\xi,\lambda)\\
		&\quad+ T_{f_j}^{(g)}(t,\xi,\lambda)\frac{\partial_t P_{2n,j}(\xi-\xi_j-\lambda_j t)}{P_{2n,j}(\xi-\xi_j-\lambda_j t)}.
	\end{align*}
	To simplify the calculation, we denote
	\[
	B_1(t,\xi,\lambda) := T_{f_1}^{(g)}(t,\xi,\lambda), \quad B_2(t,\xi,\lambda) := T_{f_2}^{(g)}(t,\xi,\lambda);
	\]
	\[
	C_j(t,\xi,\lambda) := 2\pi i \xi_j + 2\pi i\lambda_j t+ \frac{2\pi\lambda_j(\xi-\xi_j-\lambda_j t)}{\alpha+i(\lambda-\lambda_j)}, \quad j = 1,2;
	\]
	\[
	D_j(t,\xi,\lambda) :=\partial_t C_j = 2\pi i \lambda_j - \frac{2\pi\lambda_j^2}{\alpha+i(\lambda-\lambda_j)}, \quad j = 1,2;
	\]
	\[
	E_j(t,\xi,\lambda) := \frac{-2\pi(\xi-\xi_j-\lambda_j t)}{\alpha+i(\lambda-\lambda_j)},\quad j = 1,2;
	\]
	\[
	F_j(t,\xi,\lambda) := \partial_t E_j = \frac{2\pi\lambda_j}{\alpha+i(\lambda-\lambda_j)},\quad j = 1,2;
	\]
	\[
	Q_{j}(t,\xi,\lambda) := \frac{\partial_t P_{2n,j}(\xi-\xi_j-\lambda_j t)}{P_{2n,j}(\xi-\xi_j-\lambda_j t)},\quad j=1,2;
	\]
	\[
	\tilde{Q}_{j}(t,\xi,\lambda) := \frac{\partial_{\xi} P_{2n,j}(\xi-\xi_j-\lambda_j t)}{P_{2n,j}(\xi-\xi_j-\lambda_j t)},\quad j=1,2;
	\]
	\[
	Z_{j}(t,\xi,\lambda) := \partial_t Q_j,\quad j=1,2;
	\ \ 
	\tilde{Z}_{j}(t,\xi,\lambda) := \partial_t \tilde{Q}_j,\quad j=1,2.
	\]
	With these notations, we have at $(t,\xi,\lambda)$
	\begin{equation*}
		\partial_t T_{f}^{(g)} = B_1C_1 + B_1Q_1 + B_2C_2 + B_2Q_2,
	\end{equation*}
	and
	\begin{align*}
		&T_{f}^{(g)}\partial_{tt}T_{f}^{(g)} - (\partial_t T_{f}^{(g)})^2\\
		=&\, (B_1C_1^2+2B_1C_1Q_1 + B_1Q_1^2+ B_1D_1 + B_1Z_1 + B_2C_2^2+2B_2C_2Q_2 + B_2Q_2^2 + B_2D_2+B_2Z_2)(B_1+B_2)\\
		&-(B_1C_1 + B_1Q_1 + B_2C_2 + B_2Q_2)^2.
	\end{align*}
	At $(t_0,\bar{\xi},\lambda)$, since $\partial_t P_{2n,i}(\xi-\xi_i-\lambda_i t)$ is a polynomial with only odd powers of $\xi-\xi_i-\lambda_i t$, so $Q_i(t_0,\bar{\xi},\lambda)=0$ for any $\lambda\in\mathbb{R}$ and $i = 1,2$; also $C_1(t_0,\bar{\xi},\lambda) - C_2(t_0,\bar{\xi},\lambda) = 0$. So at $(t_0,\bar{\xi},\lambda)$,
	\begin{flalign*}
		&T_{f}^{(g)}\partial_{tt}T_{f}^{(g)} - (\partial_t T_{f}^{(g)})^2 &\\
		=\,& (B_1C_1^2+ B_1D_1 + B_1Z_1 + B_2C_2^2 + B_2D_2+B_2Z_2)(B_1+B_2)-(B_1C_1 + B_2C_2)^2&\\
		=\,&(B_1D_1 + B_1Z_1 + B_2D_2 + B_2Z_2)(B_1+B_2) + B_1B_2(C_1-C_2)^2&\\
		=\,& (B_1+B_2)(B_1D_1 + B_1Z_1 + B_2D_2 + B_2Z_2).
	\end{flalign*}
	Then we have for $j=1,2$
	\begin{align*}
		\partial_{\xi} T_{f_j}^{(g)}(t,\xi,\lambda) = \left(\frac{-2\pi(\xi-\xi_j-\lambda_j t)}{\alpha+i(\lambda-\lambda_j)}\right)T_{f_j}^{(g)}(t,\xi,\lambda) + T_{f_j}^{(g)}(t,\xi,\lambda)\frac{\partial_{\xi} P_{2n,j}(\xi-\xi_j-\lambda_j t)}{P_{2n,j}(\xi-\xi_j-\lambda_j t)}.
	\end{align*}
	So at any $(t,\xi,\lambda)$,
	\[
	\partial_{\xi} T_{f}^{(g)} = B_1E_1 + B_1 \tilde{Q}_1 + B_2E_2 + B_2\tilde{Q}_2,
	\]
	and 
	\begin{align*}
		&\partial_{t\xi}^2 T_{f}^{(g)}= \partial_{t}\partial_{\xi}T_{f}^{(g)}\\
		=&\, (B_1C_1 + B_1Q_1)E_1 + B_1F_1 + (B_1C_1 + B_1Q_1)\tilde{Q}_1 + B_1 \tilde{Z}_1 + (B_2C_2 + B_2Q_2)E_2+ B_2F_2\\
		&\qquad + (B_2C_2 + B_2Q_2)\tilde{Q}_2 + B_2\tilde{Z}_2.
	\end{align*}
	Similar to $Q_i$, at $(t_0,\bar{\xi}, \lambda)$, we have $\tilde{Q}_i(t_0,\bar{\xi},\lambda)=0$ for any $\lambda\in\mathbb{R}$ and $i = 1,2$. Moreover, we have $E_i(t_0,\bar{\xi},\lambda)=0$, and hence $\partial_{\xi}T_f^{(g)}(t_0,\bar{\xi},\lambda)=0$. It is easy to check that $Z_i(t_0,\bar{\xi},\lambda) = -\lambda_i \tilde{Z}_i(t_0,\bar{\xi},\lambda)$. Therefore, we have at $(t_0,\bar{\xi},\lambda)$
	\begin{align*}
		T_{f}^{(g)}\partial_{t\xi}^2T_{f}^{(g)}- \partial_{\xi}T_{f}^{(g)}\partial_t T_{f}^{(g)}= (B_1+B_2)(B_1F_1+B_1\tilde{Z}_1 +B_2F_2 + B_2\tilde{Z}_2).
	\end{align*}
	As a result, at $(t_0,\bar{\xi},\lambda)$, we have
	\begin{align*}
		\abs{\mu_f^{(g)}-\lambda_1}
		&\leq\abs{\frac{\partial_t\left(\frac{\partial_t T_f^{(g)}}{T_f^{(g)}}\right)}{2\pi i - \partial_t\left(\frac{\partial_{\xi}T_f^{(g)}}{T_f^{(g)}}\right)}-\lambda_1}\\
		&= \abs{\frac{T_{f}^{(g)}\partial_{tt}T_{f}^{(g)} - (\partial_t T_{f}^{(g)})^2}{2\pi i (T_{f}^{(g)})^2-(T_{f}^{(g)}\partial_{t\xi}^2T_{f}^{(g)}- \partial_{\xi}T_{f}^{(g)}\partial_t T_{f}^{(g)})}-\lambda_1}\\
		& = \abs{\frac{B_1D_1+B_1Z_1+B_2D_2+B_2Z_2}{2\pi i (B_1+B_2)-(B_1F_1+B_1\tilde{Z}_1 +B_2F_2 + B_2\tilde{Z}_2)}-\lambda_1}\\
		&= \abs{\frac{D_1+Z_1+\frac{B_2}{B_1}(D_2+Z_2)}{(2\pi i - F_1 -\tilde{Z}_1) + \frac{B_2}{B_1}(2\pi i - F_2 -\tilde{Z}_2)}-\frac{D_1+Z_1}{2\pi i -F_1-\tilde{Z}_1}}\\
		&= \abs{\frac{B_2}{B_1}}\abs{\frac{(D_2+Z_2)(2\pi i - F_1-\tilde{Z}_1)-(D_1+Z_1)(2\pi i - F_2-\tilde{Z}_2)}{(2\pi i -F_1-\tilde{Z}_1)\big[(2\pi i - F_1 -\tilde{Z}_1) + \frac{B_2}{B_1}(2\pi i - F_2 -\tilde{Z}_2)\big]}}
	\end{align*}
	since $\frac{D_1+Z_1}{2\pi i -F_1-\tilde{Z}_1} = \lambda_1$.
	One can verify that the constant term in $P_{2n,j}(\xi-\xi_j-\lambda_j t)$ is $\frac{c_{2n}}{(\alpha+i(\lambda-\lambda_j))^n}$ for some $c_{2n}\in\mathbb{R}$ and $j=1,2$. Therefore at $(t_0,\bar{\xi},\lambda)$, for the first term, we have
	\begin{align*}
		\abs{\frac{B_2}{B_1}} = \left(\frac{\alpha^2+(\lambda-\lambda_1)^2}{\alpha^2+(\lambda-\lambda_2)^2}\right)^{\frac{1}{4}+\frac{n}{2}}
	\end{align*}
	and the second term $\abs{\frac{(D_2+Z_2)(2\pi i - F_1-\tilde{Z}_1)-(D_1+Z_1)(2\pi i - F_2-\tilde{Z}_2)}{(2\pi i -F_1-\tilde{Z}_1)\left[(2\pi i - F_1 -\tilde{Z}_1) + \frac{B}{A}(2\pi i - F_2 -\tilde{Z}_2)\right]}}$ is uniformly bounded in $2n$ for fixed $\alpha$, since only magnitudes of $Z_i$ and $\tilde{Z}_i$ grow with $2n$, and in the numerator $-Z_2\tilde{Z}_1 + Z_1\tilde{Z}_2 = 0$.
	Hence
	$\abs{\mu_f^{(g)}(t_0,\bar{\xi},\lambda)-\lambda_1}\leq c_{\alpha}\left(\frac{\alpha^2+(\lambda-\lambda_1)^2}{\alpha^2+(\lambda-\lambda_2)^2}\right)^{\frac{1}{4}+\frac{n}{2}}$, where $c_{\alpha}$ only depends on $\alpha$.
\end{proof}

\subsection{Reconstruction}\label{reconstruction section}

A common mission in TF analysis is to reconstruct the constituent components from the recorded signal. In the SST setup, provided the window function $g\in\mathcal{S}(\mathbb{R})$ does not vanish at the origin, the reconstruction $\tilde{f}_k(x)$ of the $k$-th component $f_k(x)$ of the original signal $f(x) = \sum_{i=1}^K f_i(x)$ is 
\[
\tilde{f}(x) = \frac{1}{g(0)}\int_{\abs{\xi-\phi_k'(x)}\leq \delta}S_{f}^{1,(g,\alpha)}(t,\xi)\diff{\xi}
\]
for some small $\delta>0$. Replacing $S_{f}^{1,(g,\alpha)}(t,\xi)$ by $S_{f}^{2,(g,\alpha)}(t,\xi)$ gives us a reconstruction by the 2nd-order SST. Since the synthesis formula for the SST and the 2nd-order SST are based on the inverse Fourier transform, it fails at the time when the IFs of two components cross. 

In \cite{li2021chirplet}, a reconstruction algorithm based on the CT was proposed to handle this situation. We briefly summarize the ad hoc argument and the algorithm here. First, each component $f_k(x) = A_k(x)e^{2\pi i \phi_k(x)}$ is approximated locally by a linear chirp: $f_k(t+x)\approx f_k(t)e^{2\pi i \phi_k'(t)x + \pi i \phi_k''(t)x^2}$, and the CT of $f_k(x)$ with window $g$ is thus approximately 
\[
T_{f_k}^{(g)}(t,\xi,\lambda) \approx f_k(t)\widecheck{g}(\xi-\phi_k'(t),\lambda-\phi_k''(t))\,.
\] 
Hence, 
\[
T_{f}^{(g)}(t,\xi,\lambda) \approx \sum_{k=1}^K f_k(t)\widecheck{g}(\xi-\phi_k'(t),\lambda-\phi_k''(t))\,. 
\]
Now, if $\phi_k'(t)$ and $\phi_k''(t)$ can be well approximated by some algorithms; for example, by $\omega_k(t)$ and $\mu_k(t)$ via \eqref{definition omega mu approximation} by SCT in our case, then 
\[
T_{f}^{(g)}(t,\xi,\lambda) \approx \sum_{k=1}^K \widecheck{g}(\xi-\omega_k(t),\lambda-\mu_k(t))f_k(t)\,. 
\]
By setting $(\xi,\lambda) = (\omega_l(t),\mu_l(t))$ for $l=1,\dots,K$, we obtain a linear system:
\[
\mathbf{\hat{X}}_t = \mathbf{A}_t \mathbf{X}_t,
\]
where
\[
\mathbf{\hat{X}}_t =
\begin{bmatrix}
	T_{f}^{(g)}(t,\omega_1(t),\mu_1(t))\\
	T_{f}^{(g)}(t,\omega_2(t),\mu_2(t))\\
	\vdots\\
	T_{f}^{(g)}(t,\omega_K(t),\mu_K(t))
\end{bmatrix},\ 
\mathbf{A}_t =
\begin{bmatrix}
	a_{1,1} & a_{1,2} & \cdots & a_{1,K}\\
	a_{2,1} & a_{2,2} & \cdots & a_{2,K}\\
	\vdots & \vdots & \ddots & \vdots\\
	a_{K,1} & a_{K,2} & \cdots & a_{K,K}
\end{bmatrix},\ 
\mathbf{X}_t =
\begin{bmatrix}
	f_1(t)\\
	f_2(t)\\
	\vdots\\
	f_K(t)
\end{bmatrix}
\]
and $a_{i,j} = \widecheck{g}(\omega_i(t)-\omega_j(t),\mu_i(t)-\mu_j(t))$, which can be obtained via an explicit expression of $\widecheck{g}(\xi,\lambda)$ for some windows. For example, if $g(x) = e^{-\pi \alpha x^2}$, then
$\widecheck{g}(\xi,\lambda) = \frac{1}{\sqrt{\alpha+i\lambda}}e^{\frac{-\pi \xi^2}{\alpha+i\lambda}}$.
The reconstruction of $f_k$ at time $t$ is thus achieved by
\begin{equation}\label{reconstruction formula final}
	e_k^\top \mathbf{X}_t = e_k^\top\mathbf{A}_t^{-1} \mathbf{\hat{X}}_t\,,
\end{equation}
where $e_k$ is a unit vector with $1$ in the $k$-th entry. 

The key to the success of this algorithm is having an accurate estimate of $\phi_k'(t)$ and $\phi_k''(t)$ for the construction of $\mathbf{A}_t$ and $\mathbf{\hat{X}}_t$; in other words, an accurate knowledge of ``ridges'' associated with each component in the TFC domain. As is indicated in Lemma \ref{Proposition slow decay of CT}, the slow decay rate of CT, and hence the low contrast of the ridges and their background, might increase the difficulty of ridge detection. To this end, we propose the following solution based on the SCT.

First, we evaluate the SCT of a given signal $f$ with window $g\in\mathcal{S}(\mathbb{R})$. Second, we extract ridges by the following algorithm. Suppose we know $K\geq 2$, and we want to extract $K\geq 2$ ridges by means of a fast implementation \cite{shen2020scalability} of the (multiway) spectral clustering \cite{von2007tutorial}. Note that $K$ could be determined from the background knowledge. Without the background knowledge, an adaptive estimate of $K$ from the given data is out of the scope of this paper. 
\begin{enumerate}
	\item Select entries in the 3-dim matrix that represent the discretization of $S_f^{(g)}(t,\xi,\lambda)$ whose modules are greater than $q$ quantile of the module over all entries of the 3-dim matrix. Here we set $q=0.9995$ or higher, and denote those entries as $\{x_i\}_{i=1}^n\subset \mathbb{R}^3$.
	
	\item Form the affinity matrix $W$ for the selected points with Gaussian kernel $g(x_i,x_j) = \exp(-\|x_i-x_j\|_2^2/(2\sigma^2))$, and the associated degree matrix $D$. Here we set $\sigma$ to be about the 15\% percentile of all pairwise distances of $\{x_i\}_{i=1}^n$. 
	
	\item Compute the top $2(K-1)$ eigenvectors of the $D^{-1}W$, denoted as $u_1,\ldots,u_{2(K-1)}$. Embed $x_i$ from $\mathbb{R}^3$ to $(u_1(i),u_2(i),\ldots,u_{2(K-1)}(i))^\top \in \mathbb{R}^{2(K-1)}$. In practice, when $n$ is large, we suggest applying the Roseland algorithm \cite{shen2020scalability}. 
	
	\item Run $k$-means on the embedding to split them into $K$ clusters (each refers to one component of the signal). For each cluster, we track back to their original samples in $\mathbb{R}^3$ and set those entries as the ridge of the associated component.
\end{enumerate}
The ridges then represent the estimate $\phi_k'(t)$ and $\phi_k''(t)$ according to Theorem \ref{maintheorem}. We could then plug these estimates into \eqref{reconstruction formula final}. See \cite{ding2021impact} for an argument for the choice of the bandwidth $\sigma$.

\subsection{Numerical implementation of SCT}\label{sec:numericalSCT}

The numerical implementation of SCT is by a direct discretization, which we detail now. The Matlab codes of CT and SCT and those used to generate the figures below in Section \ref{Sect:Numerical} are available in \url{https://github.com/ziyuchen7/Synchrosqueezed-chirplet-transforms}. 

Below we index our vectors and matrices beginning with 1. Suppose the continuous signal $f$ is uniformly sampled over a discrete set
of time points with the sampling interval $\Delta_t > 0$ second. The sampling rate is thus $f_s = \Delta_t^{-1}$. We denote the discretization of $f$ as a column vector $\mathbf{f}\in\mathbb{R}^N$, where $\mathbf{f}(n) = f(n\Delta_t)$, where $1 \leq n\leq N$; that is, we ``record'' the signal for $N\Delta_t$ seconds. 
Choose a discrete window function $\mathbf{h}\in\mathbb{R}^{2K+1}$, where $K\in \mathbb{N}$. For example, $\mathbf{h}$ can be a discretization of a Gaussian, so that $\mathbf{h}(K+1)$ is the center of the Gaussian. 
Write $\mathbf{h}', \mathbf{h}''\in\mathbb{R}^{2K+1}$ for the discretization of the first and second derivative of the window function. Write $\mathbf{th}, \mathbf{th}', \mathbf{t^2h}$ for the discretization of the entry-wise product of $\mathbf{t}$ or $\mathbf{t^2}$ with $\mathbf{h}$ or $\mathbf{h}'$, where $\mathbf{t},\mathbf{t^2}\in\mathbb{R}^{2K+1}$, $\mathbf{t}(n) = n-K-1$ and $\mathbf{t^2}(n) = (n-K-1)^2$, $n = 1,2,\dots, 2K+1$.
Choose a resolution $\alpha>0$, set $M = \floor{\frac{0.5}{\alpha}}$ so that $M+1$ is the number of bins in the (positive) frequency axis and $2M$ is the number of bins in the chirp rate axis of our targeting TFC representation. Then, the CT of $\mathbf{f}$ with window $\mathbf{h}$ would be a matrix $\mathbf{T_f^h}\in\mathbb{C}^{2M\times(M+1)\times N}$, whose entries are
\[
\mathbf{T_f^h}(l,m,n) = \sum_{k=1}^{2K+1}\mathbf{f}(n+k-K-1)\mathbf{h}(k)e^{\frac{-2\pi i (k-1) (m-1)}{2M}}e^{\frac{-\pi i l(k-1)^2 }{4M^2}},
\]
where we define $\mathbf{f}(n)=0$ when $n<1$ or $n>N$, $n=1,2,\dots,N$ is the time index, $m=1,2,\dots,M+1$ is the frequency index and $l = -(M-1),\dots, M$ is the chirp rate index. The CT of $\mathbf{f}$ with windows $\mathbf{h}', \mathbf{h}'', \mathbf{th}, \mathbf{th}', \mathbf{t^2h}$, denoted by $\mathbf{T_f^{h'}}, \mathbf{T_f^{h''}}, \mathbf{T_f^{th}},\mathbf{T_f^{th'}}, \mathbf{T_f^{t^2h}}$ can be defined likewise. Let 
\begin{align*}
	&\mathbf{\lambda_{f}}(l,m,n)\\
	&= \frac{1}{2\pi}(\mathbf{T_f^{h}}(l,m,n)\mathbf{T_f^{h''}}(l,m,n)-4\pi i \frac{l}{4M^2} \mathbf{T_f^{h}}(l,m,n)\mathbf{T_f^{th'}}(l,m,n) -2\pi i \frac{l}{4M^2} (\mathbf{T_f^{h}}(l,m,n))^2\\
	&\quad+(2\pi i \frac{l}{4M^2})^2\mathbf{T_f^{h}}(l,m,n)\mathbf{T_f^{t^2h}}(l,m,n)-(\mathbf{T_f^{h'}}(l,m,n))^2 - (2\pi i \frac{l}{4M^2})^2(\mathbf{T_f^{th}}(l,m,n))^2\\
	&\quad +4\pi i \frac{l}{4M^2} \mathbf{T_f^{h'}}(l,m,n)\mathbf{T_f^{th}}(l,m,n))(-\mathbf{T_f^{h}}(l,m,n)\mathbf{T_f^{th'}}(l,m,n)\\
	&\quad +2\pi i\frac{l}{4M^2} \mathbf{T_f^{h}}(l,m,n) \mathbf{T_f^{t^2h}}(l,m,n)+ \mathbf{T_f^{th}}(l,m,n)\mathbf{T_f^{h'}}(l,m,n)\\
	&\quad-2\pi i\frac{l}{4M^2} (\mathbf{T_f^{th}}(l,m,n))^2)^{-1}.
\end{align*}
We choose a threshold $\nu>0$ and calculate the reassignment operator: 
\begin{align*}
	\mathbf{\mu_{f}}(l,m,n) = 4M^2 \text{Im}(\mathbf{\lambda_{f}}(l,m,n)),
\end{align*}
and
\begin{flalign*}
	\mathbf{\omega_f}(l,m,n)
	&= m - (2M)\text{Im}\left( \frac{-\mathbf{T_f^{h'}}(l,m,n)}{2\pi \mathbf{T_f^{h}}(l,m,n)} - \frac{(\frac{il}{4M^2} - \mathbf{\lambda_{f}}(l,m,n)) \mathbf{T_f^{th}}(l,m,n) }{\mathbf{T_f^{h}}(l,m,n)}\right),
\end{flalign*}
if $\abs{\mathbf{T_f^{h}}(l,m,n)}>\nu$; otherwise set $\mathbf{\mu_{f}}(l,m,n) = -\infty$ and $\mathbf{\omega_f}(l,m,n) = -\infty$. The SCT of $\mathbf{f}$, a matrix $\mathbf{S_f}(l,m,n)\in\mathbb{C}^{2M\times(M+1)\times N}$ is finally given by 
\[
\mathbf{S_f}(l,m,n) = \sum_{r,j: \mathbf{\mu_{f}}(r,j,n) + M = l, \mathbf{\omega_{f}}(r,j,n) = m} \mathbf{T_f^{h}}(r,j,n).
\]
The reconstruction of each component follows immediately.

\section{Numerical results}\label{Sect:Numerical}

In this section, we denote $g_k(x) = x^ke^{-\pi x^2}$ for $k\geq 0$, unless defined otherwise.
To plot the 3-dim TFC representation determined by CT or SCT, we apply the following scheme. Take $q$ to be the $99.99$\% percentile of all entries of the associated discretized TFC representation. Then, plot all entries with magnitudes between different ranges by different gray scales, where we consider ranges $(l-1)q/10$ to $lq/10$, where $l=1,\ldots,10$.

\subsection{Standard linear chirps}
We first look at the synthetic example shown in Figure \ref{fig:intro0}, where the IFs of two ICT components cross at $t_0 = 3$ and $\xi_0 = 24$.  
In Figure \ref{fig:1}, and the TF representation determined by SCT with the standard Gaussian window $g_0(x)$. Compared with the TF representation determined by other methods and CT shown in Figure \ref{fig:intro0}, the TFC representation determined by SCT around the crossover time is clearer. In particular, at time $t_0=3$ when the IFs crossover happens at frequency $\xi_0=24$, the TFC representation provided by SCT is sharper compared with that provided by CT. Specifically, the magnitude at the chirp rate $0$ is close to zero. This shows the impact of the squeezing step. 

\begin{figure}[!htbp]\centering
	\includegraphics[width=0.325\textwidth]{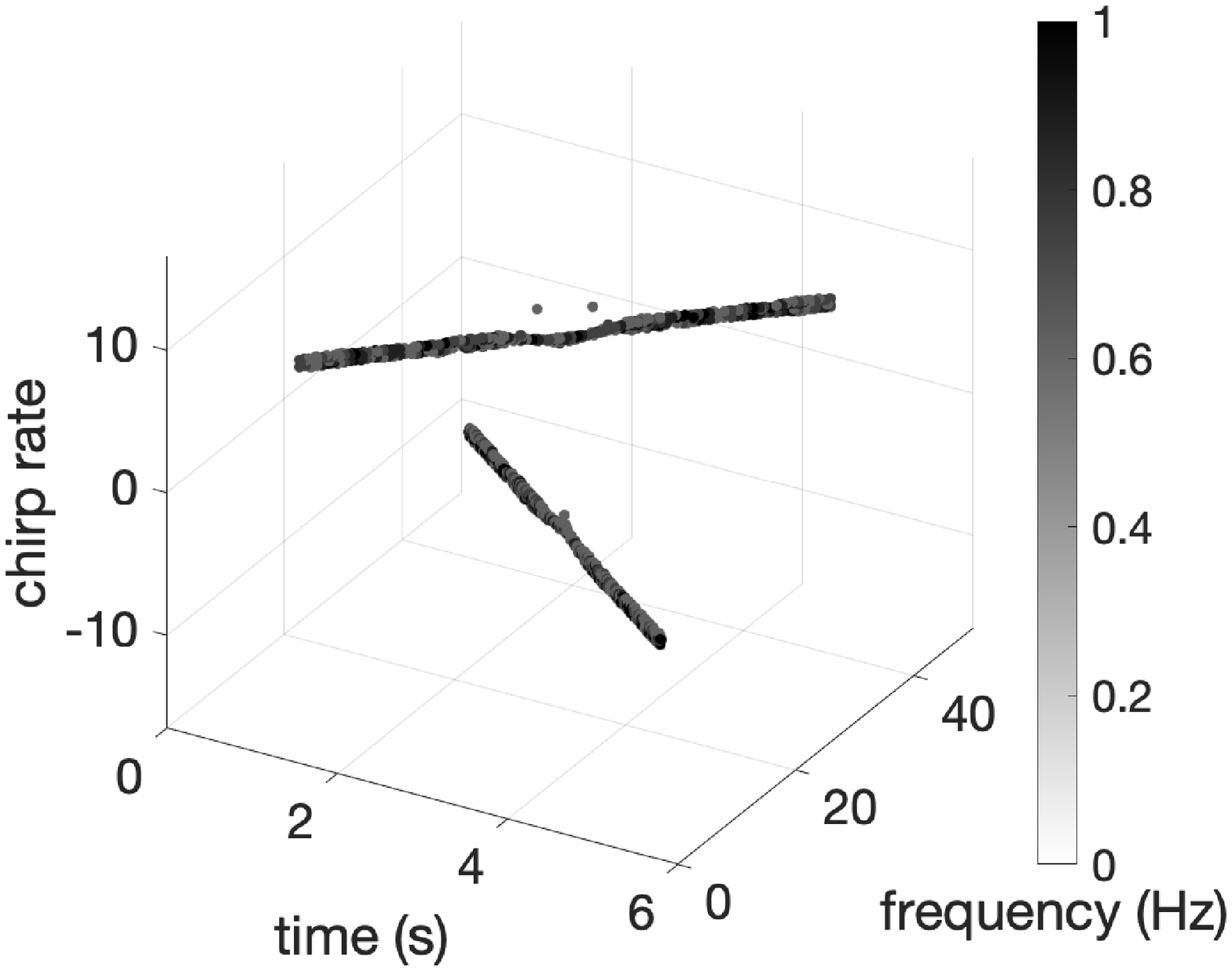}
	\includegraphics[width=0.325\textwidth]{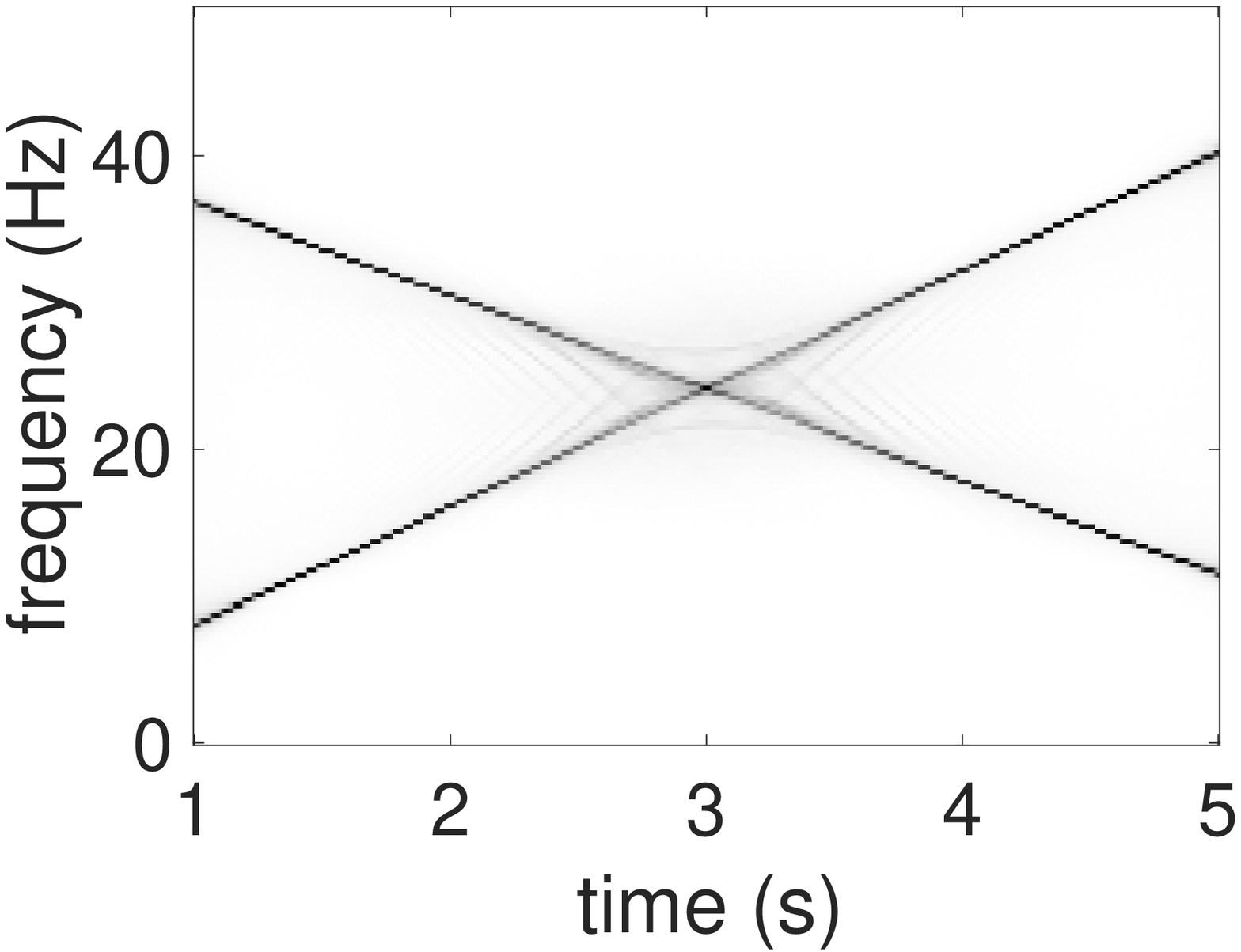}
	\includegraphics[width=.325\textwidth]{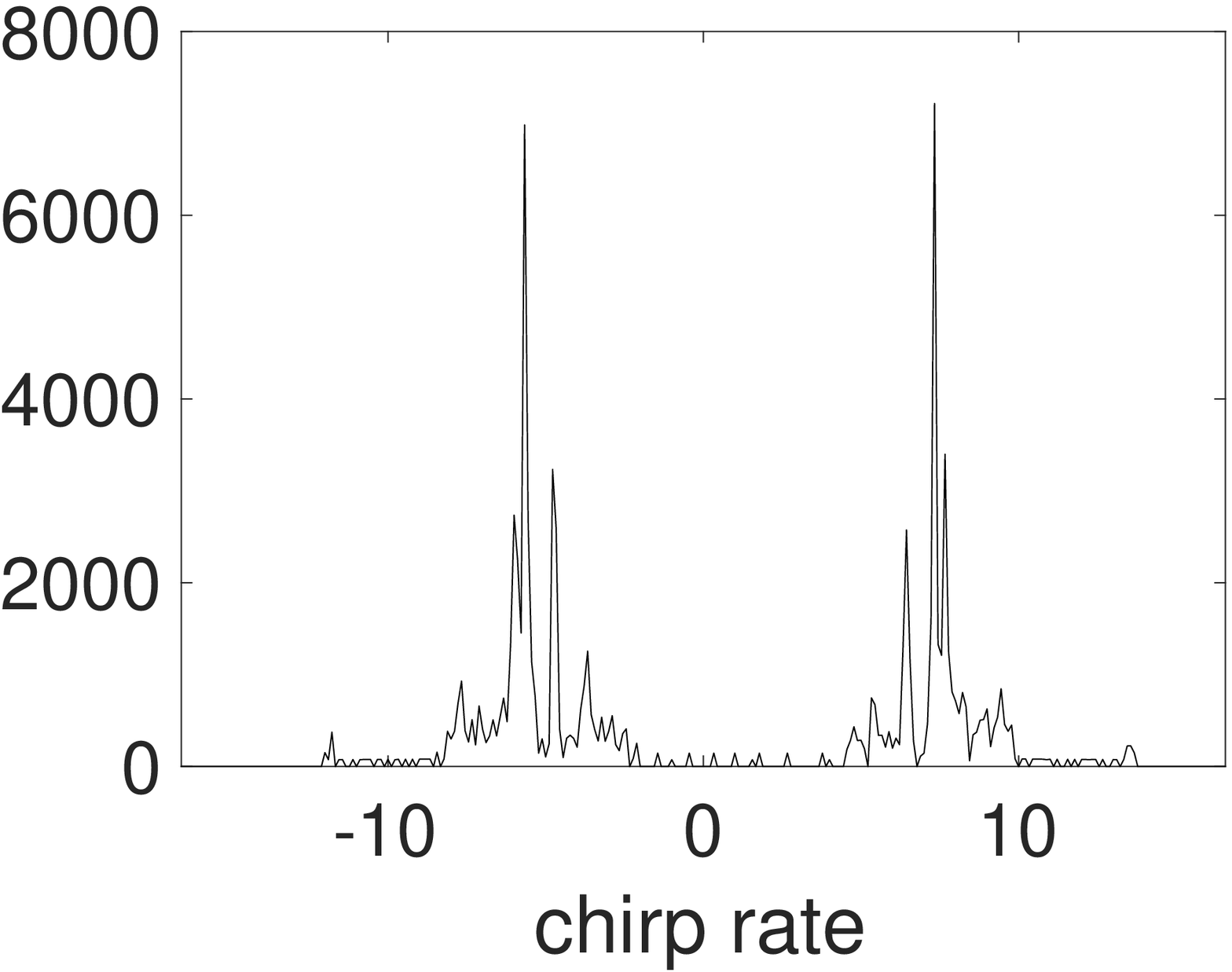}
	\caption{An illustration of SCT with the standard Gaussian window. From left to right: the 3-dim visualization of $\abs{S_f^{(g_0)}(t,\xi,\lambda)}$, the projection of $\abs{S_f^{(g_0)}(t,\xi,\lambda)}$ onto the time-frequency plane, and the plot of $\abs{S_f^{(g_0)}(t_0,\xi_0,\lambda)}$, where $t_0 = 3$ and $\xi_0 = 24$.
		For each time-frequency representation, we set all entries with the magnitude above the 99.99\% percentile of all entries to be the magnitude of the 99.99\% percentile.}
	\label{fig:1}
\end{figure}

Next, we explore the impact of different windows. Figure \ref{fig:2} shows CT and SCT of $f$ with $g_0$ and $g_2$ at different times. We observe that while both CT and SCT provide information about both components at the non-crossing time, $t_1 = 2$s, the SCT with $g_0$ clearly gives a sharper representation of the frequency and chirp rate information. At the crossing time, $t_0 = 3$s, we see a dramatic difference. On the one hand, we cannot easily distinguish the chirp rates of the two components at the crossing frequency 24 Hz when CT is applied. On the other hand, we can get the wanted information when SCT with $g_0$ is applied, and a sharper representation of the chirp rate information when $g_2$ is used as the window. 
We further show the results with $g_2$ as the window in Figure \ref{fig:3}. Compared with CT with $g_0$ shown in Figure \ref{fig:intro0}, we see that CT with $g_2$ gives a sharper TFC representation and hence the TF representation, which reflects the theoretical results shown in Proposition \ref{prop2}. On the other hand, SCT gives us a more concentrated representation of the chirp rate information compared with CT. Moreover, we see that the chirp rate locations of the peaks by SCT with $g_0$ are at $-5.67$ and 7.33, and the peak locations in the SCT with window $g_2$ are at $-6.33$ and 8, and we know that the true chirp rates of $f_1$ and $f_2$ are $-2\pi$ ($\approx -6.28$) and 8. Clearly, the SCT with $g_2$ gives a better estimate of IF and chirp rate.
Moreover, 
two ridges extracted from $|S_f^{(g_2)}(t,\xi,\lambda)|$ by our proposed algorithm in Section \ref{reconstruction section} is shown in Figure \ref{fig:3}.
Furthermore, we compare SCT with different windows, including $g_3(x) = x^2e^{-5\pi x^2}$, $g_4(x) =  x^4e^{-5\pi x^2}$ and $g_5(x) =  x^6e^{-5\pi x^2}$ at the frequency-crossover time. The results are shown in Figure \ref{fig:4}. It shows that $g_4$ seems to give the most concentrated TFC representation in the chirp rate direction in both CT and SCT. Note that while this finding seems to be against Proposition \ref{vanishorder}, which states that the window obtained by multiplying the Gaussian $e^{-\pi x^2}$ with higher power $x^n$ gives us a better TFC representation of SCT in the chirp rate direction, we mention that the analysis is asymptotic, and it is not necessarily guaranteed that increasing $n$ to a finite number with $\alpha$ fixed (as $g_4$ and $g_5$ here) will give us better concentration, particularly under the numerical precision. This result suggests the value of an extensive exploration of the window effect on CT and SCT. However, it is out of the scope of this paper, and more detailed analysis will be reported in our future work.

\begin{figure}[!htbp]\centering
	\includegraphics[width=.32\textwidth]{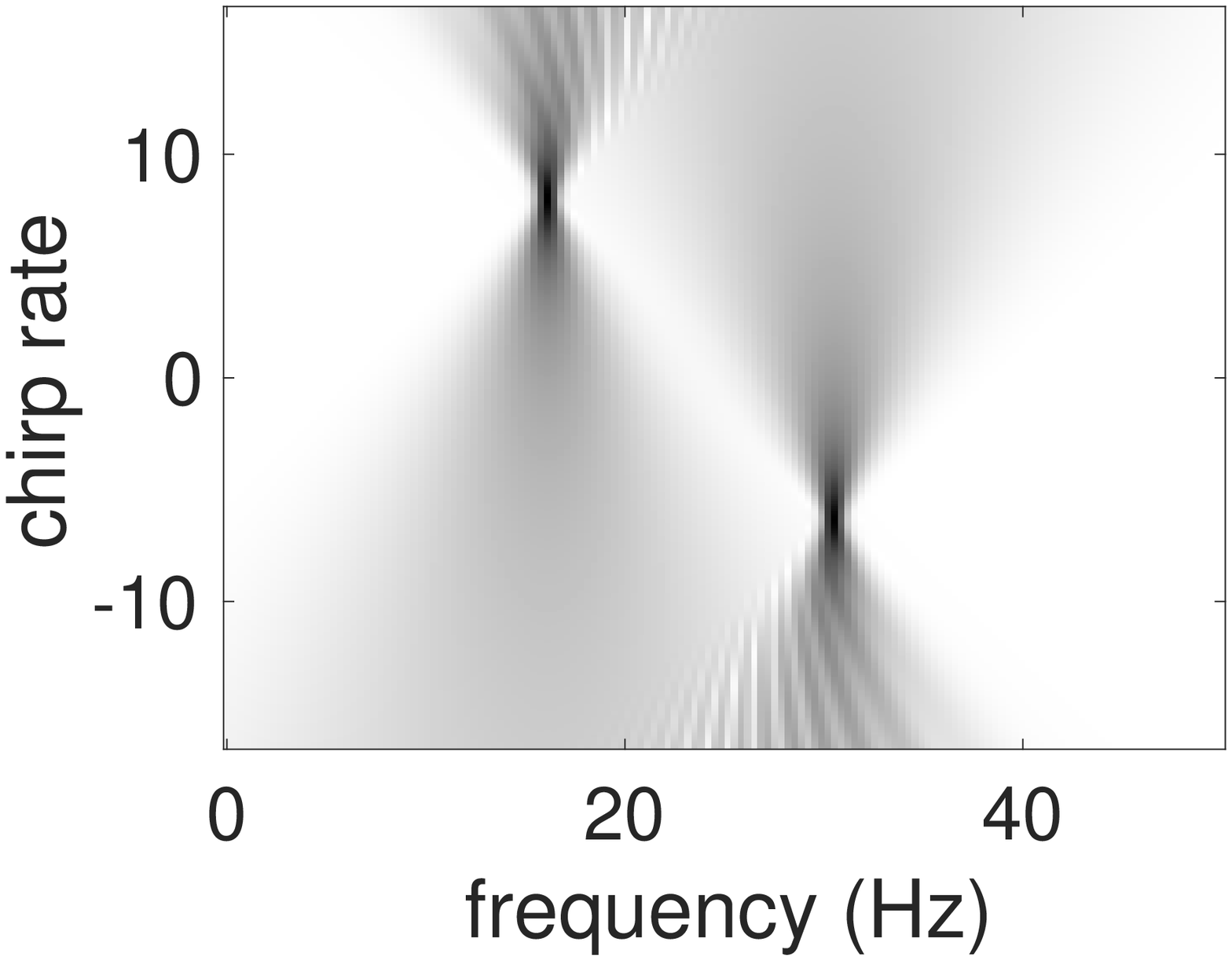}
	\includegraphics[width=.32\textwidth]{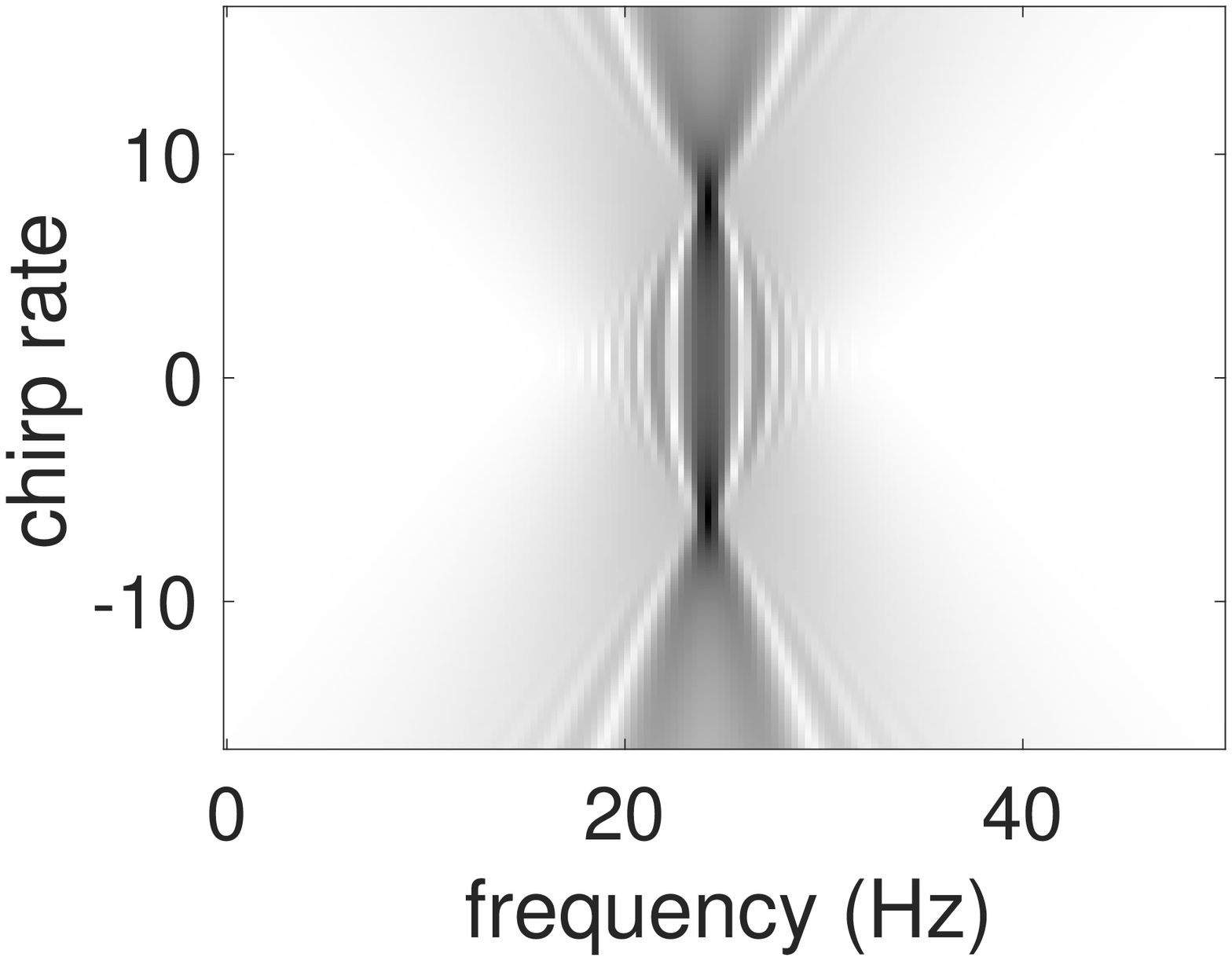}
	\includegraphics[width=.32\textwidth]{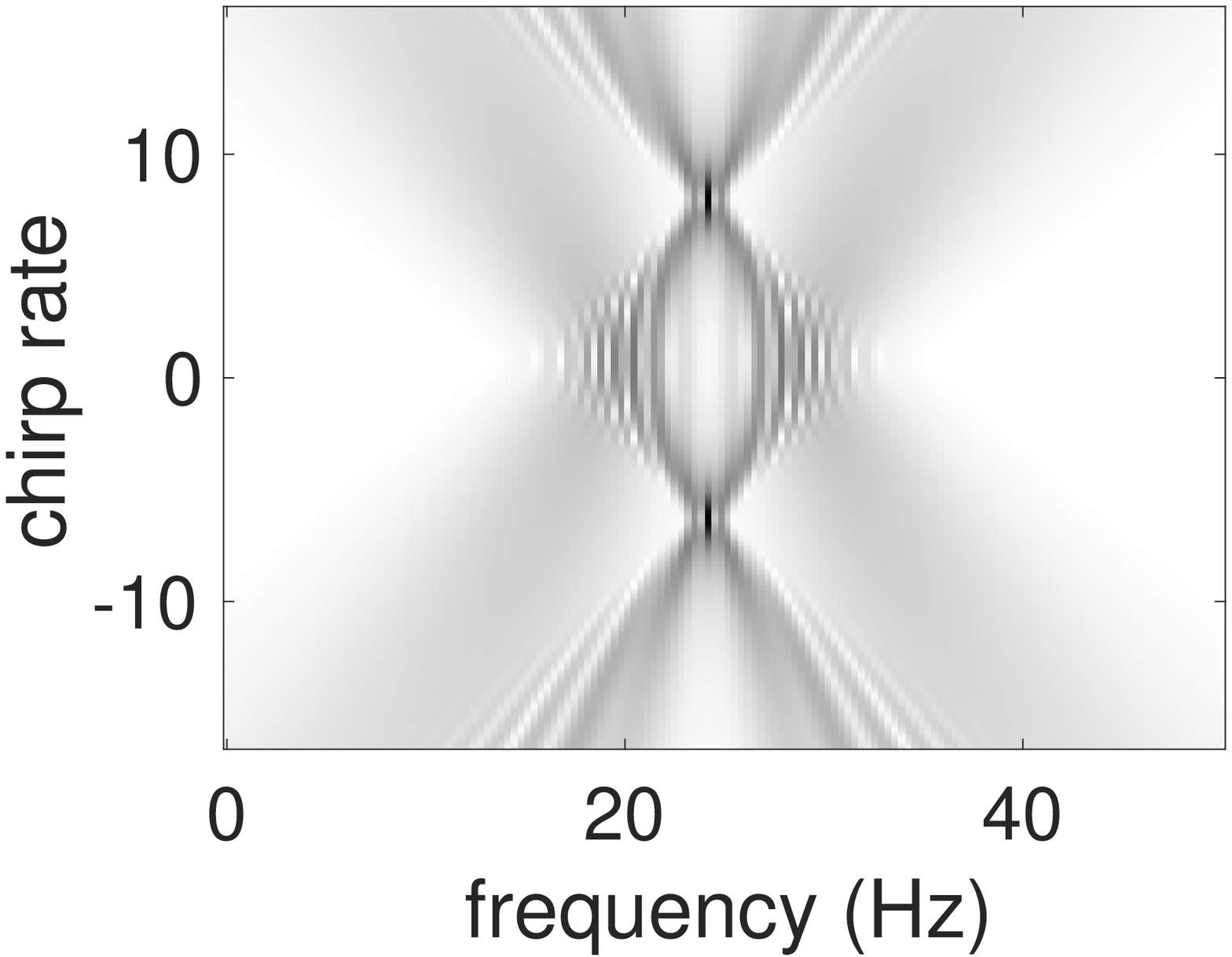}\\
	\includegraphics[width=.32\textwidth]{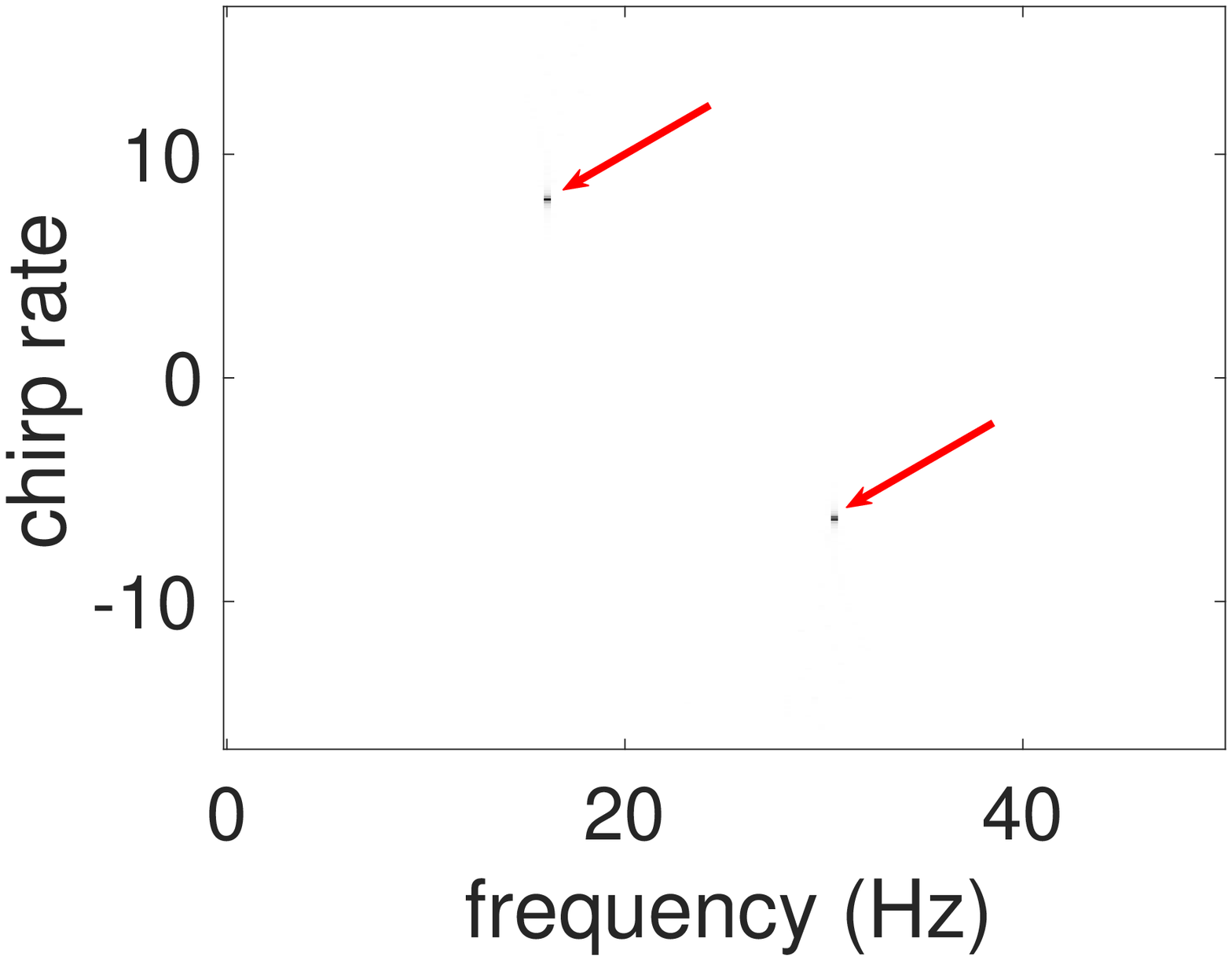}
	\includegraphics[width=.32\textwidth]{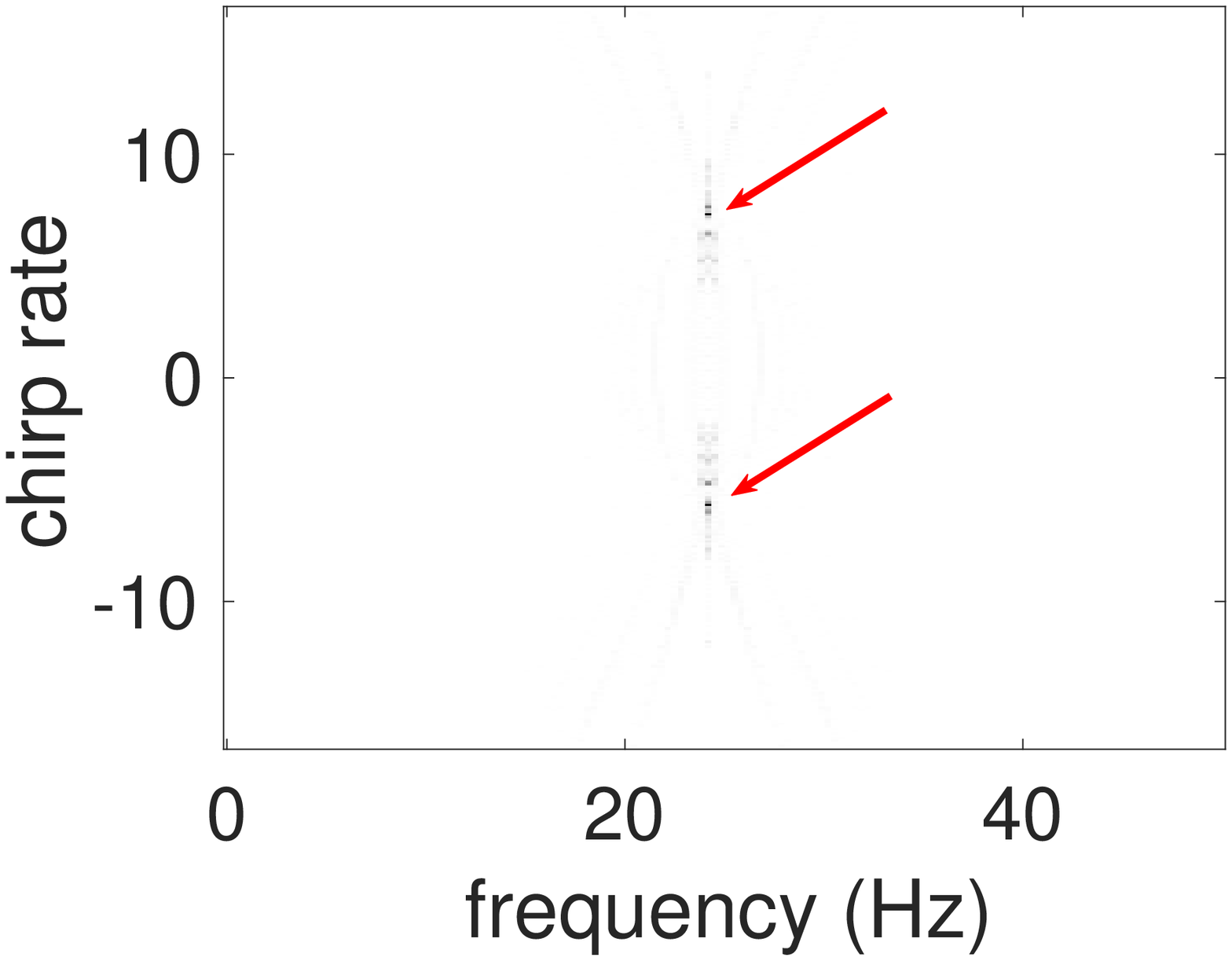}
	\includegraphics[width=.32\textwidth]{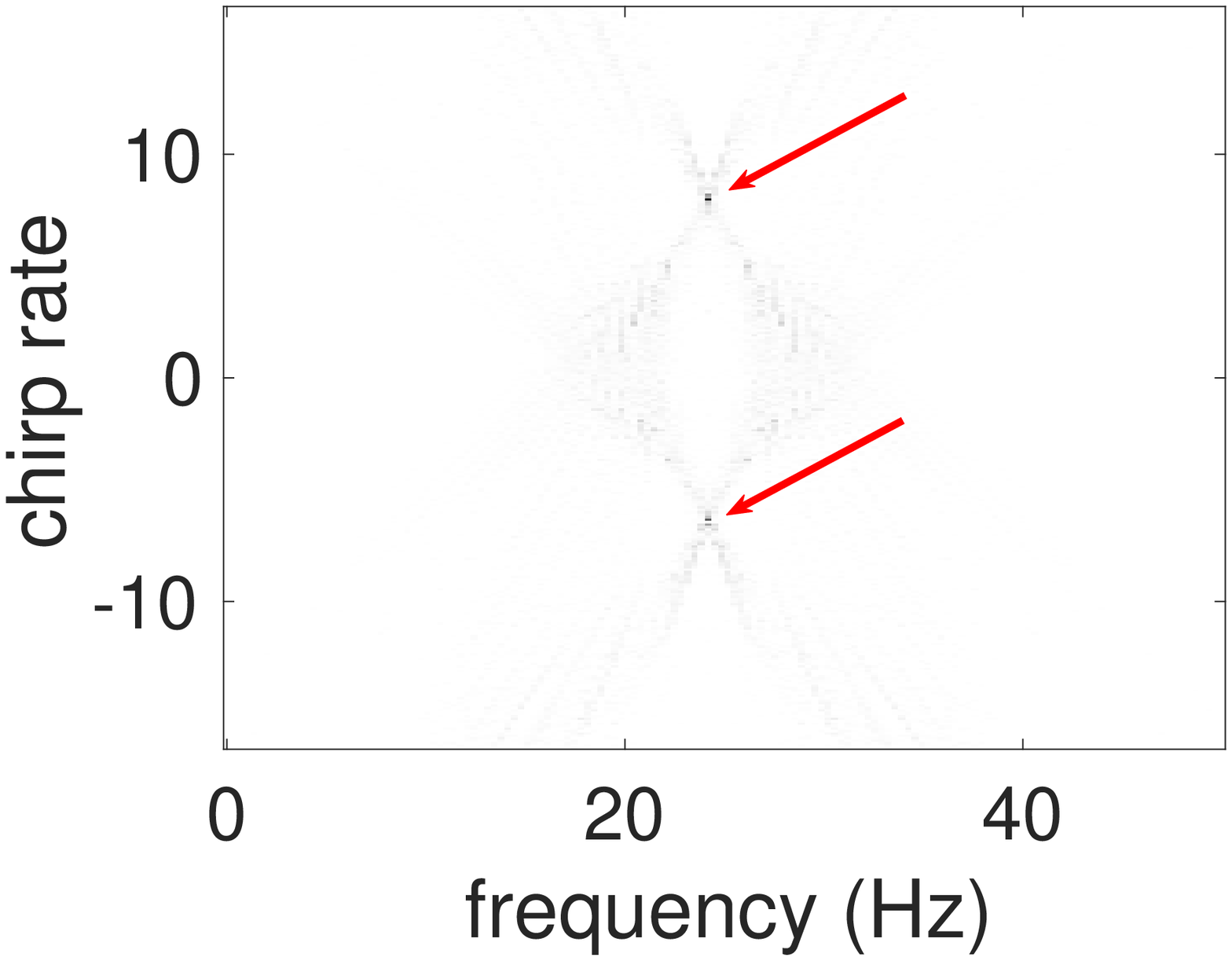}
	\caption{An illustration of the impact of windows. Consider $g_0(x) = e^{-\pi x^2}$ and $g_2(x) = x^2e^{-\pi x^2}$. Top row, from left to right: $\abs{T_f^{(g_0)}(t_1,\xi,\lambda)}$, $\abs{T_f^{(g_0)}(t_0,\xi,\lambda)}$, and $\abs{T_f^{(g_2)}(t_0,\xi,\lambda)}$. 
		Bottom row, from left to right: $\abs{S_f^{(g_0)}(t_1,\xi,\lambda)}$, $\abs{S_f^{(g_0)}(t_0,\xi,\lambda)}$, and $\abs{S_f^{(g_2)}(t_0,\xi,\lambda)}$. To enhance the visualization, the red arrows are superimposed to indicate the squeezed ``spots''.}
	\label{fig:2}
\end{figure}

\begin{figure}[!htb]\centering
	\includegraphics[width=.325\textwidth]{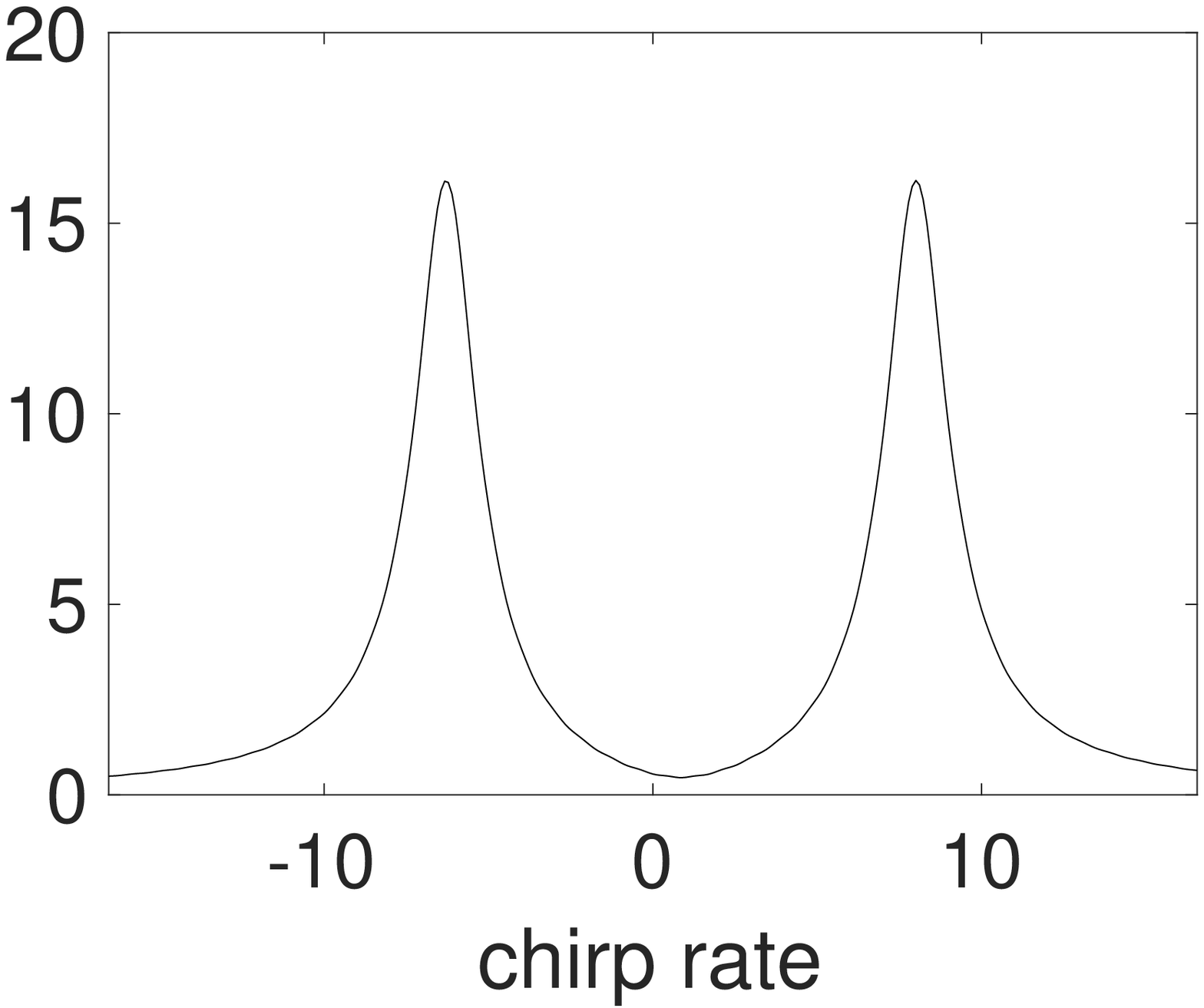}
	\includegraphics[width=.325\textwidth]{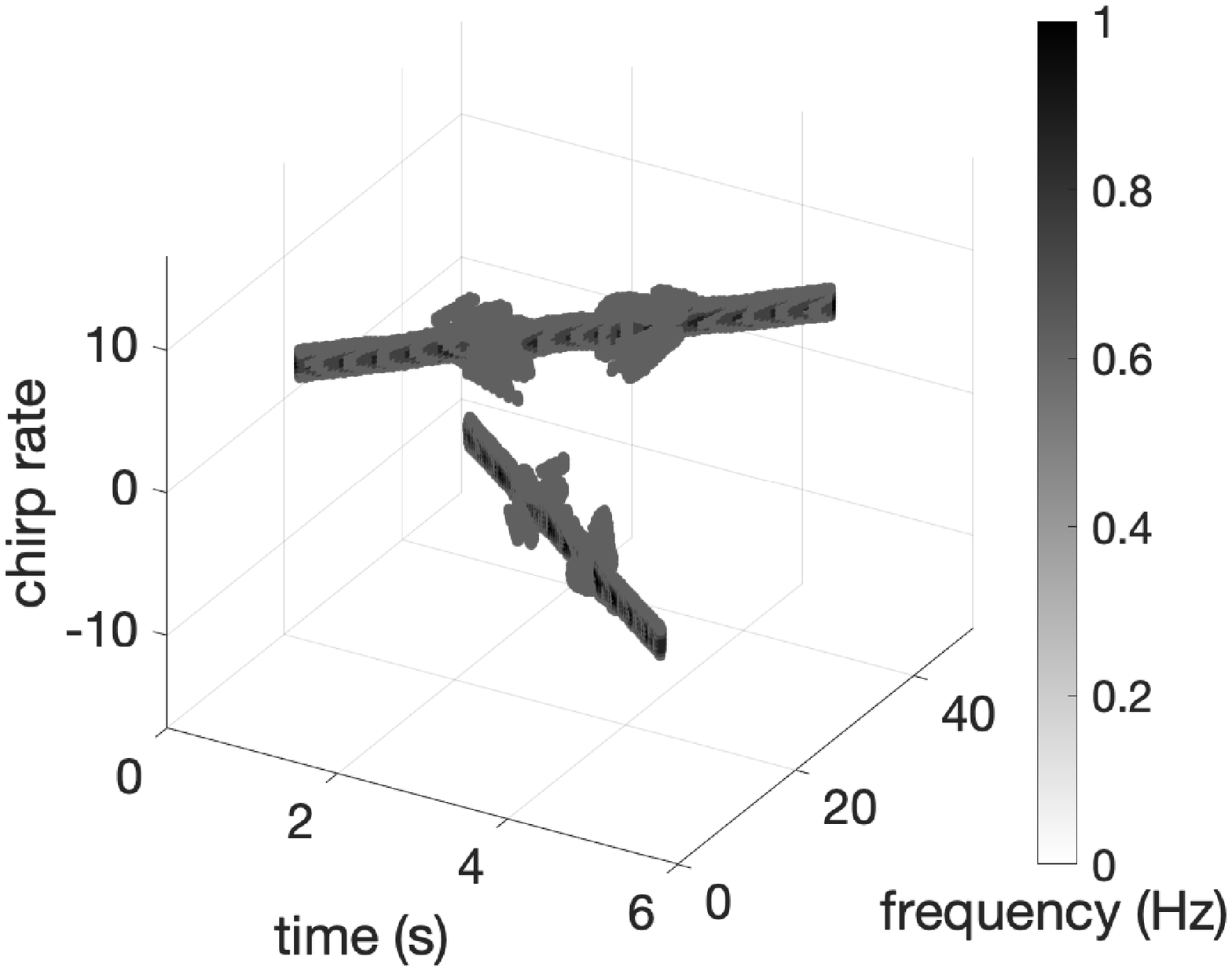}\\
	\includegraphics[width=.325\textwidth]{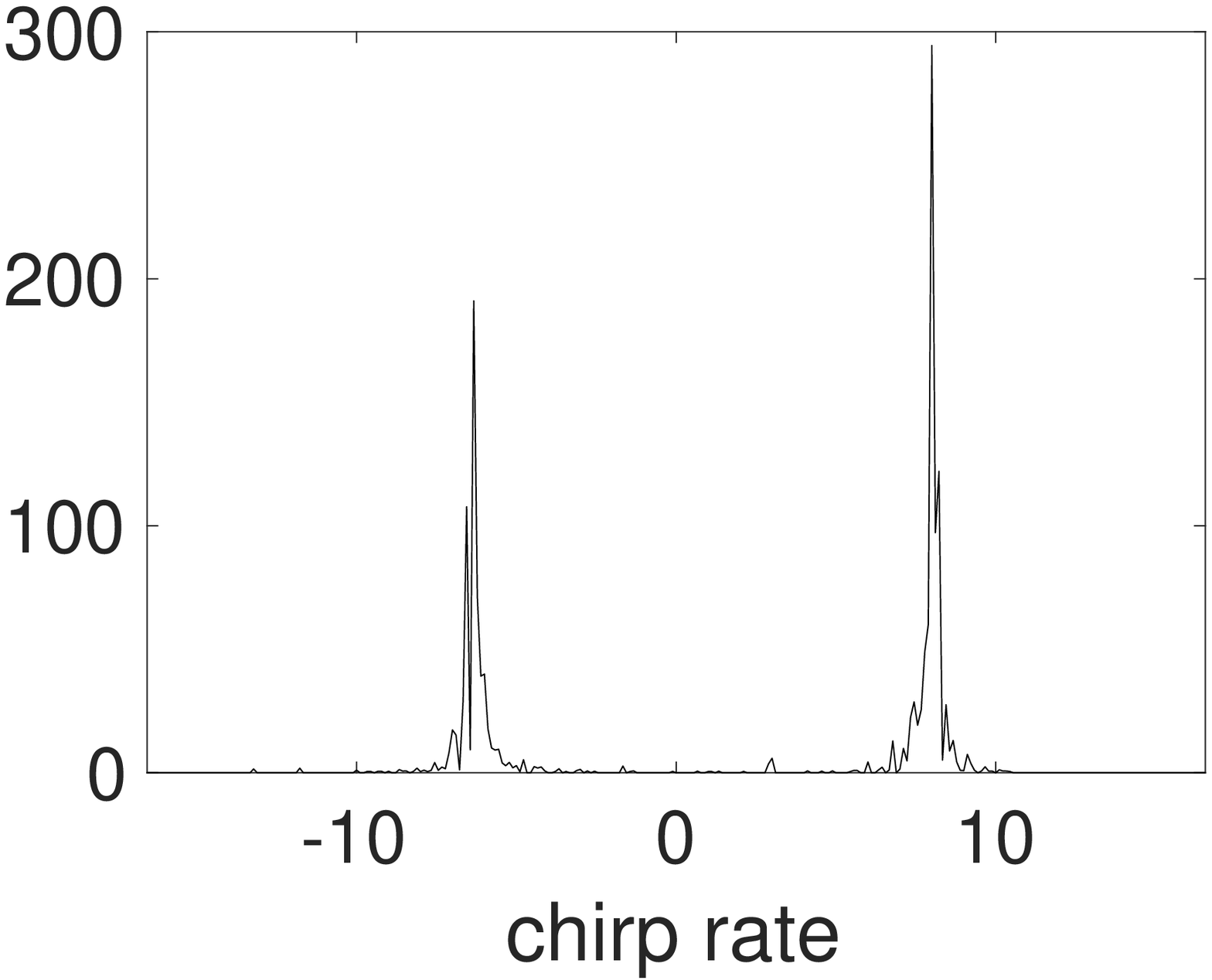}
	\includegraphics[width=.325\textwidth]{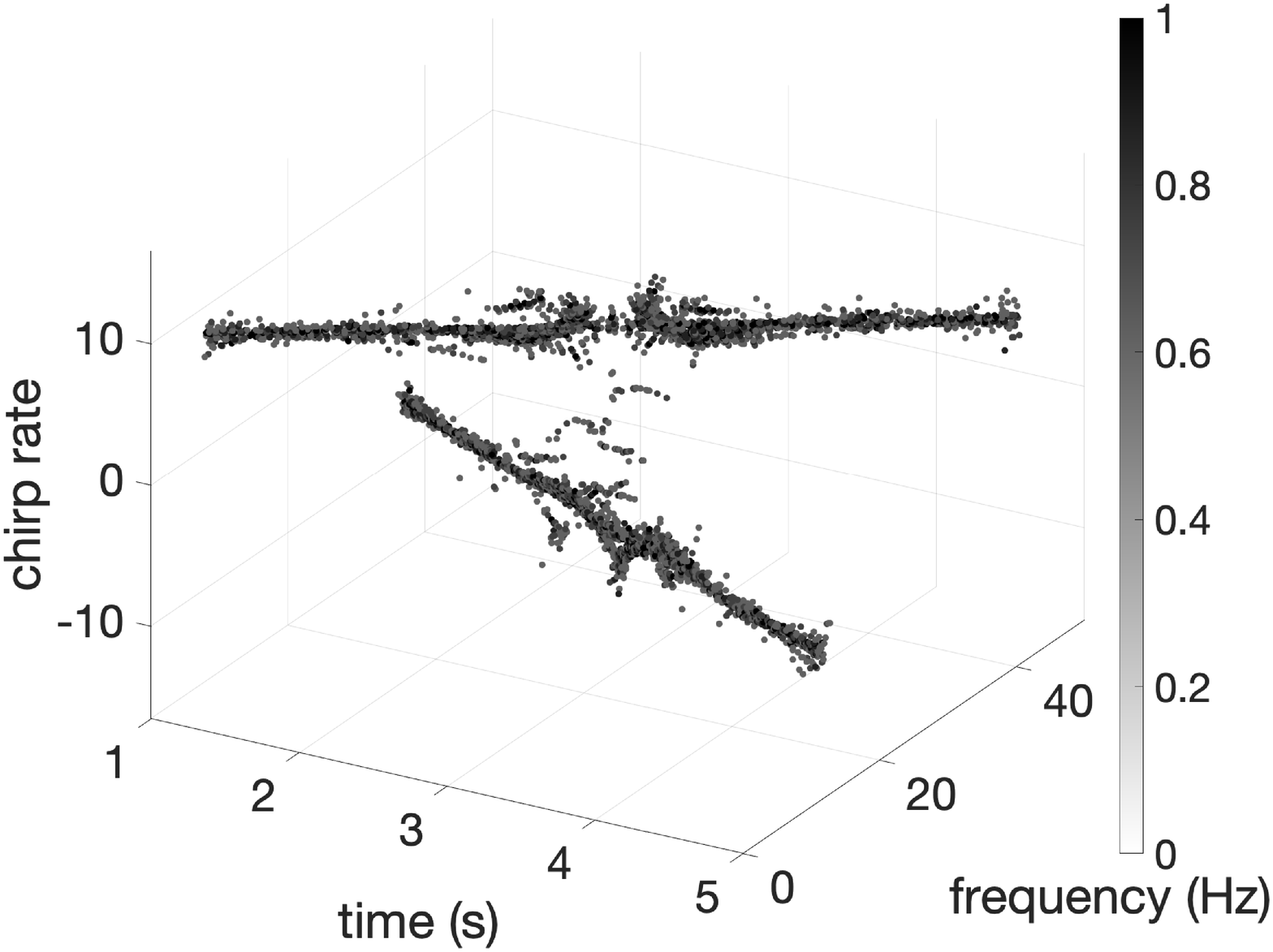}
	\includegraphics[width=.325\textwidth]{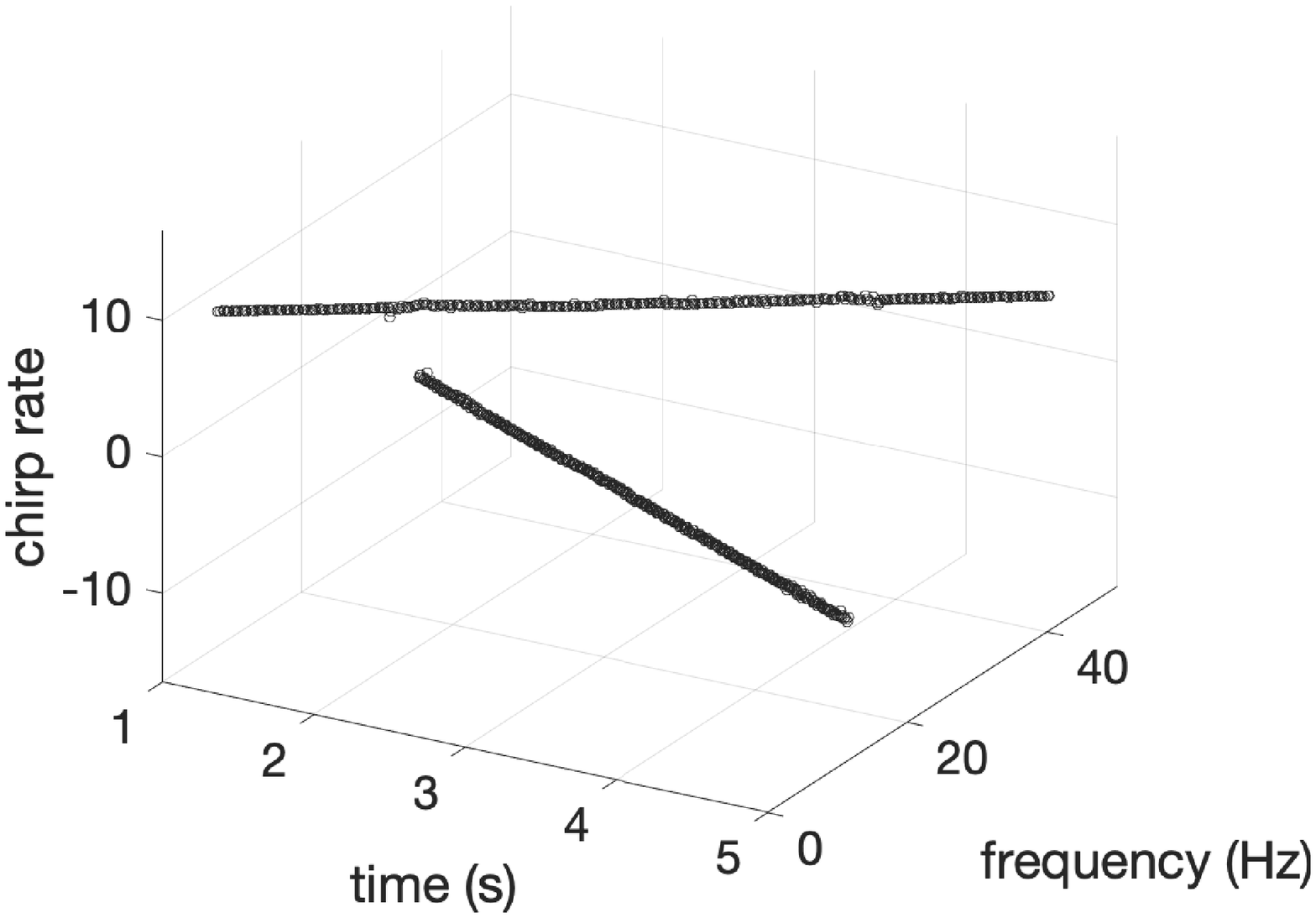}
	\caption{Top row, from left to right: $\abs{T_f^{(g_2)}(t_0,\xi_0,\lambda)}$, and the 3-dim visualization of $\abs{T_f^{(g_2)}(t,\xi,\lambda)}$. Bottom row, from left to right: $\abs{S_f^{(g_0)}(t_0,\xi_0,\lambda)}$, the 3-dim visualization of $\abs{S_f^{(g_2)}(t,\xi,\lambda)}$, and the ridges extracted from $\abs{S_f^{(g_2)}(t,\xi,\lambda)}$. }
	\label{fig:3}
\end{figure}

\begin{figure}[!htbp]
	\centering
	\includegraphics[width=.32\textwidth]{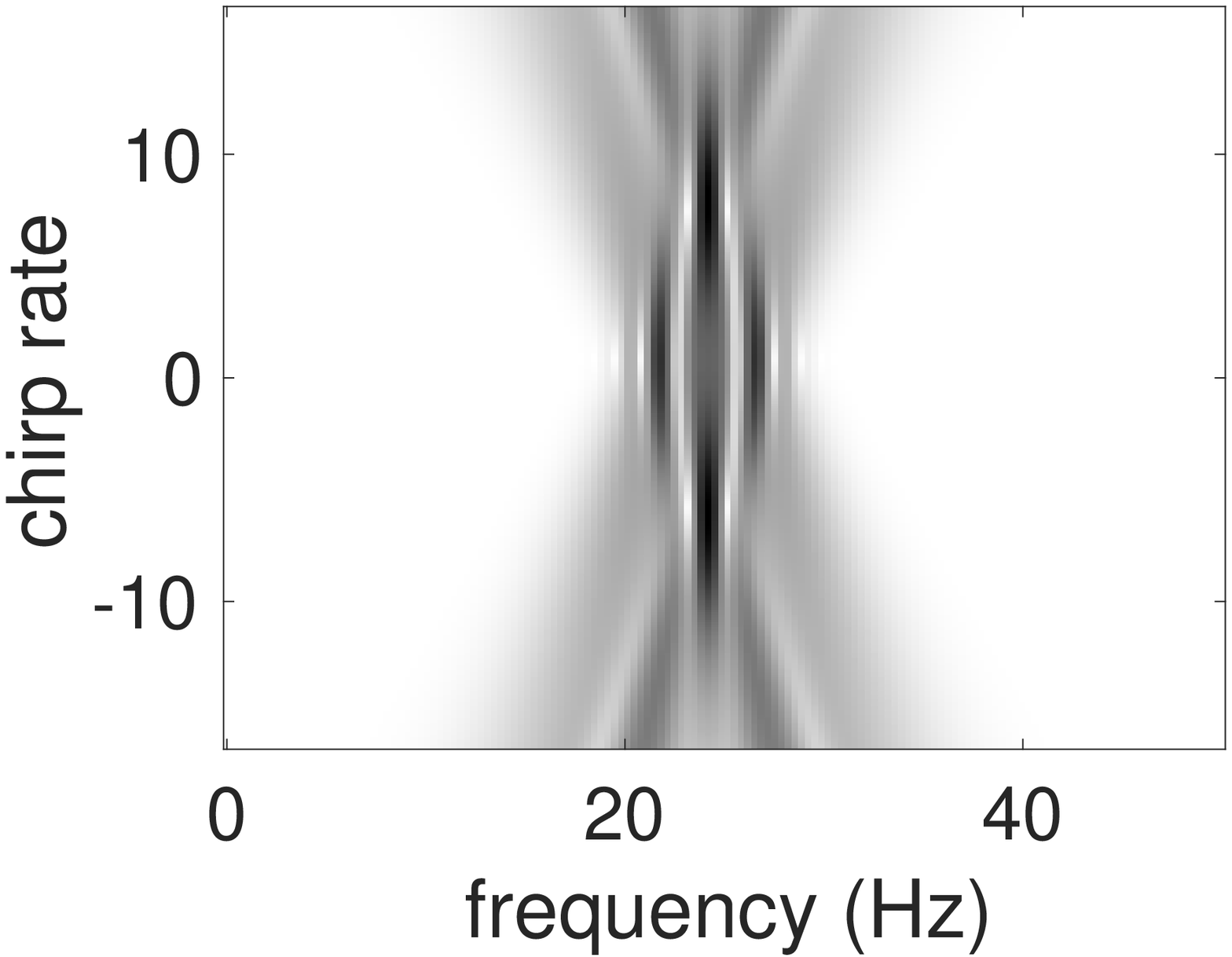}
	\includegraphics[width=.32\textwidth]{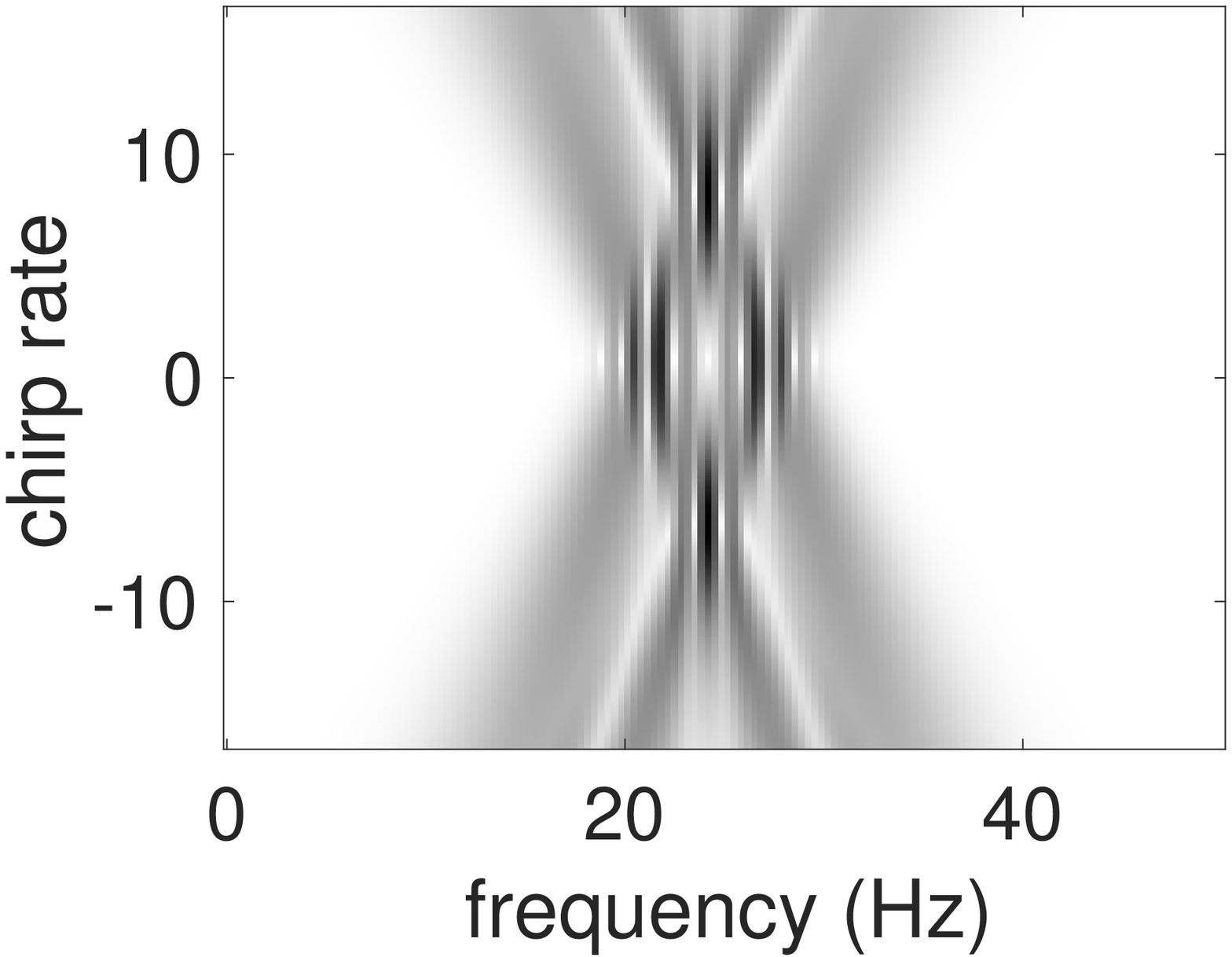}
	\includegraphics[width=.32\textwidth]{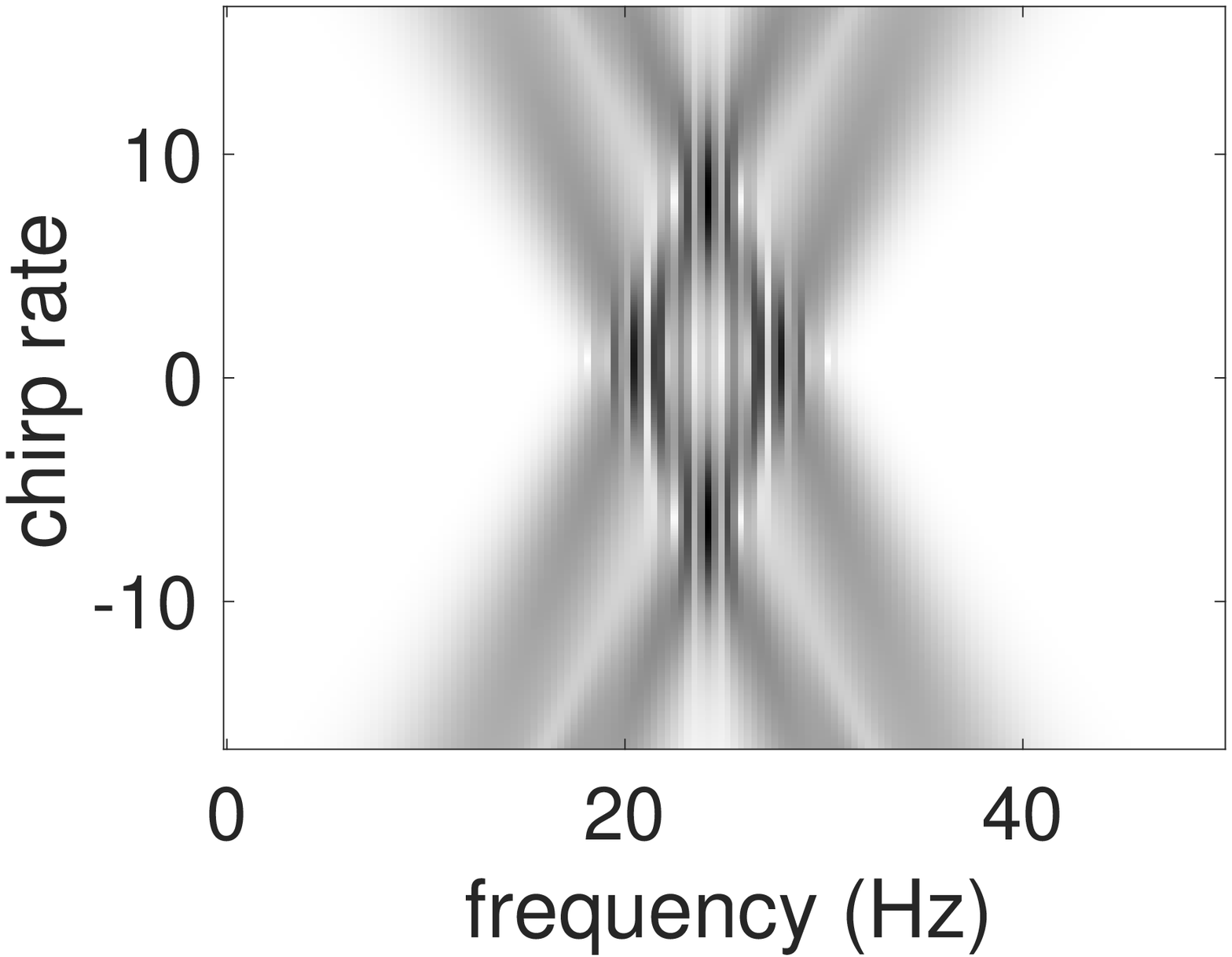}\\
	\includegraphics[width=.32\textwidth]{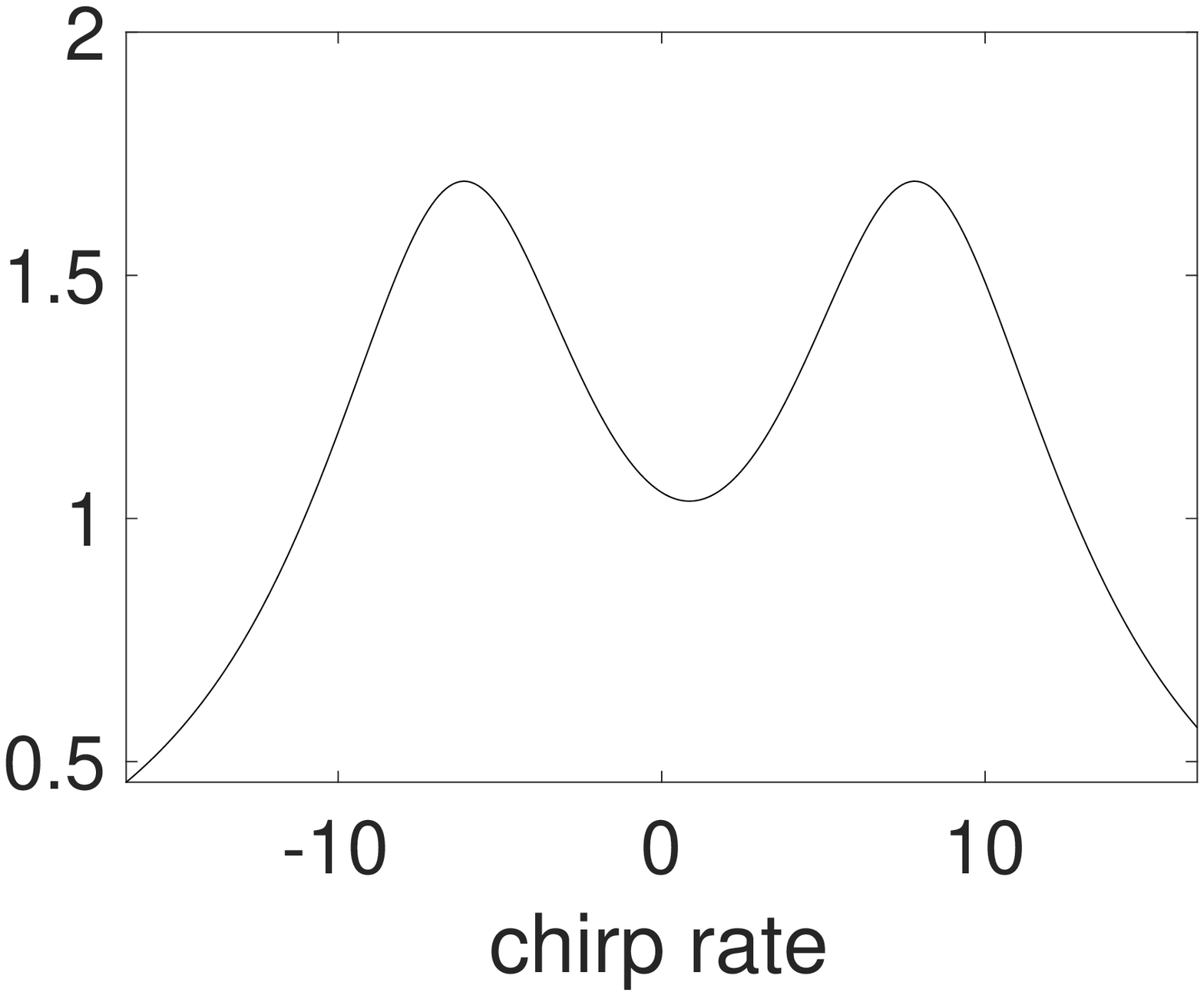}
	\includegraphics[width=.32\textwidth]{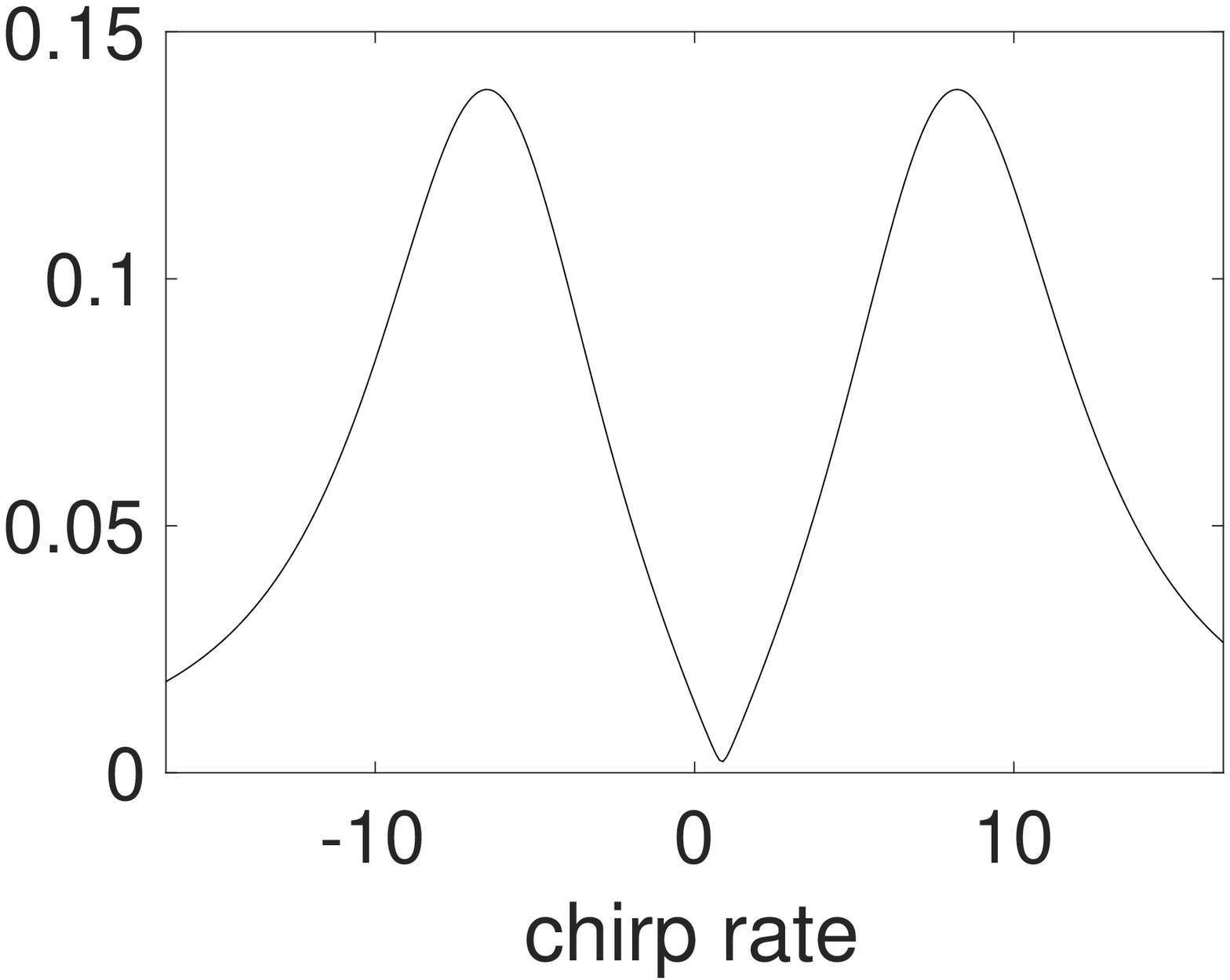}
	\includegraphics[width=.32\textwidth]{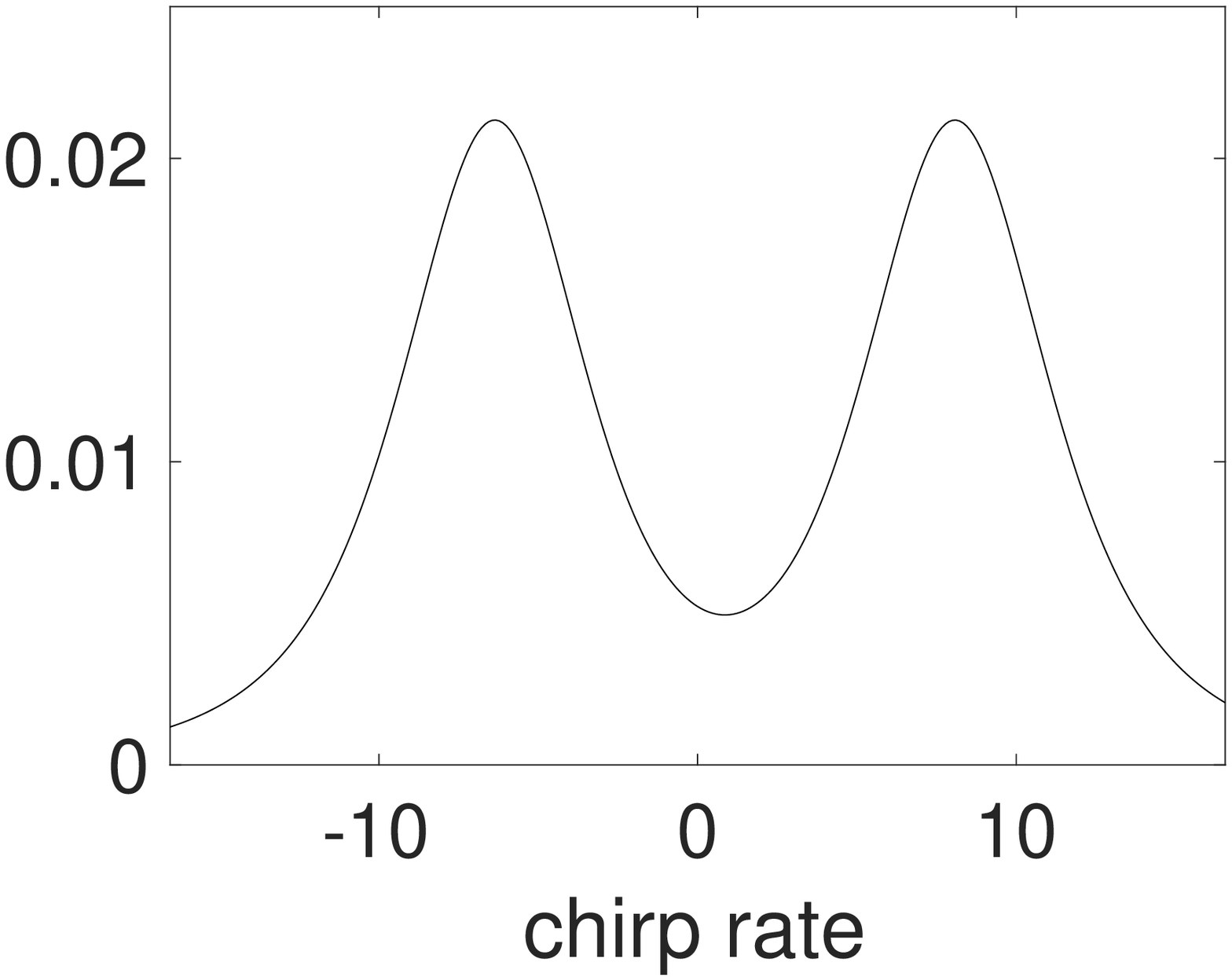}\\
	\includegraphics[width=.32\textwidth]{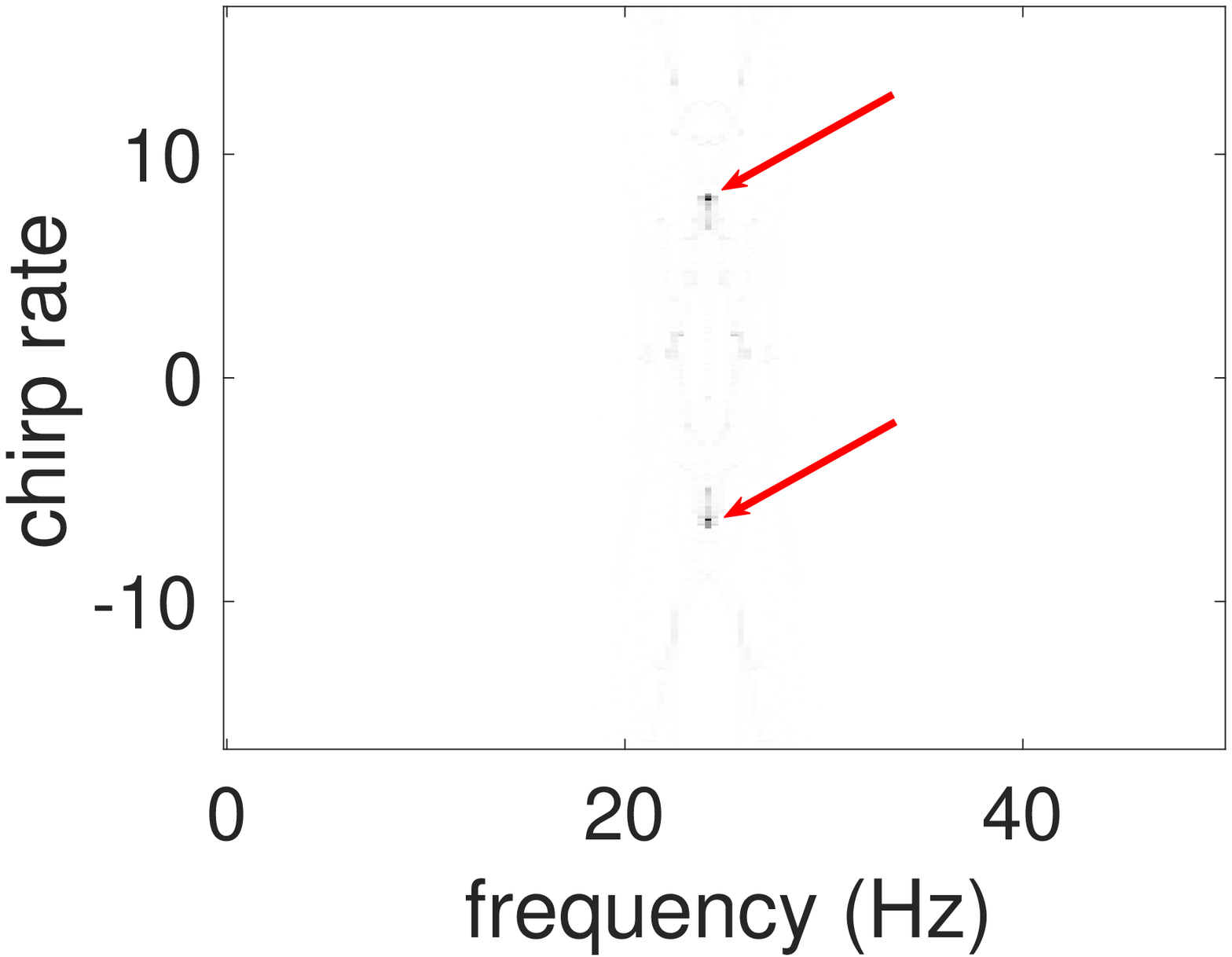}
	\includegraphics[width=.32\textwidth]{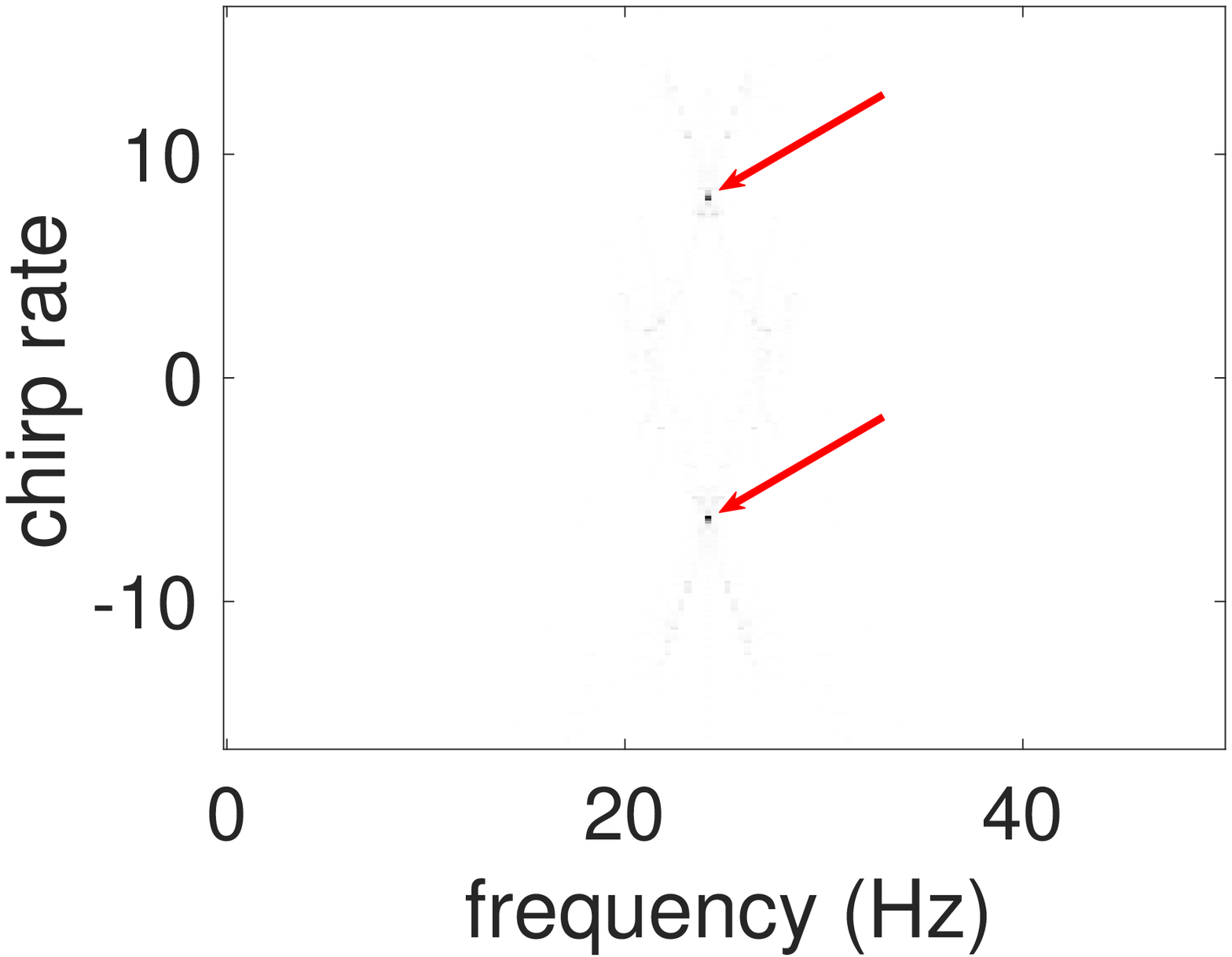}
	\includegraphics[width=.32\textwidth]{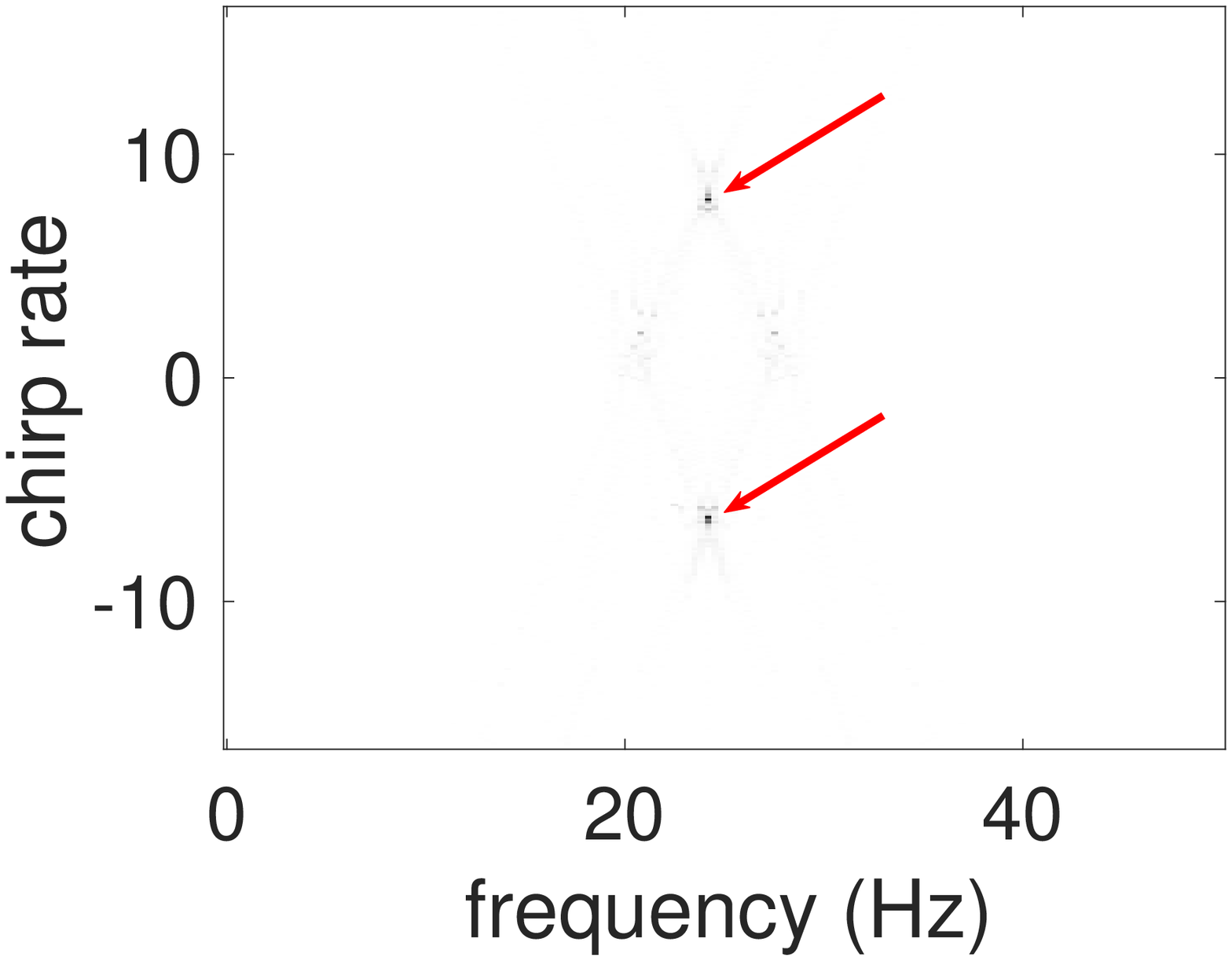}\\
	\includegraphics[width=.32\textwidth]{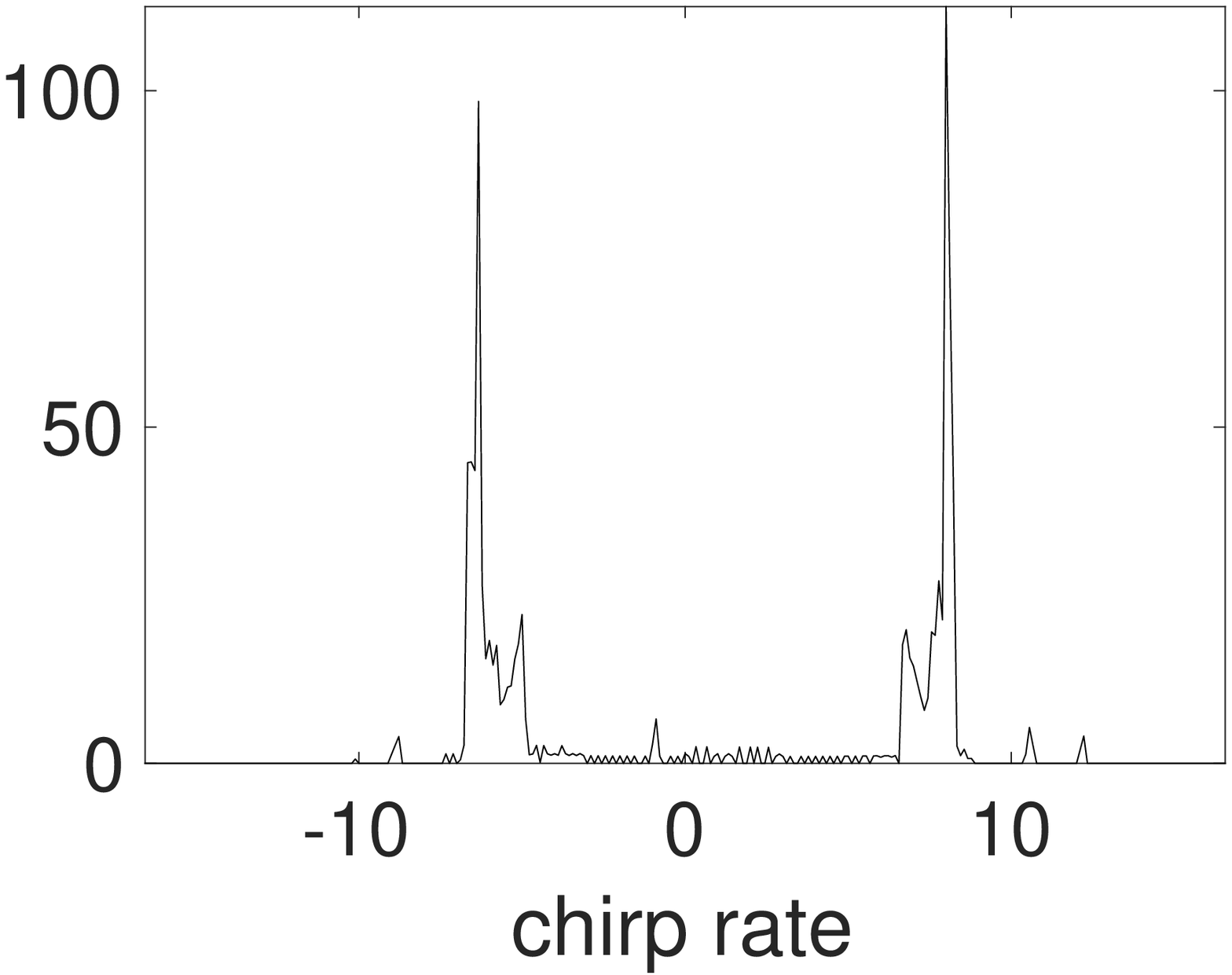}
	\includegraphics[width=.32\textwidth]{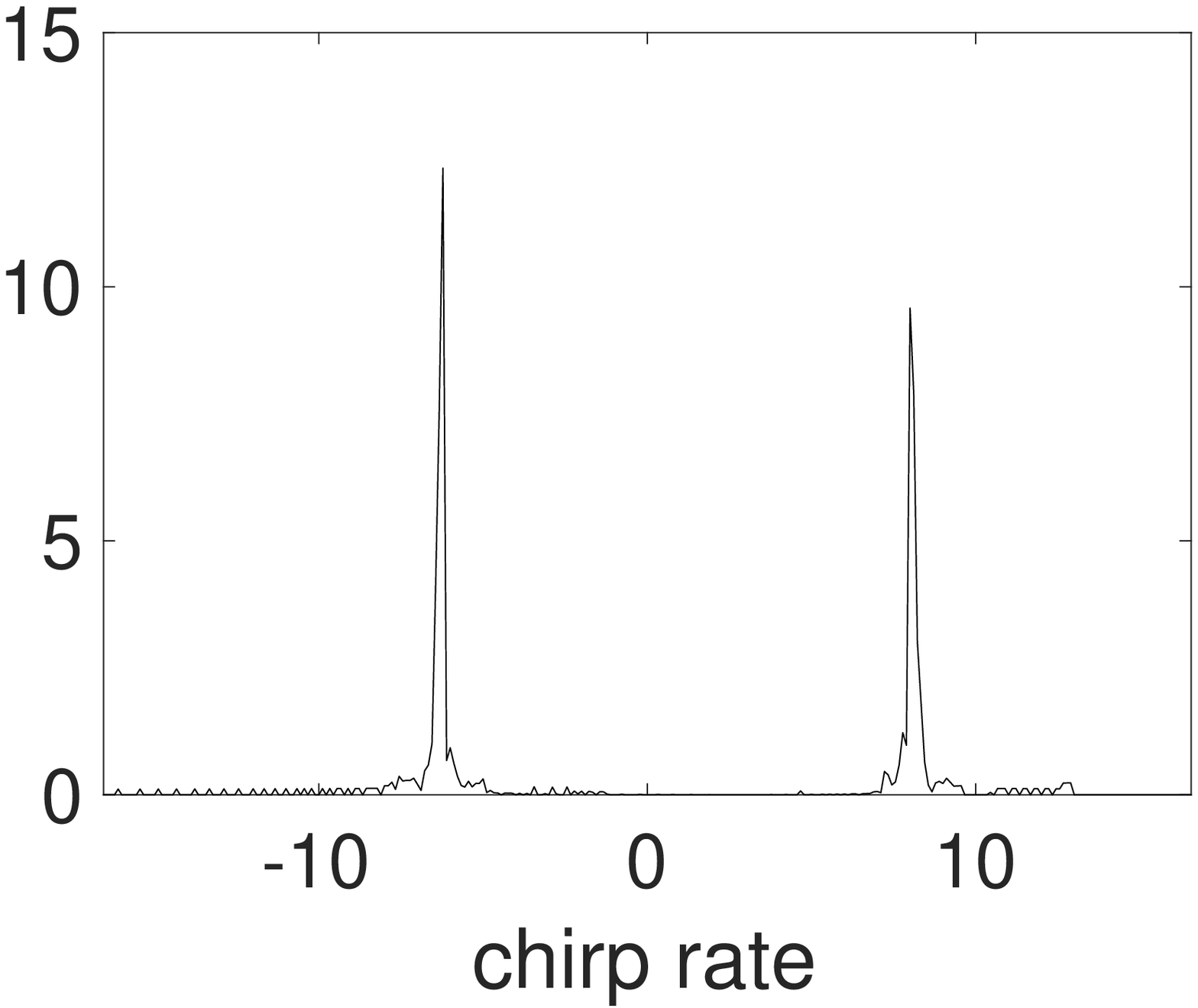}
	\includegraphics[width=.32\textwidth]{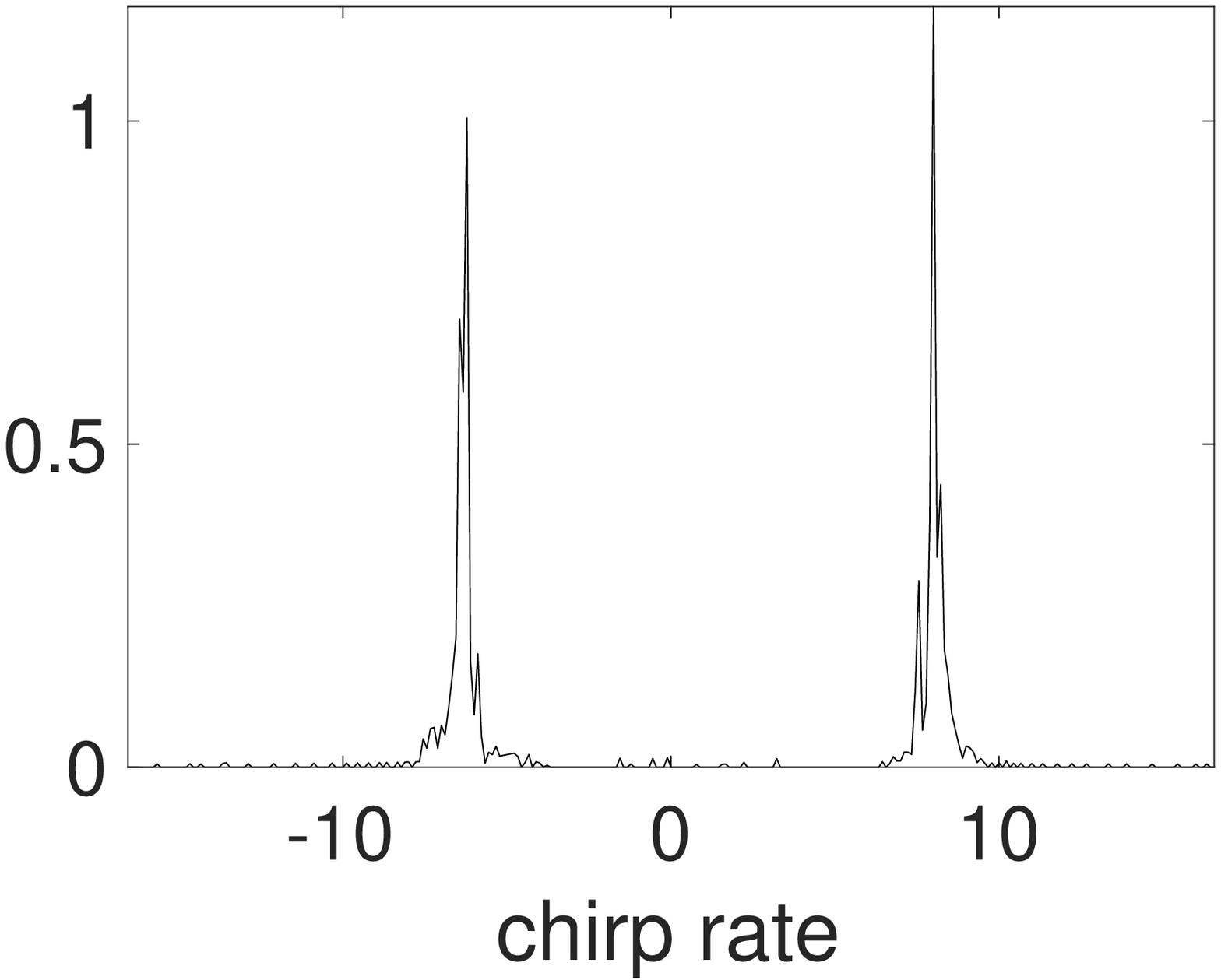}
	\caption{Top left: plot of $\abs{T_f^{(g_3)}(t_0,\xi,\lambda)}$; top middle: plot of $\abs{T_f^{(g_4)}(t_0,\xi,\lambda)}$; top right: plot of $\abs{T_f^{(g_5)}(t_0,\xi,\lambda)}$; second row left: plot of $\abs{T_f^{(g_3)}(t_0,\xi_0,\lambda)}$; second row middle: plot of $\abs{T_f^{(g_4)}(t_0,\xi_0,\lambda)}$; second row right: plot of $\abs{T_f^{(g_5)}(t_0,\xi_0,\lambda)}$; third row left: plot of $\abs{S_f^{(g_3)}(t_0,\xi,\lambda)}$; third row middle: plot of $\abs{S_f^{(g_4)}(t_0,\xi,\lambda)}$; third row right: plot of $\abs{S_f^{(g_5)}(t_0,\xi,\lambda)}$;  bottom left: plot of $\abs{S_f^{(g_3)}(t_0,\xi_0,\lambda)}$; bottom middle: plot of $\abs{S_f^{(g_4)}(t_0,\xi_0,\lambda)}$; bottom right: plot of $\abs{S_f^{(g_5)}(t_0,\xi_0,\lambda)}$.}
	\label{fig:4}
\end{figure}

Third, we evaluate how the squeezing step functions; that is, how does the squeezing step rearrange the TFC representation determined by CT. Consider the following way to {\em invert SCT}. Choose $\epsilon_1,\epsilon_2 >0$, and define the inverse SCT of $f$ with window $g$ in the vicinity of $(t,\xi,\lambda)$ as a set:
\[
Q_{\epsilon_1, \epsilon_2}(t,\xi,\lambda)=\{(t, \eta, \gamma)|\, \abs{\omega_{f}^{(g)}(t,\eta,\gamma)-\xi}<\epsilon_1\mbox{ and }\abs{\mu_f^{(g)}(t,\eta,\gamma)-\lambda}<\epsilon_2  \}\,.
\]
Figure \ref{fig:5} shows the inverted SCT of $f(x)$ in the vicinity of the ridge of $f_1$ (extracted from $S_{f}^{(g_2)}$) with $g_0$ and $g_2$ at $t_0$ and $t_1$. Here we choose $\epsilon_1 = 1, \epsilon_2 = 0.33$. It is clear that there is an interesting ``pattern'' in the nonlinear squeezing step in SCT, which comes from the nonlinearity of the squeezing step.

\begin{figure}[!htbp]
	\centering
	\includegraphics[width=.32\textwidth]{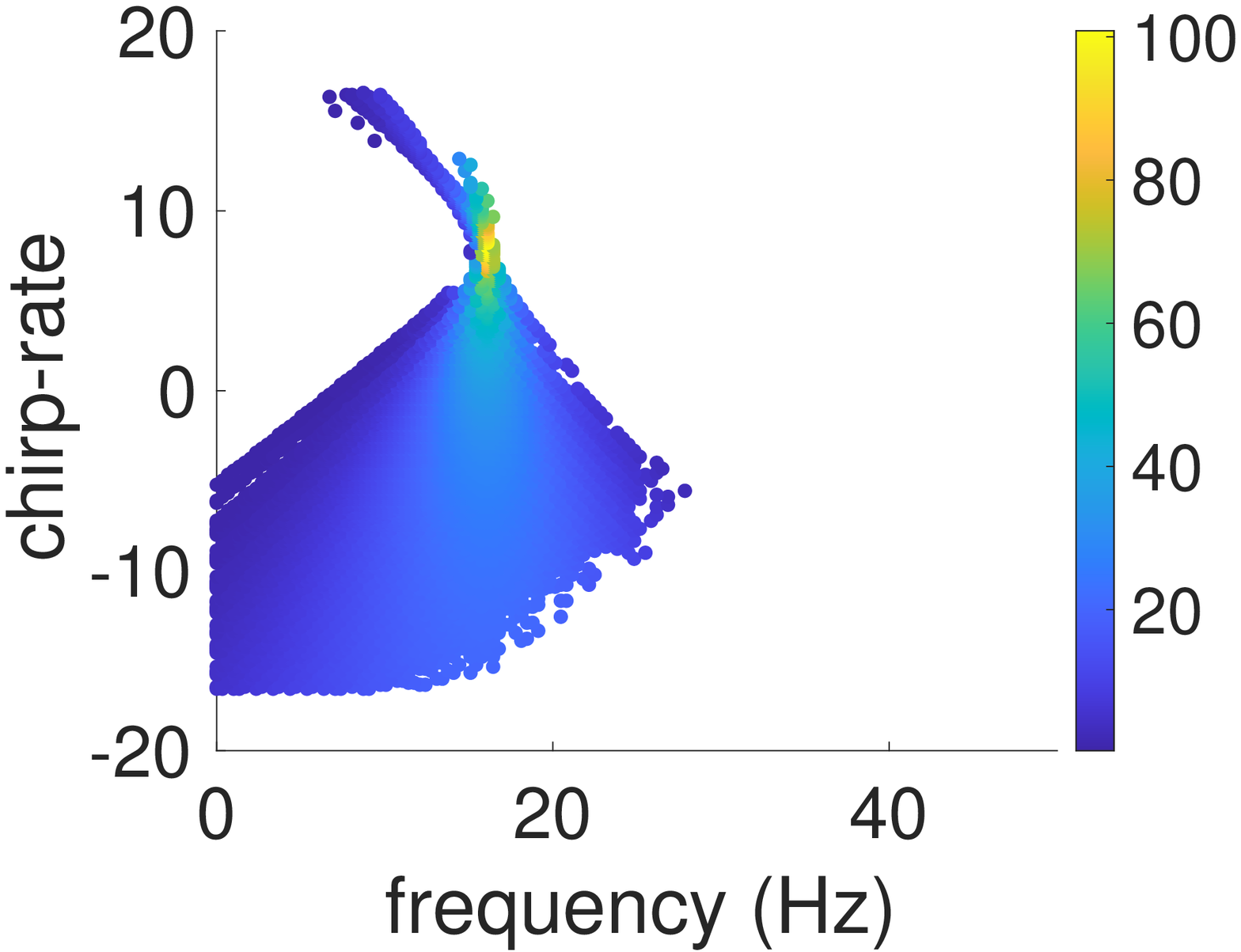}
	\includegraphics[width=.32\textwidth]{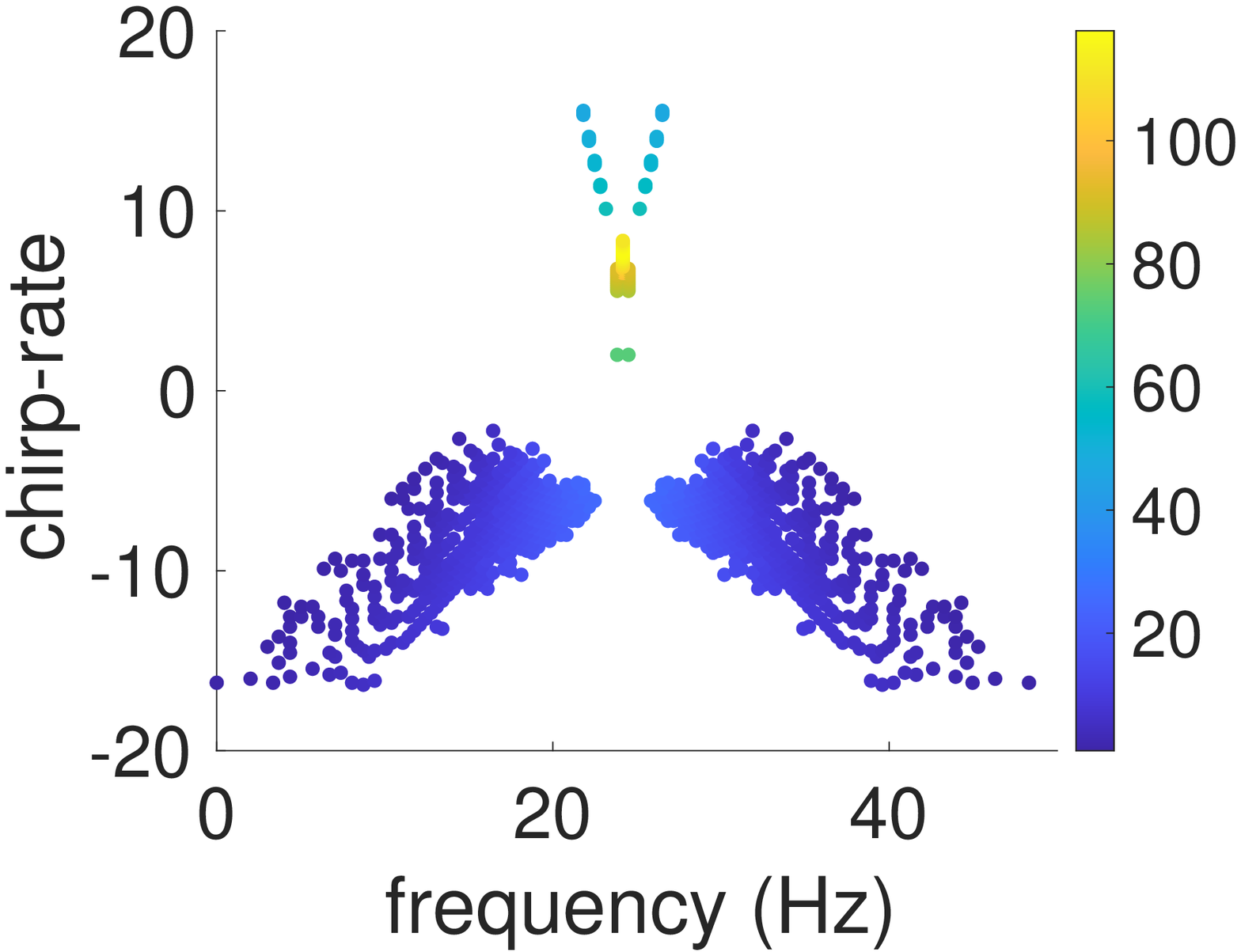}\\
	\includegraphics[width=.32\textwidth]{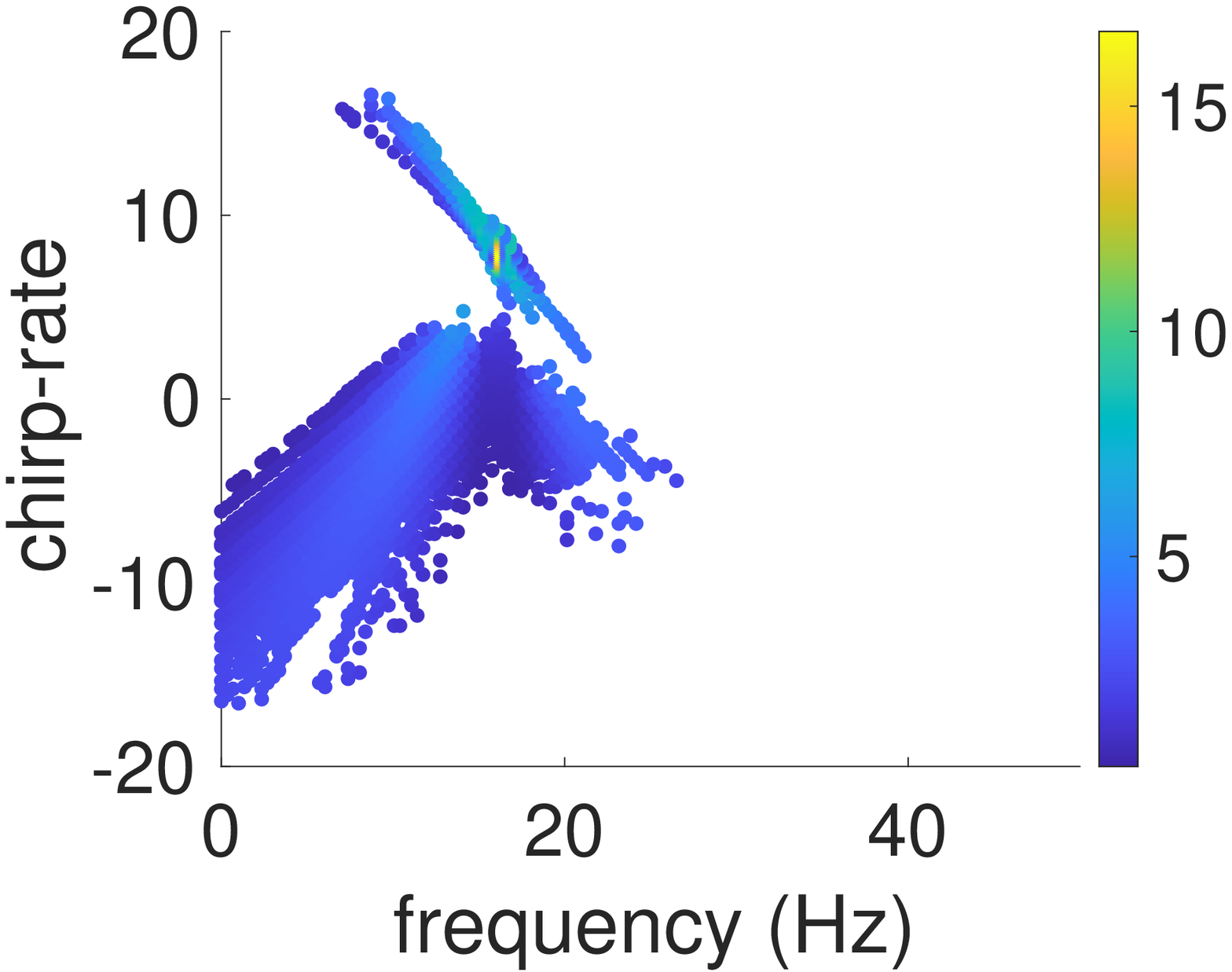}
	\includegraphics[width=.32\textwidth]{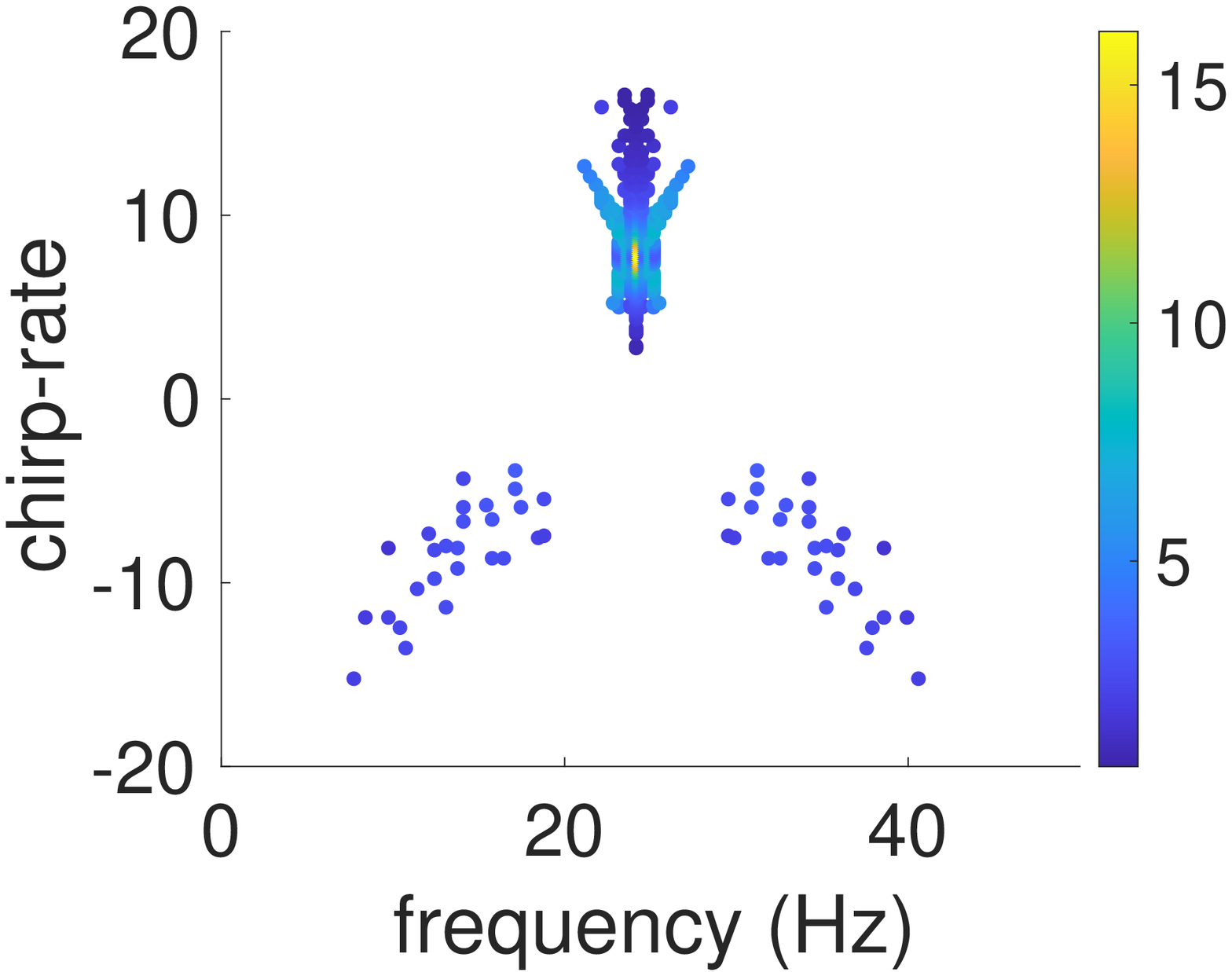}
	\caption{Top left: inverse SCT map of the vicinity of the ridge of $f_1$ with $g_0$ at $t_1 =2$s; top right: inverse SCT map of the vicinity of the ridge of $f_1$ with $g_0$ at $t_0 =3$s; bottom left: inverse SCT map of the vicinity of the ridge of $f_1$ with $g_2$ at $t_1 =2$s; bottom right: inverse SCT map of the vicinity of the ridge of $f_1$ with $g_2$ at $t_0 =3$s. Color reflects the magnitude of the CT.}
	\label{fig:5}
\end{figure}

Finally, we evaluate the proposed reconstruction approach. Figure \ref{fig:6} shows the reconstructed $f_1$ and $f_2$ by the proposed reconstruction scheme with window $g_0$ \cite{li2021chirplet}, denoted by $f^{SCT}_1, f^{SCT}_2$, where the IF and chirp rate information is obtained by the SCT of $f$ with window $g_2$.
For a comparison purpose, the reconstructions of $f_1$ and $f_2$ through the reconstruction of 2nd-order SST around ridges of $f_1$ and $f_2$, denoted by $f^{SST}_1, f^{SST}_2$, are also provided. We see that the reconstruction is good except near $t_0$ of the crossover frequency. 
Quantitatively, we evaluate the reconstruction errors over two disjoint intervals, $\mathcal{I}_1 = [2.5,3.5]$ (around the crossover time) and $\mathcal{I}_2 = [1,2.5)\cup (3.5,5]$ (away from the crossover time), where we have 
$\frac{\norm{(\Re(f^{SCT}_1)-\Re(f_1))\mathbf{1}_{\mathcal{I}_1}}_2}{\norm{\Re(f_1\mathbf{1}_{\mathcal{I}_1})}_2}=0.076$,
$\frac{\norm{(\Re(f^{SCT}_1)-\Re(f_1))\mathbf{1}_{\mathcal{I}_2}}_2}{\norm{\Re(f_1\mathbf{1}_{\mathcal{I}_2})}_2}= 0.064$,  
$\frac{\norm{(\Re(f^{SST}_1)-\Re(f_1))\mathbf{1}_{\mathcal{I}_1}}_2}{\norm{\Re(f_1\mathbf{1}_{\mathcal{I}_1})}_2}=0.458$,
and
$\frac{\norm{(\Re(f^{SST}_1)-\Re(f_1))\mathbf{1}_{\mathcal{I}_2}}_2}{\norm{\Re(f_1\mathbf{1}_{\mathcal{I}_2})}_2}=0.021$.
We see from this example that the reconstruction error by SCT is reduced at the crossing time $t_0$, but it is worse at the non-crossing time if we compare it with the reconstruction by the 2nd-order SST. 

\begin{figure}[!htbp]
	\centering
	\includegraphics[width=0.48\textwidth]{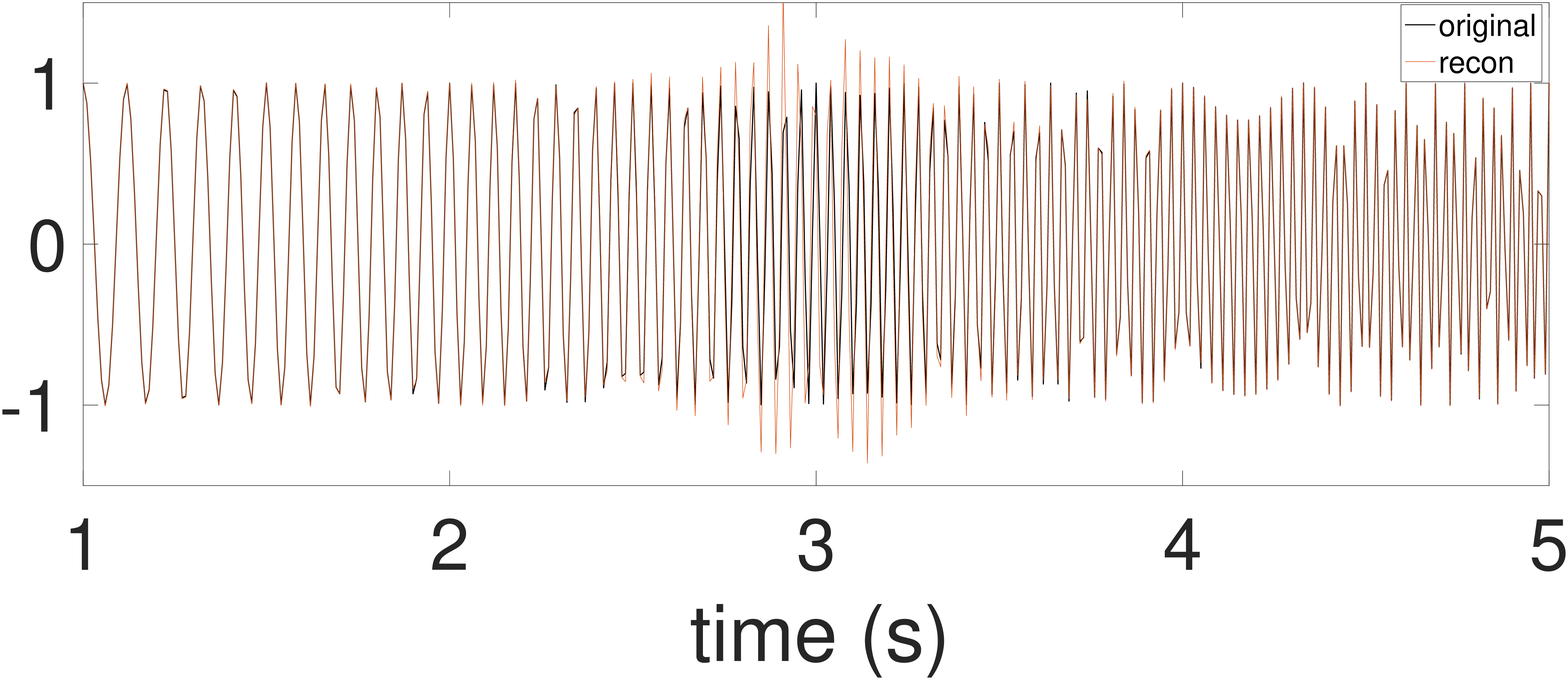}
	\includegraphics[width=0.48\textwidth]{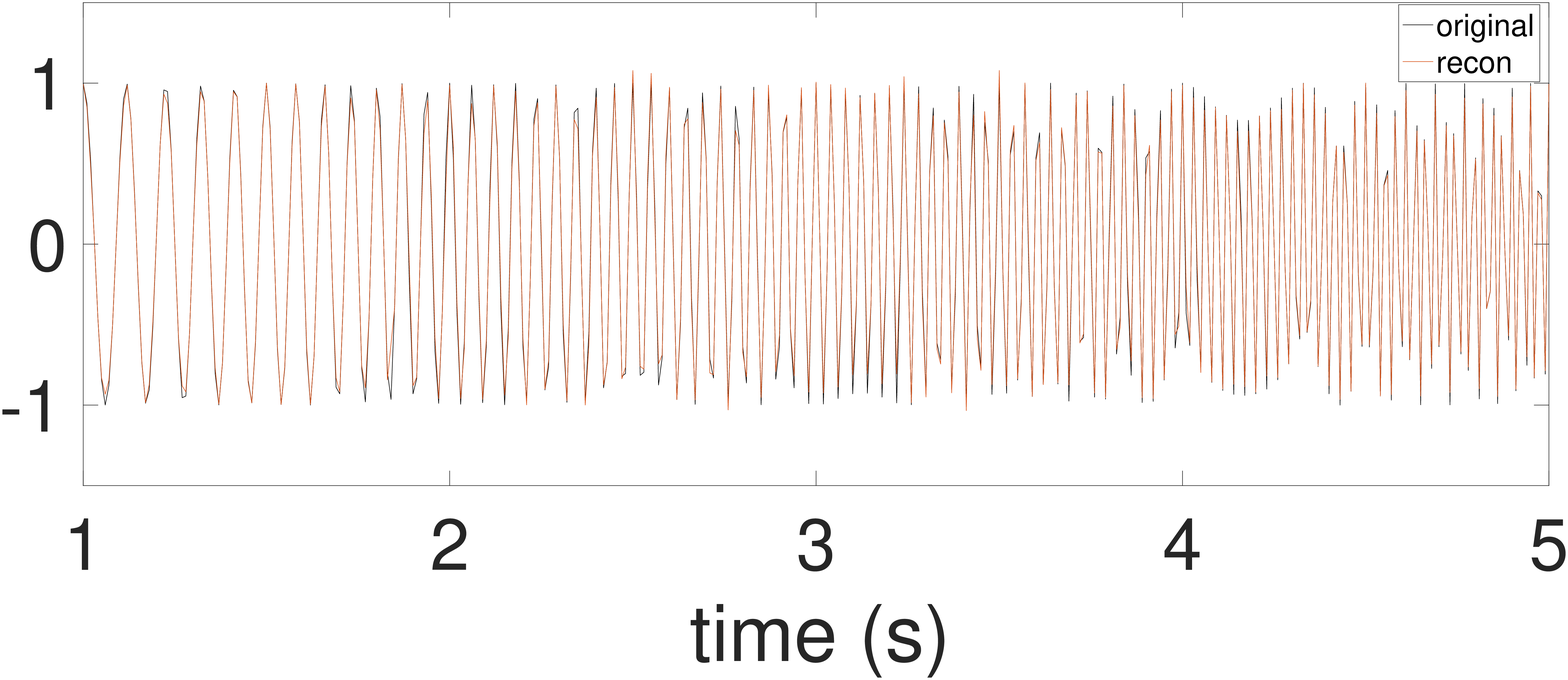}\\
	\includegraphics[width=0.48\textwidth]{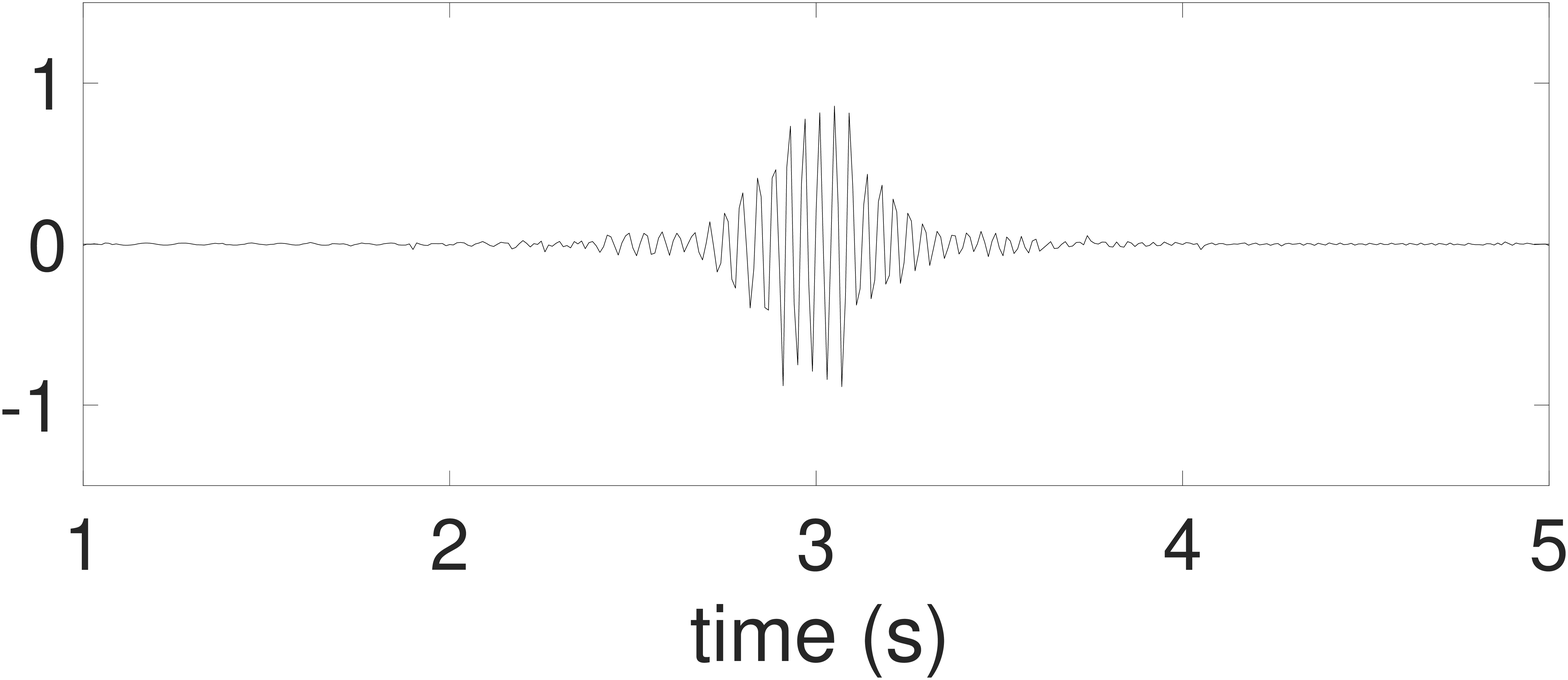}
	\includegraphics[width=0.48\textwidth]{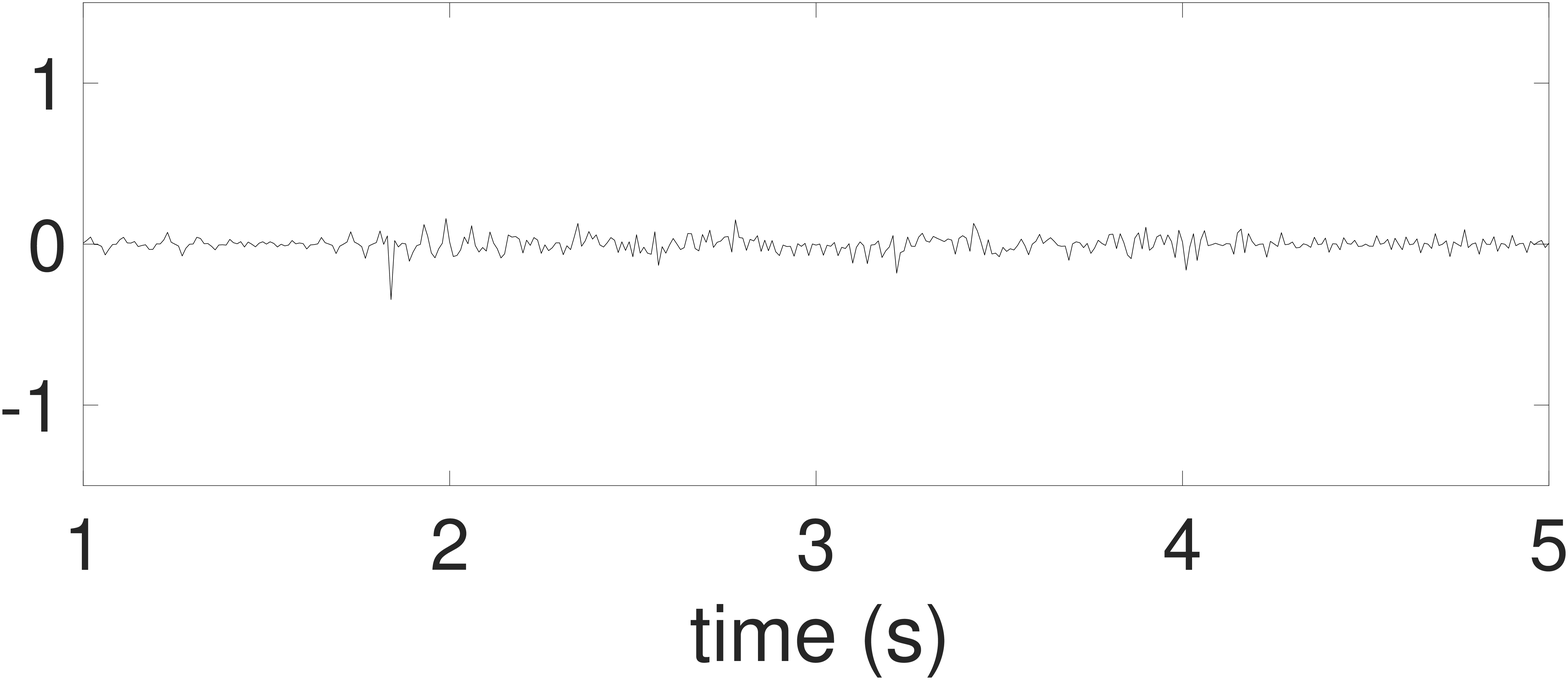}\\
	\includegraphics[width=0.48\textwidth]{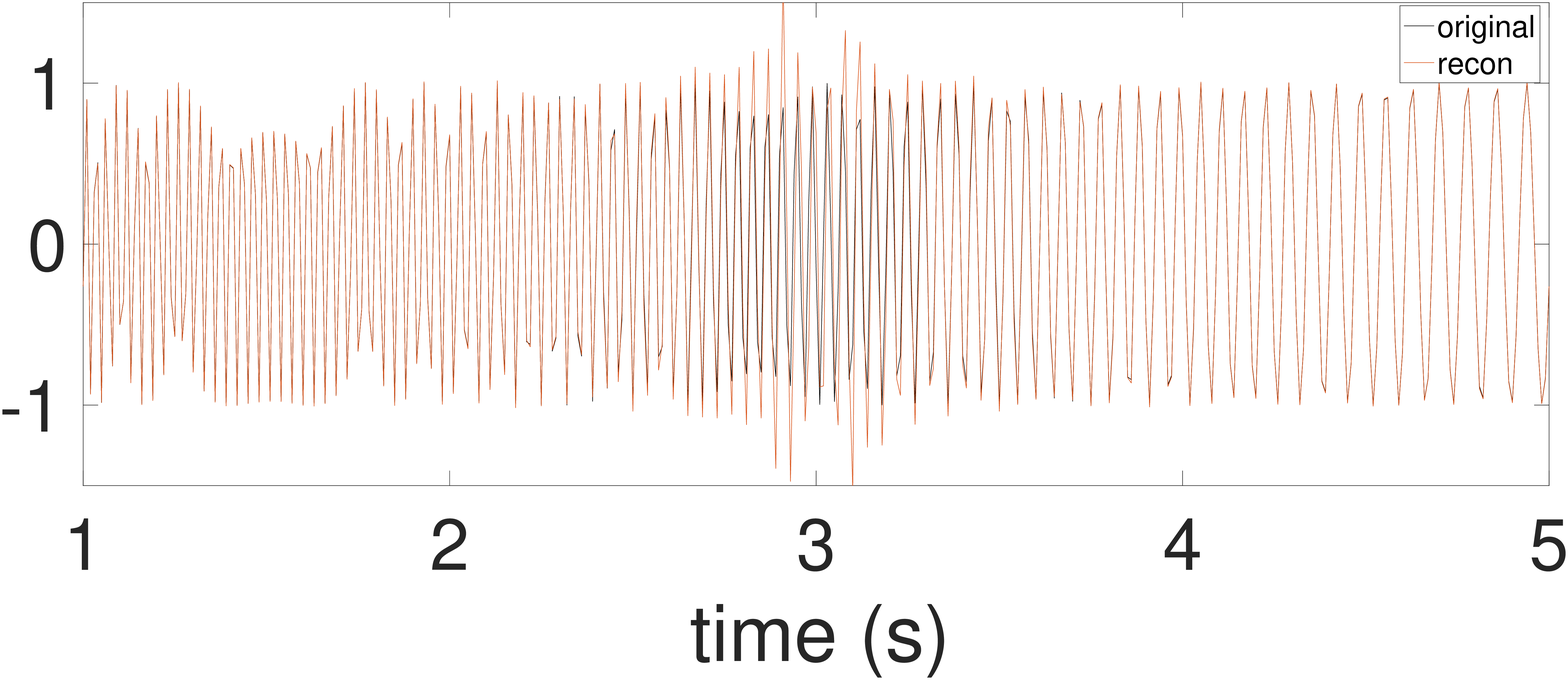}
	\includegraphics[width=0.48\textwidth]{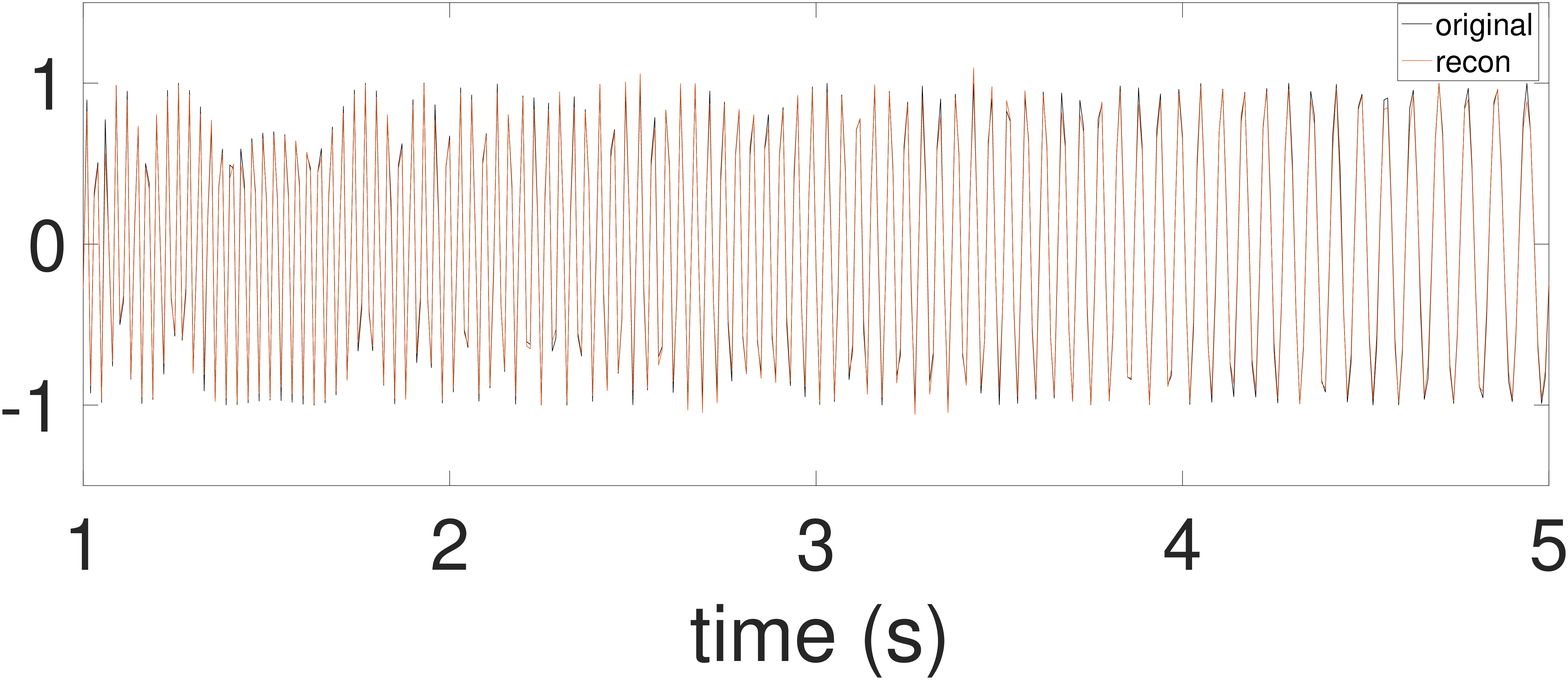}\\
	\includegraphics[width=0.48\textwidth]{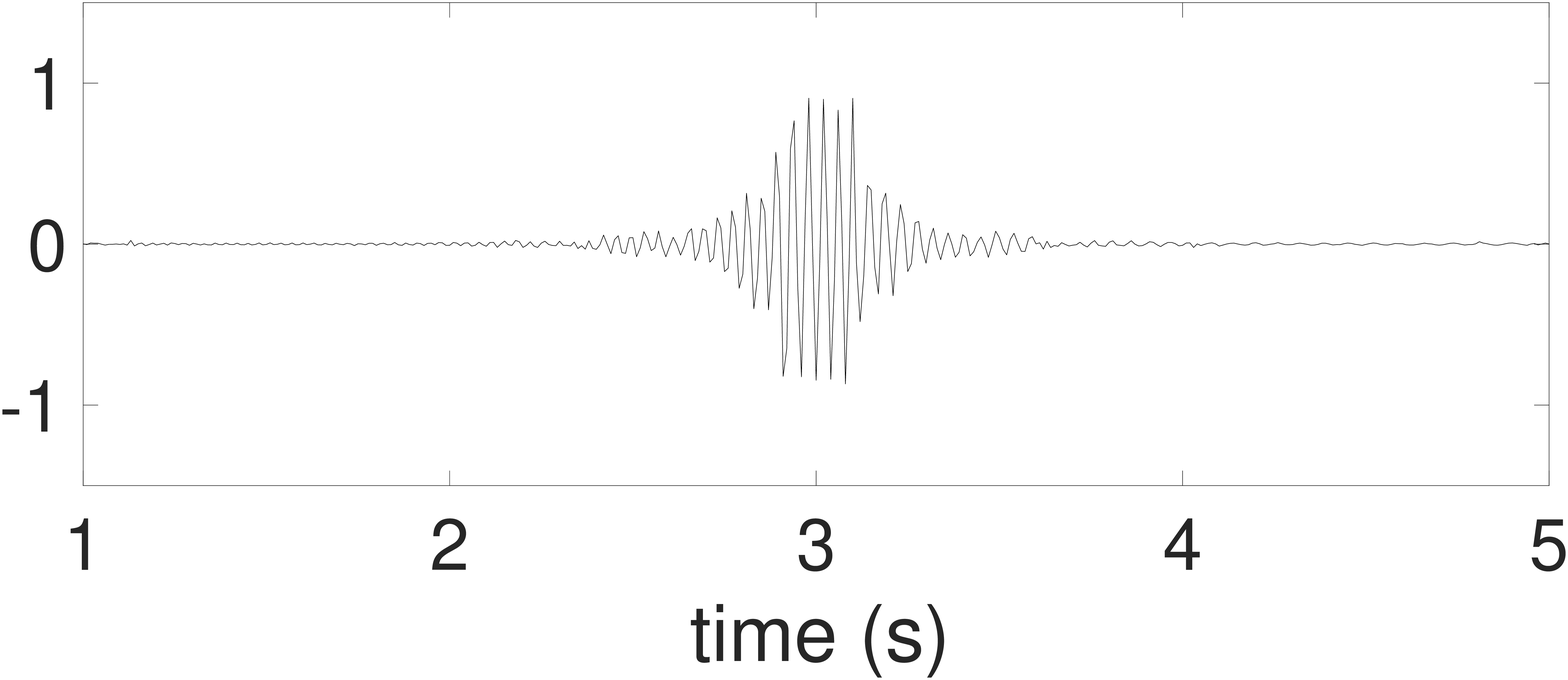}
	\includegraphics[width=0.48\textwidth]{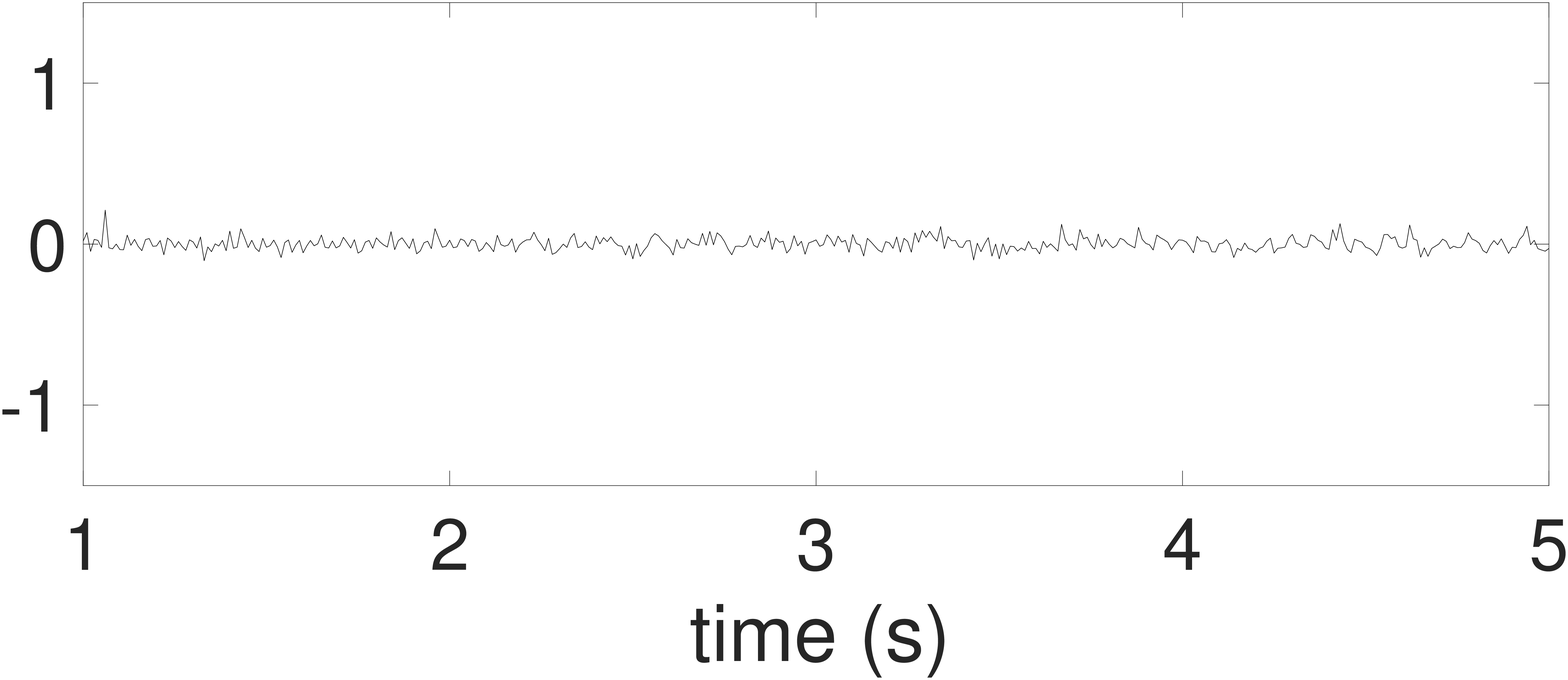}
	\caption{Left column: reconstruction of signals from STFT-based 2nd-order SST with window $g_0$. Top row: reconstruction of $\Re(f_1)$; second row: reconstruction error of $\Re(f_1)$; third row: reconstruction of $\Re(f_2)$; bottom row: reconstruction error of $\Re(f_2)$. Right column: reconstruction of signals using group idea with window $g_0$ with instantaneous frequencies and chirp rates obtained on ridges extracted from \abs{S_f^{(g_2)}(t,\xi,\lambda)} (as shown in Figure \ref{fig:3} bottom). Top row: reconstruction of $\Re(f_1)$; second row: reconstruction error of $\Re(f_1)$; third row: reconstruction of $\Re(f_2)$; bottom row: reconstruction error of $\Re(f_2)$.}
	\label{fig:6}
\end{figure}

\subsection{Simulated time-varying chirp rate signal}
Next, we demonstrate the performance of SCT on signals with time-varying amplitudes, frequencies and chirp rates. We use the smoothed Brownian path realizations to model the AM, IF and the time-varying chirp rates of the constituent components of the signal. This simulation scheme is a generalization of that proposed in \cite{DaWaWu2016} and we summarize it here. If $W$ is the standard Brownian motion defined on $[0,\infty)$, then we define the smoothed Brownian motion with bandwidth $B>0$ as $\Phi_B := W\star K_B$, where $K_B$ is the Gaussian function with standard deviation $B>0$ and $\star$ denotes the convolution operator. Given $T>0$ and parameters $\zeta_1,\dots,\zeta_6>0$, we then define the following family of random processes on $[0,T]$:
\[
\Psi_{[\zeta_1,\dots,\zeta_7]}(x):= \zeta_1 + \zeta_2 x+ \zeta_3 x^2 + \zeta_4\frac{\Phi_{\zeta_5}(x)}{\norm{\Phi_{\zeta_5}}_{L^{\infty}[0,T]}} + \zeta_6\int_0^x\int_0^u \frac{\Phi_{\zeta_7}(s)}{\norm{\Phi_{\zeta_7}}_{L^{\infty}[0,T]}}\diff{s}\diff{u}.
\]
For the amplitude $A(x)$ of each component, we set $\zeta_2 = \zeta_3 = \zeta_6 =0$; every realization then varies smoothly between $\zeta_1$ and $\zeta_1+\zeta_4$. In the example shown below, the signal consists of two components (i.e. $K=2$) on $[0,10]$; their two amplitudes are independent realizations of $\Psi_{[2,0,0,1,200,0,0]}(x)$. To simulate phase functions, we set $\zeta_1 = \zeta_4 = 0$. In this example we consider, we take for $\phi_1(x)$ a realization of $\Psi_{[0,1,4.5,0,0,0.2,400]}(x)$, and for $\phi_2(t)$ a realization of $\Psi_{[0,12,-4,0,0,0.25,300]}(x)$. The signal is then
$$s(x) := f_1(x) + f_2(x) = A_1(x)e^{2\pi i \phi_1(x)} + A_2(x)e^{2\pi i \phi_2(x)},$$
where $x\in [0,10]$, and the sampling rate is 100Hz. 
Finally, we add white noise $\eta(x)$ to $s(x)$, where the noise are identically distributed Student $t_4$ random variables, so the resulting signal is $f(x) = s(x) + \eta(x)$. 
To avoid the boundary effects of our TF method, we only show the result for $x\in [1,9]$.

Figure \ref{fig:8} gives the plot of the signals and their IFs and chirp rates. The signal-to-noise ratio is 5.20, computed as $20\log_{10}(\frac{\text{std}(f(x))}{\text{std}(\xi(x))})$, where std stands for standard deviation. Also, a comparison between TF representations determined by the 2nd-order SST and the SCT with different window functions $g_0(x)$ and $g_2(x)$ is shown. We see that with SCT, the TF representation gives sharper information around the time when IF crossover happens.
Figure \ref{fig:9} shows a comparison of CT and SCT with different window functions $g_0(x)$ and $g_2(x)$ at the IF crossover time. We see that SCT with $g_2$ gives two peaks that are exactly located at the true instantaneous chirp rates of $f_1$ and $f_2$.
Figure \ref{fig:10} shows the reconstruction of $f_1$ and $f_2$ from the proposed reconstruction scheme with window $g_0$, where we consider both the situations that we know the true IFs and instantaneous chirp rates or when the IFs and instantaneous chirp rates are estimated from $S_{f}^{(g_2)}(t,\xi,\lambda)$. 
We see that the estimated IFs and instantaneous chirp rate from SCT give almost identical results.

\begin{figure}[!htbp]
	\centering
	\includegraphics[width=.325\textwidth]{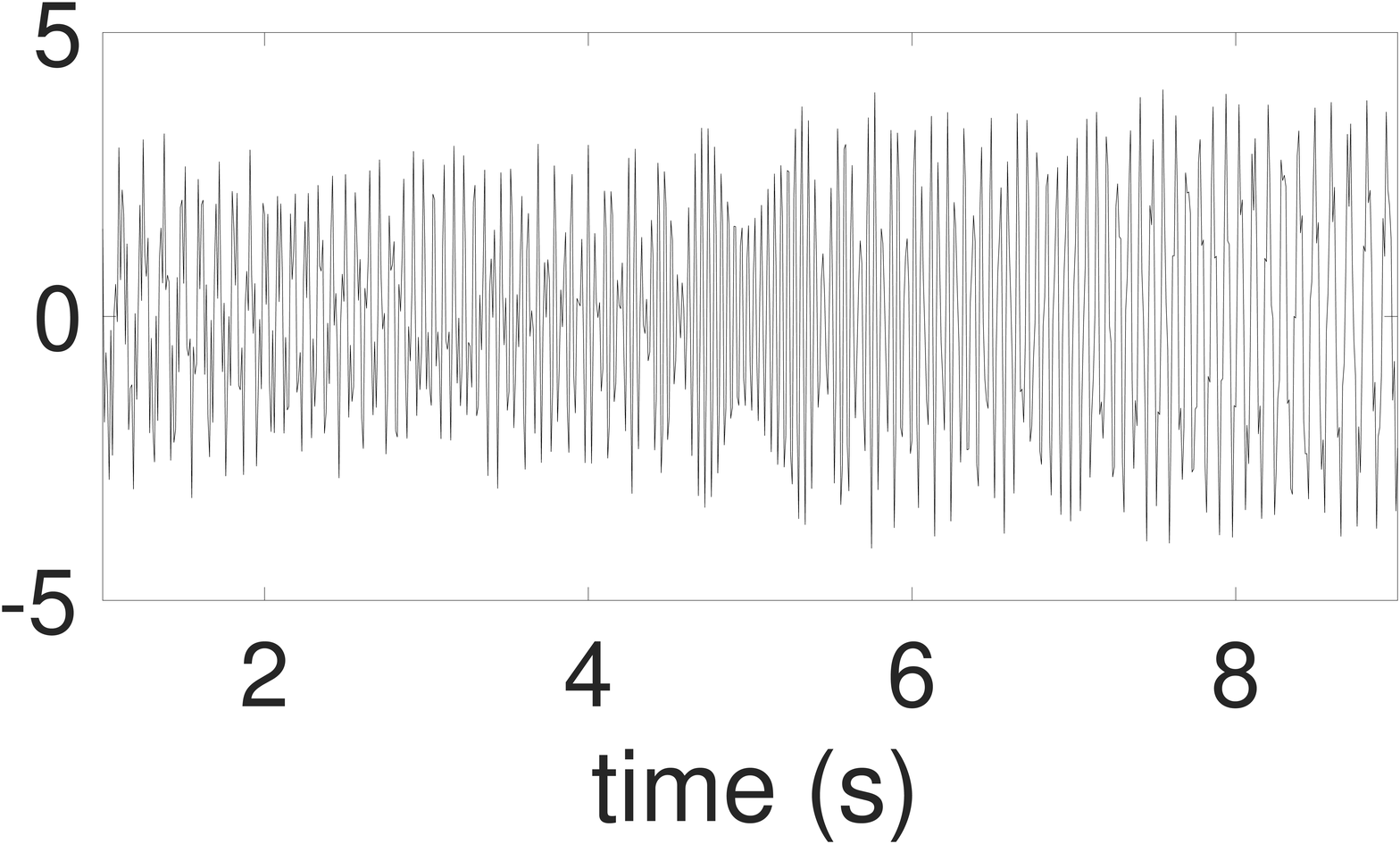}
	\includegraphics[width=.325\textwidth]{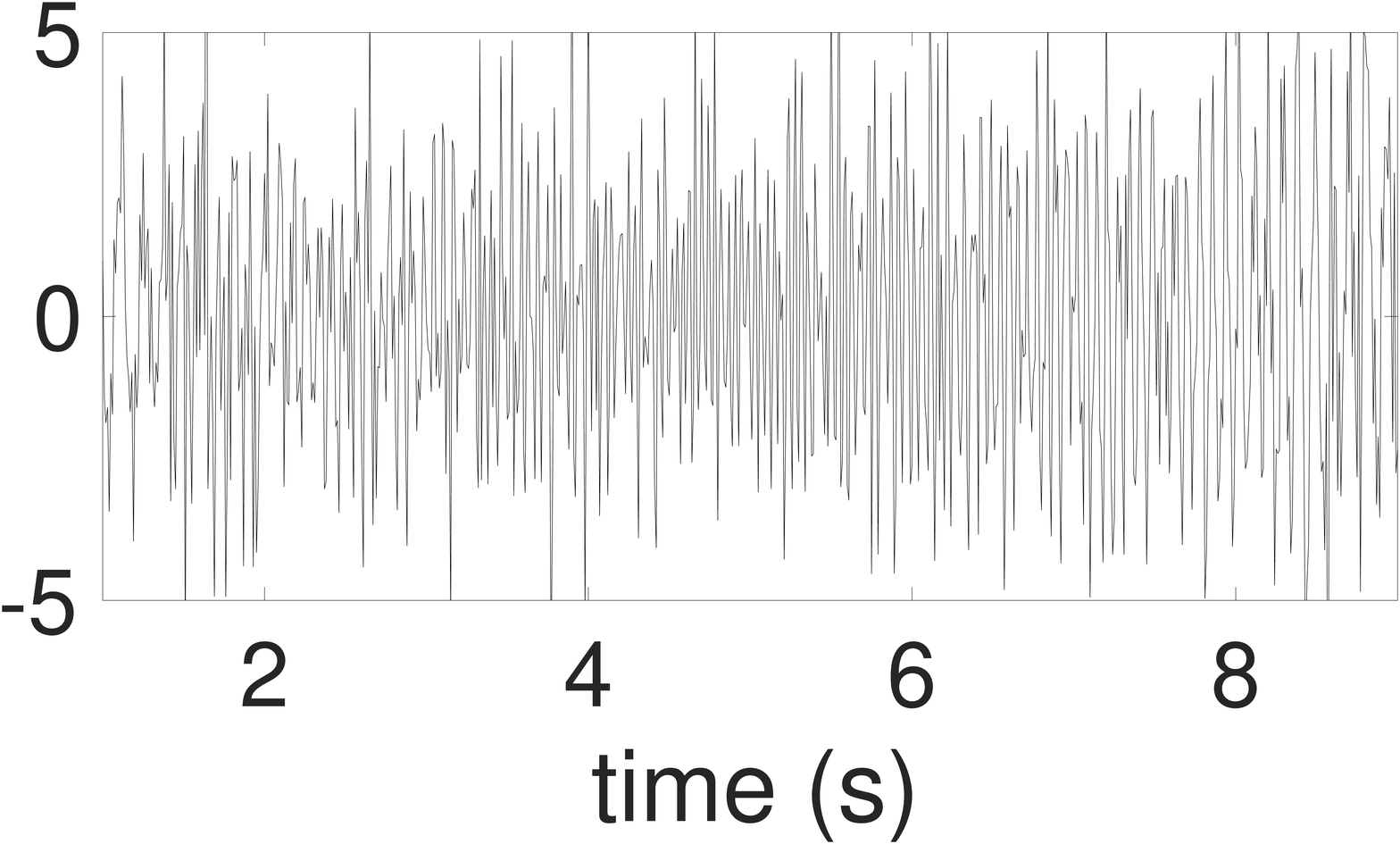}\\
	\includegraphics[width=.325\textwidth]{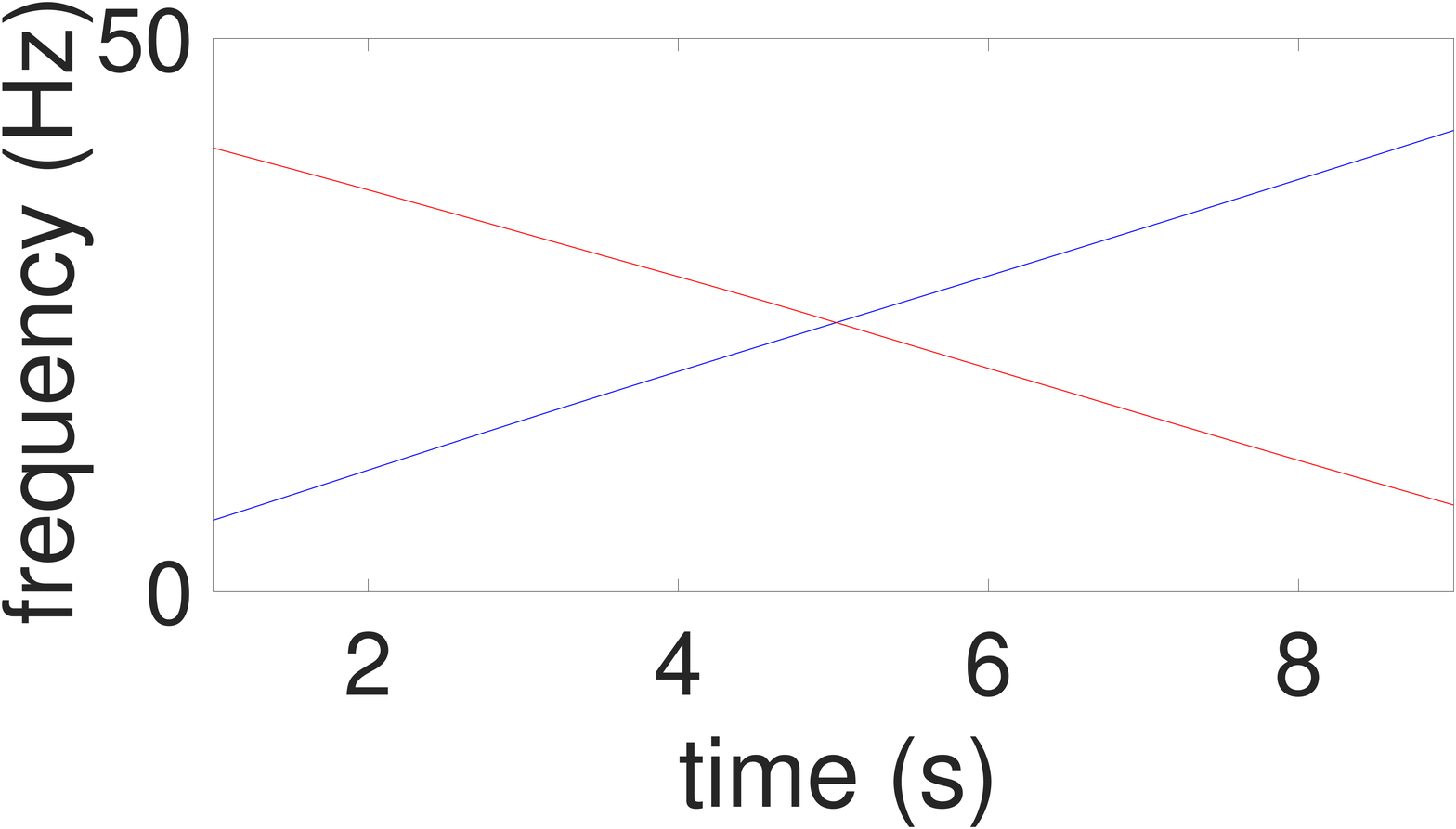}
	\includegraphics[width=.325\textwidth]{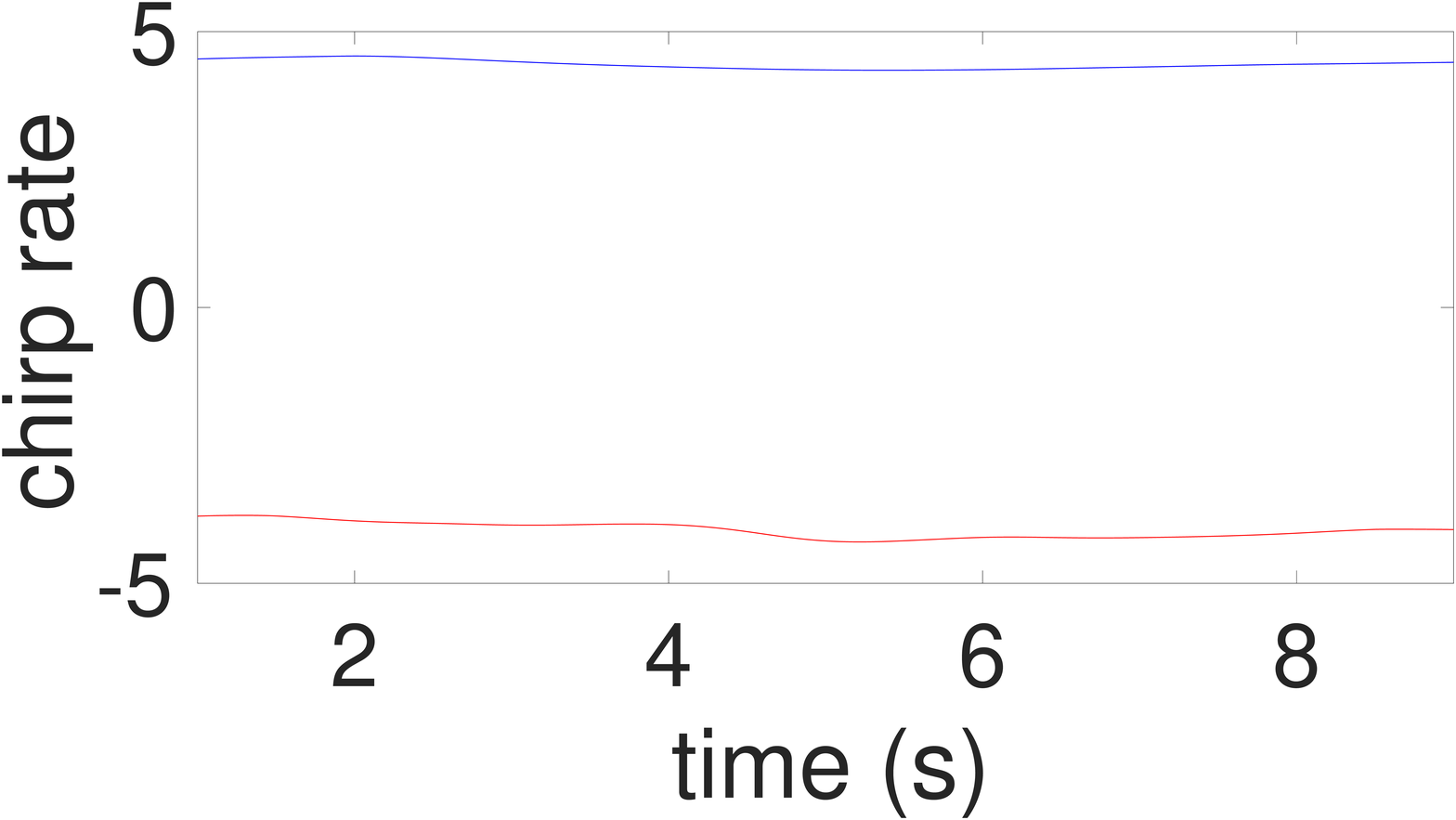}\\	\includegraphics[width=.325\textwidth]{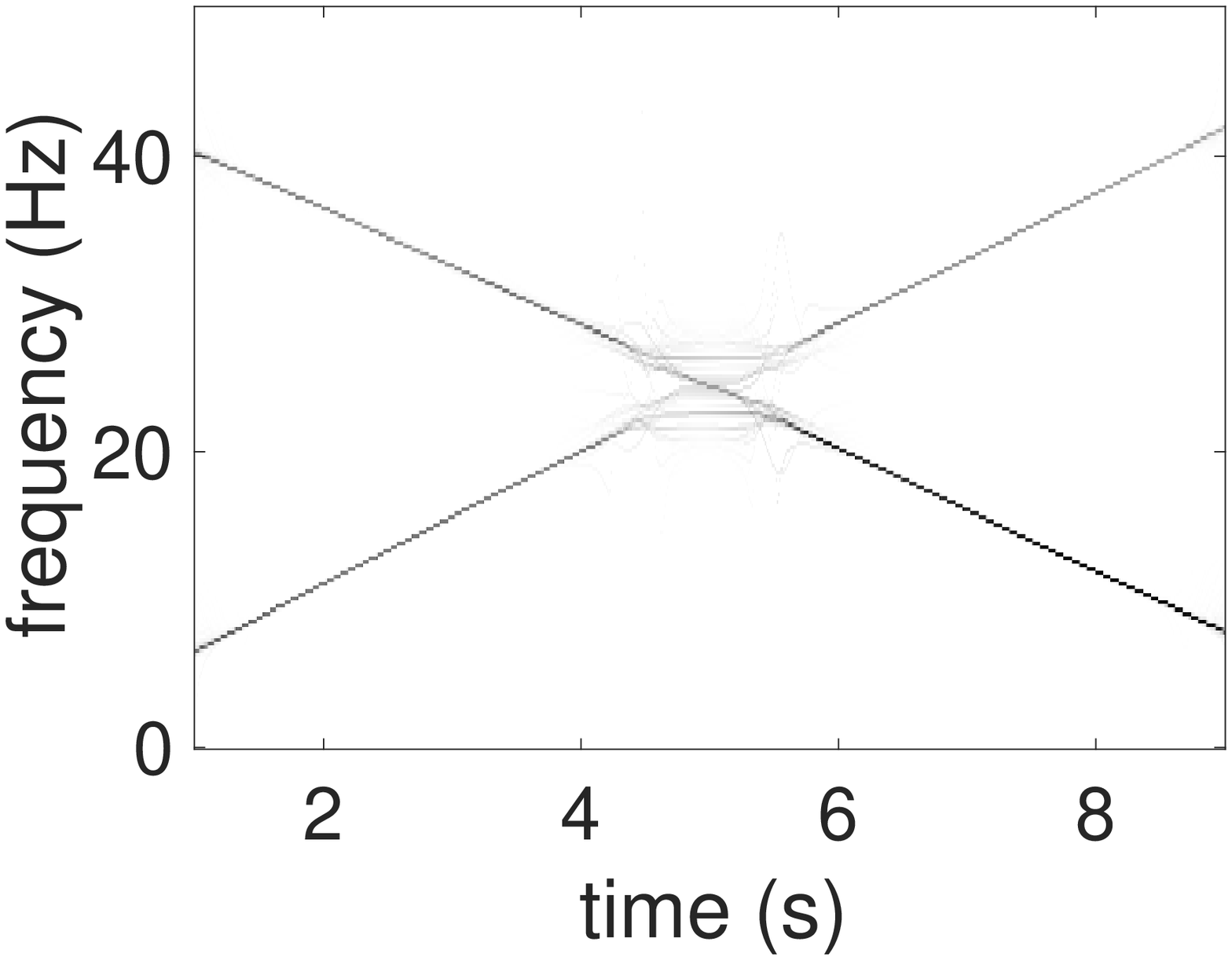}
	\includegraphics[width=.325\textwidth]{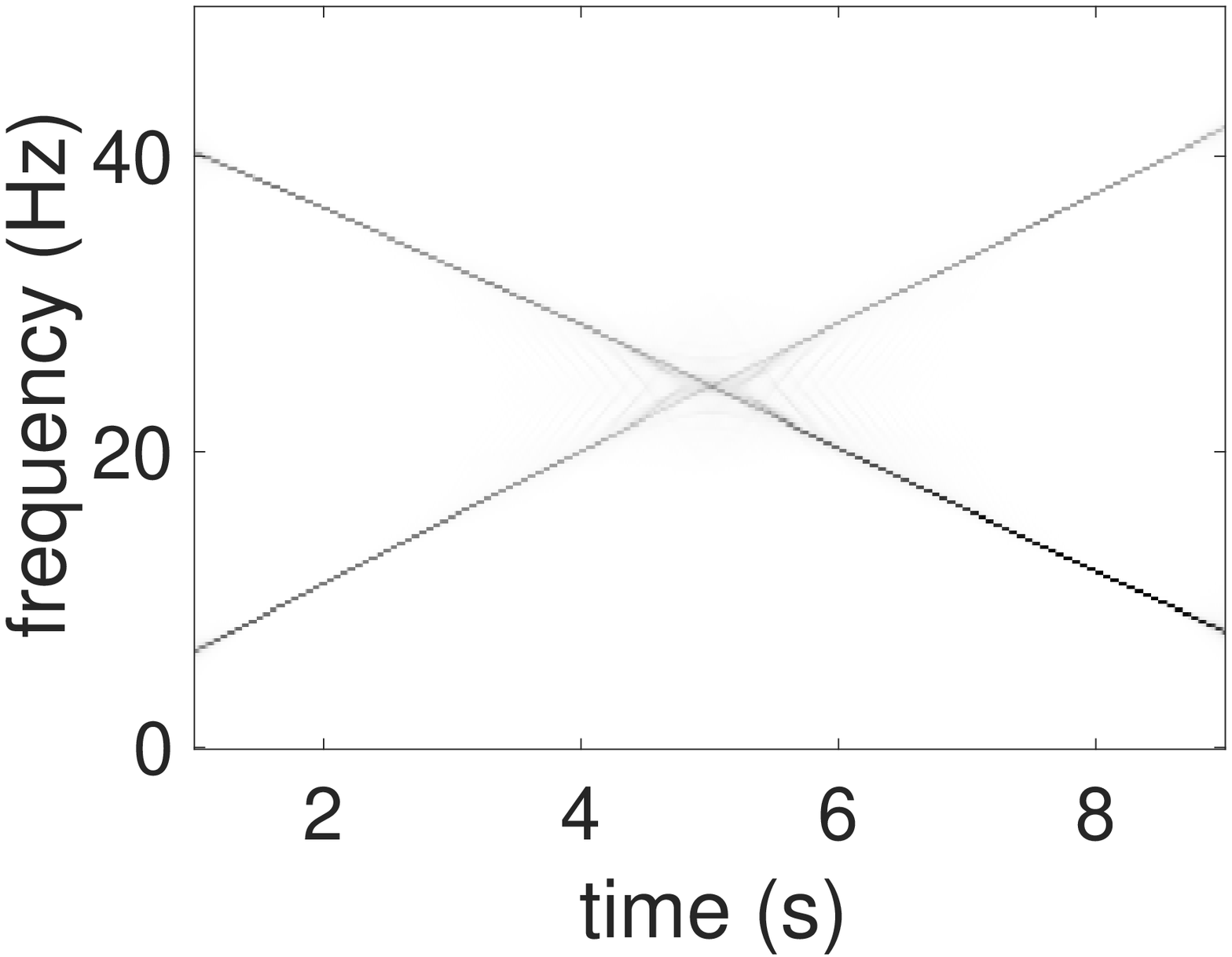}\\
	\includegraphics[width=.325\textwidth]{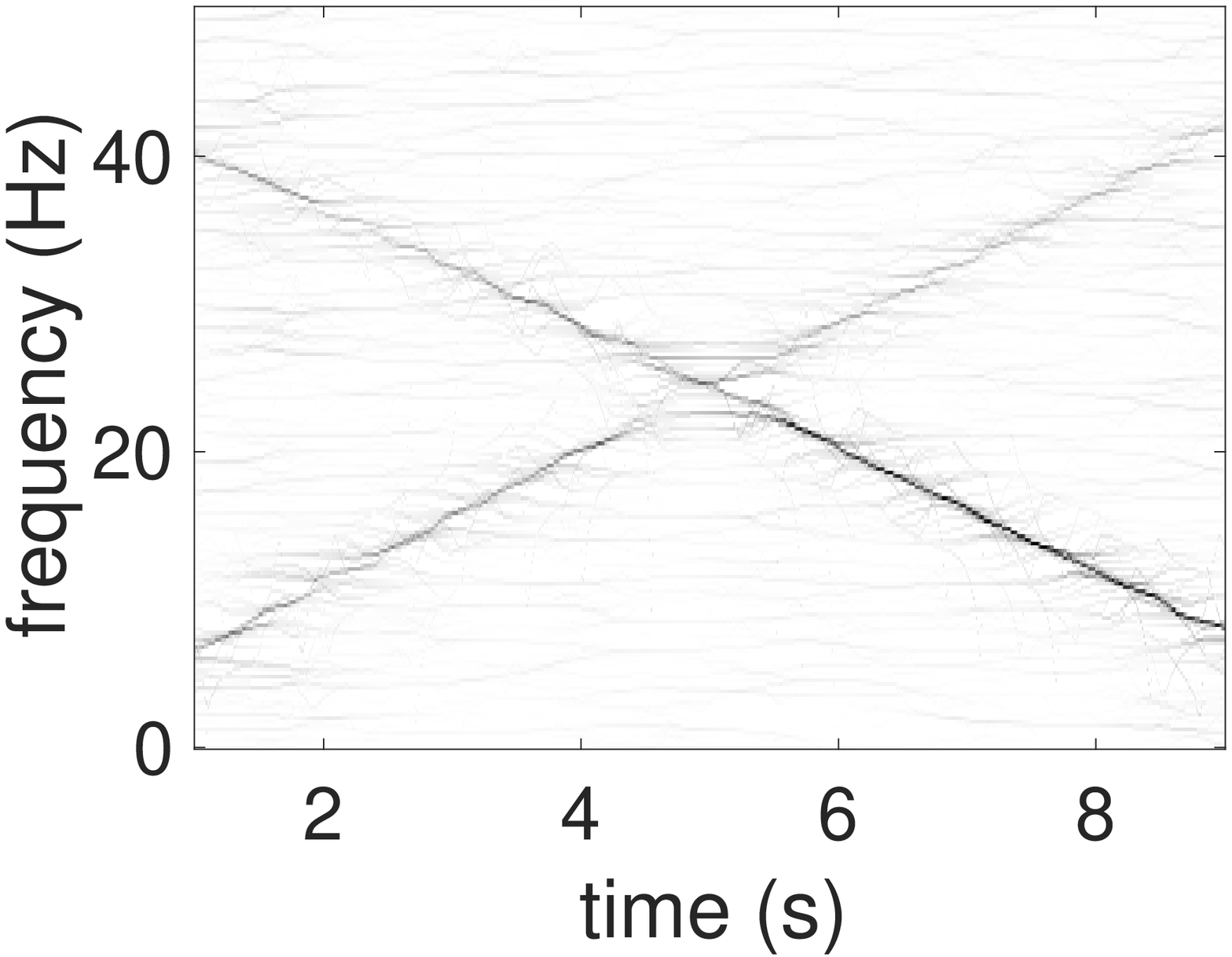}
	\includegraphics[width=.325\textwidth]{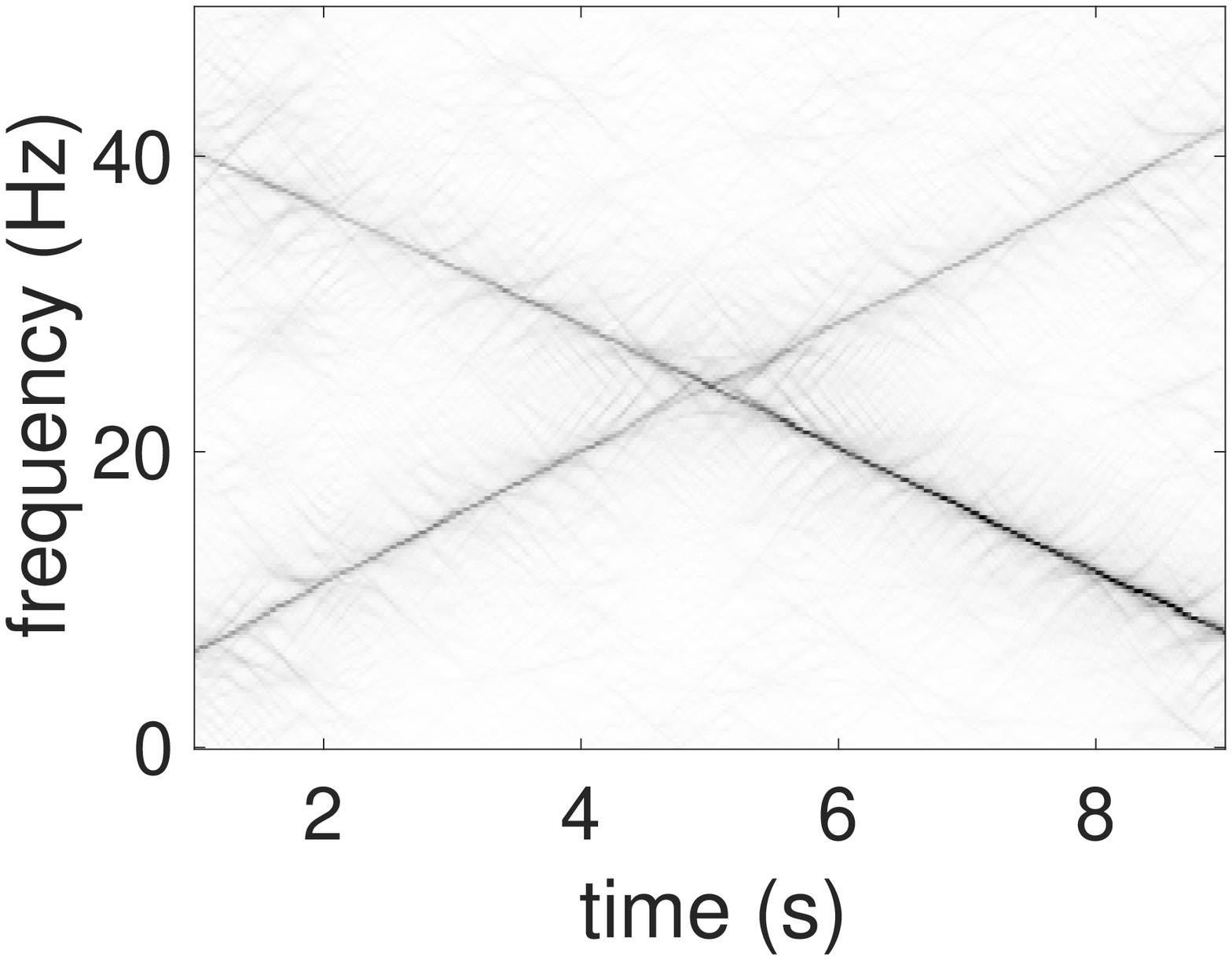}
	\caption{Top row, from left to right: the clean signal $\Re(s)$, the noisy signal $\Re(f)$. Second row, from left to right: plot of $\phi_1'$ and $\phi_2'$, plot of $\phi_1''$ and $\phi_2''$. Note that the IFs have a crossing point at $(t_0,\xi_0) = (4.66, 7.8)$.
		Third row, from left to right: 2nd-order SST of the clean signal $s$ with $g_0$ and the TF representation determined by projecting the SCT with $g_0$. Bottom row, from left to right: 2nd-order SST of the noisy signal $f$ with $g_0$ and the TF representation determined by projecting the SCT with $g_0$.}
	\label{fig:8}
\end{figure}

\begin{figure}[!htbp]
	\centering
	\includegraphics[width=.32\textwidth]{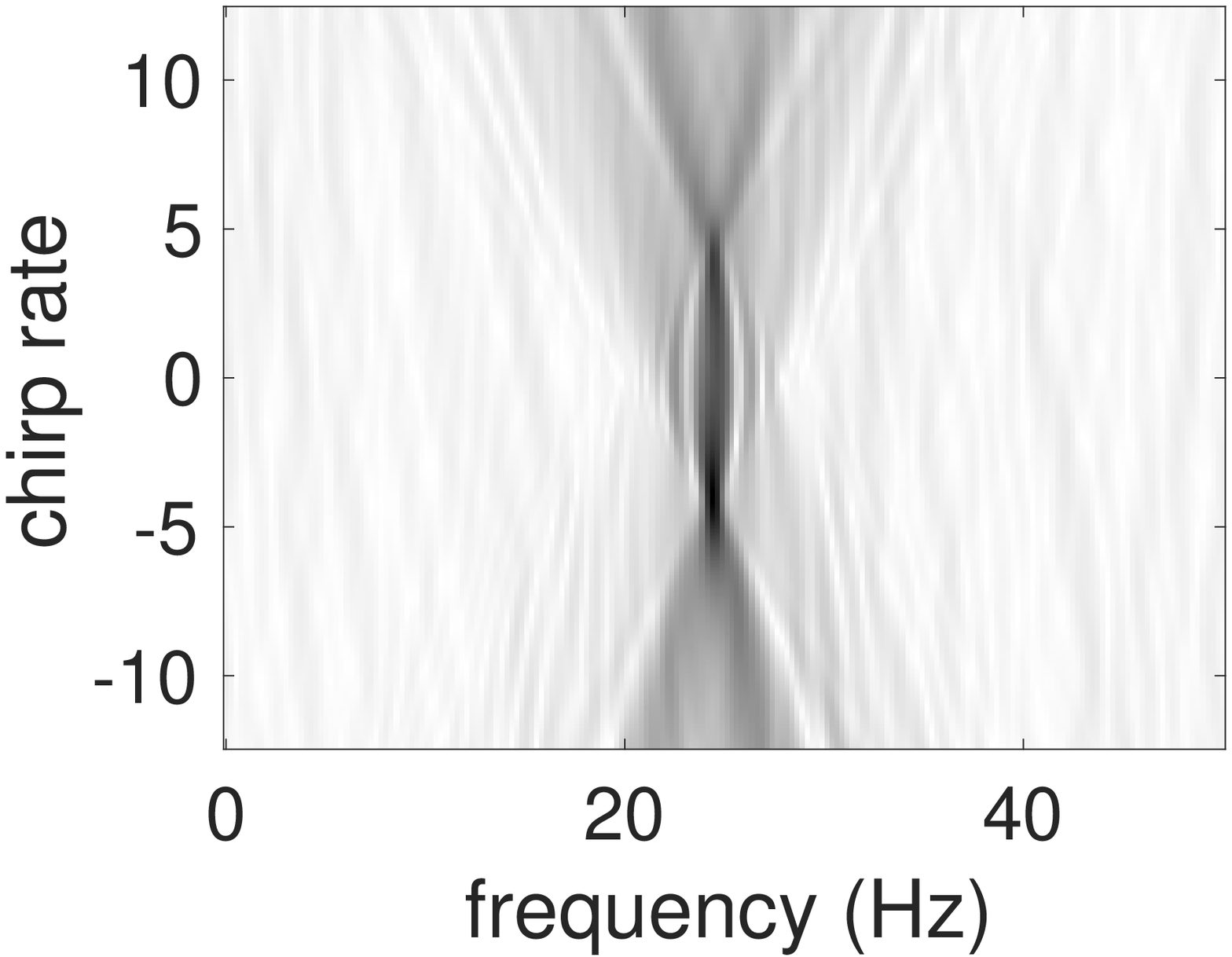}
	\includegraphics[width=.32\textwidth]{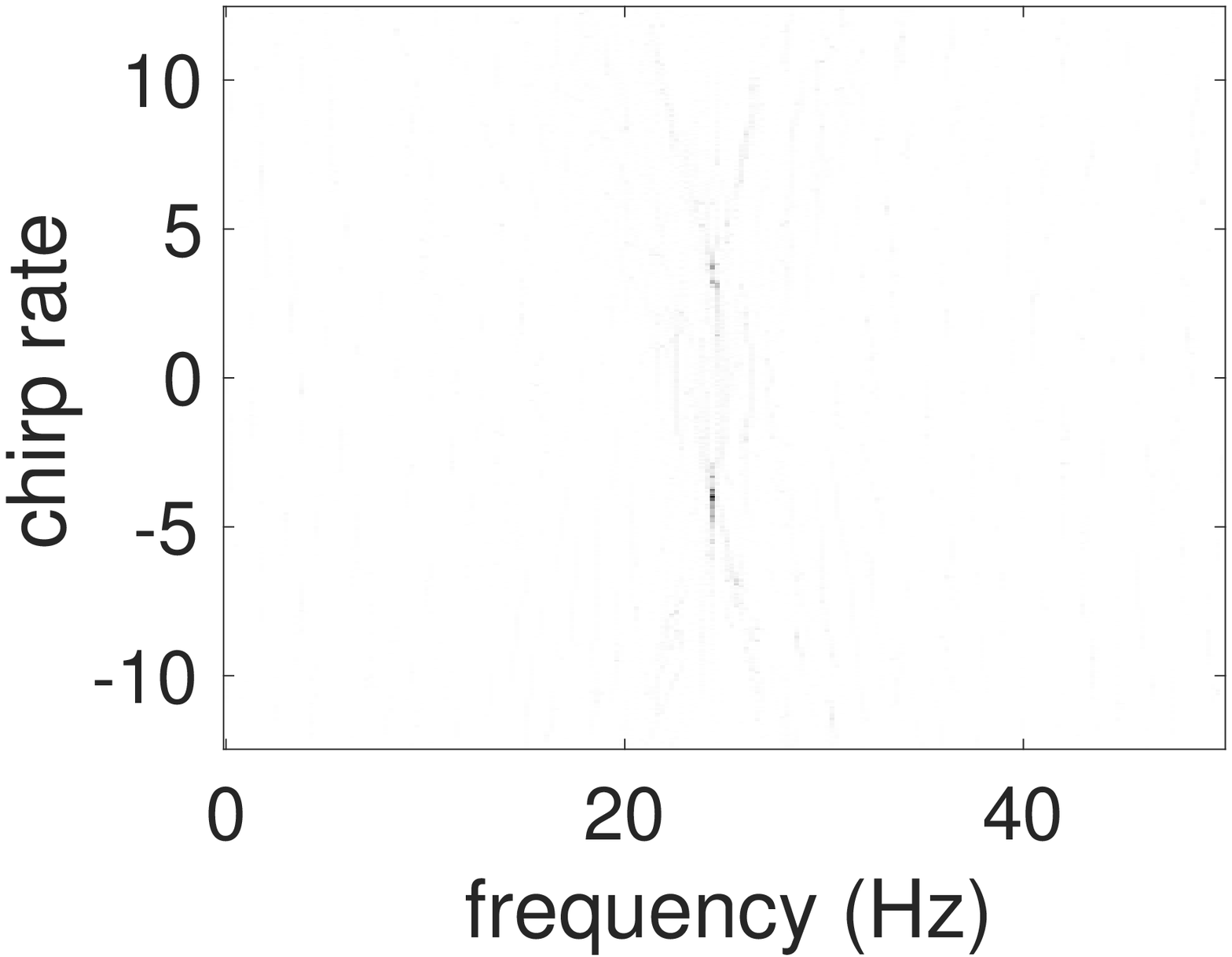}
	\includegraphics[width=.32\textwidth]{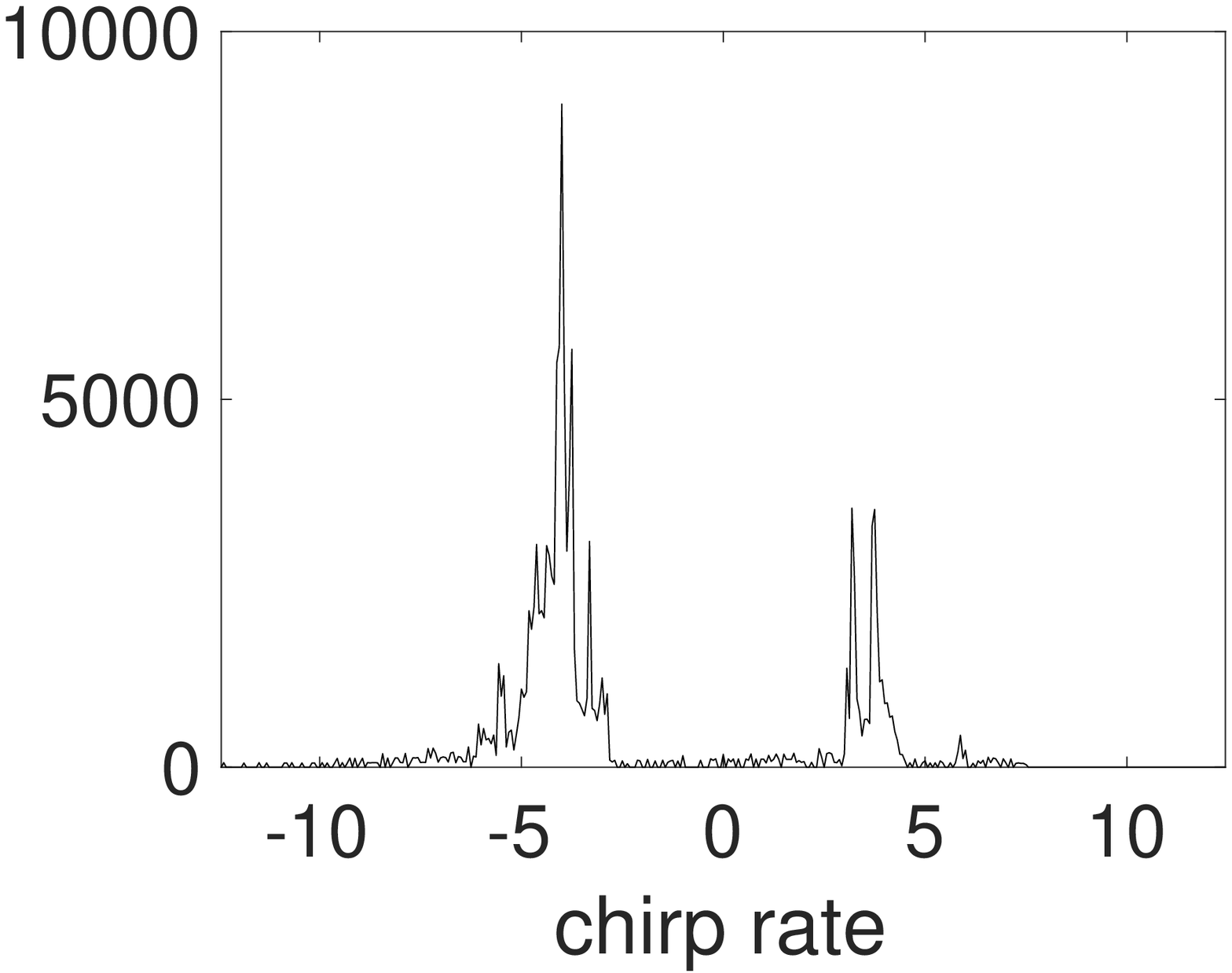}\\
	\includegraphics[width=.32\textwidth]{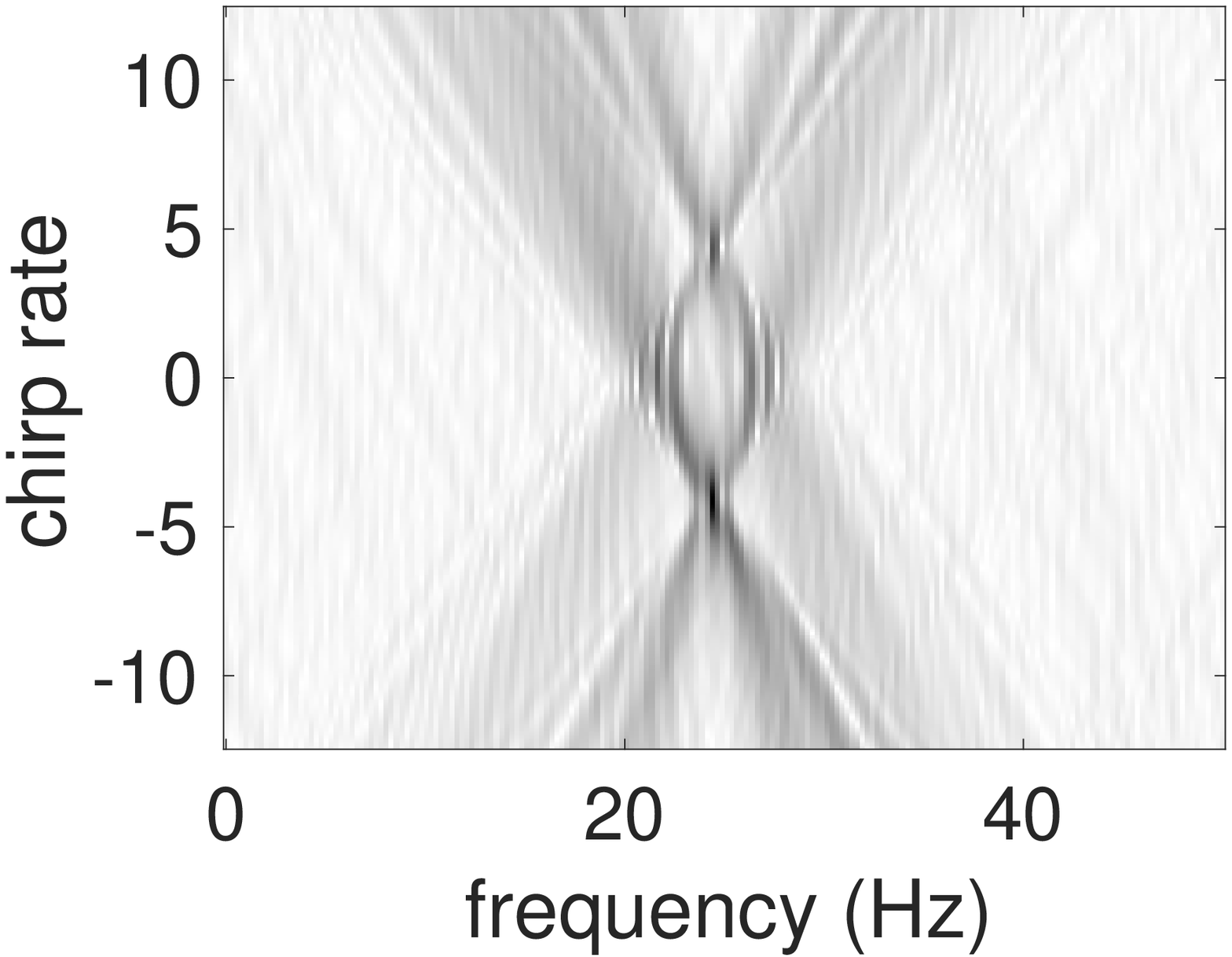}
	\includegraphics[width=.32\textwidth]{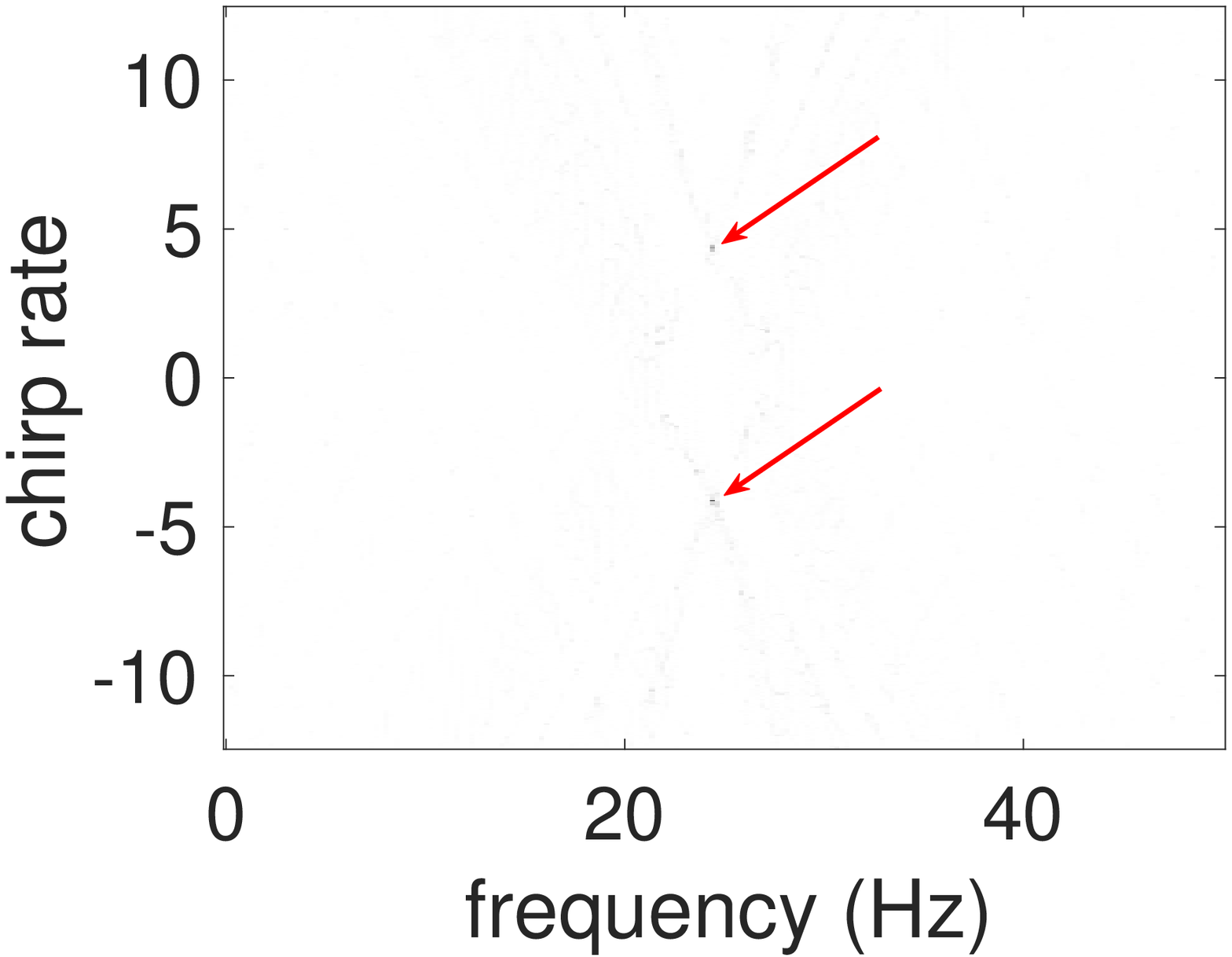}
	\includegraphics[width=.32\textwidth]{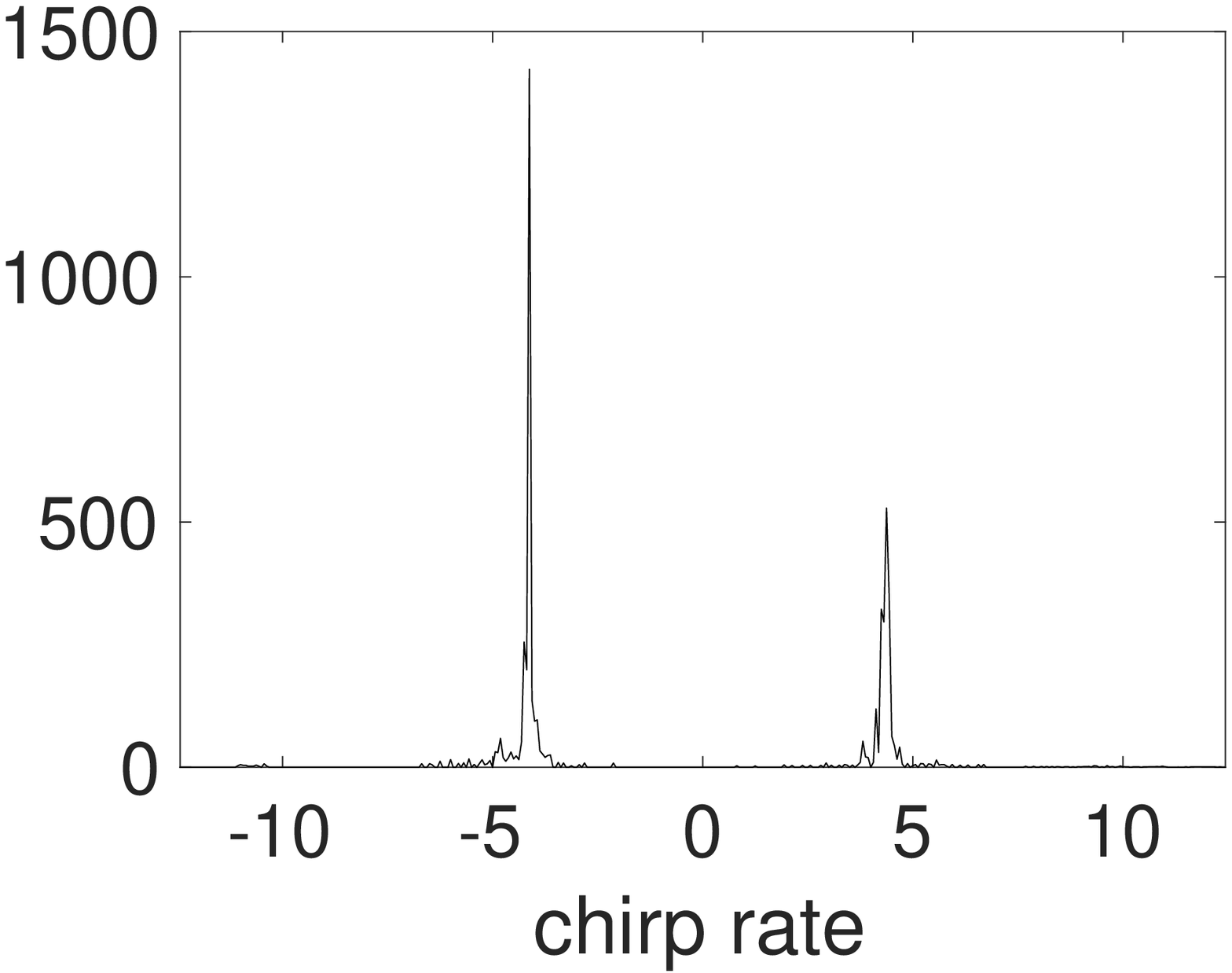}
	\caption{Top left: plot of $\abs{T_f^{(g_0)}(t_0,\xi,\lambda)}$; top middle: plot of $\abs{S_f^{(g_0)}(t_0,\xi,\lambda)}$; top right: plot of $\abs{S_f^{(g_0)}(t_0,\xi_0,\lambda)}$; bottom left: plot of $\abs{T_f^{(g_2)}(t_0,\xi,\lambda)}$; bottom middle: plot of $\abs{S_f^{(g_2)}(t_0,\xi,\lambda)}$ with arrows pointing out concentrations; bottom right: plot of $\abs{S_f^{(g_2)}(t_0,\xi_0,\lambda)}$.}
	\label{fig:9}
\end{figure}

Finally, we provide a quantitative result. Define the relative reconstruction error between the estimated signal $\tilde{f}$ and the original clean signal $f$ as $\frac{\norm{\tilde{f}-f}_2}{\norm{f}_2}$. We report the mean and standard deviation (SD) of the relative reconstruction error of $\Re(f_1)$ and $\Re(f_2)$ over 100 different realizations of the signal and noise, where their IFs and instantaneous chirp rates are extracted from SCT, $S_{f}^{(g_2)}(t,\xi,\lambda)$. The mean $\pm$ SD of $\frac{\norm{{\Re(\tilde f_1)}-\Re(f_1)}_2}{\norm{\Re(f_1)}_2}$ and $\frac{\norm{{\Re(\tilde f_2)}-\Re(f_2)}_2}{\norm{\Re(f_2)}_2}$ are  $0.154 \pm 0.091$
and $0.161\pm 0.084$ respectively.
We also report the mean and SD of the optimal transport (OT) distance between the estimated $\tilde{\phi}_k'$ from $S_{f}^{(g_2)}(t,\xi,\lambda)$ and $\phi_k'$, as well as $\hat{\phi}_k'$ from $T_{f}^{(g_2)}(t,\xi,\lambda)$ and $\phi_k'$, where $k=1,2$. The OT distance is obtained by first computing the 1-dimensional Wasserstein-1 distance between $\tilde{\phi}_k'(x)$ (or $\hat{\phi}_k'(x)$) and $\phi_k'(x)$, where $k=1,2$, at every time $x$ using the Euclidean distance as metric, and then taking average over all $x\in [1,9]$. The resulting mean $\pm$ SD for $\text{OT}(\tilde{\phi}'_1, \phi'_1)$, $\text{OT}(\hat{\phi}'_1, \phi'_1)$, $\text{OT}(\tilde{\phi}'_2, \phi'_2)$ and $\text{OT}(\hat{\phi}'_2, \phi'_2)$ are $0.424 \pm 0.280$, 
$0.854\pm 1.778$, $0.325\pm 0.252$, and $0.453\pm 0.728$ respectively.
We conclude that the IFs estimated from $S_{f}^{(g_2)}(t,\xi,\lambda)$ are more accurate than those estimated from $T_{f}^{(g_2)}(t,\xi,\lambda)$.

\begin{figure}[!htbp]
	\centering
	\includegraphics[width=0.48\textwidth]{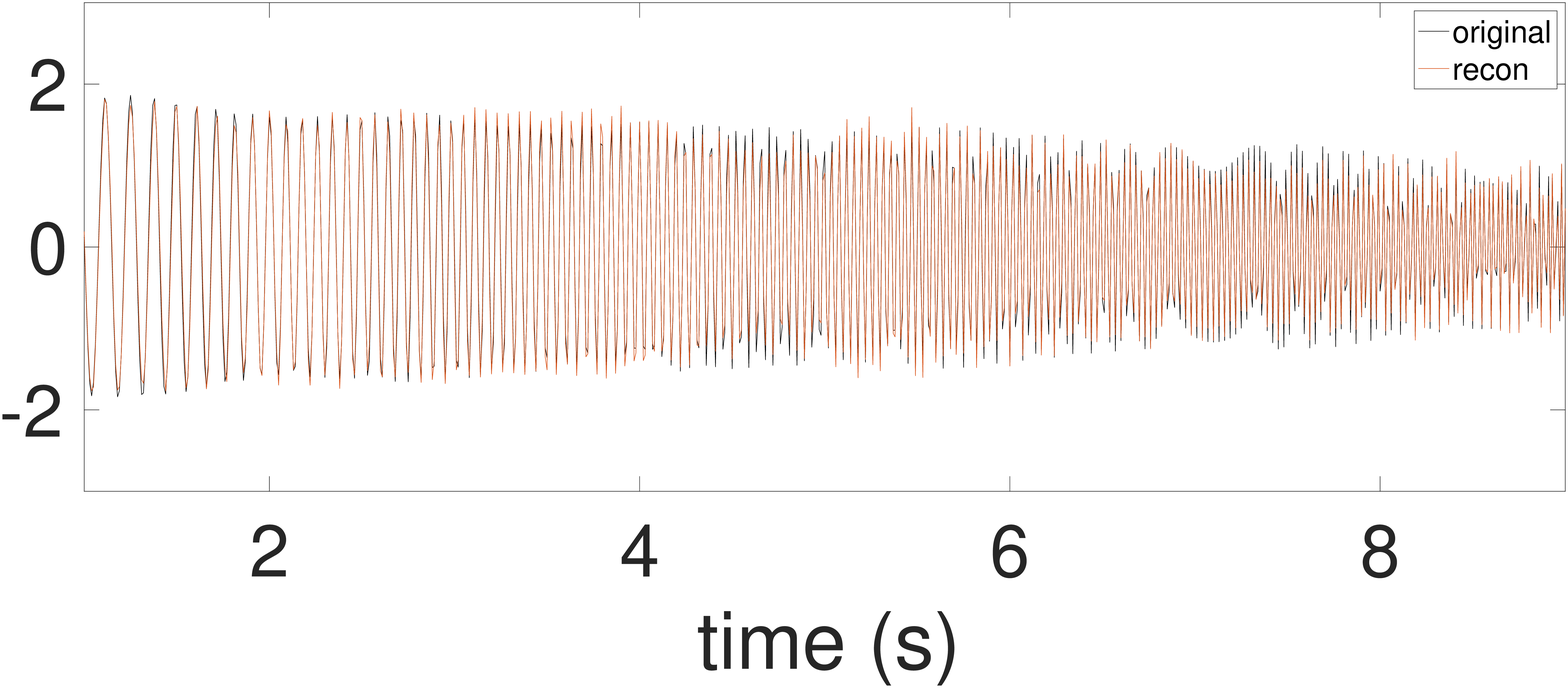}
	\includegraphics[width=0.48\textwidth]{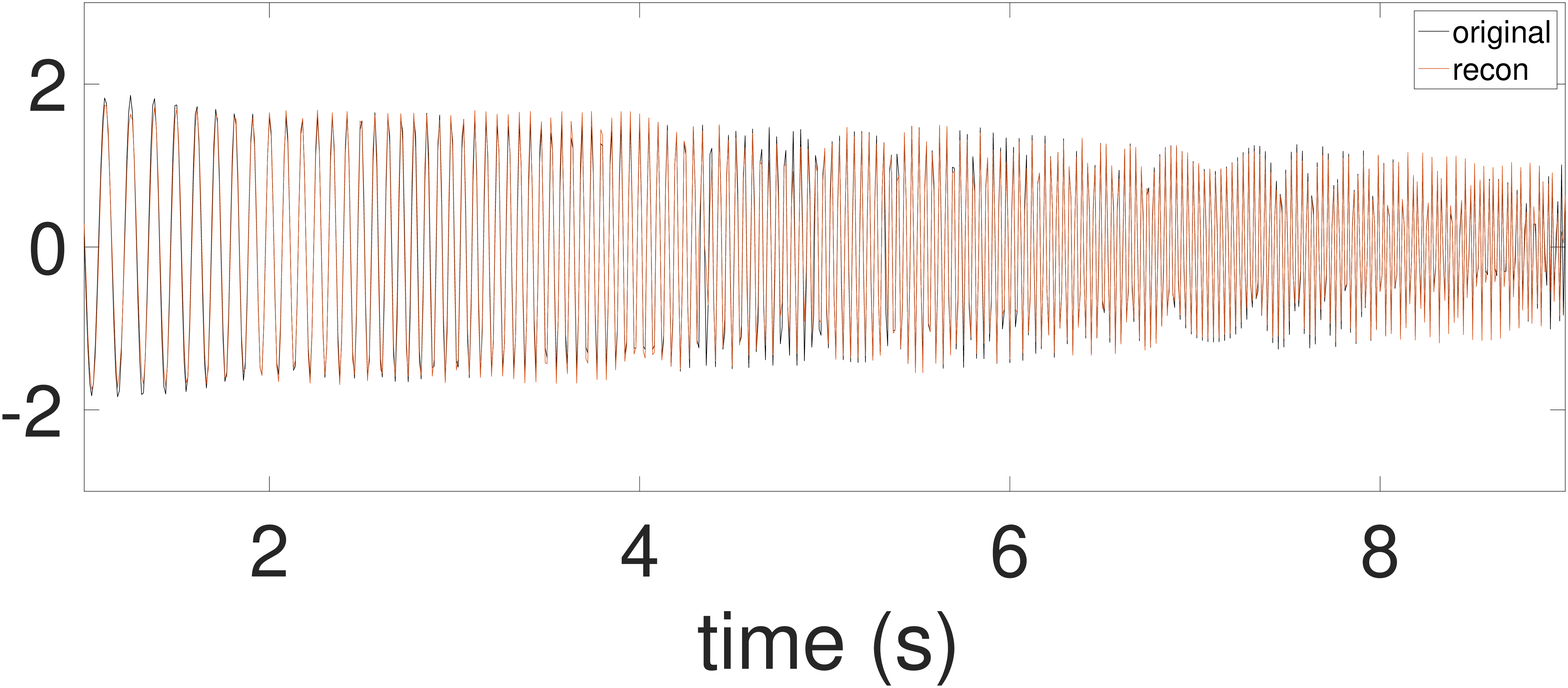}\\
	\includegraphics[width=0.48\textwidth]{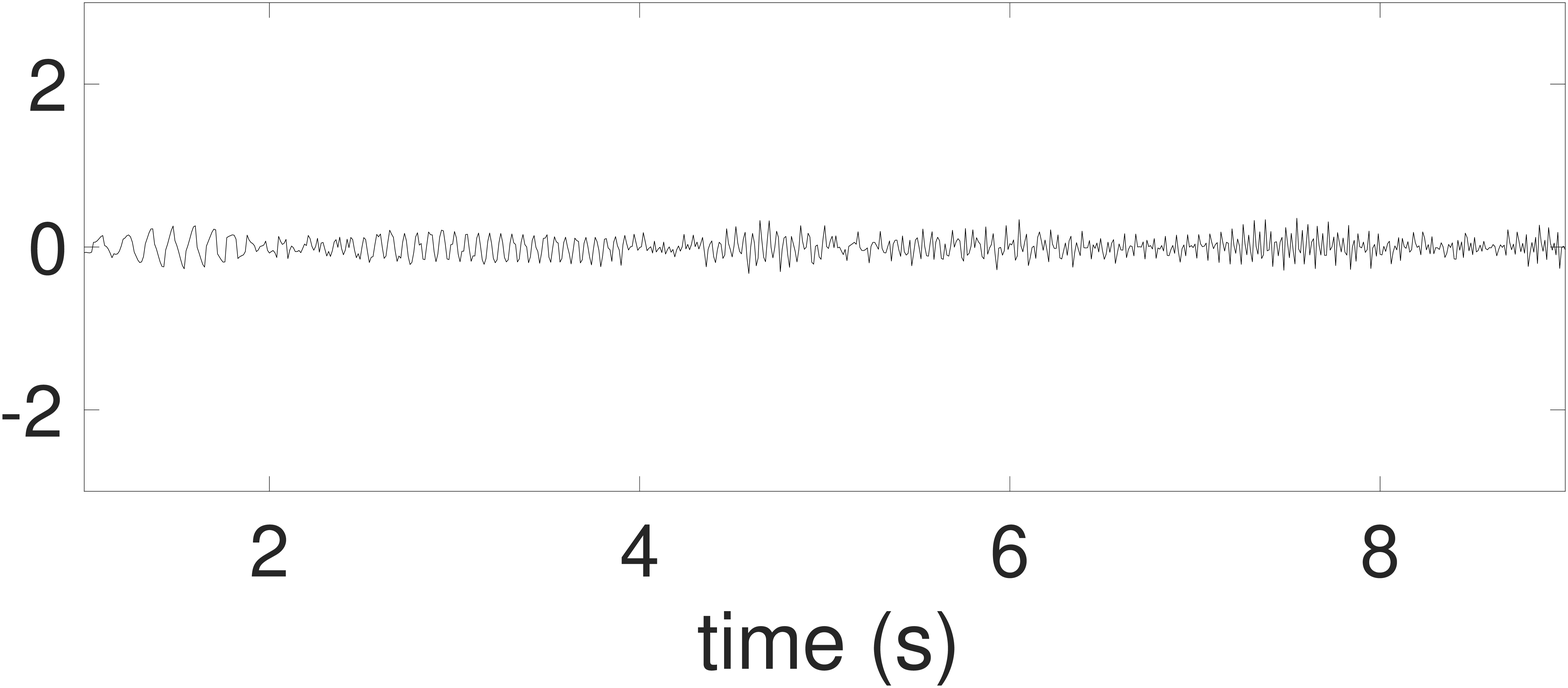}
	\includegraphics[width=0.48\textwidth]{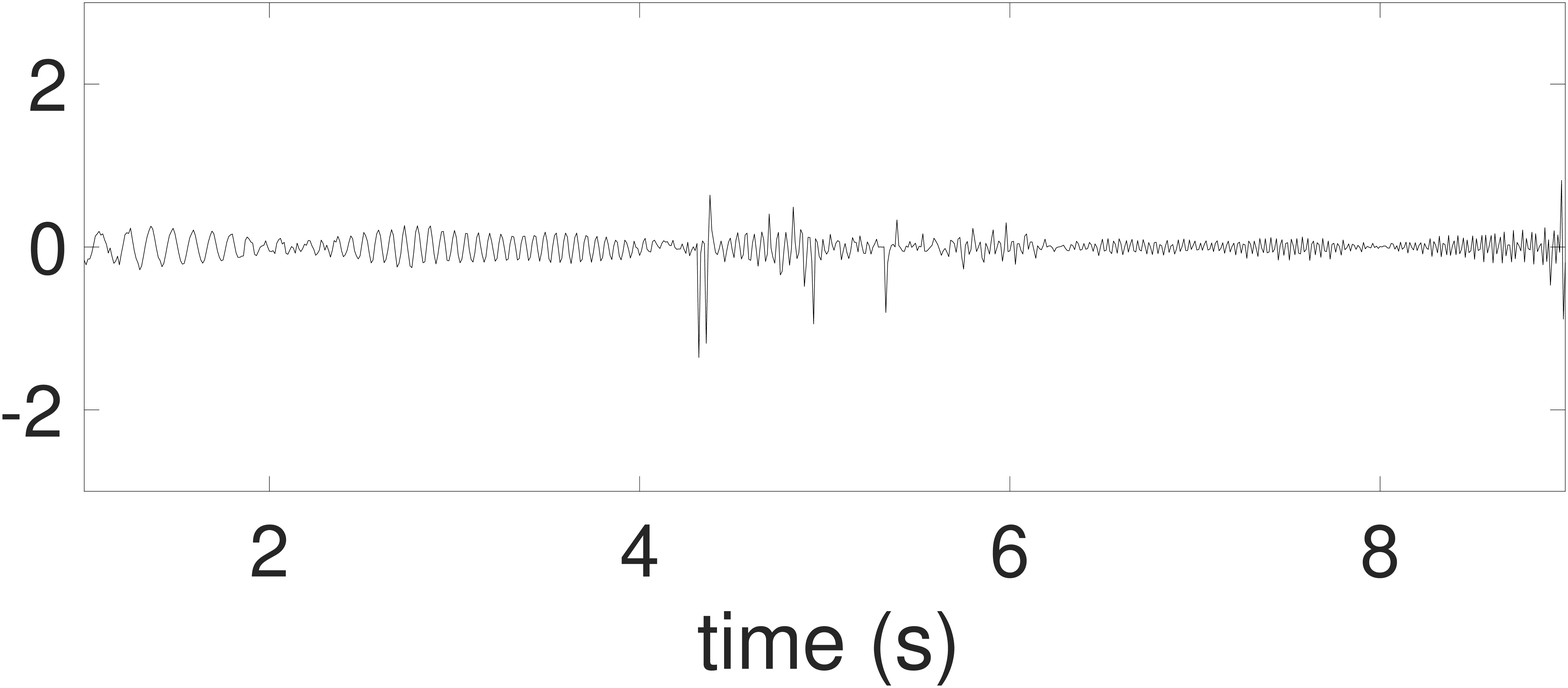}\\
	\includegraphics[width=0.48\textwidth]{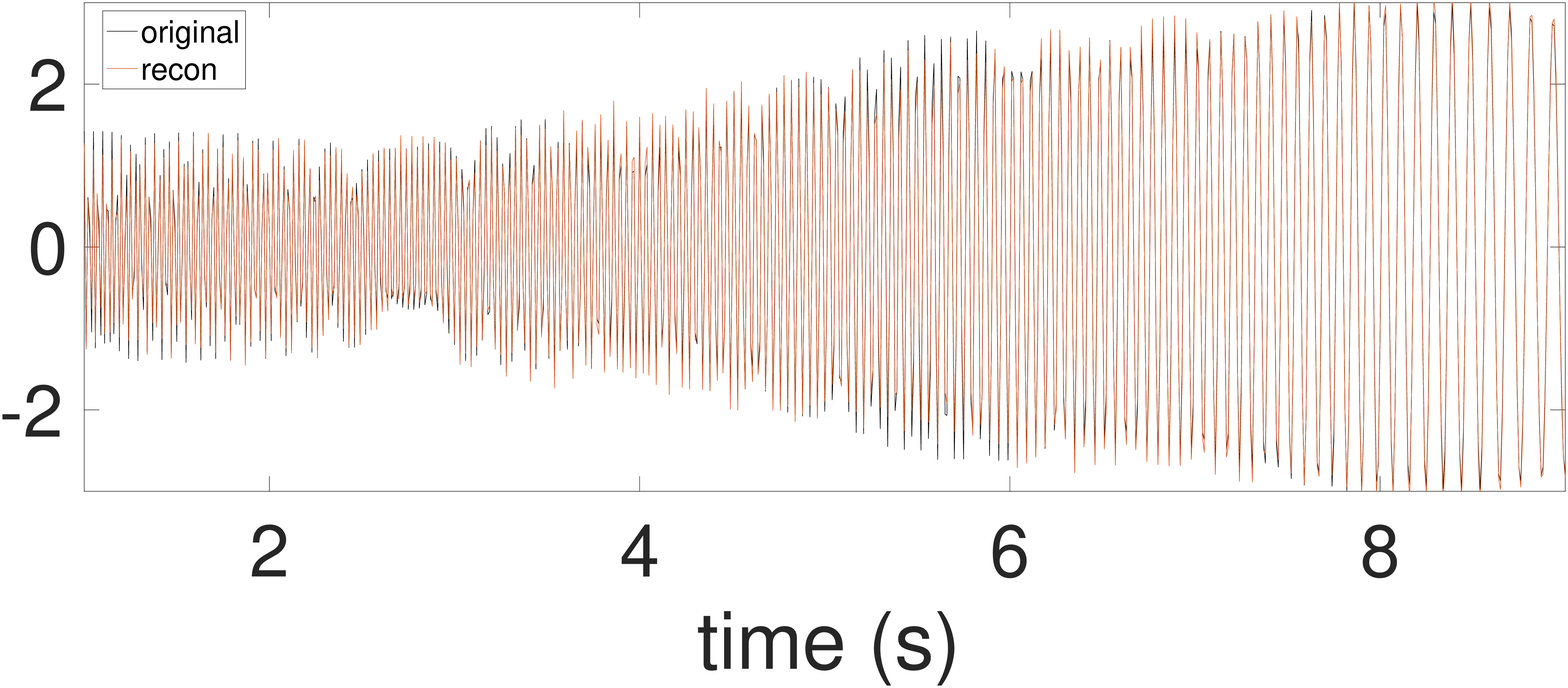}
	\includegraphics[width=0.48\textwidth]{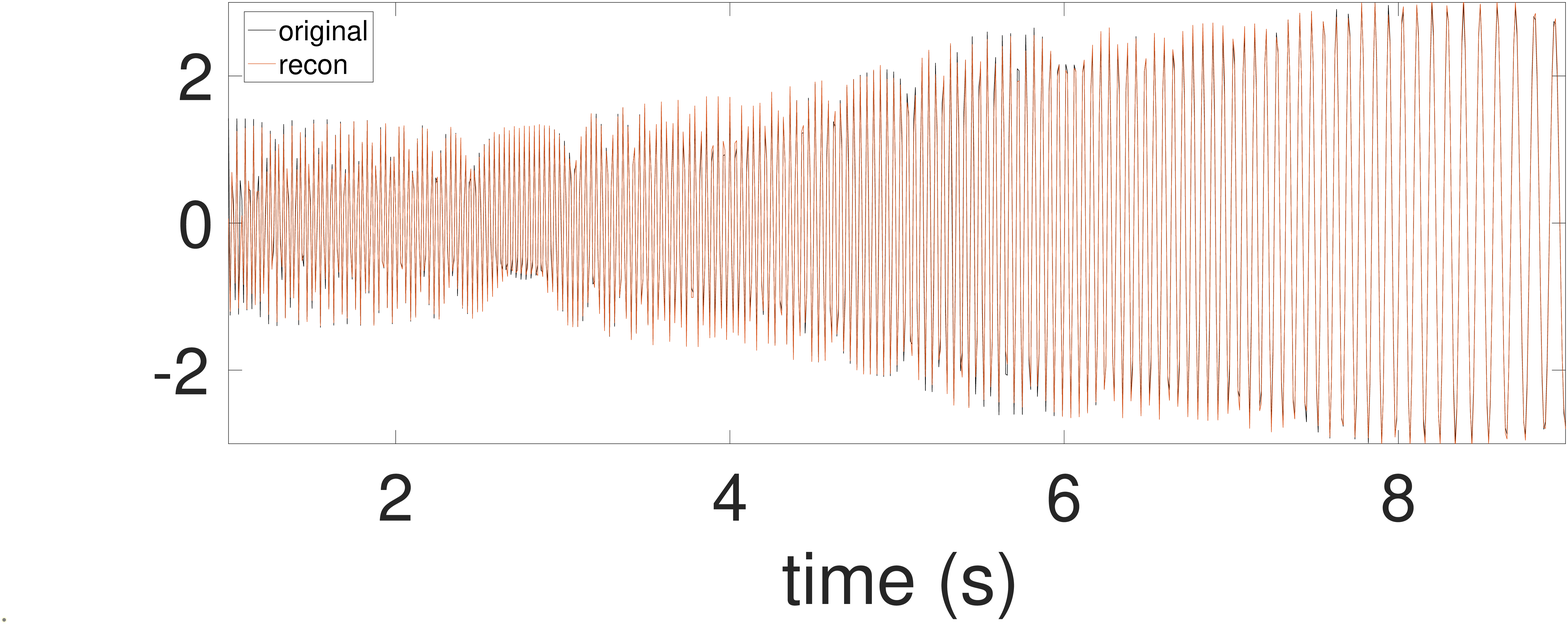}\\
	\includegraphics[width=0.48\textwidth]{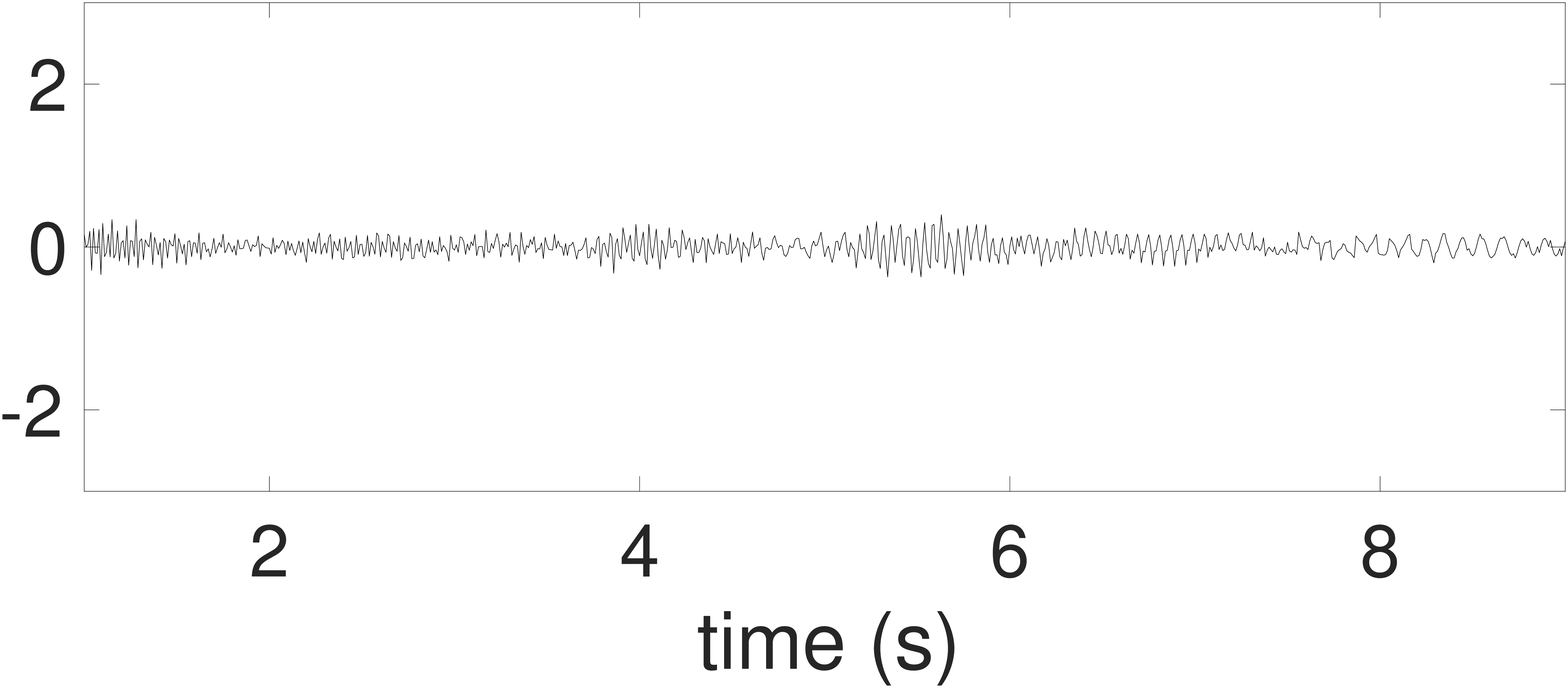}
	\includegraphics[width=0.48\textwidth]{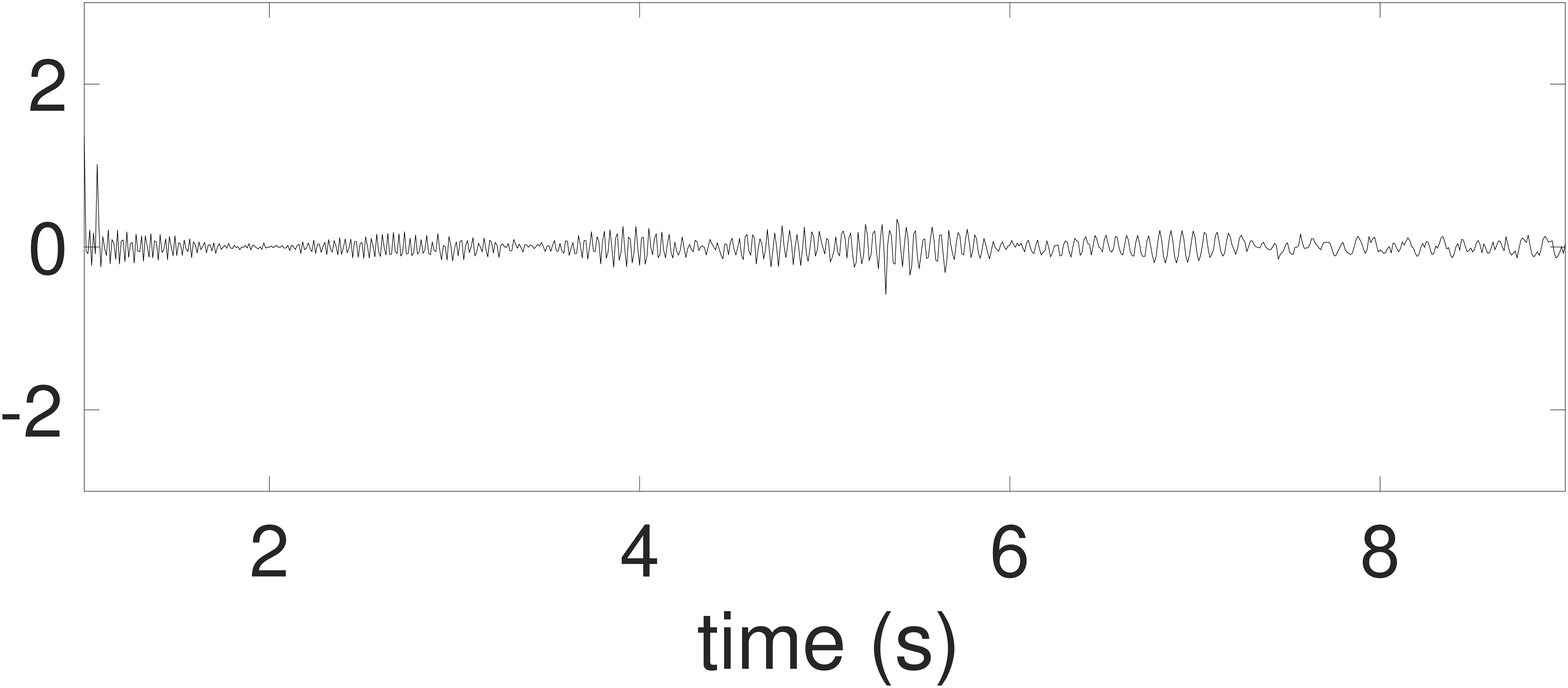}
	\caption{Left column: reconstruction of signals using group idea with window $g_0$, assuming we know the true instantaneous frequency and chirp rate. Top row: reconstruction of $\Re(f_1)$; second row: reconstruction error of $\Re(f_1)$; third row: reconstruction of $\Re(f_2)$; bottom row: reconstruction error of $\Re(f_2)$. Right column: reconstruction of signals using group idea with window $g_0$, the instantaneous frequency and chirp rate are extracted from $S_{f}^{(g_2)}$. Top row: reconstruction of $\Re(f_1)$; second row: reconstruction error of $\Re(f_1)$; third row: reconstruction of $\Re(f_2)$; bottom row: reconstruction error of $\Re(f_2)$. }
	\label{fig:10}
\end{figure}

\subsection{Wolf Howling Signal}
We now show a real world data analysis. An important issue in conservation biology is to estimate the number of wolves in the field \cite{wolf,dugnol2008chirplet}. In this final example we show the numerical result with a wolf howling signal. The sound is downloaded from Wolf Park website.\footnote{https://wolfpark.org/Images/Resources/Howls/} The signal $f$ is sampled at 8 kHz for 55 s. In this example, we downsample the signal $f$ by a factor 8 to look at the crossover IFs, which is below 500Hz. Figure \ref{fig:12} shows the signal and the TFC representation and the associated TF representation determined by CT and SCT with $g_0$. Figure \ref{fig:13} shows the frequency-chirp rate slices at three instants 16.16s, 16.43s and 16.57s with window $g_0$ and $g_2$. We see in Figure \ref{fig:13} that at 16.16s, the IFs of the two components (two wolves) are well-separated. We do not expect any problem and the SCT with $g_0$ gives a clear representation of their IFs and ICs, particularly in the chirp rate axis compared with CT. At about 16.43s, when their IFs are very close and their ICs are reasonably separated, two components are mixed in the chirp rate direction in the SCT with $g_0$; however, we can observe two places of concentration in the frequency-chirp rate slice of the SCT with $g_2$, which is also sharper compared to the CT at that time. At 16.57s near another instance of frequency crossover, their chirp rates are close, and it becomes hard to distinguish those two components even in the SCT with $g_2$.
While solving this wolf counting problem is out of the scope of this paper, we leave this problem to the future work.

\begin{figure}[!htbp]
	\centering
	\includegraphics[width=.8\textwidth]{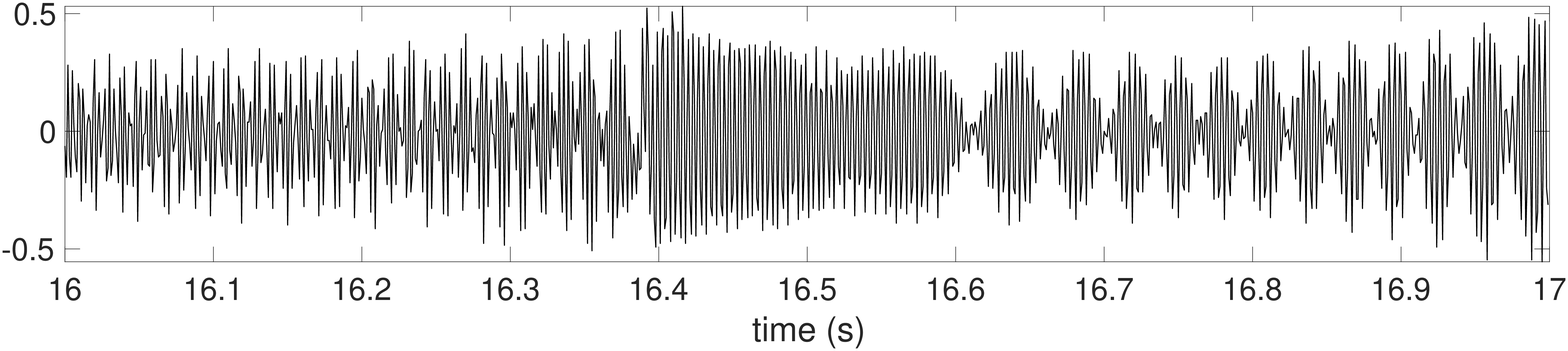}\\
	\includegraphics[width=.325\textwidth]{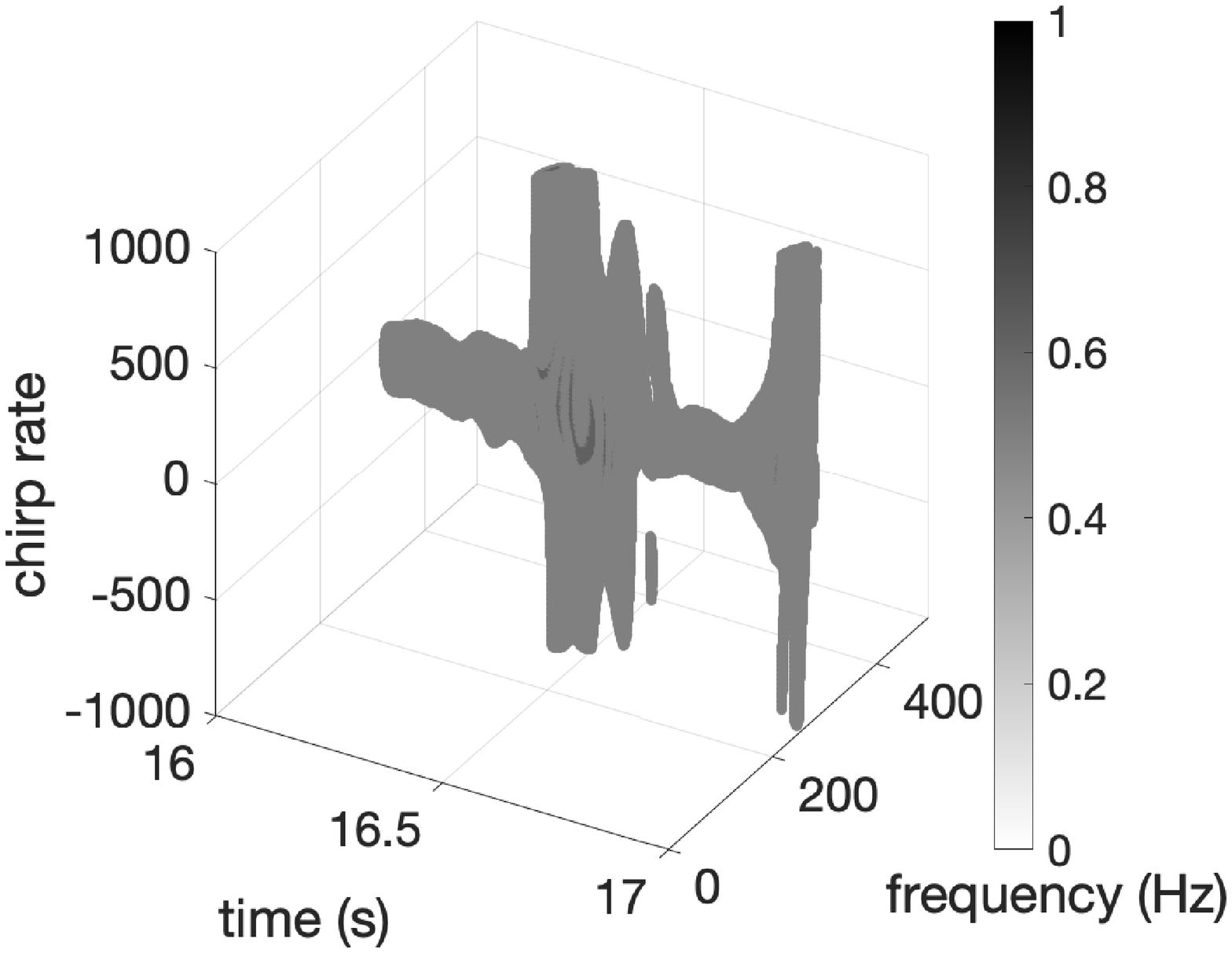}
	\includegraphics[width=.325\textwidth]{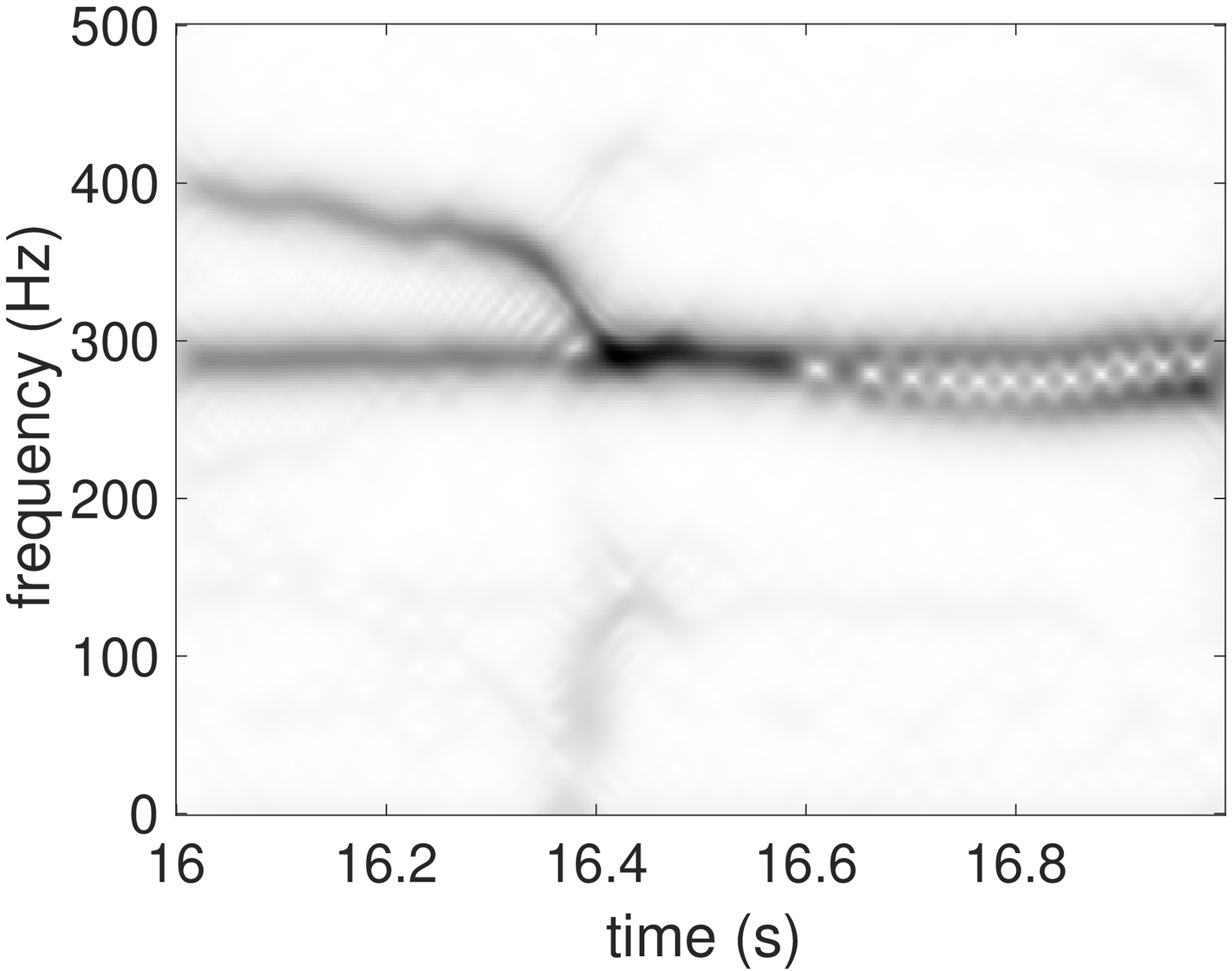}\\
	\includegraphics[width=.325\textwidth]{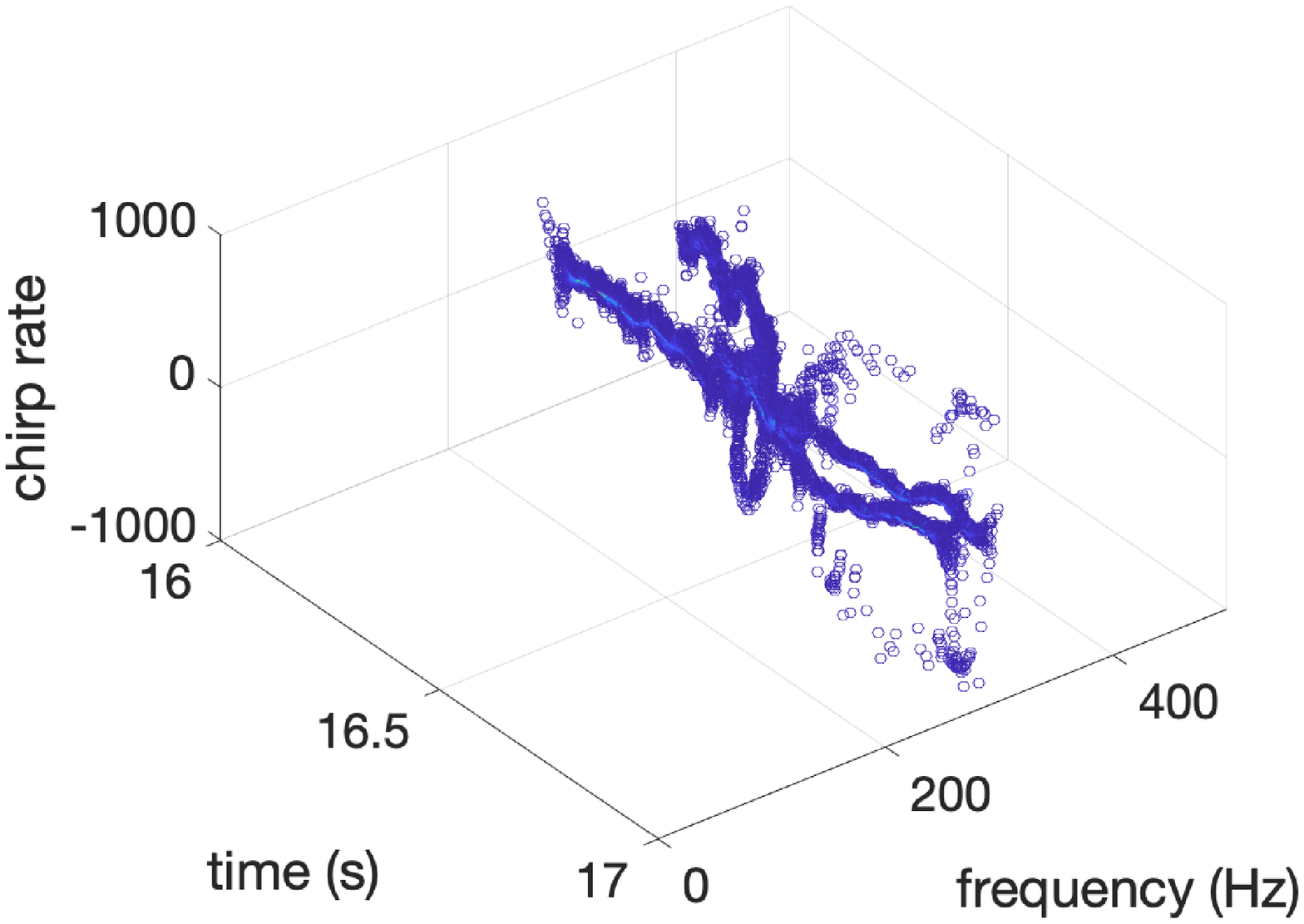}
	\includegraphics[width=.325\textwidth]{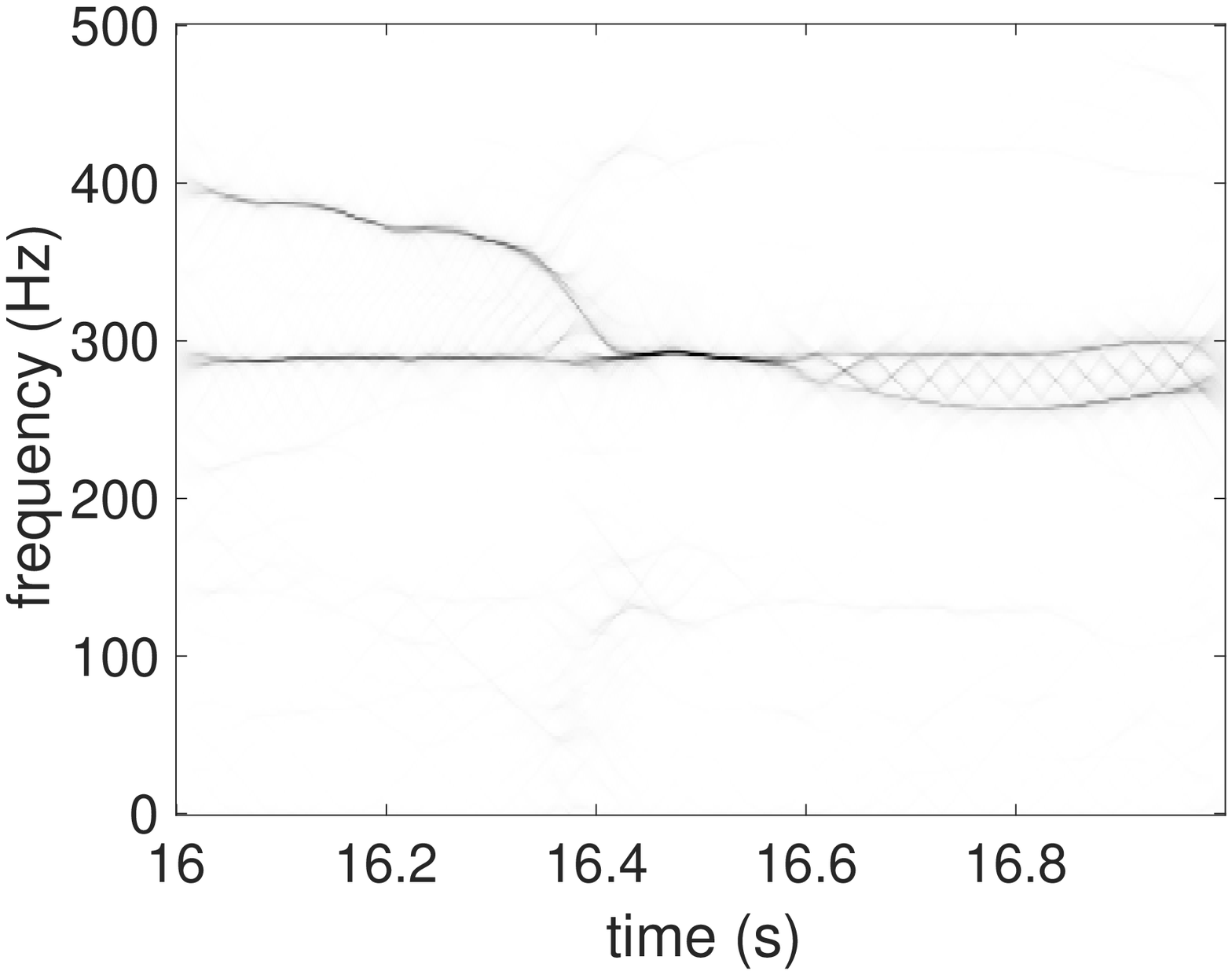}
	\caption{Top: a wolf howling signal between 16 and 17 seconds. Second row: the chirplet transform of the signal with a Gaussian window with the half window width 0.16 second and its associated TF representation. Third row: the synchrosqueezed chirplet transform of the signal with the same window and its associated TF representation.}
	\label{fig:12}
\end{figure}

\begin{figure}[!htbp]
	\includegraphics[width=.325\textwidth]{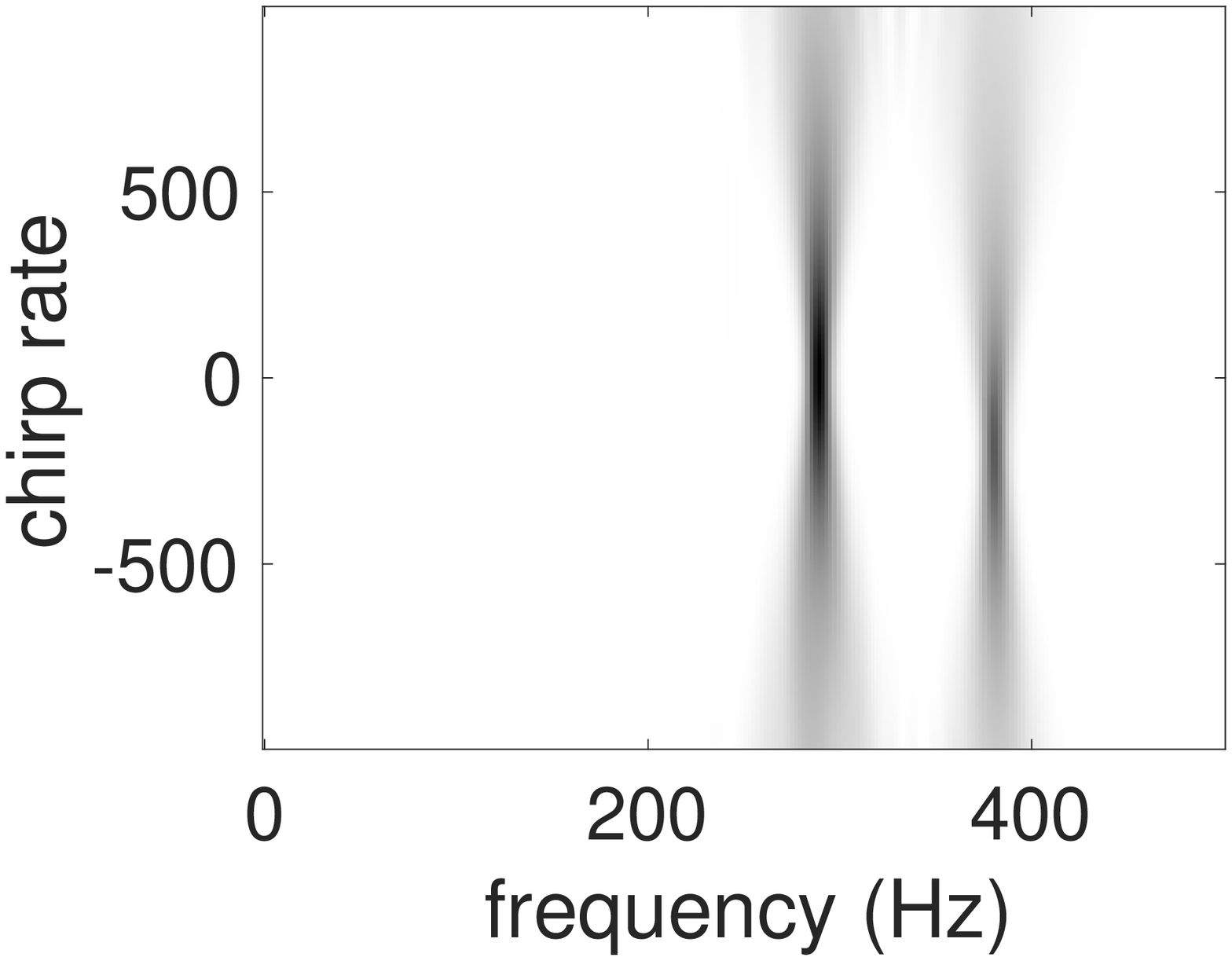}
	\includegraphics[width=.325\textwidth]{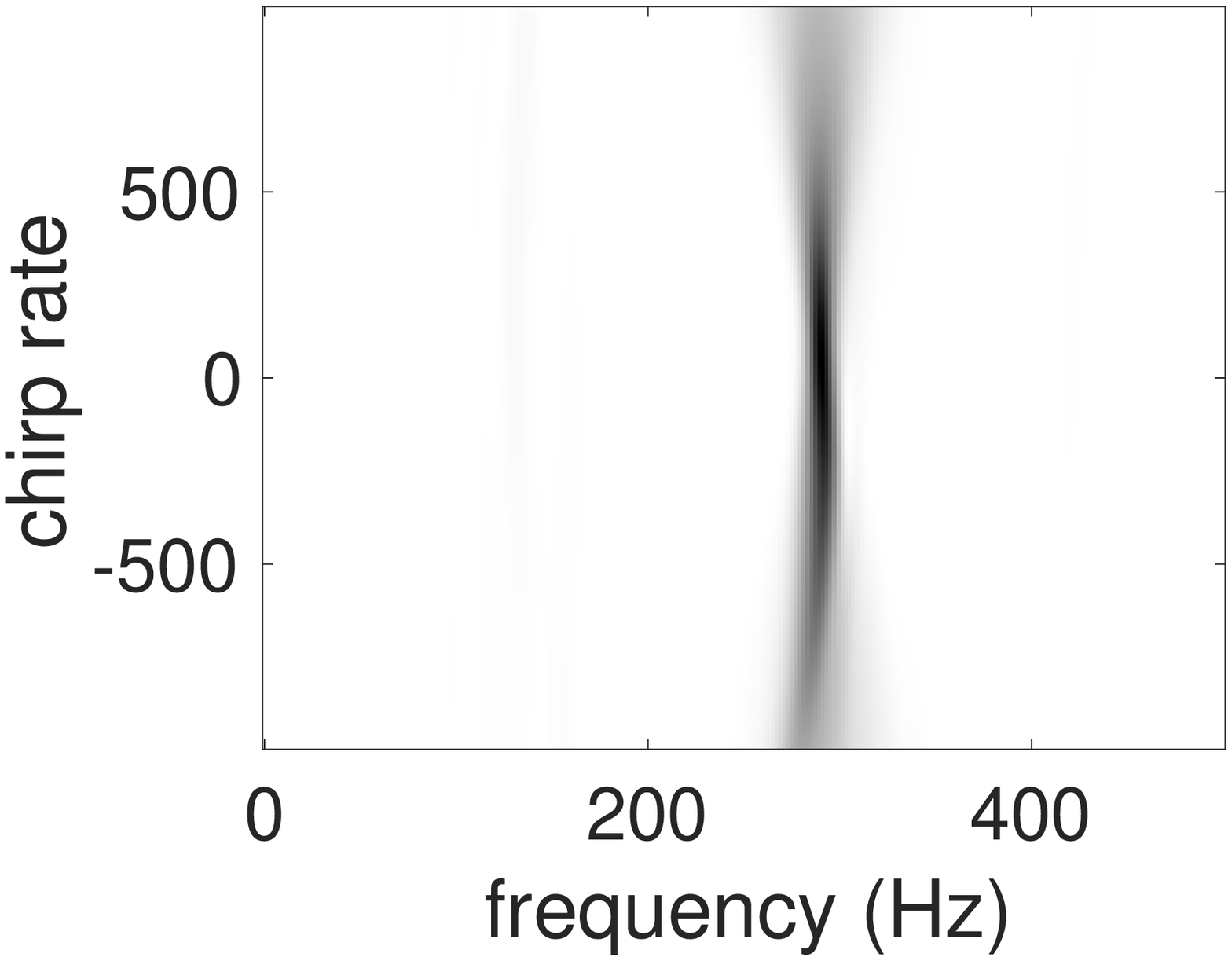}
	\includegraphics[width=.325\textwidth]{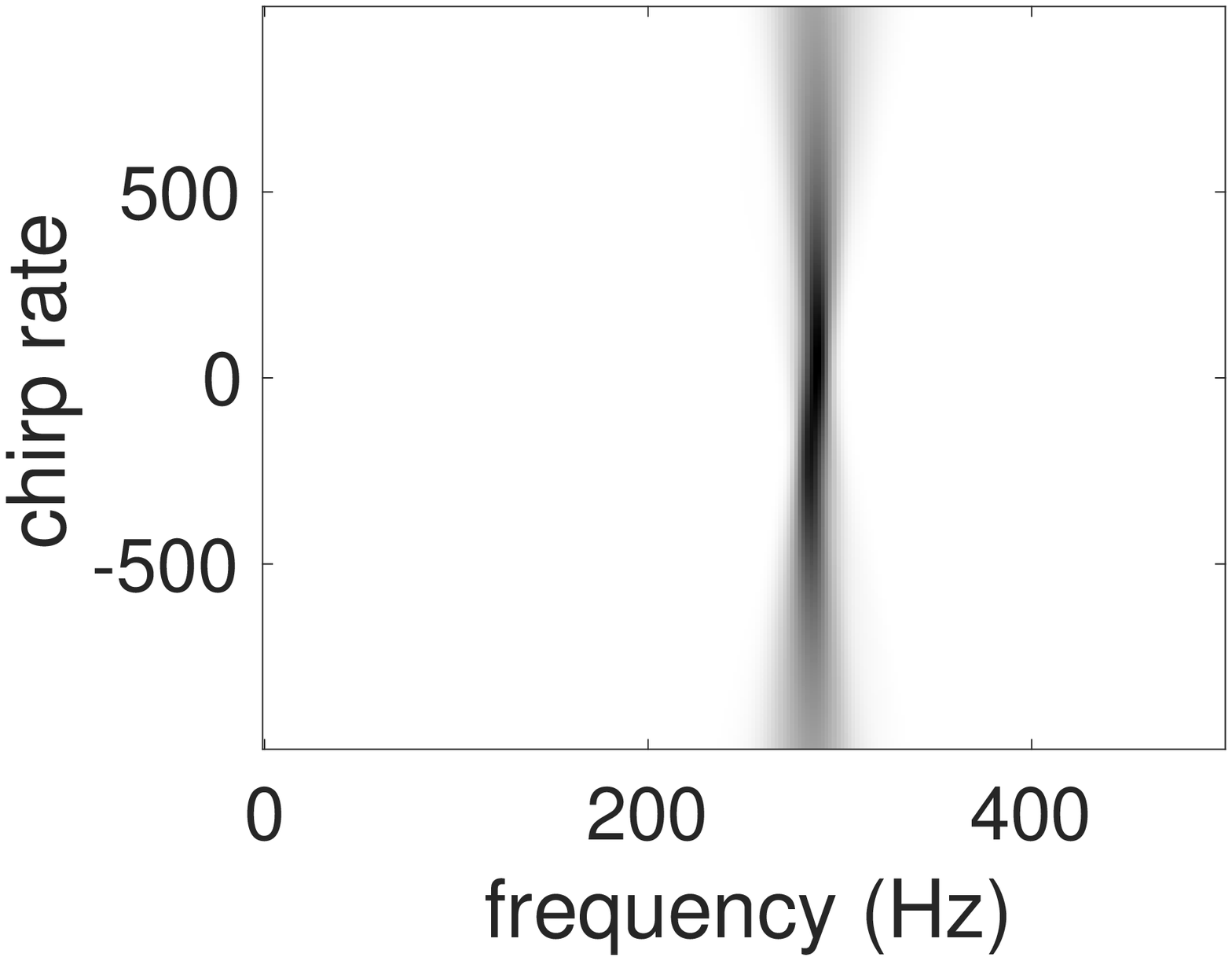}\\
	\includegraphics[width=.325\textwidth]{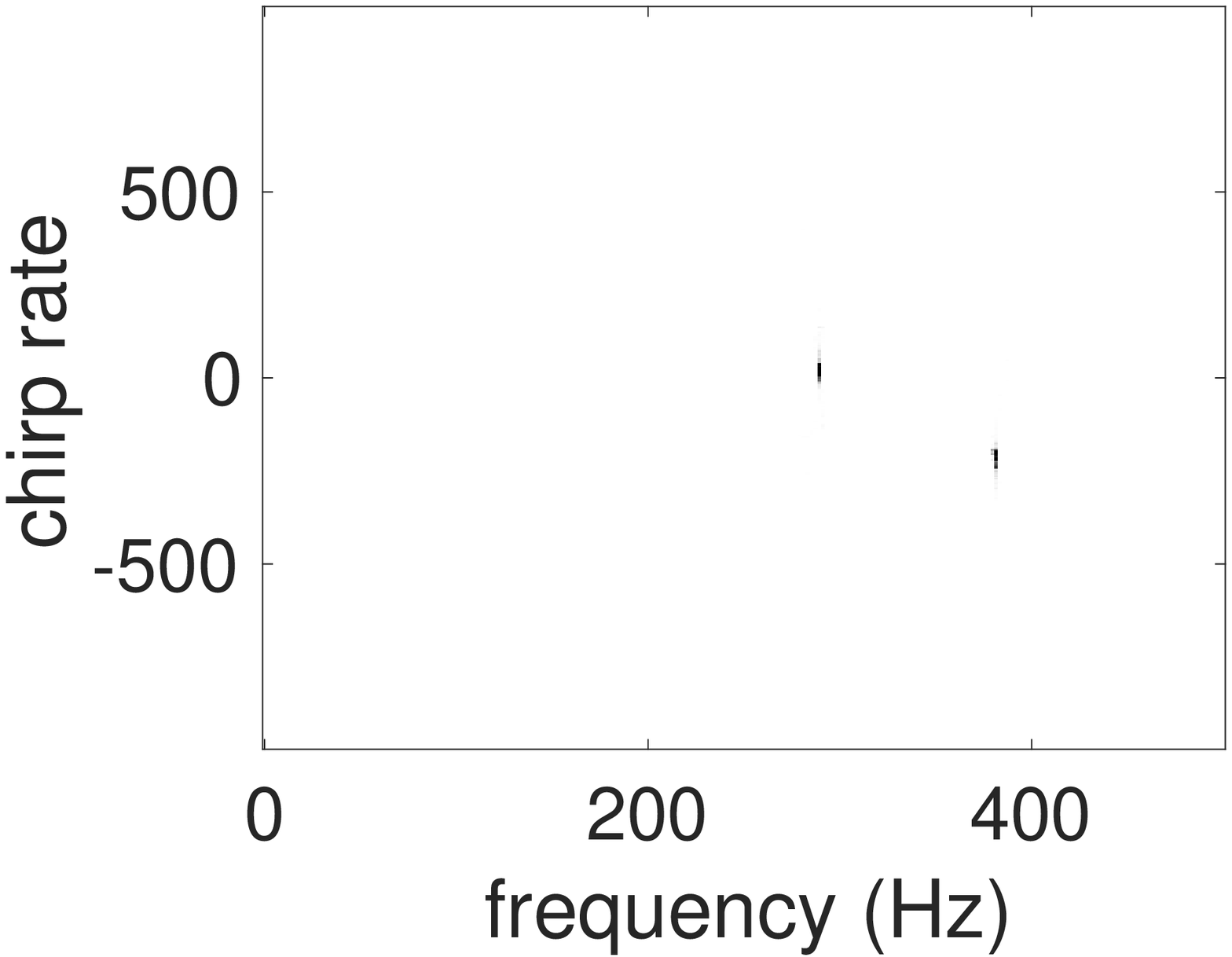}
	\includegraphics[width=.325\textwidth]{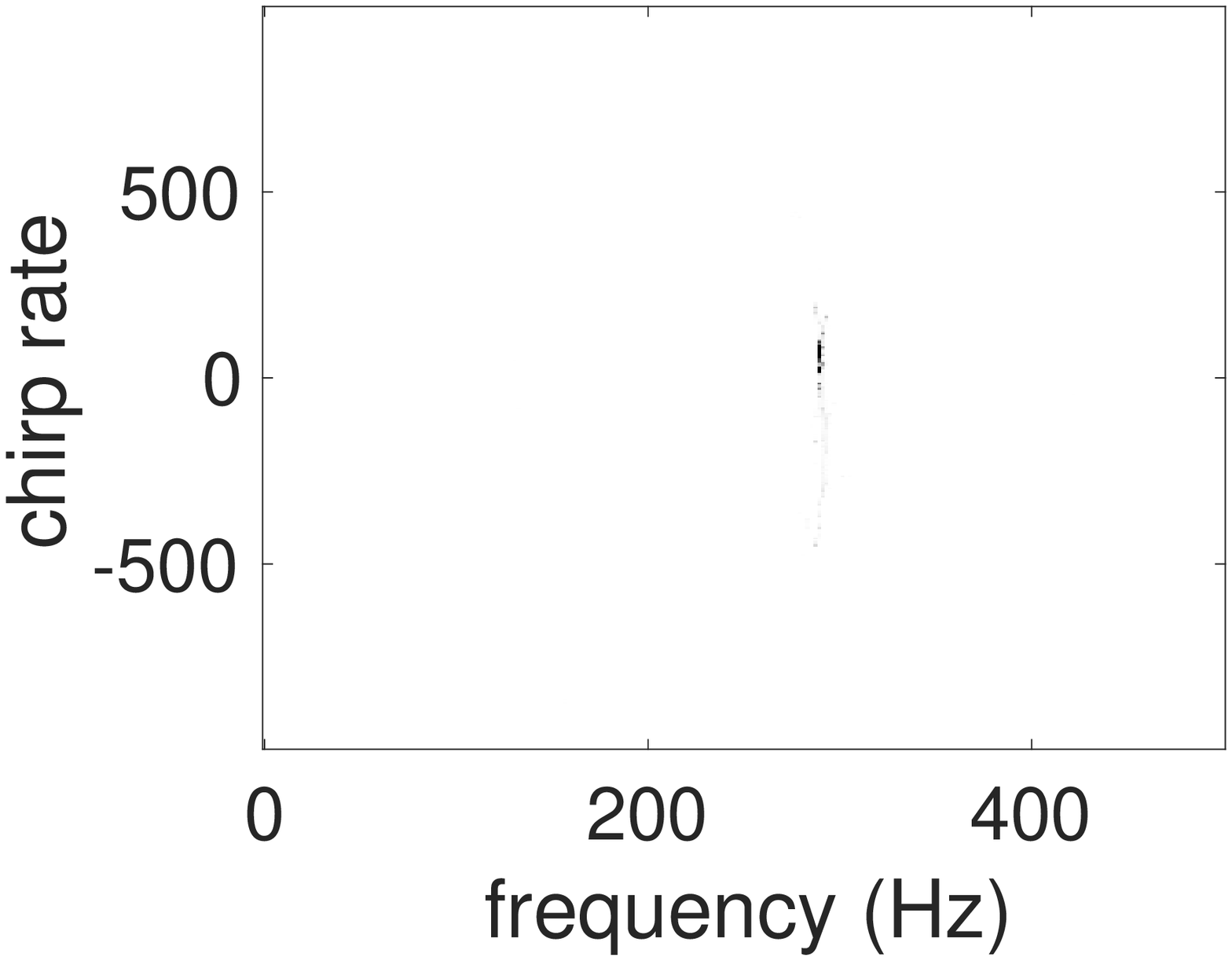}
	\includegraphics[width=.325\textwidth]{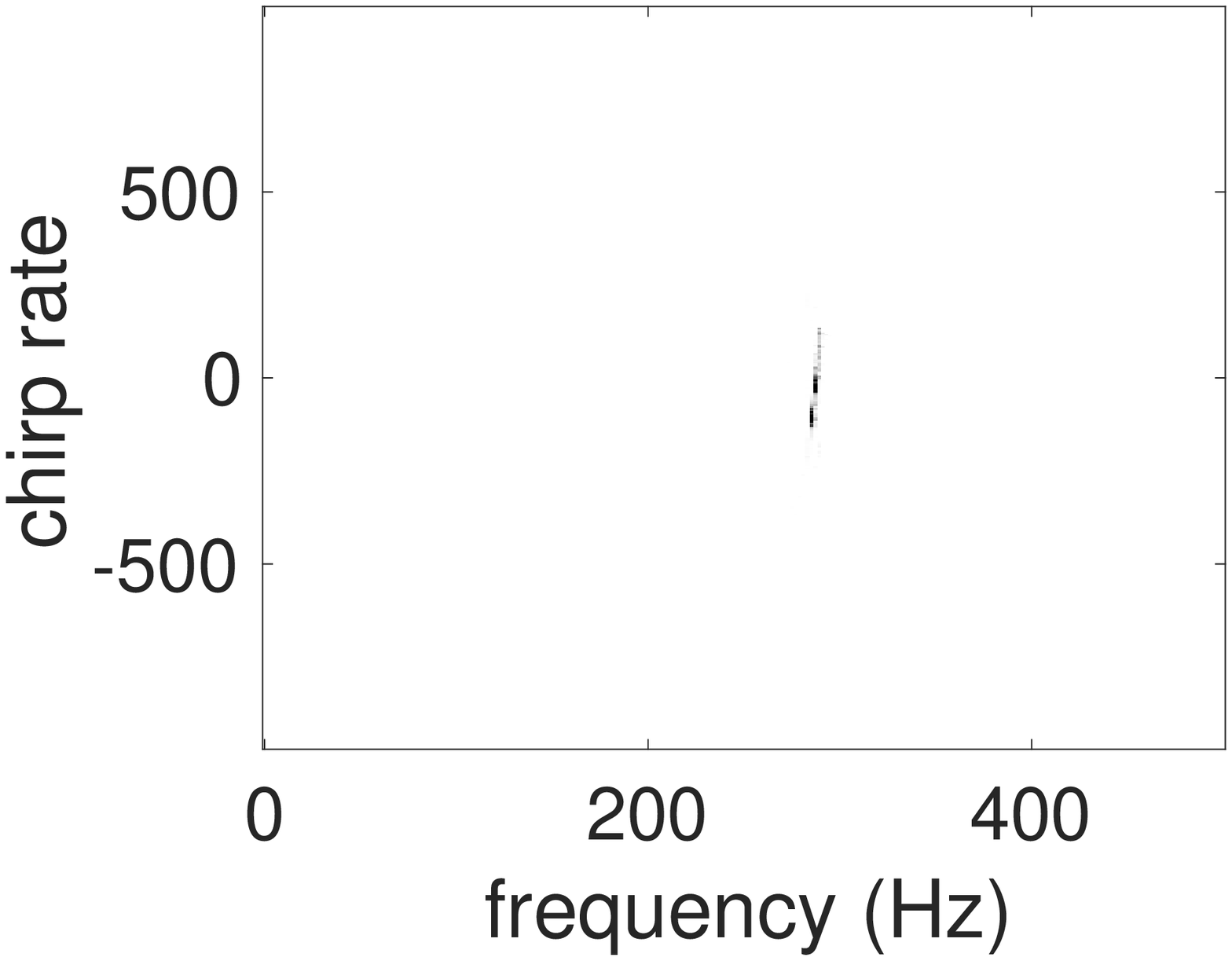}\\
	\includegraphics[width=.325\textwidth]{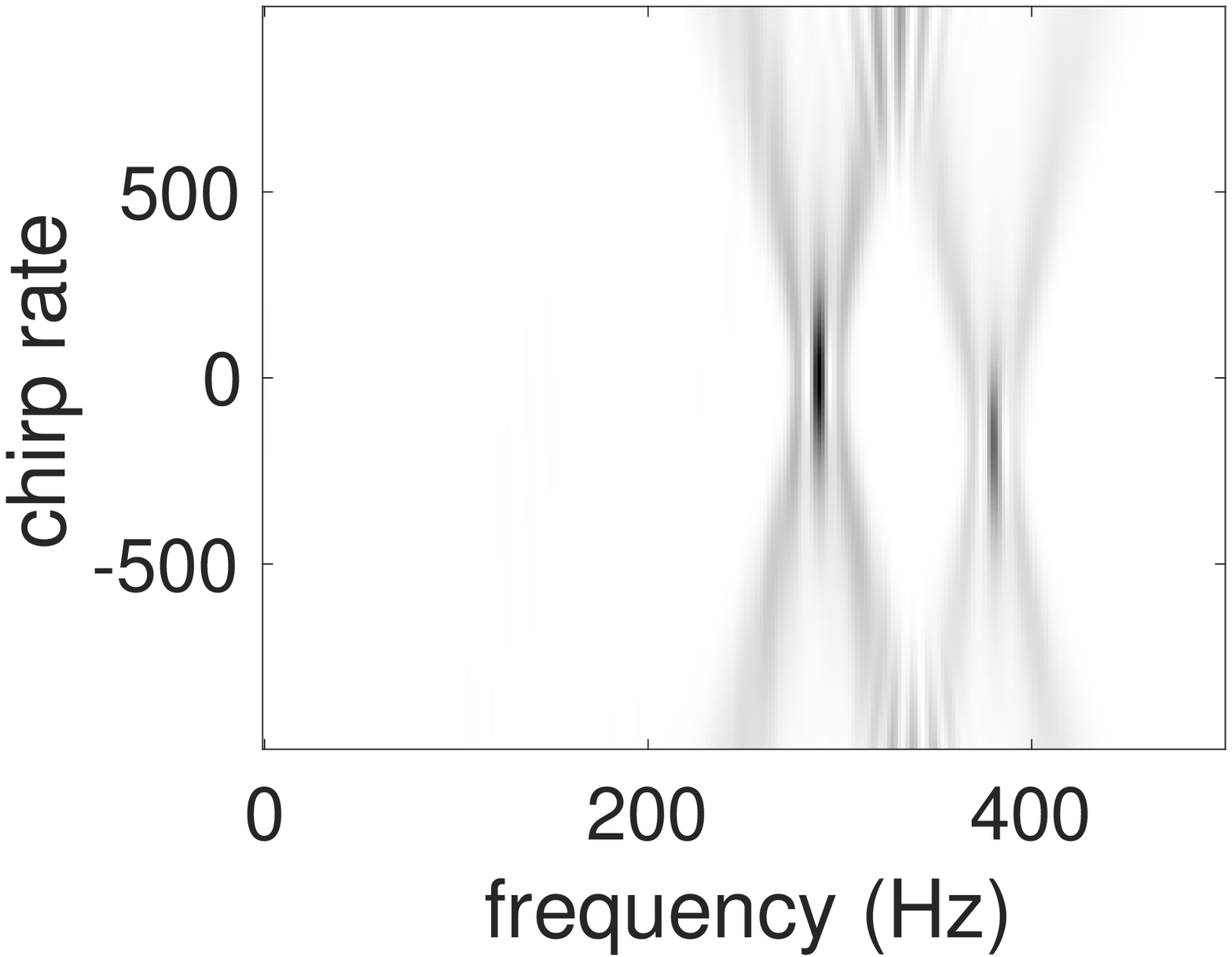}
	\includegraphics[width=.325\textwidth]{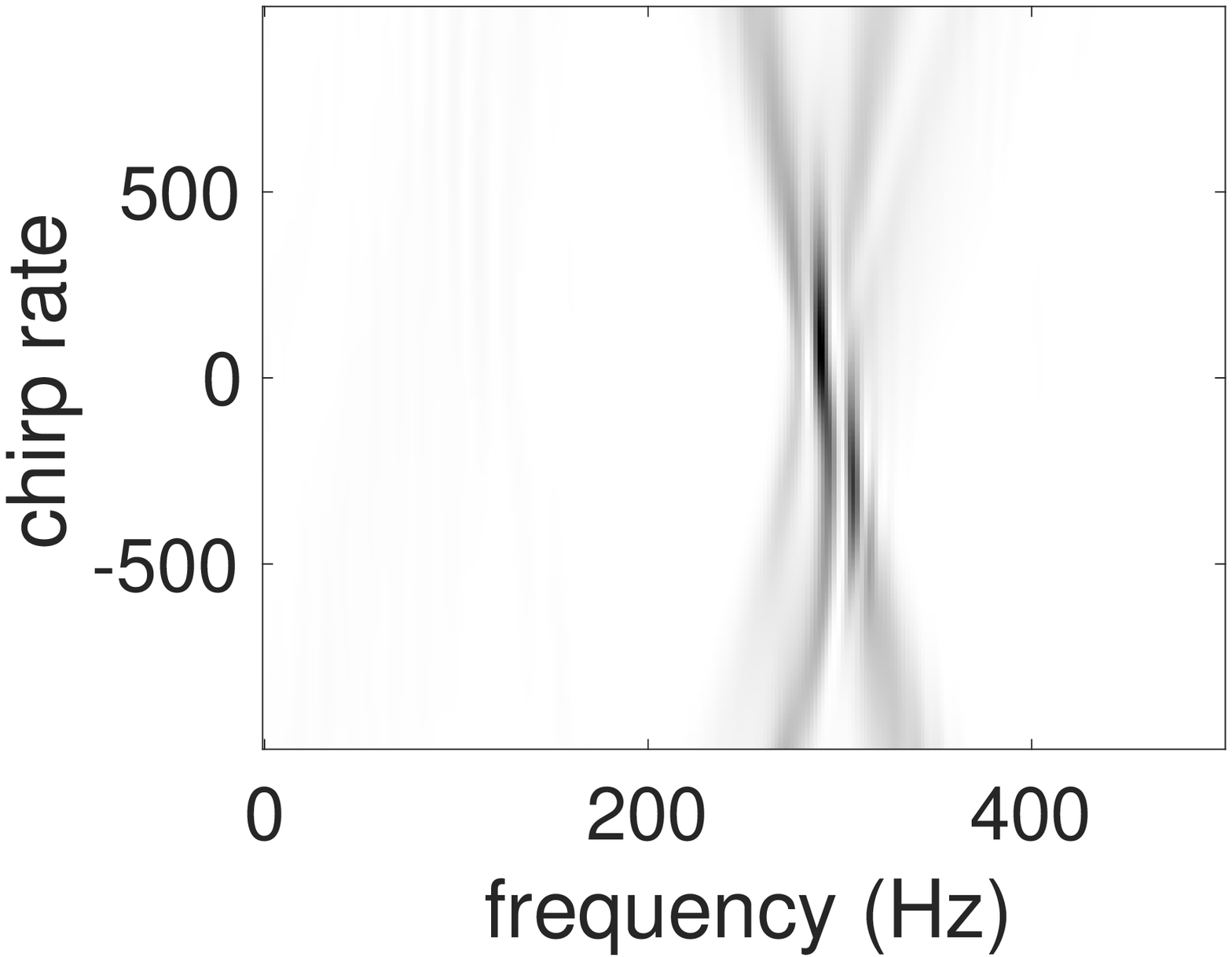}
	\includegraphics[width=.325\textwidth]{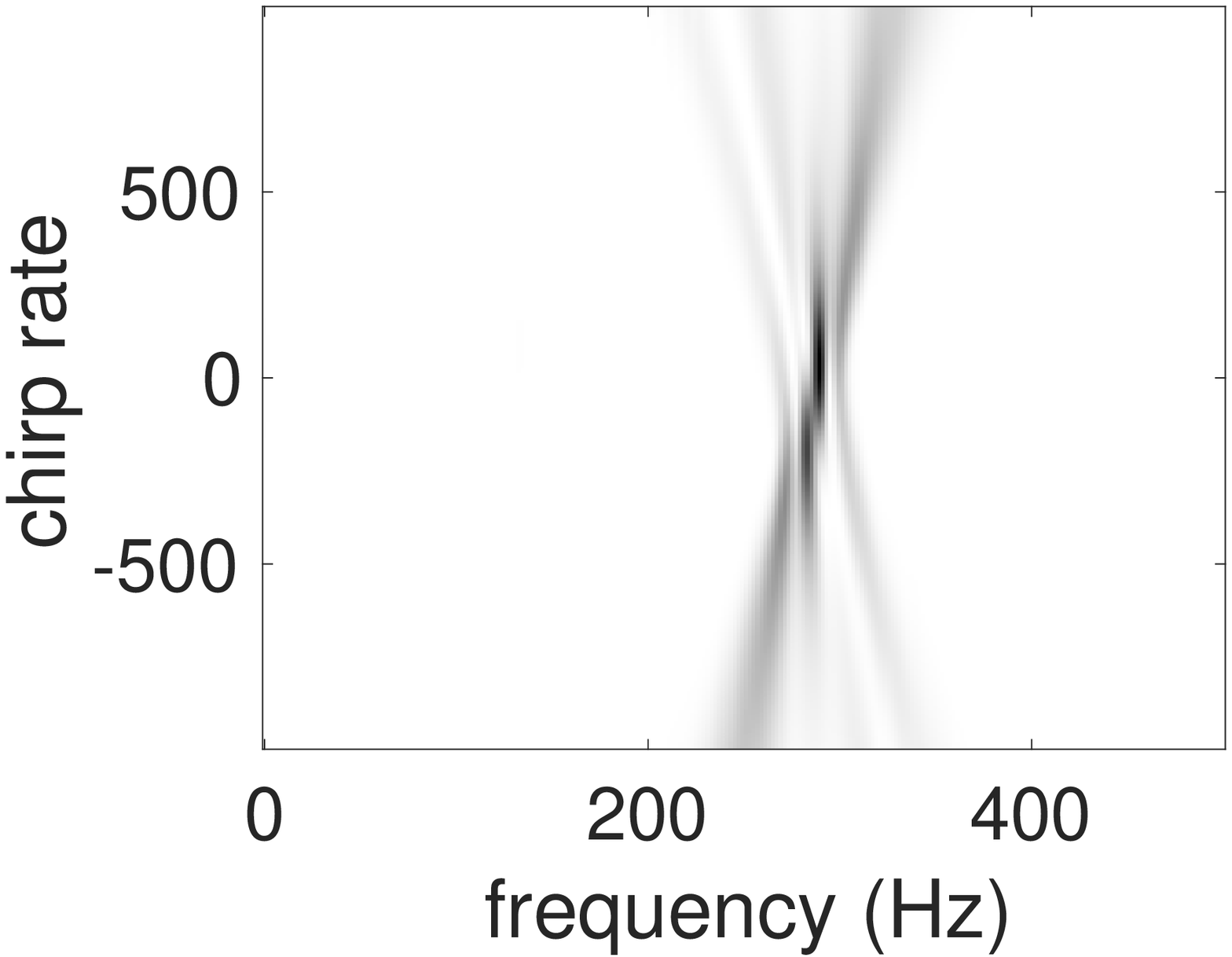}\\
	\includegraphics[width=.325\textwidth]{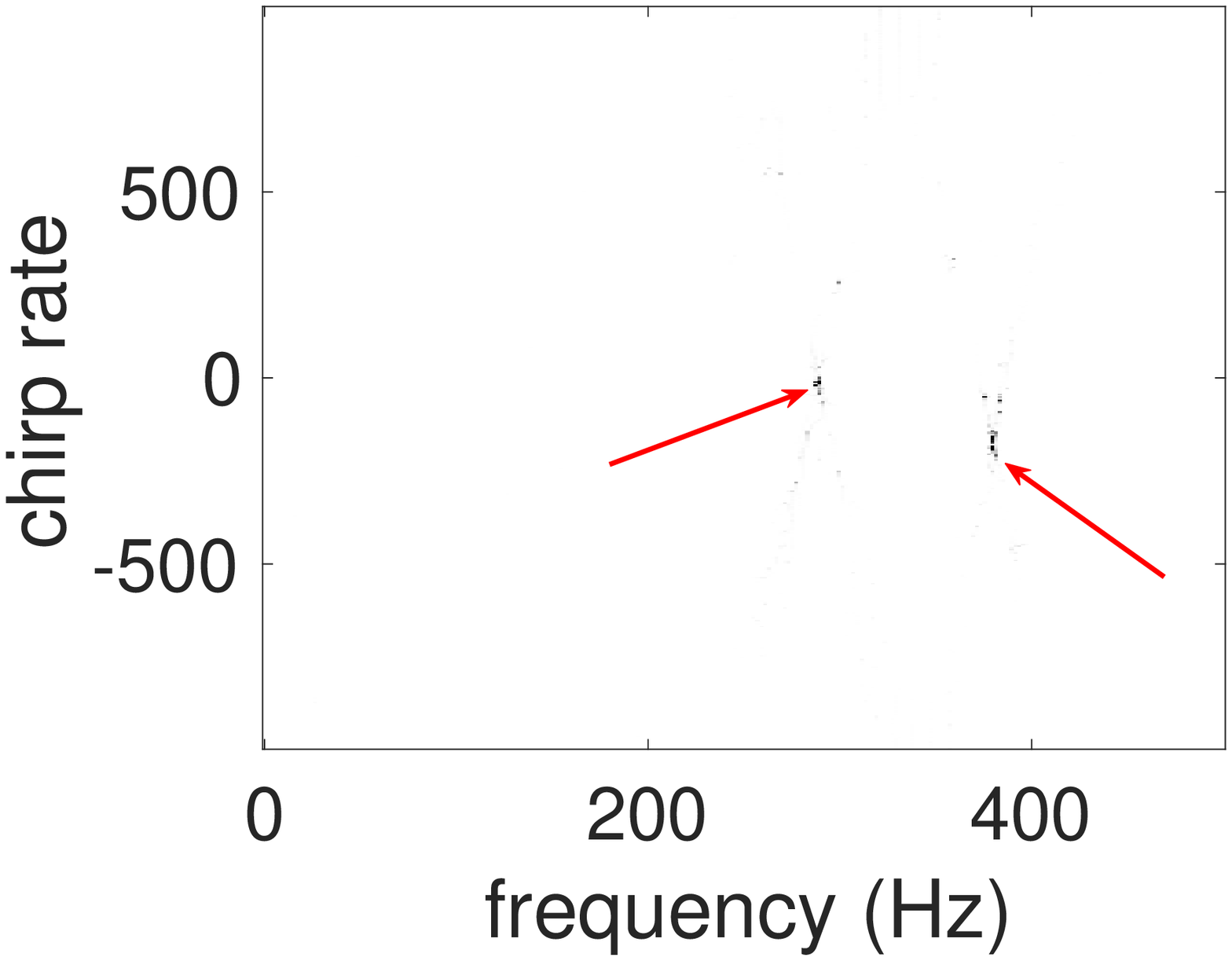}
	\includegraphics[width=.325\textwidth]{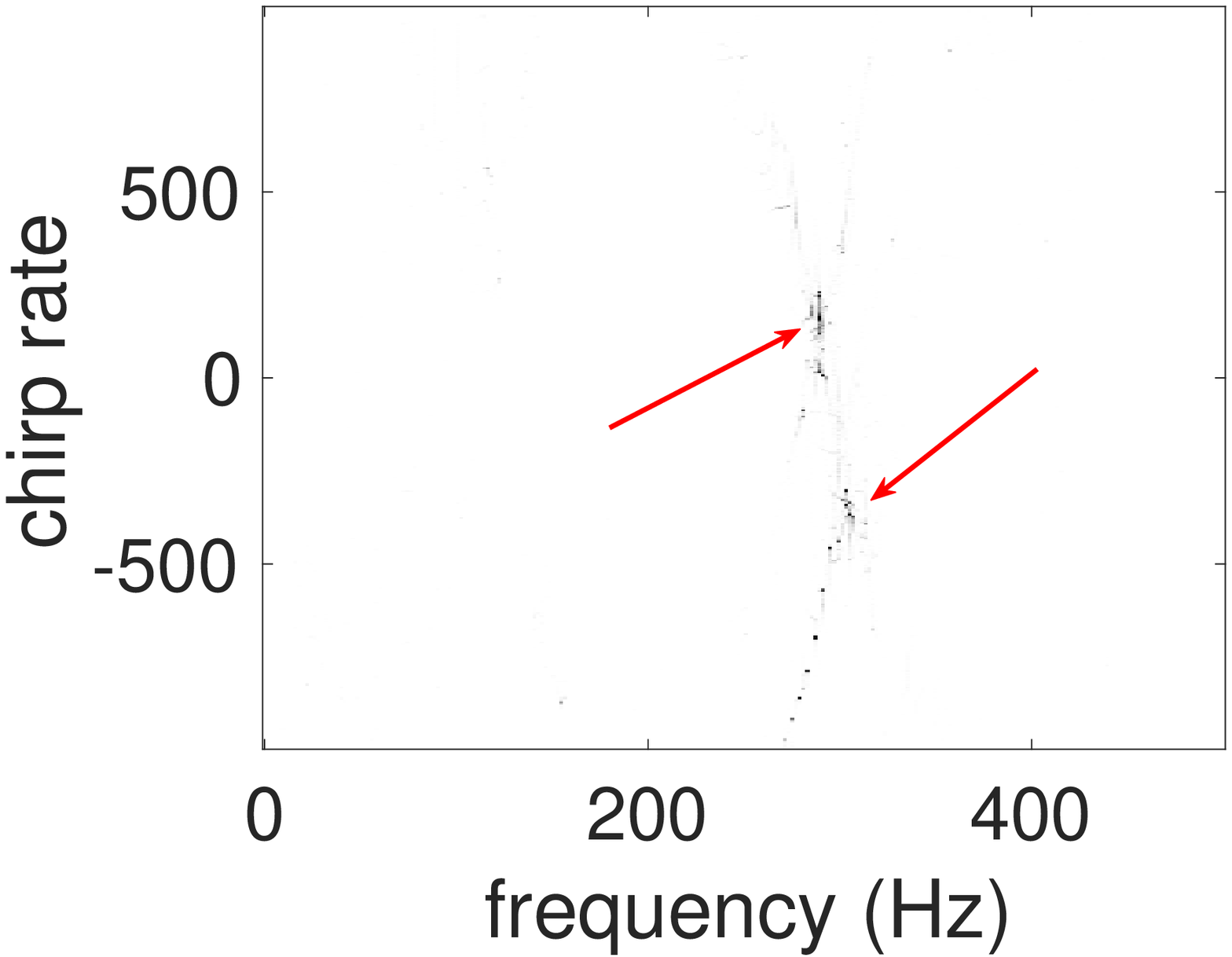}
	\includegraphics[width=.325\textwidth]{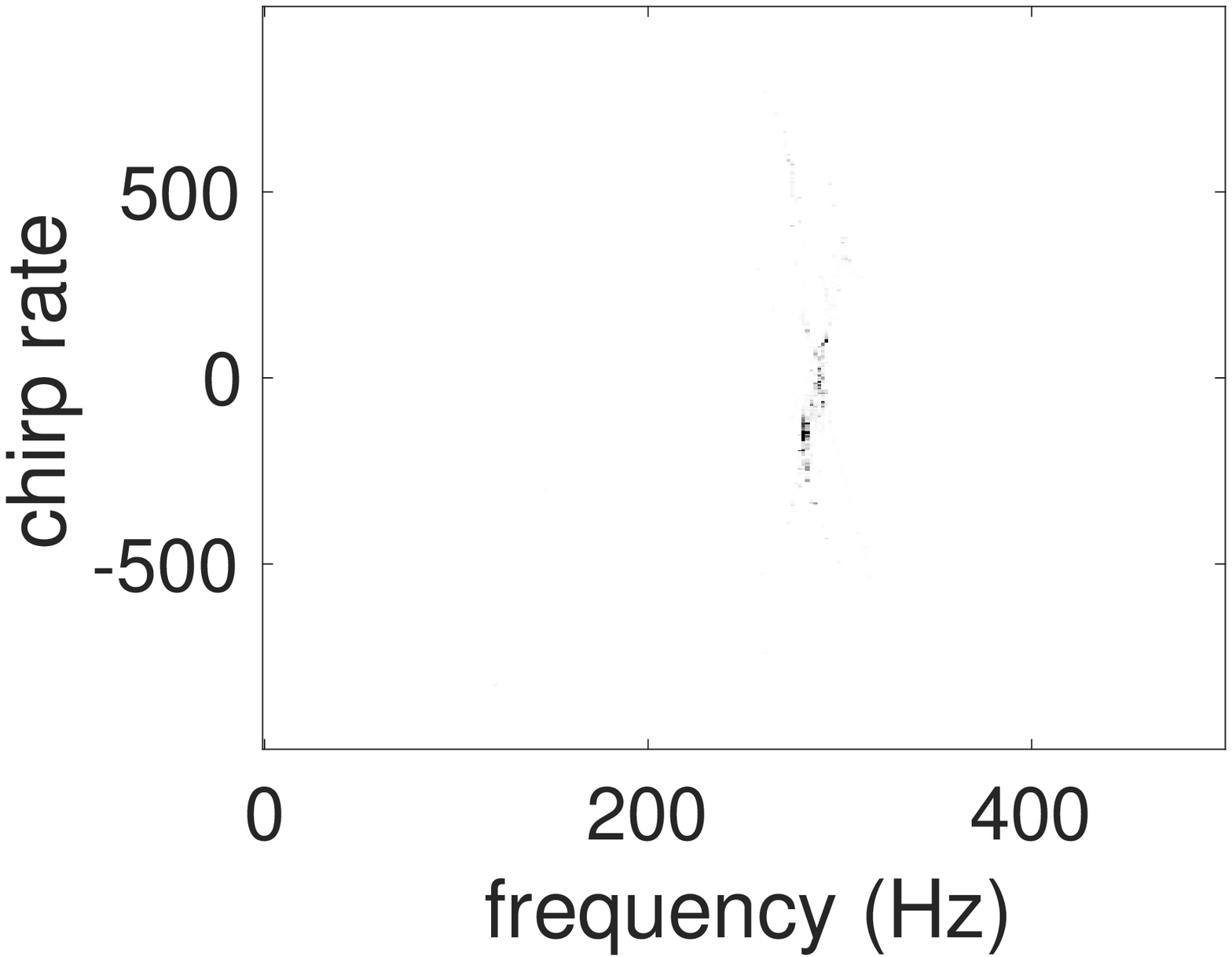}
	\caption{Top row, from left to right: $\abs{T_{f}^{(g_0)}(16.16,\xi,\lambda)}^2$, $\abs{T_{f}^{(g_0)}(16.43,\xi,\lambda)}^2$, and $\abs{T_{f}^{(g_0)}(16.57,\xi,\lambda)}^2$. Second row, from left to right:  $\abs{S_{f}^{(g_0)}(16.16,\xi,\lambda)}^2$, $\abs{S_{f}^{(g_0)}(16.43,\xi,\lambda)}^2$ and $\abs{S_{f}^{(g_0)}(16.57,\xi,\lambda)}^2$. Third row, from left to right: $\abs{T_{f}^{(g_2)}(16.16,\xi,\lambda)}^2$, $\abs{T_{f}^{(g_2)}(16.43,\xi,\lambda)}^2$, and $\abs{T_{f}^{(g_2)}(16.57,\xi,\lambda)}^2$. Bottom row, from left to right:  $\abs{S_{f}^{(g_2)}(16.16,\xi,\lambda)}^2$, $\abs{S_{f}^{(g_2)}(16.43,\xi,\lambda)}^2$ and $\abs{S_{f}^{(g_2)}(16.57,\xi,\lambda)}^2$. Red arrows are imposed to enhance the visualization.}
	\label{fig:13}
\end{figure}

\section{Discussion and future work}\label{secttion:discussion future}

While the proposed SCT algorithm could help disentangle components with crossover IFs, there are still several open problems. First, we need a fast implementation of the SCT algorithm for large databases, probably by including the existing fast implementation for CT \cite{lu2008fast}. Unlike SST and its variations, at each time, the computational time depends on the frequency axis only, in SCT, at each time, the computational time depends on both the frequency axis and chirp rate axis. An intuitive solution is taking a reliable TF representation to determine those frequency entries that are ``informative'', and focus on those frequency entries to evaluate information on the chirp rate. However, this solution might be challenged by the nonlinear relationship between frequency and chirp rate shown in Figure \ref{fig:5}. We will systematically study how to design a fast implementation of SCT (or even CT) in our future work.

The oscillatory components, or modes, considered in this paper are all sinusoidal; that is, it is the cosine function as an $1$-periodic signal that is stretched (by the phase function) and/or scaled (by the AM). In practice, we found many signals oscillate with a non-sinusoidal pattern (or wave-shape function) \cite{Wu:2013}. It would be an interesting topic to extend SCT to the signals with non-sinusoidal wave-shape functions. For example, while the cepstrum is used to handle the impact of non-sinusoidal wave-shape function on the spectrogram \cite{lin2018wave}, what is the proper geometric, or algebraic structure, on the TFC plane that we can count on to handle the non-sinusoidal wave-shape functions?

The $\epsilon$-ICT function is nothing but a generalization of the linear chirp function, where locally it behaves like a linear chirp function; that is, locally, the phase function of an $\epsilon$-ICT function is close to a quadratic function. A natural question to ask is what if the phase function is a polynomial? Such a problem has been considered in \cite{barkat1999instantaneous,tu2017instantaneous}, among many others. Since our theory for SCT is based on the $\epsilon$-ICT assumption, we cannot guarantee what may happen when we apply SCT to a signal with polynomial phase. However, based on our experience in applying SST to estimate IF when the chirp rate is non-trivial (the estimate is biased, and the bias depends on the chirp rate, which leads to the second \cite{Oberlin_Meignen_Perrier:2015} and higher  \cite{pham2017high} order SST and other variations), we would expect biased IF and chirp rate estimates under some relationships between IF and chirp rates. 

Last but not the least, while we empirically see that SCT is robust to noise (e.g. Figure \ref{fig:8}), we do not have a theoretical justification at this moment. Note that while we use differentiation to motivate the squeezing step of SCT, that squeezing step can be evaluated by different windows without any differentiation. We could thus apply the same argument as that in \cite{Chen_Cheng_Wu:2014}. Moreover, we may generalize the technique used in \cite{sourisseau2019inference} to obtain the impact of noise on the TFC representation, or the distribution of the TF representation under the non-null setup. 
We will leave the above interesting problems in our future work.

\section*{Acknowledgement}
We would like to thank the anonymous reviewers for their valuable and constructive comments.

\bibliographystyle{plain}
\bibliography{mybibfile.bib}

\clearpage
\appendix

\section{Proofs of theoretical results}\label{section:proofs}
	\subsection{Proof of Lemma \ref{Proposition slow decay of CT}}\label{proof:lemma1}
	
	Since $t$ and $\alpha>0$ are fixed, we can find $\xi$ close to $\xi_0 + \lambda_0 t$ so that 
	\[
	4\pi\alpha(\xi-\xi_0-\lambda_0 t)^2\leq \alpha^2 + (\lambda-\lambda_0)^2
	\]
	for any $\lambda$. Particularly, when $\lambda = \lambda_0$, $\alpha\geq 4\pi(\xi-\xi_0-\lambda_0 t)^2$.
	Denote $C =C(\xi)=(\xi - \xi_0 - \lambda_0 t)^2$. Clearly, $C\leq \frac{\alpha}{4\pi}$.
	For the fixed $t$ and the chosen $\xi$, we investigate the monotonicity of $M(\lambda):=\abs{T_f^{(g)}(t,\xi,\lambda)}$, which is equivalent to checking the monotonicity of $\ln(M(\lambda))$:
	\[
	\ln(M(\lambda)) = -\frac{\pi\alpha(\xi-\xi_0-\lambda_0 t)^2}{\alpha^2 + (\lambda - \lambda_0)^2}-\frac{1}{4}\log(\alpha^2+(\lambda-\lambda_0)^2),
	\]
	\begin{align*}
		\frac{\diff{\ln(M(\lambda))}}{\diff{\lambda}} &= \frac{2\pi\alpha(\xi-\xi_0-\lambda_0 t)^2(\lambda-\lambda_0)}{(\alpha^2 + (\lambda - \lambda_0)^2)^2} - \frac{\lambda - \lambda_0}{2(\alpha^2 + (\lambda-\lambda_0)^2)}\\
		&= \frac{\lambda - \lambda_0}{2(\alpha^2 + (\lambda-\lambda_0)^2)}\left(\frac{4\pi\alpha(\xi-\xi_0-\lambda_0 t)^2}{\alpha^2 + (\lambda-\lambda_0)^2}-1\right)\,,
	\end{align*}
	which indicates that $\frac{\text{d}\ln(M(\lambda))}{\text{d}\lambda}\leq 0$ when $\lambda\geq\lambda_0$, and $\frac{\text{d}\ln(M(\lambda))}{\text{d}\lambda}\geq 0$ when $\lambda\leq\lambda_0$. Since $\abs{T_f^{(g)}(t,\xi,\lambda_0)} = \frac{1}{
		\sqrt{\alpha}}e^{-\frac{\pi C}{\alpha}}$, when $\xi$ is sufficiently close to $\xi_0+\lambda_0t$, we have $\displaystyle\max_{\lambda} \abs{T_f^{(g)}(t,\xi,\lambda)} = \frac{1}{
		\sqrt{\alpha}}e^{-\frac{\pi C}{\alpha}}$.
	
	Next, we solve for $\lambda$ to see when the magnitude will decrease to $\frac{1}{\rho}$ of its maximum, where $\rho > 1$. The desired equality
	\[
	\frac{1}{\sqrt[4]{\alpha^2 + (\lambda - \lambda_0)^2}} e^{\frac{-\pi \alpha C}{\alpha^2 + (\lambda-\lambda_0)^2}}
	= \frac{1}{\rho}\max_{\lambda} \abs{T_f^{(g)}(t,\xi,\lambda)}=\frac{1}{\rho
		\sqrt{\alpha}}e^{-\frac{\pi C}{\alpha}}
	\]
	gives 
	\[
	\frac{-4\pi\alpha C}{\alpha^2+(\lambda-\lambda_0)^2} = 
	\ln{(\alpha^2+(\lambda-\lambda_0)^2)} - \ln{(\rho^4\alpha^2)}-\frac{4\pi C}{\alpha},
	\]
	which is equivalent to
	\[
	\ln{(\alpha^2 + (\lambda - \lambda_0)^2)} + \frac{4\pi\alpha C}{\alpha^2 + (\lambda- \lambda_0)^2} = \ln{(\rho^4\alpha^2) + \frac{4\pi C}{\alpha}}.
	\]
	Let $\lambda - \lambda_0 = k\alpha$, where $k\in \mathbb{R}$ will be determined. Then the above equality becomes
	\begin{equation}\label{eq CT decay in k}
		\ln{(k^2 + 1)} + \frac{4\pi C}{(k^2+1)\alpha}= \ln{(\rho^4)} + \frac{4\pi C}{\alpha}.
	\end{equation}
	Clearly, when $\xi=\xi_0+\lambda_0t$, $C=0$ and $k=\sqrt{\rho^4-1}$. In general, we need to take $C$ into account.
	Note that the derivative of the left hand side of \eqref{eq CT decay in k} in terms of $k$ satisfies 
	\[
	\partial_k\left(\ln{(k^2 + 1)} + \frac{4\pi C}{(k^2+1)\alpha}\right) = 
	\frac{2k}{k^2 + 1} + \frac{4\pi C}{\alpha}\frac{-2k}{(k^2 + 1)^2} = \left(1-\frac{4\pi C}{\alpha}\frac{1}{k^2+1}\right)\frac{2k}{k^2 + 1}>0
	\]
	since $\alpha\geq 4\pi C$ when $\xi$ is close to $\xi_0+\lambda_0t$. So $\ln{(k^2 + 1)} + \frac{4\pi C}{(k^2+1)\alpha}$ is an increasing function of $\abs{k}$. To solve the problem, we consider the following two cases.
	\begin{itemize}
		\item If $k^2 + 1 = \rho^4$, i.e. $\abs{k} = \sqrt{\rho^4-1}$, then 
		\[
		\ln{(k^2 + 1)} + \frac{4\pi C}{(k^2+1)\alpha} = \ln{(\rho^4)} + \frac{4\pi C}{\rho^4\alpha}\leq \ln{(\rho^4)} + \frac{4\pi C}{\alpha},
		\] 
		since $\rho>1$, and the equality holds when $C=0$.  
		\item If $k^2  = e\rho^4$, where $e=\exp(1)$, i.e. $\abs{k} = \sqrt{e}\rho^2$, then 
		\[
		\ln{(k^2 + 1)} + \frac{4\pi C}{(k^2+1)\alpha} - \ln{(\rho^4)} - \frac{4\pi C}{\alpha} = \ln{\left(e+\frac{1}{\rho^4}\right)} - \frac{4\pi C}{\alpha}\left(1-\frac{1}{e\rho^4+1}\right)\geq0,
		\] 
		since $\frac{4\pi C}{\alpha}(1-\frac{1}{e\rho^4+1})<\frac{4\pi C}{\alpha}\leq 1$.
	\end{itemize}
	Therefore, the $\lambda$ for $M(\lambda)$ decreasing to $\frac{1}{\rho}$ of its maximum lies between
	\[
	\sqrt{\rho^4 -1}\alpha\leq\abs{\lambda-\lambda_0}<\sqrt{e}\rho^2\alpha.
	\]
	Note that
	both the upper and the lower bounds scale linearly with $\alpha$, hence if the window is more localized in time ($\alpha$ is larger), it is less localized in the chirp rate domain. Moreover, by the monotonicity of $M(\lambda)$ we showed above, $M(\lambda) \leq \frac{1}{\rho}\max_{\lambda} \abs{T_f^{(g)}(t,\xi,\lambda)}$ if $\abs{\lambda-\lambda_0}=\sqrt{e}\rho^2\alpha$. Solving for $\rho$ produces $M(\lambda)\leq \frac{e^{\frac{1}{4}}\sqrt{\alpha}}{\sqrt{\abs{\lambda-\lambda_0}}}\max_{\lambda} \abs{T_f^{(g)}(t,\xi,\lambda)}$. On the other hand, we also have $M(\lambda) \geq \frac{1}{\rho}\max_{\lambda} \abs{T_f^{(g)}(t,\xi,\lambda)}$ if $\abs{\lambda-\lambda_0}=\sqrt{\rho^4 -1}\alpha$, which implies $M(\lambda)\geq \frac{\sqrt{\alpha}}{(\abs{\lambda-\lambda_0}^2+\alpha^2)^{\frac{1}{4}}}\max_{\lambda} \abs{T_f^{(g)}(t,\xi,\lambda)}$ as is claimed.

	\subsection{Proof of Theorem \ref{maintheorem}}\label{section:supp}
	In this supplementary section, we list the lemmas that make up the proof of Theorem \ref{maintheorem}. We first recall and define some notations that will be used in the proofs:
	\begin{itemize}
		\item The absolute constants $G_n$ are defined as in the condition of window functions in Theorem \ref{maintheorem}, i.e., $\abs{\widecheck{(x^ng)}(\xi,\lambda)}\leq \frac{G_n}{\sqrt{\abs{\xi}+\abs{\lambda}}}$ for some $G_n>0$, $n=0,1,2$.
		
		\item For a given window function $g(x)\in\mathcal{S}(\mathbb{R})$, we define \[I_n:=\int_{\mathbb{R}}\abs{x}^n\abs{g(x)}\diff{x}.\]
	\end{itemize}

\begin{lemma}\label{lemma:1}
	For any tuple $(t,\xi,\lambda)$ under consideration, there can be at most one $k\in\{ 1,\dots, K\}$ such that $\abs{\xi-\phi_k'(t)}+\abs{\lambda-\phi_k''(t)}< \Delta$, i.e. $(t,\xi,\lambda)\in Z_k$.
\end{lemma}

\begin{proof}
	Assume there exists $(t,\xi,\lambda)\in Z_k \cap Z_l$ for $k\neq l$, then by the definition of $Z_k$ we have
	\begin{align*}
		\abs{\phi_k'(t)-\phi_l'(t)}+\abs{\phi_k''(t)-\phi_k''(t)}
		&\leq \abs{\phi_k'(t)-\xi} + \abs{\xi-\phi_l'(t)} + \abs{\phi_k''(t)-\lambda} + \abs{\lambda-\phi_k''(t)}\\
		&< \Delta+\Delta = 2\Delta,
	\end{align*}
	which contradicts the separation condition of $\mathcal{A}_{\epsilon,\Delta}$.
\end{proof}

\begin{lemma}\label{lemma:2}
	If $(t,\xi,\lambda)\not\in Z_k$, then 
	$$\abs{T_{f_k}^{(x^n g)}(t,\xi,\lambda)}\leq \epsilon E_{k,n}(t),$$
	where $E_{k,n}(t):=\norm{\phi_k'}_{L^{\infty}}I_{n+1}+\left(G_n+\frac{\pi}{3}\norm{\phi_k'}_{L^{\infty}}I_{n+3}\right)A_k(t)$. Consequently, for any $(t,\xi,\lambda)\in Z_k$, we have
	$$\abs{T_{f}^{(x^n g)}(t,\xi,\lambda)-T_{f_k}^{(x^n g)}(t,\xi,\lambda)}\leq \epsilon \sum_{l\neq k}E_{l,n}(t).$$
\end{lemma}

\begin{proof}
	Using the second-order Taylor expansion of $\phi_k$, we can write
	\begin{align*}
		f_k(x) &= A_k(x)e^{2\pi i \phi_k (x)}\\
		&= (A_k(x)-A_k(t))e^{2\pi i \phi_k (x)} + A_k(t)e^{2\pi i[ \phi_k (t)+\phi_k'(t)(x-t)+\frac{1}{2}\phi_k''(t)(x-t)^2]}\\
		&\quad+A_k(t)(e^{2\pi i \phi_k(x)}-e^{2\pi i [\phi_k (t)+\phi_k'(t)(x-t)+\frac{1}{2}\phi_k''(t)(x-t)^2]}).
	\end{align*}
	Denote these three terms in the sum by $f_{k,1},f_{k,2},f_{k,3}$ respectively. Then we have
	\begin{align*}
		\abs{T_{f_{k,1}}^{(x^ng)}(t,\xi,\lambda)}&\leq \int \abs{A_k(x)-A_k(t)}\abs{x-t}^n\abs{g(x-t)}\diff{x}\\
		&\leq \epsilon\norm{\phi_k'}_{L^{\infty}}\int \abs{x-t}^{n+1}\abs{g(x-t)}\diff{x}\\
		&\leq \epsilon\norm{\phi_k'}_{L^{\infty}}I_{n+1},
	\end{align*}
	and if $(t,\xi,\lambda)\not \in Z_k$, then
		\begin{align*}
			\abs{T_{f_{k,2}}^{(x^ng)}(t,\xi,\lambda)}&= A_k(t)\abs{\widecheck{(x^ng)} (\xi-\phi_k'(t),\lambda-\phi_k''(t))} \leq A_k(t)\frac{G_n}{\sqrt{\Delta}}= A_k(t)G_n\epsilon.
		\end{align*}
	For $f_{k,3}$, we have
	\begin{align*}
		\abs{T_{f_{k,3}}^{(x^ng)}(t,\xi,\lambda)}&\leq A_k(t)\int\frac{\pi}{3}\norm{\phi_k'''}_{L^{\infty}}\abs{x-t}^{n+3}\abs{g(x-t)}\diff{x}\\
		&\leq \epsilon\frac{\pi}{3}A_k(t)\norm{\phi_k'}_{L^{\infty}}I_{n+3}.
	\end{align*}
	Therefore, for $(t,\xi,\lambda)\not \in Z_k$,
	\[
	\abs{T_{f_k}^{(x^ng)}(t,\xi,\lambda)}\leq\epsilon\left[\norm{\phi_k'}_{L^{\infty}}I_{n+1}+(G_n+\frac{\pi}{3}\norm{\phi_k'}_{L^{\infty}}I_{n+3})A_k(t)\right].
	\]
	Consequently,
	\[
	\abs{T_{f}^{(x^ng)}(t,\xi,\lambda)-T_{f_k}^{(x^ng)}(t,\xi,\lambda)}\leq\sum_{l\neq k}E_{l,n}(t).
	\]
\end{proof}

Since for each $k$, $E_{k,0}(t)$ is uniformly bounded from above and from below for all $(t,\xi,\lambda)\not\in Z_k$, it follows that if $\epsilon$ is sufficiently small, such that for all $(t,\xi,\lambda)\not\in \bigcup_{k=1}^K Z_k$,
\[
\epsilon \leq \left(\sum_{k=1}^K E_{k,0}(t)\right)^{-\frac{6}{5}},
\]
then for any $(t,\xi,\lambda)\not\in \bigcup_{k=1}^K Z_k$, $\abs{T_{f}^{(g)}(t,\xi,\lambda)}\leq\tilde{\epsilon}$.

\begin{lemma}\label{lemma:3}
	If $(t,\xi,\lambda)\in Z_k$ for some $k$, $1\leq k\leq K$, then
	\begin{equation}
		\abs{\partial_t T_{f}^{(g)}(t,\xi,\lambda) - 2\pi i (\phi_k'(t)T_{f}^{(g)}(t,\xi,\lambda) +\phi_k''(t)T_{f}^{(xg)}(t,\xi,\lambda))}\leq\epsilon B_{k,1}(t),
	\end{equation}
	where
	\begin{align*}
		B_{k,1}(t) &:= \sum_{k=1}^K\norm{\phi_k'}_{L^{\infty}}\left(I_0+\pi\norm{A_k}_{L^{\infty}}I_2\right)+2\pi\sum_{l\neq k }\left(\phi_l'(t)E_{l,0}(t)+\abs{\phi_l''(t)}E_{l,1}(t)\right)\\
		&\quad + 2\pi\left(\phi_k'(t)\sum_{l\neq k}E_{l,0}(t) + \abs{\phi_k''(t)}\sum_{l\neq k}E_{l,1}(t)\right)\,.
	\end{align*}
\end{lemma}
\begin{proof}
	If $(t,\xi,\lambda)\not\in Z_k$, then 
	\begin{align*}
		\partial_t T_{f}^{(g)}(t,\xi,\lambda)
		=&\, \partial_t\left(\sum_{k=1}^K \int A_k(x)e^{2\pi i \phi_k(x)}g(x-t)e^{-2\pi i \xi(x-t)-\pi i \lambda (x-t)^2}\diff{x}\right)\\
		=&\, \sum_{k=1}^K \bigg[\int A_k'(x)e^{2\pi i \phi_k(x)}g(x-t)e^{-2\pi i \xi(x-t)-\pi i \lambda (x-t)^2}\diff{x}\\
		&\quad + \int A_k(x)2\pi i \phi_k'(x)e^{2\pi i \phi_k(x)}g(x-t)e^{-2\pi i \xi(x-t)-\pi i \lambda (x-t)^2}\diff{x}\bigg]\\
		=&\, \sum_{k=1}^K \bigg[\int A_k'(x)e^{2\pi i \phi_k(x)}g(x-t)e^{-2\pi i \xi(x-t)-\pi i \lambda (x-t)^2}\diff{x}\\
		&\quad + \int A_k(x)2\pi i \bigg(\phi_k'(t)+\phi_k''(t)(x-t)+\int_{0}^{x-t}\big(\phi_k''(t+u)-\phi_k''(t)\big)\diff{u} \bigg)\\
		&\,\quad\quad\quad\times e^{2\pi i \phi_k(x)}g(x-t)e^{-2\pi i \xi(x-t)-\pi i \lambda(x-t)^2}\diff{x}\bigg]\,.\\
	\end{align*}
	It follows that
	\begin{align*}
		&\quad\abs{\partial_t T_{f}^{(g)}(t,\xi,\lambda) - 2\pi i\sum_{k=1}^K (\phi_k'(t)T_{f_k}^{(g)}(t,\xi,\lambda) +\phi_k''(t)T_{f_k}^{(xg)}(t,\xi,\lambda))}\\
		&\leq \sum_{k=1}^K\bigg[ \int \abs{A_k'(x)}\abs{g(x-t)}\diff{x}
		+ 2\pi\int A_k(x)\abs{\int_{0}^{x-t}\big[\phi_k''(t+u)-\phi_k''(t)\big]\diff{u}}\abs{g(x-t)}\diff{x}\bigg]\\
		&\leq \sum_{k=1}^K(\epsilon\norm{\phi_k'}_{L^{\infty}}I_0+\pi\norm{\phi_k'''}_{L^{\infty}}\norm{A_k}_{L^{\infty}}I_2)\\
		&\leq \epsilon\sum_{k=1}^K\norm{\phi_k'}_{L^{\infty}}(I_0+\pi\norm{A_k}_{L^{\infty}}I_2).
	\end{align*}
	By Lemma \ref{lemma:2}, when $(t,\xi,\lambda)\in Z_k$, we have
	\begin{align*}
		&\abs{\partial_t T_{f}^{(g)}(t,\xi,\lambda) - 2\pi i \left(\phi_k'(t)T_{f_k}^{(g)}(t,\xi,\lambda) +\phi_k''(t)T_{f_k}^{(xg)}(t,\xi,\lambda)\right)}\\
		\leq &\,\epsilon\left(\sum_{k=1}^K\norm{\phi_k'}_{L^{\infty}}(I_0+\pi\norm{A_k}_{L^{\infty}}I_2)+2\pi\sum_{l\neq k }(\phi_l'(t)E_{l,0}(t)+\abs{\phi_l''(t)}E_{l,1}(t))\right),
	\end{align*}
	so finally we have
	\begin{align*}
		&\abs{\partial_t T_{f}^{(g)}(t,\xi,\lambda) - 2\pi i (\phi_k'(t)T_{f}^{(g)}(t,\xi,\lambda) +\phi_k''(t)T_{f}^{(xg)}(t,\xi,\lambda))}\\
		\leq&\, \epsilon\left[\sum_{k=1}^K\norm{\phi_k'}_{L^{\infty}}(I_0+\pi\norm{A_k}_{L^{\infty}}I_2)+2\pi\sum_{l\neq k }(\phi_l'(t)E_{l,0}(t)+\abs{\phi_l''(t)}E_{l,1}(t))\right.\\
		&\quad \left.+ 2\pi\left(\phi_k'(t)\sum_{l\neq k}E_{l,0}(t) + \abs{\phi_k''(t)}\sum_{l\neq k}E_{l,1}(t)\right)\right].
	\end{align*}
\end{proof}

To simply the notation, we omit $(t,\xi,\lambda)$ in $T_{f}^{(g)}(t,\xi,\lambda)$ in following lemmas and proofs.
\begin{lemma}\label{lemma:4}
	If $(t,\xi,\lambda)\in Z_k$, then at $(t,\xi,\lambda)$
	\[
	\abs{\partial_{tt}^2 T_{f}^{(g)} - 2\pi i \left(\phi_k''(t)T_f^{(g)}+\phi_k'(t)\partial_t T_f^{(g)}+\phi_k''(t)\partial_t T_f^{(xg)}\right)}\leq\epsilon B_{k,2}(t),
	\]
	where
	\begin{align*}
		B_{k,2}(t)&:= \Omega_2+ 2\pi\sum_{l\neq k}\left(\abs{\phi_l''(t)}E_{l,0}(t)+\phi_l'(t)F_{l,0}(t)+\phi_l''(t)F_{l,1}(t)\right)\\
		&\qquad+2\pi\left(\abs{\phi_k''(t)}\sum_{l\neq k}E_{l,0}(t)+\phi_k'(t)\sum_{l\neq k}F_{l,0}(t)+\phi_k''(t)\sum_{l\neq k}F_{l,1}(t)\right)
	\end{align*}
	with 
	\[\Omega_2 := \sum_{k=1}^K \left[\norm{\phi_k'}_{L^{\infty}}I_0+2\pi\norm{A_k}_{L^{\infty}}\norm{\phi_k'}_{L^{\infty}}I_1 + 2\pi\norm{\phi_k'}_{L^{\infty}}^2 I_0 + \pi(\norm{A_k'}_{L^{\infty}}+2\pi\norm{\phi_k'}_{L^{\infty}}\norm{A_k}_{L^{\infty}})I_2\right],\]
	and
	\[
	F_{l,n}(t):=\norm{\phi_l'}_{L^{\infty}}(I_n+\pi \norm{A_l}_{L^{\infty}}I_{n+2}) + 2\pi(\phi_l'(t)E_{l,n}(t)+\abs{\phi_l''(t)}E_{l,n+1}(t)).
	\]
\end{lemma}

\begin{proof}
	\begin{align*}
		\partial_{tt}^2 T_{f}^{(g)}(t,\xi,\lambda)
		&= \sum_{k=1}^K \int(A_k''(x)+4\pi i \phi_k'(x)A_k'(x)+2\pi i \phi_k''(x)A_k(x)-4\pi^2\phi_k'(x)^2 A_k(x))\\
		&\qquad \qquad e^{2\pi i \phi_k(x)}g(x-t)e^{-2\pi i \xi(x-t)-\pi i \lambda(x-t)^2}\diff{x}\\
		&= \sum_{k=1}^K \big(\int A_k''(x)e^{2\pi i \phi_k(x)}g(x-t)e^{-2\pi i \xi(x-t)-\pi i \lambda (x-t)^2}\diff{x}\\
		&\qquad\quad + 2\pi i \int \phi_k''(x)A_k(x) e^{2\pi i \phi_k(x)}g(x-t)e^{-2\pi i \xi(x-t)-\pi i \lambda (x-t)^2}\diff{x}\\
		&\qquad\quad + 2\pi i \int \phi_k'(x)A_k'(x) e^{2\pi i \phi_k(x)}g(x-t)e^{-2\pi i \xi(x-t)-\pi i \lambda (x-t)^2}\diff{x}\\
		&\qquad\quad + 2\pi i \int \phi_k'(x)f_k'(x) e^{2\pi i \phi_k(x)}g(x-t)e^{-2\pi i \xi(x-t)-\pi i \lambda (x-t)^2}\diff{x}\big).
	\end{align*}
	The first term gives 
	\begin{align*}
		\abs{\sum_{k=1}^K \int A_k''(x)e^{2\pi i \phi_k(x)}g(x-t)e^{-2\pi i \xi(x-t)-\pi i \lambda (x-t)^2}\diff{x}}
		&\leq \epsilon\sum_{k=1}^K \norm{\phi_k'}_{L^{\infty}}I_0\,,
	\end{align*}
	the second term gives
	\begin{align*}
		&\abs{2\pi i \sum_{k=1}^K \int \phi_k''(x)A_k(x) e^{2\pi i \phi_k(x)}g(x-t)e^{-2\pi i \xi(x-t)-\pi i \lambda (x-t)^2}\diff{x} - 2\pi i \sum_{k=1}^K \phi_k''(t)T_{f_k}^{(g)}}\\
		\leq&\, 2\pi\sum_{k=1}^K \int \epsilon\norm{\phi_k'}_{L^{\infty}}\abs{x-t}A_k(x)\abs{g(x-t)}\diff{x}\\
		\leq&\, 2\pi\epsilon I_1 \sum_{k=1}^K \norm{A_k}_{L^{\infty}}\norm{\phi_k'}_{L^{\infty}}\,,
	\end{align*}
	the third term gives
	\begin{align*}
		\abs{2\pi i \sum_{k=1}^K \int \phi_k'(x)A_k'(x) e^{2\pi i \phi_k(x)}g(x-t)e^{-2\pi i \xi(x-t)-\pi i \lambda (x-t)^2}\diff{x}}\leq 2\pi\epsilon\sum_{k=1}^K \norm{\phi_k'}_{L^{\infty}}^2 I_0\,,
	\end{align*}
	and for the last term, note that $\partial_t T_{f_k}^{(g)} = T_{f_k'}^{(g)}$ and $\partial_t T_{f_k}^{(xg)} = T_{f_k'}^{(xg)}$, by the second order Taylor expansion of $\phi_k'(x)$, we have
	\begin{align*}
		&\abs{2\pi i \sum_{k=1}^K \int \phi_k'(x)f_k'(x) e^{2\pi i \phi_k(x)}g(x-t)e^{-2\pi i \xi(x-t)-\pi i \lambda (x-t)^2}\diff{x}-2\pi i \sum_{k=1}^K(\phi_k'(t)\partial_t T_{f_k}^{(g)} + \phi_k''(t)\partial_t T_{f_k}^{(xg)})}\\
		\leq&\, 2\pi\abs{\sum_{k=1}^K \int \frac{1}{2}\norm{\phi_k'''}_{L^{\infty}}(x-t)^2(A_k'(x)+2\pi i \phi_k'(x)A_k(x))e^{2\pi i \phi_k(x)}g(x-t)e^{-2\pi i \xi(x-t)-\pi i \lambda(x-t)^2}\diff{x}}\\
		\leq&\, \pi\epsilon \sum_{k=1}^K\left(\norm{A_k'}_{L^{\infty}}+2\pi\norm{\phi_k'}_{L^{\infty}}\norm{A_k}_{L^{\infty}}\right)I_2.
	\end{align*}
	Combine those estimates above, and we have at $(t,\xi,\lambda)$
	\[
	\abs{\partial_{tt}^2 T_{f}^{(g)} - 2\pi i\sum_{k=1}^K \left(\phi_k''(t)T_{f_k}^{(g)}+\phi_k'(t)\partial_t T_{f_k}^{(g)}+\phi_k''(t)\partial_t T_{f_k}^{(xg)}\right)}\leq\epsilon \Omega_2,
	\]
	From the proof of Lemma \ref{lemma:3}, we know that
	\[
	\abs{\partial_t T_{f_l}^{(g)}-2\pi i\left(\phi_l'(t)T_{f_l}^{(g)}+\phi_l''(t)T_{f_l}^{(xg)}\right)}\leq \epsilon \norm{\phi_l'}_{L^{\infty}}(I_0+\pi \norm{A_l}_{L^{\infty}}I_2),
	\]
	\[
	\abs{\partial_t T_{f_l}^{(xg)}-2\pi i\left(\phi_l'(t)T_{f_l}^{(xg)}+\phi_l''(t)T_{f_l}^{(x^2g)}\right)}\leq \epsilon \norm{\phi_l'}_{L^{\infty}}(I_1+\pi \norm{A_l}_{L^{\infty}}I_3).
	\]
	So using Lemma \ref{lemma:2}, if $(t,\xi,\lambda)\in Z_k$, for $l\neq k$, we have
	\begin{align*}
		\abs{\partial_t T_{f_l}^{(g)}(t,\xi,\lambda)}&\leq \epsilon\left[ \norm{\phi_l'}_{L^{\infty}}(I_0+\pi \norm{A_l}_{L^{\infty}}I_2) + 2\pi(\phi_l'(t)E_{l,0}(t)+\abs{\phi_l''(t)}E_{l,1}(t))\right]\\
		&= \epsilon F_{l,0}(t),
	\end{align*}
	and
	\begin{align*}
		\abs{\partial_t T_{f_l}^{(xg)}(t,\xi,\lambda)}&\leq \epsilon\left[ \norm{\phi_l'}_{L^{\infty}}(I_1+\pi \norm{A_l}_{L^{\infty}}I_3) + 2\pi(\phi_l'(t)E_{l,1}(t)+\abs{\phi_l''(t)}E_{l,2}(t))\right]\\
		&= \epsilon F_{l,1}(t).
	\end{align*}
	Therefore, for $(t,\xi,\lambda)\in Z_k$, we have
	\[
	\abs{\partial_t T_{f}^{(g)}(t,\xi,\lambda)-\partial_t T_{f_k}^{(g)}(t,\xi,\lambda)}\leq \epsilon\sum_{l\neq k} F_{l,0}(t),
	\]
	and
	\[
	\abs{\partial_t T_{f}^{(xg)}(t,\xi,\lambda)-\partial_t T_{f_k}^{(xg)}(t,\xi,\lambda)}\leq \epsilon\sum_{l\neq k} F_{l,1}(t).
	\]
	We then get, at $(t,\xi,\lambda)\in Z_k$
	\begin{align*}
		&\quad\abs{\partial_{tt}^2 T_{f}^{(g)} - 2\pi i \left(\phi_k''(t)T_{f}^{(g)}+\phi_k'(t)\partial_t T_{f}^{(g)}+\phi_k''(t)\partial_t T_{f}^{(xg)}\right)}\\
		&\leq\abs{\partial_{tt}^2 T_{f}^{(g)} - 2\pi i\sum_{k=1}^K \left(\phi_k''(t)T_{f_k}^{(g)}+\phi_k'(t)\partial_t T_{f_k}^{(g)}+\phi_k''(t)\partial_t T_{f_k}^{(xg)}\right)}\\
		&\qquad + \abs{2\pi i\sum_{l\neq k} \left(\phi_l''(t)T_{f_l}^{(g)}+\phi_l'(t)\partial_t T_{f_l}^{(g)}+\phi_l''(t)\partial_t T_{f_l}^{(xg)}\right)}\\
		&\qquad +\abs{2\pi i\left[\phi_k''(t)\left(T_{f}^{(g)}-T_{f_k}^{(g)}\right)+\phi_k'(t)\left(\partial_t T_{f}^{(g)}-\partial_t T_{f_k}^{(g)}\right)+\phi_k''(t)\left(\partial_t T_{f}^{(xg)}-\partial_t T_{f_k}^{(xg)}\right)\right]}\\
		&\leq\epsilon\Bigg[ \Omega_2+ 2\pi\sum_{l\neq k}\left(\abs{\phi_l''(t)}E_{l,0}(t)+\phi_l'(t)F_{l,0}(t)+\phi_l''(t)F_{l,1}(t)\right)\\
		&\qquad+2\pi\left(\abs{\phi_k''(t)}\sum_{l\neq k}E_{l,0}(t)+\phi_k'(t)\sum_{l\neq k}F_{l,0}(t)+\phi_k''(t)\sum_{l\neq k}F_{l,1}(t)\right)\Bigg].
	\end{align*}
\end{proof}

\begin{lemma}\label{lemma:5}
	For any $(t,\xi,\lambda)\in Z_k$, such that $\abs{T_f^{(g)}(t,\xi,\lambda)}>\tilde{\epsilon}$ and $2\pi\abs{1+\partial_t\left(\frac{T_f^{(xg)}(t,\xi,\lambda)}{T_f^{(g)}(t,\xi,\lambda)}\right) }>\tilde{\epsilon}$, we have
	\[
	\abs{\mu_{f}^{(g)}(t,\xi,\lambda)-\phi_k''(t)}\leq\tilde{\epsilon}\,,
	\]
	provided $\epsilon$ is sufficiently small.
\end{lemma}

\begin{proof}
	First, for any $(t,\xi,\lambda)$ satisfies the conditions of the lemma, we have at $(t,\xi,\lambda)$,
	\begin{align*}
		\abs{2\pi i \left(T_{f}^{(g)}\right)^2 +2\pi i \left(T_{f}^{(g)}\partial_t T_{f}^{(xg)}- T_{f}^{(xg)}\partial_t T_{f}^{(g)}\right)}
		= 2\pi\abs{1+\partial_t\left(\frac{T_f^{(xg)}}{T_f^{(g)}}\right)}\abs{T_f^{(g)}}^2>\tilde{\epsilon}^3.
	\end{align*}
	
	By Lemma \ref{lemma:3} and \ref{lemma:4}, we have at $(t,\xi,\lambda)\in Z_k$
	\begin{align*}
		&\quad\abs{\mu_{f}^{(g)}-\phi_k''(t)}\\
		&\leq \abs{\phi_k''(t)- \frac{T_{f}^{(g)}\partial_{tt}^2T_{f}^{(g)}-\left(\partial_t T_f^{(g)}\right)^2}{2\pi i \left(T_{f}^{(g)}\right)^2 +2\pi i \left(T_{f}^{(g)}\partial_t T_{f}^{(xg)}- T_{f}^{(xg)}\partial_t T_{f}^{(g)}\right)}}\\
		&= \left|\frac{\partial_t T_f^{(g)}\left[\partial_t T_f^{(g)}-2\pi i \phi_k''(t)T_f^{(xg)}-2\pi i\phi_k'(t)T_f^{(g)}\right]}{2\pi i \left(T_{f}^{(g)}\right)^2 +2\pi i \left(T_{f}^{(g)}\partial_t T_{f}^{(xg)}- T_{f}^{(xg)}\partial_t T_{f}^{(g)}\right)}\right.\\
		&\left.\qquad - \frac{T_{f}^{(g)}\left[\partial_{tt}^2T_{f}^{(g)}-2\pi i \phi_k''(t)T_f^{(g)}-2\pi i \phi_k'(t)\partial_t T_f^{(g)}-2\pi i \phi_k''(t)\partial_t T_f^{(xg)}\right]}{2\pi i \left(T_{f}^{(g)}\right)^2 +2\pi i \left(T_{f}^{(g)}\partial_t T_{f}^{(xg)}- T_{f}^{(xg)}\partial_t T_{f}^{(g)}\right)}\right|\\
		&\leq \epsilon B_{k,1}(t)\abs{\frac{\partial_t T_{f}^{(g)}}{2\pi i \left(T_{f}^{(g)}\right)^2 +2\pi i \left(T_{f}^{(g)}\partial_t T_{f}^{(xg)}- T_{f}^{(xg)}\partial_t T_{f}^{(g)}\right)}}\\
		&\qquad +\epsilon B_{k,2}(t)\abs{\frac{T_{f}^{(g)}}{2\pi i \left(T_{f}^{(g)}\right)^2 +2\pi i \left(T_{f}^{(g)}\partial_t T_{f}^{(xg)}- T_{f}^{(xg)}\partial_t T_{f}^{(g)}\right)}}\\
		&\leq \tilde{\epsilon}^3\left(B_{k,1}(t)\abs{\partial_t T_f^{(g)}}+B_{k,2}(t)\abs{T_f^{(g)}}\right)\,.
	\end{align*}
	By the uniform boundedness of $B_{k,n}(t)$, $\abs{\partial_t T_f^{(g)}}$ and $\abs{T_f^{(g)}}$, if we impose an extra restriction on $\epsilon$, namely that, for any $k\in\{1,\dots,K\}$ and at any $(t,\xi,\lambda) \in Z_k$, \[
	\epsilon\leq\left(B_{k,1}(t)\abs{\partial_t T_f^{(g)}}+B_{k,2}(t)\abs{T_f^{(g)}}\right)^{-3},
	\]
	then $
	\abs{\mu_{f}^{(g)}-\phi_k''(t)}\leq\tilde{\epsilon}.
	$
\end{proof}

\begin{lemma}
	For any $(t,\xi,\lambda)\in Z_k$, such that $\abs{T_f^{(g)}(t,\xi,\lambda)}>\tilde{\epsilon}$ and $2\pi\abs{1+\partial_t\left(\frac{T_f^{(xg)}(t,\xi,\lambda)}{T_f^{(g)}(t,\xi,\lambda)}\right) }>\tilde{\epsilon}$, we have
	\[
	\abs{\omega_{f}^{(g)}(t,\xi,\lambda)-\phi_k'(t)}\leq\tilde{\epsilon},
	\]
	provided $\epsilon$ is sufficiently small.
\end{lemma}

\begin{proof}
	By Lemma \ref{lemma:3} and \ref{lemma:5}, we have at $(t,\xi,\lambda)\in Z_k$
	\begin{align*}
		\abs{\omega_{f}^{(g)}-\phi_k'(t)}
		&\leq\abs{\frac{\partial_t T_f^{(g)} - 2\pi i \mu_{f}^{(g)} T_f^{(xg)}}{2\pi i T_f^{(g)}}-\phi_k'(t)}\\
		&= \abs{\frac{\partial_t T_f^{(g)}-2\pi i\phi_k'(t)T_f^{(g)}-2\pi i\phi_k''(t)T_f^{(xg)}}{2\pi i T_f^{(g)}}+ \frac{(\phi_k''(t)-\mu_{f}^{(g)})T_f^{(xg)}}{T_f^{(g)}}}\\
		&\leq \abs{\frac{\partial_t T_f^{(g)}-2\pi i\phi_k'(t)T_f^{(g)}-2\pi i\phi_k''(t)T_f^{(xg)}}{2\pi i T_f^{(g)}}}+\abs{\frac{(\phi_k''(t)-\mu_{f}^{(g)})T_f^{(xg)}}{T_f^{(g)}}}\\
		&\leq \frac{\tilde{\epsilon}^5 B_{k,1}(t)}{2\pi}+\tilde{\epsilon}^2\left(B_{k,1}(t)\abs{\partial_t T_f^{(g)}}+B_{k,2}(t)\abs{T_f^{(g)}}\right)\abs{T_f^{(xg)}}.
	\end{align*}
	By a similar argument as in the proof of Lemma \ref{lemma:5}, $\abs{\omega_{f}^{(g)}(t,\xi,\lambda)-\phi_k'(t)}\leq\tilde{\epsilon}$, if $\epsilon$ is sufficiently small.
\end{proof}


\end{document}